\documentclass[11pt]{amsart}
\usepackage[letterpaper,margin=1in]{geometry}

\usepackage{amssymb}
\usepackage{amsmath}
\usepackage{amsfonts}
\usepackage{amsthm}
\usepackage{stmaryrd}
\usepackage[all]{xy}
\usepackage{mathrsfs}
\usepackage{graphicx}
\usepackage{hyperref}
\usepackage{color}
\usepackage{multirow}
\usepackage{xcolor}
\usepackage{cleveref}

\usepackage{tikz}
\usetikzlibrary{cd}

\usepackage{comment}
\usepackage[shortlabels]{enumitem}
\usepackage[mathscr]{eucal}
\usepackage{scalerel}

\usepackage{stackengine}

\usepackage{mathtools}
\usepackage{bbm}
\DeclareMathAlphabet{\mathpzc}{OT1}{pzc}{m}{it}

\usepackage[normalem]{ulem}
\definecolor{cerisepink}{rgb}{0.93, 0.23, 0.51}
\definecolor{cerulean}{rgb}{0.0, 0.48, 0.65}
\hypersetup{
	colorlinks = true,
	linkcolor = {cerisepink},
	linkbordercolor = {white},
	citebordercolor = {white},
	citecolor = {cerulean},
}

\numberwithin{equation}{subsection}
\newtheorem{theorem}[subsubsection]{Theorem}

\newtheorem{corollary}[subsubsection]{Corollary}
\newtheorem{lemma}[subsubsection]{Lemma}
\newtheorem{proposition}[subsubsection]{Proposition}
\newtheorem{definition}[subsubsection]{Definition}
\newtheorem{claim}[subsubsection]{Claim}

\theoremstyle{definition}

\newtheorem{construction}[subsubsection]{Construction}

\newtheorem{example}[subsubsection]{Example}

\newtheorem{remark}[subsubsection]{Remark}

\newcommand{\ra}{\rightarrow}

\def\CC{\mathbb{C}}

\def\LL{\mathbb{L}}

\def\NN{\mathbb{N}}

\def\QQ{\mathbb{Q}}

\def\TT{\mathbb{T}}

\def\ZZ{\mathbb{Z}}

\def\calA{\mathcal{A}}

\def\calC{\mathcal{C}}
\def\calD{\mathcal{D}}
\def\calE{\mathcal{E}}
\def\calF{\mathcal{F}}
\def\calG{\mathcal{G}}
\def\calH{\mathcal{H}}
\def\calI{\mathcal{I}}
\def\calJ{\mathcal{J}}

\def\calM{\mathcal{M}}

\def\calO{\mathcal{O}}

\def\scrA{\mathscr{A}}

\def\scrC{\mathscr{C}}
\def\scrD{\mathscr{D}}

\def\scrT{\mathscr{T}}
\def\scrU{\mathscr{U}}

\def\scrX{\mathscr{X}}
\def\scrY{\mathscr{Y}}

\def\rmh{\mathrm{h}}
\def\rmN{\mathrm{N}}

\def\rmH{\mathrm{H}}

\def\ul{\underline}
\def\wt{\widetilde}
\def\wh{\widehat}

\DeclareMathOperator{\Cone}{Cone}

\DeclareMathOperator{\Gal}{Gal}

\DeclareMathOperator{\Hom}{Hom}

\DeclareMathOperator{\id}{id}

\DeclareMathOperator{\rank}{rank}

\DeclareMathOperator{\Spec}{Spec}

\DeclareMathOperator*{\motimes}{\text{\raisebox{0.25ex}{\scalebox{0.8}{$\bigotimes$}}}}
\DeclareMathOperator*{\moplus}{\text{\raisebox{0.25ex}{\scalebox{0.8}{$\bigoplus$}}}}

\newcommand{\Q}{\mathbb{Q}}
\newcommand{\Z}{\mathbb{Z}}

\newcommand{\Sh}{\mathrm{Sh}}

\newcommand{\Sym}{\mathrm{Sym}}

\newcommand{\rra}{\longrightarrow}
\newcommand{\lmt}{\longmapsto}

\newcommand{\ol}{\overline}
\newcommand{\cH}{\mathrm{H}}

\newcommand{\frakp}{\mathfrak{p}}

\newcommand{\frakP}{{\mathfrak{P}}}

\newcommand{\Sing}{\textrm{Sing}}
\newcommand{\Fun}{\textrm{Fun}}

\newcommand{\et}{{\mathrm{\acute{e}t}}}
\newcommand{\pe}{{\mathrm{pro\acute{e}t}}}

\newcommand{\Spa}{\mathrm{Spa}}

\newcommand{\Bl}{\mathrm{Bl}}

\newcommand{\eh}{{\mathrm{\acute{e}h}}}

\newcommand{\Ainf}{{\mathrm{A_{inf}}}}

\newcommand{\Bdr}{{\mathrm{B_{dR}^+}}}

\newcommand{\an}{{\mathrm{an}}}
\newcommand{\colim}{{\mathrm{colim}}}

\newcommand{\cosk}{{\mathrm{cosk}}}

\newcommand{\RSe}{{\mathrm{Rig_{\Sigma_e}}}}
\newcommand{\Rei}{{\mathrm{Rig_{\Sigma_e, \Inf}}}}
\newcommand{\Coh}{{\mathrm{Coh}}}
\newcommand{\rig}{{\mathrm{rig}}}
\newcommand{\Rig}{{\mathrm{Rig}}}

\newcommand{\Ch}{{\mathrm{Ch}}}
\newcommand{\Tot}{{\mathrm{Tot}}}
\newcommand{\Afe}{{\mathrm{Alg}_{\mathrm{tfp},e}}}
\newcommand{\Rlim}{{\mathrm{R\varprojlim}}}
\newcommand{\gr}{{\mathrm{gr}}}
\newcommand{\Fil}{{\mathrm{Fil}}}
\newcommand{\dR}{{\wh{\mathrm{dR}}^\an}}
\newcommand{\op}{{\mathrm{op}}}
\newcommand{\Cech}{{\mathrm{\check{C}ech}}}
\newcommand{\ic}{{$\infty$-category }}
\newcommand{\Ainfe}{{\mathrm{A_{inf,e}}}}

\newcommand{\Bdre}{{\mathrm{B^+_{dR,e}}}}

\newcommand{\Map}{{\mathrm{Map}}}

\newcommand{\aff}{{\mathrm{aff}}}

\newcommand{\Inf}{{\mathrm{INF}}}
\newcommand{\VVec}{{\mathrm{Vec}}}
\newcommand{\SE}{{\mathrm{SE}}}
\newcommand{\red}{{\mathrm{red}}}
\newcommand{\BBdr}{{\mathbb{B}_{\mathrm{dR}}}}
\newcommand{\BBdrp}{{\mathbb{B}^+_{\mathrm{dR}}}}
\newcommand{\Bdrr}{{\mathrm{B}_{\mathrm{dR}}}}
\newcommand{\DF}{{\mathscr{DF}}}
\newcommand{\Sch}{{\mathrm{Sch}}}
\newcommand{\Zar}{{\mathrm{Zar}}}

\begin{document}
	\title{Crystalline cohomology of rigid analytic spaces}
	\date{}
	
	\author{Haoyang Guo}

	\begin{abstract}
		In this article, we introduce the infinitesimal cohomology for rigid analytic spaces that are not necessarily smooth, with coefficients in a $p$-adic field or in Fontaine's de Rham period ring $\Bdr$.
	\end{abstract}

	\maketitle
	\tableofcontents

	\section{Introduction}
	\subsection{Background}\label{subsec back}
	
	Let $X$ be a complex algebraic variety.
	Attached to the set of $\CC$-points of $X$, 
	there is a natural analytic structure which makes $X(\CC)$ a complex analytic space.
	This allows us to obtain a topological invariant of $X$ via singular cohomology $\rmH^i_{\Sing}(X(\CC),\CC)$, which is computed transcendentally.
	
	As the topological space $X(\CC)$ comes from an algebraic variety,
	it is natural to ask if we could compute this singular cohomology algebraically.
	When $X$ is a smooth algebraic variety over $\CC$, 
	it is a result of Grothendieck (\cite{Gr66}) that singular cohomology is isomorphic to algebraic de Rham cohomology.
	Namely, there exists a natural isomorphism
	\[
	\rmH^i_{\Sing}(X(\CC),\CC) \simeq \rmH^i(X, \Omega_{X/\CC}^\bullet),
	\]
	where $\Omega_{X/\CC}^i$ is the sheaf of the $i$-th algebraic K\"ahler differentials over the variety $X$, and $\Omega_{X/\CC}^\bullet$ is the algebraic de Rham complex.
	As a consequence, we get a purely algebraic way to compute singular cohomology group.
	
	However, if $X$ is non-smooth, 
	the cohomology of the usual algebraic de Rham complex may fail to compute singular cohomology of $X(\CC)$ (c.f. \cite[Example 4.6]{AK11}).
	To get the correct answer, in particular to get an algebraic cohomology theory which computes singular cohomology, 
	there are several methods modifying algebraic de Rham cohomology in the non-smooth setting:
	\begin{enumerate}[(1)]
		\item In \cite{Har75}, Hartshorne discovered that if $X$ admits a closed immersion into a smooth variety $Y$,
		then the formal completion $\wh{\Omega_{Y/\CC}^\bullet}$ of the de Rham complex $\Omega_{Y/\CC}^\bullet$ along $X\ra Y$ computes singular cohomology.
		Precisely, there exists the following isomorphism
		\[
		\rmH^i_\Sing(X(\CC),\CC) \simeq \rmH^i(X, \wh{\Omega_{Y/\CC}^\bullet}).
		\]
		The result was obtained independently by Deligne (unpublished), and by Herrera--Lieberman \cite{HL71}.
		
		In the general case when $X$ is not necessarily embeddable, 
		there exists a ringed \emph{infinitesimal site $(X/\CC_{\inf}, \calO_{X/\CC})$} (or the crystalline site in characteristic zero) introduced by Grothendieck \cite{Gr68}.
		It can be shown that its cohomology $\rmH^i(X/\CC_{\inf}, \calO_{X/\CC})$ coincides with $\rmH^i(X, \wh{\Omega_{Y/\CC}^\bullet})$ whenever $X\ra Y$ is a closed immersion into a smooth variety as above.
		In particular we obtain a conceptual cohomology theory that is independent of immersions.
		Moreover, the method allows us to compute cohomology with nontrivial coefficients, where we could replace $\calO_{X/\CC}$ by vector bundles with flat connections (or in other words \emph{crystals}).
		
		\item Extending the de Rham complex of smooth $\mathbb{C}$-algebras via simplicial resolutions, one obtains the \emph{(Hodge-completed) derived de Rham complex}, first invented by Illusie \cite{Ill71}.
		To any scheme $X$ over $\CC$, we can associate a filtered derived algebra $\wh{\mathrm{dR}}_{X/\CC}$ to it.
		It was shown by Illusie in loc.\ cit.\ that the cohomology of the derived de Rham complex $\wh{\mathrm{dR}}_{X/\CC}$ is isomorphic to the Hartshorne's cohomology, assuming $X$ is a local complete intersection.
		Later on, using Adams completion from the algebraic topology, 
		Bhatt \cite{Bha12} shows that the comparison is true for any finite type scheme in characteristic zero, without the l.c.i condition.
		In particular, for an arbitrary variety $X/\CC$, we get the isomorphism
		\[
		\rmH^i_\Sing(X(\CC),\CC) \simeq \rmH^i(X, \wh{\mathrm{dR}}_{X/\CC}).
		\]
		Here we mention that the first graded piece of $\wh{\mathrm{dR}_{X/\CC}}$ is the cotangent complex $\LL_{X/\CC}$ up to a shift, 
		which plays an important role in the deformation theory of schemes.
		Moreover, similarly to the universal property of the algebraic de Rham complex, it could be shown that the derived de Rham complex is the initial object among all filtered (derived) algebras $\calA$ over $X$ that is equipped with a homomorphism $\calO_X \ra \gr^0 \calA$ (\cite{Rak20}).
		
		\item  Another modification of the de Rham complex is called the \emph{Deligne--Du Bois complex}, introduced by Deligne and studied by Du Bois (\cite{DB81}).
		The Deligne--Du Bois complex is defined via the cohomological descent for resolution of singularties.
		More precisely, the Deligne--Du Bois complex for $X$ is defined as  the limit of the de Rham complex of $X_n$
		\[
		R\lim_{[n]\in \Delta} Rf_{n *}\Omega_{X_n/\CC}^\bullet,
		\]
		where  $f_\bullet:X_\bullet \ra X$ is a simplicial variety constructed using blowups at smooth nowhere dense centers, such that each $X_n$ is smooth over $\CC$.
		It could be shown that singular cohomology of $X$ is isomorphic to the cohomology of the Deligne--Du Bois complex.
		Moreover, Deligne--Du Bois complex admits a finite \emph{Hodge--Deligne filtration} where each graded piece is a bounded complex of coherent sheaves in the derived category.
		The induced filtration on cohomology is the Hodge filtration for the underlying mixed Hodge structure.
		Furthermore, Deligne--Du Bois complex together with its filtration also admits a site-theoretical interpretation via the h-topology, where the latter is introduced by Voevodsky in \cite{Voe96}.
		The theory of $h$-cohomology of $X$ is studied for example in \cite{HJ14} and \cite{Lee09}.
	\end{enumerate}
	
	The above provides three algebraic methods computing singular cohomology of a complex variety that is not necessarily smooth.
	It is then natural to ask if we have an analogous picture in the non-archimedean geometry.
	The goal of our article is to study the theory of cohomology for non-smooth rigid spaces in non-archimedean geometry, analogous to the three modifications for complex algebraic varieties as above.
	
	To start, let us fix some notations.
	Let $K$ be a $p$-adic extension of $\mathbb{Q}_p$; namely $K$ is a field that is complete with respect to a non-archimedean valuation extending that of $\mathbb{Q}_p$.
	In 1960s, Tate introduced the notion of the \emph{rigid analytic space} that forms a natural $p$-adic analogue of the complex analytic space.
	Here similarly to complex analytic spaces, examples of rigid analytic spaces include analyfications of algebraic varieties over $K$.

	As a $p$-adic field is totally disconnected, singular cohomology of a rigid space over $K$ is not an interesting object.
	However, we could still define a de Rham complex $\Omega_{X/K}^\bullet$ of $X$, where each $\Omega_{X/K}^i$ is the sheaf of K\"ahler differentials  that are continuous under $p$-adic topology.
	When $X$ is smooth and proper over $K$, it can be shown that the cohomology of its de Rham complex behaves very well: it is living within cohomological degrees $[0,2\dim(X)]$, such that each cohomology group is a finite dimensional $K$-vector space.
	In particular, when $X$ is the analytification of a smooth proper algebraic variety over $K$, the $p$-adic analytic de Rham cohomology of $X$ and the algebraic de Rham cohomology of the variety match up, so we get the correct Betti numbers.
	
	In the following, we consider general rigid spaces over $K$ that may not be smooth.
	\subsection{Main results}
	Let $X$ be a rigid space over $K$.
	We introduce the \emph{infinitesimal site $X/K_{\inf}$}, defined on the category of all pairs of rigid spaces $(U,T)$ for nil closed immersions $U\ra T$ over $K$, such that $U$ is an open subset in $X$. 
	Here a collection of maps $\{(U_i, T_i) \ra (U, T)\}$ is a covering in this site if $\{T_i\ra T\}$ is an open covering for the rigid space $T$.
	The infinitesimal site $X/K_{\inf}$ naturally admits an \emph{infinitesimal structure sheaf $\calO_{X/K}$}, sending a thickening $(U,T)$ to the ring of global sections $\calO_T(T)$. 
	Moreover, the infinitesimal structure sheaf admits a surjection onto $\calO_X$, and the kernel ideal $\calJ_{X/K}$ defines a natural filtration on $\calO_{X/K}$.
	The induced filtration on the cohomology of $\calO_{X/K}$ is called the \emph{infinitesimal filtration}.
	
	Now we can formulate our first main result.
	\begin{theorem}\label{main1}
		There is a $K$-linear cohomology theory 
		\[
		X\mapsto R\Gamma_{\inf}(X/K):=R\Gamma(X/K_{\inf},\calO_{X/K})
		\]
		together with the filtration defined by $R\Gamma(X/K_{\inf},\calJ^\ast_{X/K})$, for rigid spaces $X$ over $K$ and is taking values in the filtered complete derived category of $K$-vector spaces.
		It satisfies the following properties:
		\begin{enumerate}[label=\upshape{(\roman*)}]
			
			\item \emph{Explicit formula (\Cref{glo-coh}):} Assume $X\ra Y$ is a closed immersion into a smooth rigid space $Y$, and $J$ is the ideal sheaf defining $X$.
			Then $R\Gamma_{\inf}(X/K)$ is filtered isomorphic to the cohomology of the formal completion of the de Rham complex $\Omega_{Y/K}^\bullet$ along $X\ra Y$:
			\[
			R\Gamma_{\inf}(X/K) \rra  R\Gamma(X, \wh{\Omega_{Y/K}^\bullet }),
			\]
			where the $j$-th filtration on the right side is $R\Gamma(X, J^{j-\bullet}\wh{\Omega_{Y/K}^{\bullet} })$.
			
			In particular, when $X$ itself is smooth over $k$, the infinitesimal cohomology coincides with the de Rham cohomology with its Hodge filtration.
			\item \emph{Derived de Rham comparison (\Cref{inf-ddR}):} There exists a natural filtered morphism from cohomology of the analytic derived de Rham complex to $R\Gamma_{\inf}(X/K)$:
			\[
			R\Gamma(X, \dR_{X/K}) \rra R\Gamma_{\inf}(X/K).
			\]
			The map induces an isomorphism on their underlying complexes, and is a filtered isomorphism if $X$ is a local complete intersection.
			
			\item \emph{\'Eh comparison (\Cref{inf-eh}):} The cohomology $R\Gamma_{\inf}(X/K)$ admits a natural filtered morphism to \'eh de Rham cohomology introduced in \cite{Guo19}, inducing an isomorphism on their underlying complexes:
			\[
			R\Gamma_{\inf}(X/K) \rra R\Gamma(X_\eh,\Omega_\eh^\bullet).
			\]
			In particular, the underlying complex of $R\Gamma_{\inf}(X/K)$ satisfies the \'eh-hyperdescent for rigid spaces $X$ over $K$.
			
			\item \emph{Finiteness (\Cref{fin Bdre}):} When $X$ is proper of dimension $n$ over $K$, the underlying complex of $R\Gamma_{\inf}(X/K)$ is a perfect $K$-complex that lives in cohomological degree $[0,2\dim(X)]$.
			
			\item \emph{Base extension (\Cref{base ext}):} Assume $K_0$ is a subfield of $K$ and $X$ is a proper rigid space over $K_0$.
			Then the natural base extension map is a filtered isomorphism
			\[
			R\Gamma_{\inf}(X/K_0)\otimes_{K_0} K \rra R\Gamma_{\inf}(X_K/K).
			\]
			
			\item \emph{Comparison with singular cohomology (\Cref{ana-sing}):} Assume there exists  an abstract isomorphism of fields $K\ra \CC$, and $X$ is the analytification of a proper algebraic variety $\scrX$ over $K$.
			Then we have a filtered isomorphism
			\[
			\rmH^i_{\inf}(X/K)\otimes_K \CC \simeq \rmH^i_\Sing(\scrX(\CC),\CC),
			\] where the latter is filtered by the algebraic infinitesimal filtration (cf. \cite{Bha12} and \cite{Har75}).
		\end{enumerate}
	\end{theorem}
	Here the \emph{underlying complex} of a filtered object is defined as the complex forgetting its filtration.

	\begin{remark}[Analytic derived de Rham complex]
		In Theorem \ref{main1}.(ii), the usual notion of the derived de Rham complex of Illusie is not suitable when we are working with rigid analytic spaces;
		instead we modify the construction so that it is continuous under the $p$-adic topology.
		Our strategy is to first apply the (derived) $p$-adic completion to the algebraic derived de Rham complex for the rings of definitions over $\mathcal{O}_K$, then consider the filtered completion of its generic fiber.
		This produces a filtered $\mathbb{E}_\infty$-algebra $\dR_{X/K}$ in the derived ($\infty$) category of sheaves of $K$-modules over $X$, 
		whose graded pieces are wedge products of the analytic cotangent complex $\LL_{X/K}^\an$ introduced by Gabber--Ramero in \cite[Section 7]{GR}.
		We refer readers to Subsection \ref{subsec aff ddR} for details.
		
		This notion has been considered in \cite{GL20}, where Shizhang Li and the author show that applying at affinoid perfectoid algebras, the analytic derived de Rham complex can recover the de Rham period sheaves $\BBdrp$ and $\calO\BBdrp$ over the pro-\'etale site, previously introduced in \cite{Bri08} and \cite{Sch13}.
		We also want to mention that a version of $\dR_{X/K}$ for derived analytic stacks $X$ has been considered independently by Jorge Ant\'onio in \cite{Ant20}.
	\end{remark}
	
	\begin{remark}[\'Eh topology]
		The \'eh cohomology in Theorem \ref{main1}.(iii) is the cohomology of the \'eh-sheafification of the de Rham complex over the \'eh site, where the latter is analogous to Voevodsky's h-topology for algebraic schemes in the method (3) of last subsection.
		It is designed as the minimal refinement of \'etale topology that is locally smooth (thanks to resolution of singularities of rigid spaces in for example \cite{BM08} and \cite{Tem12}), and is defined by adding  universal homeomorphisms and coverings associated with blowups.
		It can be shown that when $X$ is the analytification of a proper algebraic variety, \'eh de Rham cohomology in Theorem \ref{main1}.(iii) is filtered isomorphic to the cohomology of (analytification) of the Deligne--Du Bois complex (\cite[Corollary 5.2.2]{Guo19}).
		So we recover the algebraic Deligne--Du Bois complex in the $p$-adic analytic setting.
	\end{remark}
	
	\begin{remark}[\'Eh de Rham complex]
		Let $\pi:X_\eh \ra X_\rig$ be the natural map from the \'eh site to the rigid site introduced in \cite{Guo19}.
		The \'eh de Rham complex in Theorem \ref{main1}.(iii) is obtained by applying the derived global section to the complex 
		\[
		R\pi_*\Omega_\eh^\bullet.
		\]
		Here each $\Omega_{\eh}^i$ is the \'eh-sheafification of the continuous sheaf of differentials for rigid spaces, 
		and it is shown in \cite[Section 6]{Guo19} that each $R\pi_*\Omega_\eh^i$ is a complex of coherent $\calO_X$-modules that lives in cohomological degree $[0,\dim(X)]$, and vanishes for $i>\dim(X)$.
		As a consequence, analogous to the classical Hodge--de Rham filtration on de Rham cohomology of complex varieties, (the underlying complex of) the infinitesimal cohomology $R\Gamma_{\inf}(X/K)$ can be computed via cohomology of (finite amount of) coherent sheaves and Theorem \ref{main1}.(iii).
	\end{remark}
	
	\begin{remark}[Various filtrations]
		\label{rmk direct summand}
		Infinitesimal cohomology $R\Gamma_{\inf}(X/K)$ admits a natural filtered map to the usual de Rham cohomology of $X$, and the latter maps to \'eh de Rham cohomology.
		So summarizing the various filtrations in Theorem \ref{main1}, we get the following sequence of maps in filtered derived category:
		\[
		R\Gamma(X, \dR_{X/K}) \xrightarrow{(1)} R\Gamma_{\inf}(X/K) \xrightarrow{(2)} R\Gamma(X,\Omega_{X/K}^\bullet) \xrightarrow{(3)} R\Gamma(X_\eh ,\Omega_\eh^\bullet),
		\]
		enhancing the canonical maps on the underlying complexes.
		In particular, thanks to \Cref{main1}.(ii, iii), both the map (1) and the composition of (2) and (3) induce isomorphisms on the underlying complexes.
		As a consequence, (the underlying complex of) the infinitesimal cohomology forms a direct summand of the usual continuous de Rham cohomology.
	\end{remark}
	
	\begin{remark}[Crystals]
		Similarly to the schematic crystalline theory, we also have a theory of crystals over the infinitesimal site.
		Moreover, it could be shown that analogous statements in Theorem \ref{main1}.(i), (iii),  (iv) and (v) hold true for cohomology of crystals.
		It is expected that Theorem \ref{main1}.(ii) for derived de Rham complex also admits a generalization with coefficients, via an approximate extension of classical crystals to the simplicial world.
		We leave this question to a future investigation.
	\end{remark}

	Now let $K$ be a complete and algebraically closed $p$-adic field.
	Let $\Bdr$ be Fontaine's de Rham period ring of $K$ in $p$-adic Hodge theory (\cite{Fon94}), and let $\xi$ be a fixed generator of the kernel for the canonical surjection $\Bdr\ra K$ (see Section \ref{sec inf} for definition).
	Recall for a smooth rigid space $X$ over $K$, Bhatt--Morrow--Scholze introduced the \emph{crystalline cohomology of $X$ over $\Bdr$} (\cite[Section 13]{BMS}).
	It is defined locally via the inverse limit 
	\[
	\varprojlim_{e\in \NN} \wh{\Omega_{Y_e/\Sigma_e}^\bullet}.
	\]
	Here $\{Y_e\}$ is a compatible family of smooth adic spaces over $\Sigma_e:=\Spa(\Bdr/\xi^e)$ that admit a closed immersion from $X$, 
	and $\wh{\Omega_{Y_e/\Sigma_e}^\bullet}$ is the formal completion of the continuous de Rham complex $\Omega_{Y_e/\Sigma_e}^\bullet$ along $X\ra Y_e$.
	It is shown in loc.$~$cit.$~$that the cohomology is independent of choices of closed immersions.
	Moreover,  after inverting $\xi$, 
	there exists a natural isomorphism between the crystalline cohomology of $X$ in \cite{BMS}, 
	and pro-\'etale cohomology of the de Rham period sheaves $\BBdr$, for quasi-compact quasi-separated rigid spaces.
	
	As the construction is analogous to the computation of the crystalline cohomology of schemes (see Subsection \ref{subsec back}), 
	in loc.$~$cit.$~$Bhatt--Morrow--Scholze expect that there is a conceptual crystalline theory for rigid spaces, defined similarly to the infinitesimal cohomology in the schematic theory, whose cohomology is isomorphic to the crystalline cohomology in \cite{BMS} (see \cite[Remark 13.2]{BMS}).
	Our next goal is to answer this question.
	In fact, our project was initiated after the author read the question from \cite{BMS}.
	
	Let us consider the infinitesimal site $X/\Sigma_{\inf}$, which is defined on the category of nil closed immersions $(U,T)$ such that $U$ is open in $X$, $U\ra T$ is a nil closed immersion, and $T$ is a locally topological finite type adic space over $\Bdr/\xi^e$ for some $e\in \NN$.
	The covering structure of $X/\Sigma_{\inf}$ is defined by the open coverings of adic spaces as the one in $X/K_{\inf}$, 
	and we can equip $X/\Sigma_{\inf}$ with a natural infinitesimal structure sheaf $\calO_{X/\Sigma}$ and analogously an infinitesimal ideal sheaf $\mathcal{J}_{X/\Sigma}$.
	
	Now we can state the next main result.
	\begin{theorem}\label{main2}
		There is a $\Bdr$-linear cohomology theory 
		\[
		X\mapsto R\Gamma_{\inf}(X/\Sigma):=R\Gamma(X/\Sigma_{\inf}, \calO_{X/\Sigma}),
		\] together with the filtration defined by $R\Gamma(X/\Sigma_{\inf}, \calJ^\ast_{X/\Sigma})$, for rigid spaces $X$ over $K$, and is taking values in the filtered complete derived category over $\Bdr$ (where the latter is equipped with $\xi$-adic filtration).
		It satisfies the following properties:
		\begin{enumerate}[label=\upshape{(\roman*)}]
			\item \emph{Reduction to $K$ (\Cref{inf-lev coh}.(iii)):} There exists a natural filtered base change isomorphism 
			\[
			R\Gamma_{\inf}(X/\Sigma)\otimes^L_{\Bdr} K \rra R\Gamma_{\inf}(X/K),
			\]
			where $R\Gamma_{\inf}(X/K)$ is the infinitesimal cohomology with its infinitesimal filtration in Theorem \ref{main1}.

			\item \emph{Explicit formula (\Cref{inf-lev coh}.(ii)):} Assume $\{X\ra Y_e\}_e$ is a system of closed immersions from $X$ to smooth adic spaces $Y_e$ over $\Sigma_e$, such that they are compatible via isomorphisms $Y_{e+1}\times_{\Sigma_{e+1}} \Sigma_e\simeq Y_e$.
			Then there is a natural filtered isomorphism 
			\[
			R\Gamma_{\inf}(X/\Sigma) \rra R\Gamma(X, \varprojlim_{e\in \NN} \wh{\Omega_{Y_e/\Sigma_e}^\bullet}),
			\]
			where $\wh{\Omega_{Y_e/\Sigma_e}^\bullet}$ is the formal completion of the de Rham complex $\Omega_{Y_e/\Sigma_e}^\bullet$ along $X\ra Y_e$.
			
			\item\emph{Derived de Rham comparison (\Cref{inf-lev ddR}):} There exists a natural filtered morphism inducing an isomorphism on underlying complexes 
			\[
			R\Gamma_{\inf}(X/\Sigma) \rra  R\Gamma(X, \dR_{X/\Sigma}),
			\]
			where $\dR_{X/\Sigma}$ is defined as the derived inverse limit of the underlying complexes of $\dR_{X/\Sigma_e}$ over $e\in \NN$.
			
			\item\emph{\'Eh hyperdescent (\Cref{inf-lev des 2}):} The underlying complex of $R\Gamma_{\inf}(X/\Sigma)$ satisfies the \'eh-hyperdescent for rigid spaces $X$ over $K$.
			
			\item \emph{Pro-\'etale comparison (\Cref{inf-proet}):} There exists a natural $\Bdr$-linear map from the underlying complex of $R\Gamma_{\inf}(X/\Sigma)$ to the pro-\'etale cohomology 
			\[
			R\Gamma_{\inf}(X/\Sigma) \rra R\Gamma(X_\pe, \BBdrp).
			\]
			For quasi-compact quasi-separated rigid spaces $X$, the above induces an isomorphism after inverting $\xi$:
			\[
			R\Gamma_{\inf}(X/\Sigma)[\frac{1}{\xi}] \simeq R\Gamma(X_\pe, \BBdr).
			\]
			The map is $\Gal(K/K_0)$-equivariant when $X$ is isomorphic to $X_0\times_{K_0} K$ for a rigid space $X_0$ over $K_0$ and a subfield $K_0$ of $K$.
			
			\item \emph{Finiteness (\Cref{fin Bdr}) and Torsion-freeness (\Cref{torsionfree and HdR}):} When $X$ is proper of dimension $n$ over $K$, each cohomology group $\rmH^i_{\inf}(X/\Sigma)$ is a finite free $\Bdr$-module, and it vanishes for $i\notin [0,2\dim(X)]$.
		\end{enumerate}
	\end{theorem}
	\begin{remark}
		The comparison with the derived de Rham complex over $\Bdr$ in Theorem \ref{main2}.(iii) is compatible with the one in Theorem \ref{main1}.(ii) over $K$, under the natural base change map in Theorem \ref{main2}.(i).
	\end{remark}
	\begin{remark}[Torsion-freeness]
		\Cref{main2}.(vi) shows that each cohomology group $\mathrm{H}^i_{\inf}(X/\Sigma)$ is finite free over $\Bdr$ for $X$ proper over $K$.
		This could be surprising at the first sight, however in the special case when $X$ is defined over $K_0$ as in \Cref{main3}, the result follows easily from the base change formula as in \Cref{main3}.
	\end{remark}
	\begin{remark}[Degeneracy of Hodge--\'eh de Rham]
		As a byproduct of proving \Cref{main2}.(vi), we also show  (in \Cref{torsionfree and HdR}) that Hodge--\'eh de Rham spectral sequence for $X/K$ splits, where $X$ is a proper rigid space.
		This strengthens the result of \cite[Proposition 8.0.8]{Guo19} where we assumed $X$ to be defined over $K_0$.
	\end{remark}
	A consequence of the explicit computation in Theorem \ref{main2}.(ii) is the following.
	\begin{corollary}
		The infinitesimal cohomology $R\Gamma_{\inf}(X/\Sigma)$ is isomorphic to the crystalline cohomology of $X$ over $\Bdr$ in the sense of \cite[Section 13]{BMS}.
	\end{corollary}
	\begin{remark}
		A variant of the infinitesimal site for smooth rigid spaces has been considered by Zijian Yao in \cite[Section 5]{Yao19}, where the crystalline cohomology of \cite[Section 13]{BMS} is reconstructed conceptually, and is compared with pro-\'etale cohomology of the de Rham period sheaf.
		Using the \v{C}ech-Alexander complex, it can be shown that our $R\Gamma_{\inf}(X/\Bdr)$ for smooth rigid spaces coincides with the crystalline cohomology of \cite{Yao19}. 
	\end{remark}
	
	At last, we comment on the case when $X$ is defined over a discretely valued subfield.
	By the Primitive Comparison Theorem of Scholze \cite{Sch13b} and the \'eh-hyperdescent in Theorem \ref{main2}.(iv), 
	we can compare our infinitesimal cohomology with \'etale cohomology.
	\begin{corollary}\label{main3}
		Let $K_0$ be a discretely valued subfield of $K$ that has a perfect residue field, and let $X$ be a proper rigid space over $K_0$.
		\begin{enumerate}[label=\upshape{(\roman*)}]
			\item The cohomology $R\Gamma_{\inf}(X_K/\Sigma)$ can be defined over $K_0$.
			Namely, there exists a natural base extension formula for underlying complexes
			\footnote{The curious reader might ask how the filtrations relate to each other under this isomorphism. In fact, using the recent advance of condensed mathematics of Clausen--Scholze, one can extend the tensor product formula to a filtered enhancement, as in the sequel \cite[Cor. 8.2.5, Rmk.\ 8.2.6]{Guo21}.}
			\[
			R\Gamma_{\inf}(X_{K}/\Sigma) \simeq R\Gamma_{\inf}(X/K_0)\otimes_{K_0} \Bdr. 
			\]
			\item For every $n\in \NN$, there exists a natural $\Gal(K/K_0)$-equivariant isomorphism 
			\[
			\rmH^n_{\inf}(X_{K}/\Sigma)[\frac{1}{\xi}] \simeq  \rmH^n_\et(X_{K},\QQ_p)\otimes_{\QQ_p} \mathrm{B}_{\mathrm{dR}}.
			\]
		\end{enumerate}
	\end{corollary}
	\begin{proof}
		For (i), there exists a natural map from the right to the left, via a natural functor of infinitesimal sites defined by the base extension, and both sides are complete with respect to $\xi$-adic topology.
		When $X$ is smooth proper over $K_0$, the infinitesimal cohomology complex can be computed using de Rham cohomology by Theorem \ref{main1}.(i), so after a base change along $\Bdr\ra K$, part (i) can be reduced to the following known base change formula for the continuous de Rham cohomology
		\[
		R\Gamma_{\mathrm{dR}}(X/K_0)\otimes_{K_0} K \simeq R\Gamma_{\mathrm{dR}}(X_{K}/K).
		\]
		In general, part (i) follows from the \'eh-hyperdescent among rigid spaces over $K_0$ in Theorem \ref{main1}.(iii) and Theorem \ref{main2}.(iv).
		Part (ii) follows from the Primative Comparison Theorem (\cite[Thm.\ 3.17]{Sch13b}) and the \'eh-pro\'etale comparison in \cite[Theorem 1.1.4]{Guo19}.
	\end{proof}
	\begin{remark}[Hodge--Tate filtration]
		The natural filtration of $\mathrm{B}_{\mathrm{dR}}$ induces a $\Gal(K/K_0)$-equivariant filtration on \'etale cohomology $\rmH^n_\et(X_{K},\QQ_p)\otimes_{\QQ_p} \mathrm{B}_{\mathrm{dR}}$. 
		In fact, it is shown in \cite[Theorem 1.1.4]{Guo19} that this filtration is isomorphic (via Corollary \ref{main3}) to the product filtration for the filtration on $\mathrm{B}_{\mathrm{dR}}$ and the \'eh--Hodge filtration on the infinitesimal cohomology $R\Gamma_{\inf}(X/K_0)$, defined by extending the Hodge filtration of the de Rham cohomology of smooth rigid spaces via \'eh hyperdescent.
		In particular, the graded pieces of the filtration on \'etale cohomology can be understood via the Hodge--Tate decomposition and \'eh cohomology (\cite[Theorem 1.1.4]{Guo19}):
		\[
		\cH^n_\et(X_{K},\Q_p)\otimes_{\Q_p} K=\bigoplus_{i+j=n}\cH^i(X_\eh,\Omega_{ \eh, K_0}^j)\otimes_{K_0} K(-j).
		\]
	\end{remark}

	\subsection{Summary of sections}
	Now we give a summary of each section in this article.
	
	We start in Section \ref{sec inf} by introducing the (small) infinitesimal site $X/\Sigma_{e \inf}$ and its big site version $X/\Sigma_{e \Inf}$.
	Here the base space $\Sigma_e$ is defined as the adic space $\Spa(\Bdre)$, for the $p$-adic Huber ring $\Bdre:=\Bdr/\xi^e$ and $e\in \NN_{>0}$.
	The discussion in this section is analogous to the discussion of the crystalline site and its topos for a pair of schemes,
	and we verify various formal properties, including the relation with the rigid topos, comparison between the big and the small topoi, and the functoriality of the infinitesimal topoi.
	We also introduce the envelope for an immersion of rigid spaces $X\ra Y$, regarded as either a colimit of representable sheaves in the infinitesimal topos, 
	or a locally ringed space defined over the underlying topology of $X$.
	Here we mention that it is slightly different from the crystalline theory of a scheme over a $p$-nilpotent basis, that
	an envelope in the infinitesimal site is almost never representable (see discussions in Subsection \ref{subsec env}).
	
	In Section \ref{sec cry}, we study the coherent crystals over the infinitesimal site $X/\Sigma_{e \inf}$.
	As the base $\Bdre$ has nilpotent elements, a coherent crystal $\calF$ over $X/\Sigma_{e \inf}$ may not always be locally free.
	We show that $\calF$ is a crystal in vector bundles if and only if it is flat over $\Bdre$ (see Definition \ref{cry flat} and Theorem \ref{cry, vec-bun}) for details).
	Moreover we prove that the category of coherent crystals is equivalent to the category of coherent sheaves with integrable connections over an envelope (Theorem \ref{cry}).
	
	Section \ref{sec coh} is devoted to a sheafified version of Theorem \ref{main1}.(i) for general coherent crystals over $X/\Sigma_{e \inf}$ (cf. Theorem \ref{glo-coh}).
	Here we adapt the idea from \cite{BdJ} for the computation: first we relate the \v{C}ech--Alexander complex and the de Rham complex of the envelope via a bicomplex, and then we show that the associated total complex converges to both of them.
	In particular, we improve the loc.$~$cit.$~$to a filtered isomorphism for the infinitesimal filtration, via a finer and concrete study on graded pieces of the completed de Rham complex.
	We also obtain a base change formula for the cohomology sheaf over $\Bdre$ for different $e\in \NN$ (cf. Proposition \ref{coh, change of bases}).
	
	In Section \ref{sec ddR}, we develop the foundations of the analytic derived de Rham complex, for a map of locally topological finite type adic spaces over $\Sigma_e$.
	We first recall the basics of the analytic cotangent complex introduced by Gabber--Ramero in \cite{GR}, of which we make a slight generalization from $K$-affinoid algebras to $\Bdre$-affinoid algebras.
	Then we introduce the analytic derived de Rham complex for affinoid algebras and show various properties of it.
	We use the technique of hypersheaves and hyperdescent to extend the affinoid construction to the global setting,  which we explain in Subsection \ref{sub glo ddr}.
	At last, we compare the cohomology of the analytic derived de Rham complex to the infinitesimal cohomology in Subsection \ref{sub inf-ddr}, and thus prove Theorem \ref{main1}.(ii).
	Here we mention that we will use mildly the language of $\infty$-category in this section, which is mainly to incorporate the use of hypersheaves via a procedure of unfolding.
	
	After that, we prove the \'eh-hyperdescent for the infinitesimal cohomology in Section \ref{sec eh}.
	Here we first show the descent along a blowup square of rigid spaces over $K$ in Theorem \ref{des blowup}, following the strategy in \cite[Chapter II]{Har75}.
	The hyperdescent along an \'eh-hypercovering for the cohomology of a general crystals is then shown in Theorem \ref{inf-eh, Bdre}, 
	where we use the base change formula to reduce to the $K$-linear case in Theorem \ref{inf-eh}.
	This in particular implies Theorem \ref{main1}.(iii).
	The rest of the section is devoted to the finiteness of the infinitesimal cohomology, the comparison with the algebraic cohomology, and a base field extension result; 
	namely item (iv), (v) and (vi) in Theorem \ref{main1}.
	We want to mention that as an object $X'\in X_\eh$ is not necessarily an open subset of $X$, in order to make sense of \'eh descent, we need to use crystals over the big infinitesimal site $X/K_{\Inf}$ in this section.
	However, we will not lose anything, for crystals and their cohomology are independent of working over big or small sites, thanks to Proposition \ref{cry, big and small} and Corollary \ref{coh, big small}.

	At the end of the article in Section \ref{sec Bdr}, we consider the infinitesimal cohomology of a rigid space with coefficients in the de Rham period ring $\Bdr$.
	Though the topology on the ring $\Bdr$ is not $p$-adic, we could still define the infinitesimal site of $X$ over $\Bdr$, denoted as $X/\Sigma_{\inf}$, by taking the union of all $X/\Sigma_{e \inf}$.
	We show that the cohomology of a crystal over $X/\Sigma_{\inf}$ is in fact the limit of the cohomology of its restrictions onto $X/\Sigma_{e \inf}$ in Theorem \ref{inf-lev coh}.
	In this way, we could apply results of previous sections to study the cohomology of $X/\Sigma_{\inf}$, and thus prove the Theorem \ref{main2} in the last two subsections.
	
	As a convention, we will use the language of adic spaces throughout the article.
	We refer the reader to Huber's book \cite{Hu96} for basics of the theory.

	\subsection{Acknowledgements}
	The starting point of this project is a remark in \cite{BMS}, and I thank Bhatt--Morrow--Scholze for their beautiful theory of integral $p$-adic Hodge theory, and for proposing this question. 
	I also thank Berthelot--Ogus, from whose textbook \cite{BO78} I first learned the theory of the crystalline cohomology.
	The intellectual debt of this article that owes to \cite{BdJ} and \cite{Har75} is very obvious, and I thank Bhatt--de Jong and Hartshorne for teaching us various techniques adapted in the article.
	I am in particular extremely grateful to my advisor Bhargav Bhatt, for encouraging me to pursue these topics, and for continuous support and uncountable discussions regarding the project.
	Thanks also go to Shizhang Li for various helpful conversations during the preparation of this project.
	Finally, I express my gratitude to the anonymous referees for their very careful reading and for many comments and suggestions, which tremendously help improve the article.  
	
	The project is part of the author's thesis. 
	It is partially funded by a Graduate Student Summer Fellowship from the Department of Mathematics, University of Michigan, and by NSF grant DMS 1801689 through Bhargav Bhatt.
	During the final revision process, the author is supported by the University of Chicago.

	\section{Infinitesimal geometry over $\Bdre$}\label{sec inf}
	In this section, we introduce the basics around the infinitesimal geometry over the de Rham period ring $\Bdre:=\Bdr/\xi^e$ and over a $p$-adic extension of $\mathbb{Q}_p$.

	\subsection{de Rham period ring and infinitesimal sites}
	We first introduce the big and the small infinitesimal sites of a rigid space over them, and study two natural maps between their topoi.
	
	\noindent
	\textbf{de Rham period rings.}
	As a setup, we recall the basics of the de Rham period ring.
	A more detailed introduction of the de Rham period ring could be found in \cite{Fon94}.
	
	Let $K$ be a $p$-adic valuation extension of $\Q_p$ that is complete and algebraically closed.
	Denote by $\calO_K$  the ring of integers of $K$.
	Then we can define the $p$-adic ring $\Ainf(\calO_K)$ as
	\[
	\Ainf:=W(\varprojlim_{x\mapsto x^p}\calO_K).
	\]
	There exists a canonical continuous surjection $\theta:\Ainf\ra \calO_K$, where the kernel $\ker(\theta)$ is principal.
	Fix a compatible system of $p^n$-th root of unity $\{\zeta_{p^n}\}_n$ in $K$.
	Then the element $\xi:=\frac{[\epsilon]-1}{[\epsilon]^{\frac{1}{p}}-1}$ generates the ideal $\ker(\theta)$, where $[\epsilon]$ is the Teichm\"uller lift of the element $(\zeta_1,\zeta_p,\ldots)$ in $\Ainf$.
	
	The \emph{de Rham period ring} $\Bdr$ is defined as the $\xi$-adic completion of the ring $\Ainf[\frac{1}{p}]$.
	By abuse of the notation, we write $\theta:\Bdr\ra K$ as the canonical continuous surjection induced from $\Ainf\ra \calO_K$.
	Note that for each $n\in \NN$, we have
	\[
	\Bdr/\xi^e=\Ainf[\frac{1}{p}]/\xi^e,
	\]
	which is a $p$-adic Tate ring with a canonical ring of definition $\Ainf/\xi^e$ in it.
	So we can form a Huber pair $(\Bdr/\xi^e,(\Bdr/\xi^e)^\circ)$ over $(\Q_p,\Z_p)$ for $n\in \NN$.
	The adic space $\Sigma_e:=\Spa(\Bdr/\xi^e,(\Bdr/\xi^e)^\circ)$ is a nilpotent extension of $\Spa(K,\calO_K)$.
	
	In the rest of the article, we often use $\Ainfe$ and $\Bdre$ to denote quotient rings $\Ainf/\xi^e$ and $\Bdr/\xi^e$ respectively, in order to simplify the notations.
	
	\noindent
	\textbf{Infinitesimal topology.} We now introduce the infinitesimal site for rigid spaces.
	\begin{definition}
		Let $e$ be a positive integer.
		A \emph{rigid space over $\Bdre$} is defined as an adic space of topological finite presentation over $\Sigma_e$.
		Namely $X$ can be covered by affinoid open subspaces which are of the form 
		\[
		\Spa(\Bdre\langle t_1,\ldots,t_n\rangle/I),
		\]
		where $I$ is a (finitely generated) ideal in $\Bdre\langle t_1,\ldots,t_n\rangle.$
		
		The category of rigid spaces over $\Sigma_e$ is denoted by $\RSe$.
	\end{definition}
	Recall that for a map of rigid spaces $f:U\ra T$, it is called a \emph{nil closed immersion} if $f$ is a closed immersion (defined by the vanishing of a coherent ideal $\calI$ in $\calO_T$), such that $T$ admits an open covering $\{T_i,~i\}$ with $\calI|_{T_i}$ being nilpotent.
	The closed immersion is called \emph{nilpotent} if there is an integer $n\in \mathbb{N}$, such that $I^n=0$.
	Note that a nilpotent closed immersion is always a nil closed immersion.
	The converse is true locally or assuming the  quasi-compactness of the target space.

	\begin{definition}
		\label{def inf site}
		\begin{enumerate}[(a)]
			\item
			Let $X$ be a rigid space over $\Sigma_e$.
			The \emph{(small) infinitesimal site} $X/\Sigma_{e \inf}$ is the site defined as follows:
			\begin{itemize}
				\item The underlying category of $X/\Sigma_{e \inf}$ is the collection of pairs $(U,T)$, called the \emph{infinitesimal thickening}, where $T$ is a rigid space over $\Sigma_e$, $U$ is an open subspace of $X$ and a closed analytic subspace of $T$, such that $U\ra T$ is a nil closed immersion.
				
				Here morphisms between $(U_1,T_1)$ and $(U_2,T_2)$ are defined as maps of pairs over $\Sigma_e$ such that $U_1\ra U_2$ is an open immersion inside $X$.
				\item A collection of morphism $(U_i,T_i)\ra (U,T)$ in $X/\Sigma_{e \inf}$ is a covering if both $\{T_i\ra T,~i\}$ and $\{U_i\ra U,~i\}$ are open coverings for the rigid spaces $T$ and $U$ respectively.
			\end{itemize}	
			\item The \emph{big infinitesimal site} $\Rei$ over $\Sigma_e$ is defined on the category of all of the pairs $(U,T)$ for $U\ra T$ being a nil closed immersion of rigid spaces over $\Sigma_e$, with the same covering structure as above.
			\item The \emph{big infinitesimal site  $X/\Sigma_{e \Inf}$ of $X$} is defined as the localization $\Rei|_X$ of the big site $\Rei$ at $X$. 
			Namely it is defined on  the category of all of the tuples $\{(U,T),f:U\ra X\}$, where $(U,T)$ is an object in $\Rei$, and $f:U\ra X$ is a map of rigid spaces over $\Sigma_e$.
			The covering structure is induced from that of $\Rei$.
		\end{enumerate}
	\end{definition}
	By the definition of the infinitesimal site above, a sheaf $\calF$ over the infinitesimal site is equivalent to the data of $\calF_T$ and $\varphi_g$ as below:
	\begin{itemize}
		\item a sheaf $\calF_T$ over the rigid space $T$, for each infinitesimal thickening $(U,T)\in X/\Sigma_{e \inf}$;
		\item a map of sheaves over $T_1$:
		\[
		\varphi_g:g^{-1}\calF_{T_2} \rra \calF_{T_1},
		\] for a given morphism of infinitesimal thickenings $(i,g):(U_1,T_1)\ra (U_2,T_2)$ in $X/\Sigma_{e \inf}$;
	\end{itemize}
	together with the natural cocycle condition, namely for given morphisms $(U_1,T_1)\xrightarrow{(i,g)}(U_2,T_2)\xrightarrow{(j,h)}(U_3,T_3)$ of infinitesimal thickenings, we have equalities of maps
	\[
	\varphi_{h\circ g} = g^{-1}\varphi_h\circ \varphi_g.
	\]
	The same holds for a sheaf over the big infinitesimal site.
	We call the category of sheaves on $X/\Sigma_{e \inf}$ (or $X/\Sigma_{e \Inf}$) the \emph{infinitesimal topos}, and denote it by $\Sh(X/\Sigma_{e \inf})$ (or $\Sh(X/\Sigma_{e \Inf})$).
	
	For two sheaves $\calF$ and $\calG$ over the infinitesimal site, we sometimes use the notation $\calF(\calG)$ to denote the set of homomorphisms
	\[
	\Hom(\calG,\calF).
	\]
	In the case when $\calG$ is a representable sheaf $h_T$ for an infinitesimal thickening $(U,T)$, the above hom set is the set of sections
	\[
	\Hom(h_T,\calF)=\calF(U,T).
	\]
	
	There is a natural  \emph{structure sheaf $\calO_{X/\Sigma_e}$} over the big or small infinitesimal site, which is defined as 
	\[
	\calO_{X/\Sigma_{e}}(U,T):=\calO_T(T),~(U,T)\in X/\Sigma_{e \inf}.
	\]
	Here we note that by the equivalent description right below \Cref{def inf site}, the above formula is naturally a sheaf.
	On the other hand, one can define the \emph{analytic structure sheaf $\mathcal{O}_X$} on the infinitesimal sites, via the formula
	\[
	\mathcal{O}_X(U,T):= \mathcal{O}_U(U), ~(U,T)\in X/\Sigma_{e \inf}.
	\]
	By construction, there is a natural surjection of sheaves $\mathcal{O}_{X/\Sigma_e} \to \mathcal{O}_X$, and we define the \emph{infinitesimal ideal sheaf $\mathcal{J}_{X/\Sigma_e}$} to be the sheaf of kernel ideals for the surjection.
	\begin{remark}\label{base K0}
		It is clear from the above definition that the infinitesimal site can be defined for any pair of analytic adic spaces $X\ra Z$, not just $X\ra \Sigma_e$.
		In particular, when $Z=\Spa(K_0)$ is a discretely valued field, and $X$ is a rigid space over $K_0$, we get the analogous version of the infinitesimal site of $X$ over $K_0$.
		Moreover, there exists a natural map of sites $X_K/K_{\inf} \ra X/K_{0, {\inf}}$, defined by the base field extension.
	\end{remark}

	Here are some basic properties of the infinitesimal sites.
	\begin{lemma}\label{site finite limit}
		Let $X$ be a rigid space over $\Sigma_e$.
		Then we have
		\begin{enumerate}[(i)]
			\item The fiber product exists in the big and the small infinitesimal site of $X$ over $\Sigma_e$, and is compatible with the inclusion functor between the big and the small sites.
			\item  The equalizer exists in the big and the small infinitesimal site of $X$ over $\Sigma_e$, and is compatible with the inclusion functor as in (i).
			\item The nonempty finite product is ind-representable in the big and the small infinitesimal site of $X$ over $\Sigma_e$, and is compatible with the inclusion functor as in (i).
		\end{enumerate}
	\end{lemma}
	\begin{proof}
		\begin{enumerate}[(i)]
			\item Let $(V_i,T_i)$ for $i=0,1,2$ be three objects in the big infinitesimal site $X/\Sigma_{e \Inf}$, with arrows $g_i:(V_i,T_i)\ra (V_0,T_0)$ for $i=1,2$.
			Namely each $V_i$ admits a map to $X$, and $V_i\ra T_i$ is a closed immersion that has a nil defining ideal.
			Then we can form the fiber products of rigid spaces $V_3:=V_1\times_{V_0} V_2$ and $T_3:=T_1\times_{T_0}T_2$ over $\Sigma_e$, together with a natural map $V_3\ra T_3$.
			Here the existence of fiber products is guaranteed in \cite[Prop.\ 1.2.2]{Hu96}.
			Any infinitesimal thickening $(V,T)$ that admits a compatible family of maps $(V,T)\ra (V_i,T_i)$ for $i=0,1,2$ would produce a commutative diagram
			\[
			\xymatrix{V\ar[r] \ar[d]& T\ar[d]\\
				V_3\ar[r] & T_3.}
			\]
			So it is left to show that $V_3\ra T_3$ is a nil closed immersion, which can be checked locally by choosing affinoid open subsets of $T_i, i=0,1,2$, as we shall explain.
			
			For $i=0,1,2$, we let $T_i$ be the affinoid adic space $\Spa(B_i)$ and let each closed subspace $V_i$ be defined by a nilpotent ideal $I_i\subset B_i$.
			We let $A_i$ for each $i=0,1,2$ be a compatible choice of subrings of definition of $B_i$, and let $J_i$ be the intersection $I_i\cap A_i$, which satisfies the equality that $J_i[1/p]=I_i$ by construction and in particular is nilpotent.
			Under the assumption, the fiber product $T_3$ is $\Spa( (A_1\otimes_{A_0} A_2)^\wedge_p[1/p])$, and the fiber product $V_3$ is $\Spa( (A_1/J_1\otimes_{A_0/J_0} A_2/J_2)^\wedge_p[1/p])$.
			Notice that as the kernel ideal $J$ of the surjection $A_1\otimes_{A_0} A_2 \to A_1/J_1\otimes_{A_0/J_0} A_2/J_2$ is generated by the image of the nilpotent ideals $J_0$, $J_1$, and $J_2$, the ideal $J$ in particular is nilpotent itself.
			As a consequence, since the kernel for the $p$-completed surjection $(A_1\otimes_{A_0} A_2)^\wedge_p\to (A_1/J_1\otimes_{A_0/J_0} A_2/J_2)^\wedge_p$ is generated by the image of $J$, we see the closed immersion $V_3\to T_3$ is indeed a nilpotent closed immersion.

			Finally we note that when $(V_1,T_i)$ comes from the small site for $i=0,1,2$ (namely the map $V_i\ra X$ is an open immersion for $i=0,1,2$), then the fiber product $V_3=V_1\times_{V_0} V_2$ is also open in $X$.
			In particular, the fiber product in this case is lying in the small site $X/\Sigma_{e \inf}$.
			
			\item 
			For the equalizer, consider the two arrows $\alpha,\beta:(V_1,T_1)\rightrightarrows (V_2,T_2)$ in $X/\Sigma_{e \Inf}$.
			Here both $V_1$ and $V_2$ admits a map to $X$, and $V_i\ra T_i$ are nil closed immersions.
			We can first form the equalizer $V_3$ of $V_1\rightrightarrows V_2$ and $T_3$ of $T_1\rightrightarrows T_2$ in the category of rigid spaces over $\Sigma_e$, by the pullback diagram
			\[
			\xymatrix{
				V_3\ar[r] \ar[d] &V_1\ar[d]\\
				V_2\ar[r] &V_2\times_{\Sigma_e} V_2,}~~~
			\xymatrix{T_3 \ar[r] \ar[d] &T_1\ar[d]\\
				T_2\ar[r] & T_2\times_{\Sigma_e} T_2,}
			\]
			where the bottom horizontal maps in both diagrams are diagonal embeddings.
			The left diagram admits a natural map to the right.
			Moreover, we notice that $V_3\ra T_3$ is a nil closed immersion, as all of other three terms in the diagram of $V_3$ are nil immersed into the diagram of $T_3$.
			Furthermore, as the map $V_1\ra V_2\times_{\Sigma_e} V_2$ factors through $V_2\times_X V_2\ra V_2\times_{\Sigma_e} V_2$, the pullback $V_3$ is also isomorphic to the equalizer of $V_1\rightrightarrows V_2$ in the category of rigid spaces over $X$.
			In this way, the object $(V_3,T_3)\in X/\Sigma_{e \Inf}$ obtained above forms the equalizer of $\alpha,\beta$ in the category.
			
			We at last note that the case when $\alpha,\beta$ comes from the small site is exactly when both of the arrows $V_1\rightrightarrows V_2$ are open immersions (hence they are the same), where the obtained base change $V_3\simeq V_2\underset{V_2\times_X V_2}{\times}V_1\simeq V_2\times_{V_2}V_1=V_1$ is also open in $X$.
			Thus the construction of the equalizer is compatible with the one in the small site.

			\item Let $(V_i,T_i)$ for $i=1,2$ be two objects in the big infinitesimal site $X/\Sigma_{e \Inf}$.
			Then we can form the fiber product $V_3:=V_1\times_X V_2$ over $X$, and the fiber product $T_1\times_{\Sigma_e} T_2$ over $\Sigma_e$ together with a natural map from $V_3$, such that any object $(V',T')\in X/\Sigma_{e \Inf}$ that admits a map to $(V_i,T_i)$ for $i=1,2$ will admit a unique map onto the pair of rigid spaces $(V_3, T_1\times_{\Sigma_e} T_2)$.
			
			Now the only problem is that the pair $(V_3, T_1\times_{\Sigma_e} T_2)$ is almost never a pair of infinitesimal thickening.
			However, notice that the map $V_3 \ra T_1\times_{\Sigma_e} T_2$ can be written as the composition
			\[
			V_3=V_1\times_X V_2\rra V_1\times_{\Sigma_e} V_2\rra  T_1\times_{\Sigma_e} T_2,
			\]
			where the first map is a locally closed immersion (a composition of a closed immersion and an open immersion) and the second map is a nil closed immersion.
			This allows us to form the direct limit
			\footnote{The direct limit is called \emph{envelope} for the locally closed immersion and will be formally introduced in \Cref{subsec env}, to which we refer the reader for detailed discussions.} $\varinjlim_m Y_m$ of all infinitesimal neighborhoods of $V_3$ into $T_1\times_{\Sigma_e} T_2$, where each $Y_m$ is the $m$-th infinitesimal neighborhood of $V_3$ inside of $T_1\times_{\Sigma_e} T_2$.
			In this way, the fiber product of $(V_1,T_1)$ and $(V_2,T_2)$ is ind-represented by the colimit of $(V_3, Y_m)$, for locally each map from an object $(V',T')$ onto the pair $(V_3,T_1\times_{\Sigma_e} T_2)$ factors through some $(V_3, Y_m)$ by \Cref{rep}.
			
			At last, we note that the construction is independent of big or small infinitesimal sites.
			Moreover, when $V_1$ and $V_2$ are open in $X$, from the construction above the rigid space $V_3$ is also open in $X$.
			Thus the nonempty finite product is compatible between the big and the small sites.

		\end{enumerate}
		
	\end{proof}
	\begin{remark}
		In fact, the ind-representable sheaf for the directed limit $\varinjlim_m Y_m$ is the \emph{envelope} of the immersion $V_3\ra T_1\times_{\Sigma_e} T_2$, which we will introduce in Definition \ref{env} soon.
	\end{remark}

	\noindent
	\textbf{Relation between big and small sites/topoi.}\label{site big small}
	Given a rigid space $X$ over $\Sigma_e$, there are two natural morphisms of topoi between the big infinitesimal topos $\Sh(X/\Sigma_{e \Inf})$ and the small infinitesimal topos $\Sh(X/\Sigma_{e \inf})$ of $X$.
	To see this, we first notice that by constructions, there exists a natural inclusion functor 
	\[
	X/\Sigma_{e \inf}\rra X/\Sigma_{e \Inf}.
	\]
	The inclusion functor is \emph{continuous} in the sense of \cite[Tag 00WV]{Sta}, and thus induces two functors between their topoi (\cite[Tag 00WU]{Sta})
	\begin{itemize}
		\item For a sheaf $\calF\in \Sh(X/\Sigma_{e \inf})$ over the small site, there exists a preimage functor $\mu^{-1}$ with $\mu^{-1}\calF$ being the sheaf associated with the presheaf
		\[
		X/\Sigma_{e \Inf}\ni (V,S)\lmt \varinjlim_{\substack{(V,S)\ra (U,T)\\ (U,T)\in X/\Sigma_{e \inf}}} \calF(U,T).
		\]
		By \Cref{site finite limit} the functor $\mu^{-1}$ commutes with nonempty finite limits.
		\item The direct image functor $\mu_*$, which is the right adjoint of $\mu^{-1}$ and is computed by the restriction.
		Namely for a sheaf $\calG\in \Sh(X/\Sigma_{e \Inf})$ over the big site, we have $\mu_*\calG(U,T)=\calG(U,T)$.
	\end{itemize}
	This pair of adjoint functors in fact forms a morphism of topoi
	\[
	\mu:\Sh(X/\Sigma_{e \Inf})\rra \Sh(X/\Sigma_{e \inf}).
	\]
	To see this, we claim the following:
	\begin{lemma}\label{site exactness}
		The left adjoint functor $\mu^{-1}:\Sh(X/\Sigma_{e \inf})\ra \Sh(X/\Sigma_{e \Inf})$ commutes with any nonempty finite limit.
	\end{lemma}
	\begin{proof}
		To see this, we first notice that as a left adjoint functor commutes with any small colimit, by \cite[\href{https://stacks.math.columbia.edu/tag/0GLW}{Tag 0GLW}]{Sta} it suffices to show this for a finite diagram of representable sheaves.
		Moreover, as a nonempty finite limit can be formed by finite many of nonempty finite products and equalizers (\cite[Tag 04AS]{Sta}), it suffices to show that $\mu^{-1}$ commutes with finite products and equalizers of representable sheaves, which is given by Lemma \ref{site finite limit}.
		So we are done.
	\end{proof}
	The above by definition means that the left adjoint functor $\mu^{-1}$ is exact once passing to the minimal enlargement of the infinitesimal sites by adding the final object (cf. \cite[\href{https://stacks.math.columbia.edu/tag/03A1}{Tag 03A1}]{Sta})
	\footnote{Precisely, as both infinitesimal sites do not admit the final object (equivalently the empty product), it does not make sense to talk about the right exactness of the functor $\mu^{-1}$. To remedy this, one can enlarge the sites by adding the final objects simultaneously, which will not change the corresponding topoi by loc.\ cit.. In particular, the induced functor of $\mu^{-1}$ on the enlarged sites preserves all finite limits. }, and hence we get a morphism of topoi (\cite[Tag 00X1]{Sta})
	\[
	\mu:\Sh(X/\Sigma_{e \Inf})\rra \Sh(X/\Sigma_{e \inf}).
	\]

	On the other hand, the inclusion functor  is \emph{cocontinuous} in the sense of \cite[Tag 00XJ]{Sta}.
	This is because if a collection of thickenings $\{(U_i,T_i)\}\subset X/\Sigma_{e \Inf}$ covers a given $(U,T)\in X/\Sigma_{e \inf}$ in the big site, then each $(U_i,T_i)$  is also an object  in the small site which together forms a covering of $(U,T)$.
	So by \cite[Tag 00XO]{Sta}, the inclusion functor induces another map of topoi 
	\[
	\iota:\Sh(X/\Sigma_{e \inf})\rra \Sh(X/\Sigma_{e \Inf}),
	\]
	consists of the following adjoint pairs of functors
	\begin{itemize}
		\item The functor $\iota^{-1}=\mu_*:\Sh(X/\Sigma_{e \Inf})\ra \Sh(X/\Sigma_{e \inf})$ is the restriction functor, which commutes with any finite limits.
		\item The functor $\iota_*:\Sh(X/\Sigma_{e \inf})\ra \Sh(X/\Sigma_{e \Inf})$, which is the right adjoint of the functor $\iota^{-1}$, sending a sheaf $\calF$ over the small site to the sheaf $\iota_*\calF$ with the equality
		\[
		X/\Sigma_{e \Inf}\ni(V,S)\lmt \varprojlim_{\substack{(V,S)\ra (U,T)\\ (U,T)\in X/\Sigma_{e \inf}}} \calF(U,T).
		\]
	\end{itemize}
	Here we notice that when the thickening $(V,S)$ is an object coming from the small site $X/\Sigma_{e \inf}$ (namely $V\ra X$ is an open immersion), from the description above we then have
	\[
	(\iota_*\calF)(V,S)=\calF(V,S).
	\]

	Furthermore, notice that given an arrow $(V_1, T_1)\ra (V_2,T_2)$ in the big infinitesimal site $X/\Sigma_{e \Inf}$, the associated morphism of rigid spaces $V_1\ra V_2$ is a $X$-morphism.
	This in particular implies that the inclusion functor $X/\Sigma_{e \inf}\ra X/\Sigma_{e \Inf}$ is fully faithful, as when $(V_1,T_1)$ and $(V_2,T_2)$ come from the small site, the only $X$-morphism between $V_1$ and $V_2$ is the open immersion.
	So by \cite[Tag 00XS, Tag 00XT]{Sta} and Lemma \ref{site finite limit}, we have \footnote{In the notation of \cite[Tag 00XR]{Sta}, the functor $\mu^{-1}$ is equal to the functor $\iota_!$.}
	\begin{itemize}
		\item The functor $\mu^{-1}$ commutes with fiber products and equalizers (so with all finite connected limits).
		\item The canonical natural transformations below are isomorphisms of functors:
		\[
		\id \rra \mu_*\circ \mu^{-1};~~~\iota^{-1}\circ \iota_*=\mu_*\circ \iota_*\rra \id.
		\]
	\end{itemize}

	\subsection{Envelopes}\label{subsec env}
	Analogous to the infinitesimal theory of complex varieties in \cite{Gr68} and the crystalline theory of schemes in positive characteristic in \cite{BO78}, we can define the envelope for a locally closed immersion $X\ra Y$ of rigid spaces.
	
	\begin{definition}\label{env}
		Let $Y$ be a rigid space over $\Sigma_e$, and $X$ be a locally closed analytic subspace in $Y$, defined by a coherent ideal $I$ in $\calO_U$ for $U$ an open subset inside of $Y$.
		We denote by $Y_n$ the $n$-th infinitesimal neighborhood of $X$ in $Y$, which form an object $(X,Y_n)$ in $X/\Sigma_{e \inf}$ and is defined by the ideal $I^{n+1}$.
		
		The \emph{envelope $D_X(Y)$ of $X$ in $Y$}, is an object  in the infinitesimal topos $\Sh(X/\Sigma_{e \inf})$, defined by the colimit of the direct system of representable sheaves $h_{Y_n}$ of $(X,Y_n)$ in $\Sh(X/\Sigma_{e \inf})$:
		\[
		D_X(Y):= \varinjlim_{n\in \NN} h_{Y_n}.
		\]
	\end{definition}
	Note that the definition also works for the big infinitesimal topos $\Sh(X/\Sigma_{e \Inf})$, and under the natural inclusion functor $X/\Sigma_{e \inf}\ra X/\Sigma_{e \Inf}$ the notions of the envelopes coincide.

	\begin{remark}\label{env space}
		In many situations, it is convenient to regard $D_X(Y)$ as an actual locally ringed space, 
		instead of a direct limit of representable sheaves in the infinitesimal topos.
		Here the associated ringed space structure of the envelope $D_X(Y)$ has the same topological space as the adic space $X$, 
		and the structure sheaf $\calD=\varprojlim_n \calO_{Y_n}$ is the inverse limit of structure sheaves of infinitesimal neighborhoods $Y_n$.
	\end{remark}
	
	\begin{remark}
		The existence of the colimit in the topos is guaranteed by \cite[Tag 00WI]{Sta}.
	\end{remark}
	
	\begin{remark}
		Here we want to mention that different from the crystalline theory of a scheme over $\ZZ_p/p^e$, the envelope is almost \emph{never} representable.
		In the mixed characteristic case, 
		the divided-power structure enforces the defining ideal for a divided power thickening to be nilpotent.
		However, in equal characteristic zero such a condition is lost and the envelope is not an infinitesimal thickening.
		This in particular appears when we consider the crystalline theory of a scheme over $\CC$.
	\end{remark}
	Though the envelope fails to be representable, we do have a description of an envelope that is similar to a representable sheaf:
	\begin{lemma}\label{rep}
		For a closed immersion $X\ra Y$ of rigid spaces over $\Sigma_e$, the envelope $D_X(Y)$ is isomorphic to the sheaf on $X/\Sigma_{e \inf}$ (and $X/\Sigma_{e \Inf}$), defined by
		\[
		(U,T)\lmt \Hom((U,T),(X,Y)),
		\]
		where $\Hom((U,T),(X,Y))$ is the set of commutative diagrams of $\Sigma_e$-rigid spaces
		\[
		\xymatrix {T\ar[r] & Y\\
			U\ar[r] \ar[u] &X \ar[u]},
		\]
		with $U\ra X$ being the structure morphism for the object $(U,T)$.
	\end{lemma}
	\begin{proof}
		We first notice that we have a natural map
		\[
		D_X(Y)((U,T))=\varinjlim_{n\in \NN} \Hom((U,T),(X,Y_n))\rra \Hom((U,T),(X,Y)),
		\]
		induced by closed immersions $Y_n\ra Y_{n+1}\ra Y$.
		So it suffices to check that for a pair of affinoid rigid spaces $(U,T)=(\Spa(R/J), \Spa(R))$ in the infinitesimal site, the above is an isomorphism.
		
		For the surjection, we notice that since $(U,T)$ is affinoid rigid space over $\Sigma_e$, the ring $R$ is noetherian and $J$ is nilpotent.
		In particular, there exists an $n\in \NN$, such that $J^{n+1}=0$.
		So the map $\Spa(R)\ra Y$ factors through a map $\Spa(R)\ra Y_n$.
		
		For the injection, assume there are two maps $\alpha,\beta:T\to Y_n$ of rigid spaces over $\Sigma_e$ whose compositions with $Y_n\ra Y$ are equal.
		Note that since $Y_n\ra Y$ is a closed immersion, by restricting to an affinoid open covering of $Y$ (thus $Y_n$), the compositions can be translated into the following maps of $\Bdre$-algebras
		\[
		A\ra A/I^{n+1}\to R.
		\]
		So the equality of the maps $A\to R$ implies that the maps $A/I^{n+1}\to R$ are equal, and hence implies the equality of $\alpha,\beta$.
	\end{proof}
	The following simple observations justify this name of the envelope:
	\begin{lemma}\label{eff-epi}
		Assume $Y$ is smooth over $\Sigma_e$.
		Then the envelope $D_X(Y)$ for a closed immersion of $X$ in $Y$ covers the final object in the infinitesimal topoi $\Sh(X/\Sigma_{e \inf})$ and $\Sh(X/\Sigma_{e \Inf})$.
		In other words, the map from $D_X(Y)$ onto the final object in the infinitesimal topoi is an epimorphism of sheaves.
	\end{lemma}
	\begin{proof}
		We denote by $1$  the final object in $\Sh(X/\Sigma_{e \inf})$ or $\Sh(X/\Sigma_{e \Inf})$.
		Then to show the surjection of the map of sheaves
		\[
		D_X(Y)\rra 1,
		\]
		it suffices to show that any object $(U,T)$ in the infinitesimal site locally admits a morphism to $D_X(Y)$; namely there is an open cover $(U_i,T_i)$ of $(U,T)$ such that each $(U_i,T_i)$ admits a map to $D_X(Y)$.
		
		For an affinoid thickening $(U,T)=(\Spa(R/I),\Spa(R))$ with an open immersion $U\ra X$, since $U\ra T$ is a nil closed immersion and $R$ is noetherian, there exists an integer $m$ such that $I^{m+1}=0$ in $R$.
		By assumption that $Y$ is smooth,
		locally there exists a morphism from $\Spa(R)$ to $Y$ that makes the following diagram commute (cf. \cite[Def.\ 1.6.5]{Hu96})
		\[
		\xymatrix{
			\Spa(R/I)\ar[r] \ar[d]& \Spa(R) \ar[d]\\
			X \ar[r]& Y.}
		\]
		By the nilpotence of the ideal $I$, the map $\Spa(R) \ra Y$ factors canonically through $Y_n$ for $n\geq m$.
		Thus the map $\Spa(R)\ra Y$ factors through the direct limit $D_X(Y)=\varinjlim_{n\in \NN} h_{Y_n}\ra Y$.
		
	\end{proof}
	
	The above allows us to give a very general formula to compute the cohomology over the infinitesimal site, using the \v{C}ech nerve for an envelope.
	\begin{proposition}\label{env, coh}
		Let $X\ra Y$ be a closed immersion into a smooth rigid space $Y$ over $\Sigma_e$.
		For $n\in \Delta$, we  denote $D(n)$ to be the simplicial space where each $D(n)$ is the envelope of $X$ in $Y(n):=Y^{\underset{\Sigma_e}{\times} n+1}$.
		Then there is a natural isomorphism of cohomology for a sheaf $\mathcal{F}$ over the small infinitesimal site
		\[
		R\Gamma(X/\Sigma_{e \inf}, \calF) \rra R\lim_{[n]\in \Delta} R\Gamma(D(n), \calF).
		\]
		Similarly it holds for the big infinitesimal site.
	\end{proposition}
	Here we want to mention that for each $n\in \Delta$, the derived section functor $R\Gamma(D(n),\calF)$ is computed via the inverse limit
	\[
	R\varprojlim_{m\in \NN} R\Gamma((X,Y({n})_m), \calF),
	\]
	where each $Y(n)_m$ is the $m$-th infinitesimal neighborhood of $X$ in $Y(n)$.
	\begin{proof}
		We first notice that $D(n)$ is in fact the $(n+1)$-fold self-product of $D_X(Y)$ in the infinitesimal topos $\Sh(X/\Sigma_{e \inf})$ (or $\Sh(X/\Sigma_{e \Inf})$ respectively).
		This is because by Lemma \ref{rep}, we know
		\[
		D(n)=\Hom(-,(X,Y^{ n+1})),
		\]
		which is the same as the contravariant functor $\Hom(-,(X,Y))^{n+1}$ on the infinitesimal site.
		So the simplicial object $D(\bullet)$ is in fact the coskeleton $\cosk_0(D_X(Y))$ over the final object (in other words, the \v{C}ech nerve for the map of sheaves $D_X(Y)\ra 1$).
		In this way, since $D_X(Y)\ra 1$ is an effective epimorphism (Lemma \ref{eff-epi}), by the \cite[Tag 09VU]{Sta}, 
		$D(\bullet)\ra 1$ is a hypercovering, and we get a natural equivalence of derived functors
		\[
		R\Gamma(X/\Sigma_{e \inf},-) \simeq R\Gamma(D(\bullet),-)=R\lim_{[n]\in \Delta} R\Gamma(D(n),-).
		\]
	\end{proof}
	
	As an upshot, we see that the restriction functor from the big infinitesimal topos to the small one preserves the cohomology.
	\begin{corollary}\label{coh, big small}
		Let $\calF$ be an object in the derived category of sheaves over $X/\Sigma_{e \Inf}$.
		Then the restriction functor $\iota^{-1}=\mu_*:\Sh(X/\Sigma_{e \Inf}) \ra \Sh(X/\Sigma_{e \inf})$ (cf. Paragraph \ref{site big small}) induces the following isomorphism
		\[
		R\Gamma(X/\Sigma_{e \inf}, \mu_*\calF) \rra R\Gamma(X/\Sigma_{e \Inf}, \calF).
		\]
	\end{corollary}
	\begin{proof}
		We first assume $X$ admits a closed immersion into a smooth rigid space over $\Sigma_e$.
		The claim in this case then follows from Proposition \ref{env, coh}, as an envelope is a direct limit $\varinjlim_{m\in \NN} h_{Y(n)_m}$ of representable objects in the big and the small sites, and the restriction functor produces the natural equivalence
		\[
		R\Gamma(D(n), \calF) \rra R\Gamma(D(n), \mu_*\calF).
		\]
		In general, we may take a hypercovering by affinoid open spaces of $X$ first to reduce to the above special cases, since an affinoid rigid space $\Spa(A)$ is topologically of finite type and thus the ring $A$ admits a surjection from $\Bdre \langle T_1,\ldots,T_n \rangle$ for some $n\in \mathbb{N}$.
	\end{proof}

	\subsection{Infinitesimal and rigid topology}\label{sub inf-rig}
	In this subsection, we relate the infinitesimal topos and the rigid topos together.
	
	Let $X$ be a rigid space over $\Sigma_e$.
	Recall that there is a Grothendieck topology $X_\rig$ on the category of open subsets in $X$, called the \emph{rigid site} $X_\rig$.
	
	Consider the following two functors:
	\begin{align*}
		u_{X/\Sigma_e *}:\Sh(X/\Sigma_{e \inf})&\rra \Sh(X_\rig);\\
		\calF&\lmt (U\mapsto \Gamma(U/\Sigma_{e \inf},\calF|_{U/\Sigma_{e \inf}})).\\
		u_{X/\Sigma_e}^{-1}:\Sh(X_\rig)&\rra \Sh(X/\Sigma_{e \inf});\\
		\calE&\lmt ((U,T)\mapsto \calE(U)).
	\end{align*}

	For a given infinitesimal thickening $(U,T)$, since $(u_{X/\Sigma_e}^{-1}\calE)_T$ is equal to the sheaf $\calE|_U$ on $U_\rig\simeq T_\rig$, the functor $u_{X/\Sigma_e}^{-1}$ commutes with the finite inverse limit.
	Notice that the pair $(u_{X/\Sigma_e}^{-1},~u_{X/\Sigma_e *})$ is adjoint.
	Thus we get a morphism of topoi (\cite[Tag 00XA]{Sta})
	\[
	u_{X/\Sigma_e}: \Sh(X/\Sigma_{e \inf})\rra \Sh(X_\rig),
	\]
	which we follow \cite{BO78} and call the \emph{projection morphism}.
	
	The projection morphism $u_{X/\Sigma_e}$ admits a section.
	Consider the functor $X/\Sigma_{e \inf}\ra X_\rig$, sending $(U,T)$ onto the open subset $U$ of $X$.
	By the definition of $X/\Sigma_{e \inf}$, a covering of $(U,T)$ is mapped onto a covering of $U$.
	In particular, the map of sites is continuous in the sense of \cite{Sta}.
	So we get a morphism of sites
	\[
	i_{X/\Sigma_e}:X_\rig\rra X/\Sigma_{e \inf}.
	\]
	The morphisms induces a map of topoi, in a way that for $\calE\in \Sh(X_\rig)$,  
	\[
	i_{X/\Sigma_e *}\calE(U,T)=\calE(U),
	\]
	and for $\calF\in\Sh(X/\Sigma_{e \inf})$, we have
	\[
	i_{X/\Sigma_e}^{-1}\calF(U)=\varinjlim_{(U,U)\ra (V,T)} \calF(V,T)=\calF(U,U).
	\]
	From the description, we see the functor $i_{X/\Sigma_e}^{-1}$ is the restriction functor sending a sheaf $\calF$ over $X/\Sigma_{e \inf}$ to its restriction $\calF_X$ on the rigid space $X$.

	\begin{remark}
		By the construction of $i_{X/\Sigma_e}$ and $u_{X/\Sigma_e}$, on the rigid topos $\Sh(X_\rig)$ we have
		\[
		u_{X/\Sigma_e *}\circ i_{X/\Sigma_e *}=id,~~~i_{X/\Sigma_e}^{-1}\circ u_{X/\Sigma_e}^{-1}=id,
		\]
		which implies that those morphisms of topoi satisfy
		\[
		u_{X/\Sigma_e}\circ i_{X/\Sigma_e}=id.
		\]
		This justify the name of the projection morphism.
	\end{remark}
	
	\begin{remark}
		The construction here naturally generalizes to two morphisms between big infinitesimal site $X/\Sigma_{e \Inf}$ and the big rigid site $\Rig_{\Sigma_e}|_X$, for a given rigid space over $X$.
	\end{remark}
	
	\subsection{Functoriality}\label{subsec functoriality}
	In this subsection, we introduce natural maps of infinitesimal topoi associated with a map of rigid spaces, similarly to the construction in \cite[Tag 07IC, 07IK]{Sta}.

	Let $f:X\ra Y$ be a map of rigid spaces over $\Sigma_{e'}$, and assume the structure map $X\ra \Sigma_{e'}$ factors through $\Sigma_e$  for non-negative integers $e\leq e'$.
	By the construction of the big infinitesimal site, the map $f$ induces a natural functor between $X/\Sigma_{e \Inf}$ and $Y/\Sigma_{e' \Inf}$, satisfying
	\[
	X/\Sigma_{e \Inf}\ni((U,T),U\ra X)\lmt ((U,T),U\ra X\ra Y),
	\]
	where the map $U\ra X\ra Y$ is the composition of the map $f$ with the structure map of $(U,T)\in X/\Sigma_{e \Inf}$.
	Then it is easy to check that this functor is both continuous and cocontinuous, and commutes withe fiber products and equalizers (cf. Lemma \ref {site finite limit}).
	This in particular implies that the functor above induces a morphism of topoi (\cite[Tag 00XN, 00XR]{Sta})
	\[
	f_{\Inf}:\Sh(X/\Sigma_{e \Inf}) \rra \Sh(Y/\Sigma_{e' \Inf}),
	\]
	such that
	\begin{itemize}
		\item The inverse image functor $f_{\Inf}^{-1}$ commutes with arbitrary limits and colimits, such that for a sheaf $\calG$ over $Y/\Sigma_{e' \Inf}$, we have
		\[
		f_{\Inf}^{-1}\calG(U,T)=\calG(U,T),
		\]
		where the second $(U,T)$ is regarded as an object in $Y/\Sigma_{e' \Inf}$ by $U\ra X\ra Y$.
		\item The direct image functor $f_{\Inf *}$, which is the right adjoint to the functor $f_{\Inf}^{-1}$, sends a sheaf $\calF\in \Sh(X/\Sigma_{e \Inf})$ to the sheaf $f_{\Inf *}\calF$ such that the section is given by
		\[
		Y/\Sigma_{e' \Inf}\ni(V,S)\lmt \varprojlim_{\substack{(U,T)\ra (U,T)\\ (V,S)\in X/\Sigma_{e \Inf},\\
				V\ra U~compatible~with~f}} \calF(U,T).
		\]
	\end{itemize}
	
	Now we consider  the small topoi $\Sh(X/\Sigma_{e \inf})$ and $\Sh(Y/\Sigma_{e' \inf})$.
	Analogous to \cite[Tag 07IK]{Sta}, we use the map of big topoi to connect them.
	Consider the following diagram
	\[
	\xymatrix{
		\Sh(X/\Sigma_{e \Inf}) \ar[r]^{f_{\Inf}} & \Sh(Y/\Sigma_{e' \Inf}) \ar[d]^{\mu_Y} \\
		\Sh(X/\Sigma_{e \inf}) \ar[u]^{\iota_X} \ar@{.>}[r]_{f_{\inf}} &\Sh(X/\Sigma_{e \inf}).}
	\]
	Here we define the morphism of topoi $f_{\inf}:\Sh(X/\Sigma_{e \inf})\ra \Sh(Y/\Sigma_{e' \inf})$ to be the composition
	\[
	f_{\inf}=\mu_Y\circ f_\Inf \circ \iota_X.
	\]
	Then by the definition of those functors, we have
	\begin{itemize}
		\item For a sheaf $\calG\in \Sh(Y/\Sigma_{e' \inf})$, the inverse image $f_{\inf}^{-1}\calG$ is given by the ``restriction" of $\mu_Y^{-1}\calG$ to the category $X/\Sigma_{e \inf}$ via the map $f$, and is equal to the sheaf associated with the presheaf
		\[
		X/\Sigma_{e \inf}\ni (U,T) \lmt \varinjlim_{\substack{(U,T)\ra (V,S)\\ (V,S)\in Y/\Sigma_{e' \inf},\\
				U\ra V~compatible~with~f}} \calG(V,S).
		\]
		\item The direct image functor $f_{\inf *}$ sends a sheaf $\calF\in \Sh(X/\Sigma_{e \inf})$ to the sheaf 
		\[
		f_{\inf *}\calF(V,S)=\varprojlim_{\substack{(U,T)\ra (V,S)\\ (U,T)\in X/\Sigma_{e \Inf}\\ U\ra V~compatible~with~f}} \calF(U,T).
		\]
	\end{itemize}
	
	\begin{remark}

		In the special case when $\calG=h_S$ is the representable sheaf of $(V,S)\in Y/\Sigma_{e' \inf}$, its inverse image $f_{\inf }^{-1}h_S$ has a simpler formula by
		\[
		f_{\inf}^{-1}h_S(U,T)=\Hom_Y((U,T),(V,S)):=\{commutative~diagrams\xymatrix{X\ar[d]_f & U\ar[l] \ar[r]^{nil} \ar@{.>}[d]& T \ar@{.>}[d]\\
			Y & V \ar[l] \ar[r]^{nil} & S}\}.
		\]
		Here the notation ``nil" means the arrows below are nil closed immersions, and the map $T\ra S$ is a map of rigid spaces over $\Sigma_{e'}$.
	\end{remark}
	
	\begin{remark}

		The functoriality of infinitesimal topoi is compatible with the projection morphism to  the rigid topos and its section.
		Namely the following two diagrams are commutative
		\[
		\xymatrix{\Sh(X/\Sigma_{e \inf}) \ar[r]^{f_{\inf}} \ar[d]_{u_{X/\Sigma_e}} & \Sh(Y/\Sigma_{e' \inf}) \ar[d]^{u_{Y/\Sigma_{e'}}}\\
			\Sh(X_\rig)\ar[r]& \Sh(Y_\rig)},
		~~~\xymatrix{\Sh(X_\rig) \ar[r] \ar[d]_{i_{X/\Sigma_e}} & \Sh(Y_\rig)\ar[d]^{i_{Y/\Sigma_{e'}}}\\
			\Sh(X/\Sigma_{e \inf}) \ar[r]^{f_{\inf}}  & \Sh(Y/\Sigma_{e' \inf}).}
		\]
		Here the commutativity can be checked readily using explicit formulae above, and we will not expand but refer the reader to  \cite[\href{https://stacks.math.columbia.edu/tag/07KL}{Tag 07KL}]{Sta} for the analogue in the classical crystalline theory.
	\end{remark}
	
	We also want to mention that $f_{\inf}$ is naturally a map of ringed topoi under the infinitesimal structure sheaves, and we could define the \emph{pullback functor $f_{\inf}^*$} on the category of $\calO_{Y/\Sigma_e'}$-sheaves, similarly to the scheme theory.
	Here given a sheaf of $\calO_{Y/\Sigma_e'}$-modules $\calG$ and an object $(U,T)\in X/\Sigma_{e \inf}$, the restriction of the pullback $f_{\inf}^*\calG$ at the rigid space $T$ is equal to the colimit
	\[
	\varinjlim_{\substack{h:(U,T)\ra (V,S)\\ (V,S)\in Y/\Sigma_{e' \inf}}} h^*(\calG_S).
	\]
	
	\begin{remark}\label{change of bases}
		Here we remark that when $f:X \ra X$ is the identity map but $e$ is strictly smaller than $e'$, the transition morphism $f_{\inf}:\Sh(X/\Sigma_{e \inf}) \ra \Sh(X/\Sigma_{e' \inf})$ is induced from the map of sites $f_{\inf}:X/\Sigma_{e \inf} \ra X/\Sigma_{e' \inf}$, where the corresponding functor sends $(V,S)\in X/\Sigma_{e' \inf}$ onto the thickening $(V,S\times_{\Sigma_{e'}} \Sigma_e)$.
		
	\end{remark}
	
	\section{Crystals}\label{sec cry}
	In this section, we study the coherent crystal and its canonical connection.
	
	Before we start, we mention that though stated for rigid spaces over $\Bdre$, the results and proofs in this section hold true verbatimly for rigid spaces over an arbitrary $p$-adic extension $K$ of $\mathbb{Q}_p$; namely $K$ is a field that is complete with respect to a non-archimedean valuation extending that of $\mathbb{Q}_p$.
	
	\subsection{Crystals and their connections}\label{sub con}
	We first introduce the coherent crystal and a canonical connection associated with it.

	\textbf{Sheaf of differentials.}

	\begin{definition}
		The \emph{infinitesimal sheaf of differentials} $\Omega_{X/\Sigma_{e \inf}}^i$ is a sheaf of $\calO_{X/\Sigma_e}$-modules on $X/\Sigma_{e \inf}$ defined as 
		\[
		\Omega_{X/\Sigma_{e \inf}}^i(U,T):=\Omega_{T/\Sigma_e}^{i, cont}(T), \quad (U,T)\in X/\Sigma_{e \inf}
		\]
		locally given by the continuous differentials over $\Sigma_e$.
	\end{definition}
	Similarly we could define the infinitesimal sheaf of differentials $\Omega_{X/\Sigma_{e \Inf}}^i$ over the big site $X/\Sigma_{e \Inf}$.
	It can be checked easily that the restriction $\iota^{-1}\Omega_{X/\Sigma_{e \Inf}}^i$ at the small site is equal to $\Omega_{X/\Sigma_{e \inf}}^i$.
	
	Here we recall the definition of the sheaf of continuous differentials as follows.
	Let $T$ be a rigid space over $\Sigma_e$, and $T(m)$ be the $m+1$-th self product of $T$ over $\Sigma_e$, which is equipped with $m+1$ projection maps onto $T$ and the diagonal map from $T$.
	For each $m\in \NN$, we denote $T(n)_m$ as the $m$-th infinitesimal neighborhood of $T$ in $T(n)$.
	Then each infinitesimal thickening $(T,T(n)_m)$ is an object in $X/ \Sigma_{e \inf}$.
	
	Let $I_T$ be the coherent sheaf of ideals in $\calO_{T(1)}$, defined as the kernel of the map $\calO_{T(1)}\ra \calO_T$ given by the diagonal $T\ra T(1)_1$. 
	Then the sheaf of continuous differentials $\Omega_{T/\Sigma_e}^{1,cont}$ is the coherent sheaf $I_T/I_T^2$ over $T$.
	It can be checked that the sheaf of continuous differentials satisfy the universal property among continuous $\Bdre$-linear derivatives.
	Without mentioning, we will use $\Omega_{T/\Sigma_e}^{i}$ to denote the $i$-th continuous differentials to simplify the notation.
	
	\textbf{Crystals.}

	\begin{definition}\label{cry def}
		Let $\calF$ be a coherent sheaf over  $X/\Sigma_{e \inf}$ (resp. $X/\Sigma_{e \Inf}$).
		Namely $\calF$ is a sheaf on the infinitesimal site $X/\Sigma_{e \inf}$ (resp. $X/\Sigma_{e \Inf}$) such that $\calF_T$ is a coherent $\calO_T$-module for each infinitesimal thickening $(U,T)$ in $X/\Sigma_{e \inf}$ (resp. in $X/\Sigma_{e \Inf}$).
		We call $\calF$ a \emph{coherent crystal} if for each morphism of thickenings $(i,g):(U_1,T_1)\ra (U_2,T_2)$ in $X/\Sigma_{e \inf}$ (resp. in $X/\Sigma_{e \Inf}$), the natural map
		\[
		g^*\calF_{T_2}=\calO_{T_1}\otimes_{g^{-1}\calO_{T_2}} g^{-1}\calF_{T_2}\rra \calF_{T_1}
		\]
		is an isomorphism of $\calO_{T_1}$ modules.
	\end{definition}
	
	\begin{example}
		The easiest example of coherent crystal is the infinitesimal structure sheaf $\calO_{X/\Sigma_{e}}$, defined either over the small or big infinitesimal sites of $X$ over $\Sigma_e$.
	\end{example}
	\begin{remark}
		The infinitesimal sheaf of differentials is not a crystal in general, though it is a coherent sheaf over $\calO_{X/\Sigma_{e}}$.
	\end{remark}
	
	Here it is not hard to see that the pullback of a coherent crystal is a crystal.
	\begin{lemma}\label{cry pullback}
		Let $f:X\ra Y$ be a map of rigid spaces over $\Sigma_{e'}$, and assume the structure map $X\ra \Sigma_{e'}$ factors through $\Sigma_e$  for non-negative integers $e\leq e'$.
		Let $\calG$ be a coherent crystal over $Y/\Sigma_{e' \inf}$.
		We denote $f_{\inf}$ to be the functoriality map of infinitesimal topoi $f_{\inf}:\Sh(X/\Sigma_{e \inf}) \ra \Sh(Y/\Sigma_{e' \inf})$.
		\begin{enumerate}[(i)]
			\item Let $(g,h):(U, T)\ra (V,S)$ be a map of thickenings for $(U,T)\in X/\Sigma_{e \inf}$ and $(V,S)\in Y/\Sigma_{e' \inf}$ respectively such that $g:U\ra V$ is compatible with $f:X\ra Y$.
			Then the restriction of $f_{\inf}^*\calG$ at $T$ is naturally isomorphic to the pullback $h^*(\calG_S)$ of the coherent sheaf $\calG_S$ over $S$ along the map of rigid spaces $h:T \ra S$.
			\item The pullback $f_{\inf}^*\calG$ is a coherent crystal over $X/\Sigma_{e \inf}.$
			\item Both $(i)$ and $(ii)$ hold true for $f_{\Inf}:\Sh(X/\Sigma_{e \Inf}) \ra \Sh(Y/\Sigma_{e' \Inf})$ and a coherent crystal $\calG$ over big infinitesimal sites.
		\end{enumerate}
	\end{lemma}
	\begin{proof}
		Let $(U,T)$ be an object in the infinitesimal site $X/\Sigma_{e \inf}$.
		By the construction, we know the restriction of the pullback $f_{\inf}^*\calG$ at the rigid space $T$ is equal to the colimit
		\[
		\varinjlim_{\substack{h:(U,T)\ra (V,S)\\ (V,S)\in Y/\Sigma_{e' \inf}}} h^*(\calG_S),
		\]
		where $h:T\ra S$ is the map of rigid spaces over $\Sigma_e'$.
		On the one hand,  by the definition of the coherent crystal, for a commutative diagram of infinitesimal thickenings
		\[
		\xymatrix{&(V,S) \ar[rd]^{h''}&\\
			(U,T) \ar[rr]_{h'}  \ar[ru]^h&&(V',S')}
		\]
		that is compatible with $f:X\ra Y$, the pullback $(h'')^*(\calG_{S'})$ is equal to the coherent sheaf $\calG_S$ over $S$.
		On the other hand, as in Lemma \ref{site finite limit} the finite products are ind-representable in the small site $Y/\Sigma_{e' \inf}$.
		In particular, given two maps of thickenings $h_i:(U,T) \rightrightarrows (V,S)$ where $U\ra V$ is the compatible with $f$, both $h_i$ locally factor through a thickening $(V,S(1)_m)$ for an infinitesimal neighborhood $S(1)_m$ of $V$ in $S(1)=S\times_{\Sigma_{e'}} S$.
		As an upshot, the pullback $h^*\calG_S$ is independent of the map $h$.
		In this way, the restriction of $f_{\inf}^*\calG$ at $T$, which is equal to the colimit above, is naturally isomorphic to the coherent sheaf $h^*(\calG_S)$ over $T$ for any map of thickenings $h:(U,T)\ra (V,S)$,
		where $(U,T)\in X/\Sigma_{e \inf}$ and $(V,S)\in Y/\Sigma_{e' \inf}$.
		This finishes the proof of $(i)$.
		
		To check the crystal condition of $f_{\inf}^*\calG$, it suffices to note that given a map of objects $g:(U_1,T_1)\ra (U_2,T_2)$ in $X/\Sigma_{e \inf}$ and a compatible map of infinitesimal thickenings $h:(U_2,T_2)\ra (V,S)$ for $(V,S)\in Y/\Sigma_{e' \inf}$, we have
		\[
		(f_{\inf}^*\calG)_{T_1}\simeq g^*h^*(\calG_S)\simeq g^*(f_{\inf}^*\calG)_{T_2}.
		\]
		
		At last, notice that the proof is applicable no matter whether the structure maps $U\ra X$ and $V\ra Y$ are open immersions.
		So we are done.
	\end{proof}
	
	\begin{example}\label{cry, over big site}
		An example of a coherent crystal over the big site is the pullback of a crystal from the small site
		\[
		\mu^*\calF=\mu^{-1}\calF\underset{\mu^{-1}\calO_{X/\Sigma_{e \inf}}}{\motimes} \calO_{X/\Sigma_{e \Inf}},
		\] where $\mu:\Sh(X/\Sigma_{e \Inf})\ra \Sh(X/\Sigma_{e \inf})$ is the canonical map from the big topos to the small topos, as in Subsection \ref{site big small}.
		Here the proof is identical to that of Lemma \ref{cry pullback}. 
		
		In particular, the pullback $\mu^*\calF$ locally satisfies the same formula as in Lemma \ref{cry pullback}.(i).
		For a crystal $\calF$ over the small site $X/\Sigma_{e \inf}$ and a thickening $(U,T)\in X/\Sigma_{e \Inf}$ in the big site, the restriction of the infinitesimal sheaf $\mu^*\calF$ on $T$ is naturally isomorphic to the pullback $g^*(\calF_S)$.	 
		Here the map $g:T\ra S$ of rigid spaces  comes from an arbitrary commutative diagram of objects $(i,g):(U,T)\ra (V,S)$ in $X/\Sigma_{e \Inf}$ that is compatible with their structure maps $U\ra X$ and $V\ra X$, such that $V\ra X$ is an open immersion.
		
	\end{example}
	In fact, we have the following results about crystals over big and small infinitesimal sites.
	\begin{proposition}\label{cry, big and small}
		Let $X$ be a rigid space over $\Bdre$.
		There exists a natural equivalence as below
		\begin{align*}
			\{coherent~crystals~over~X/\Sigma_{e \inf} \} & \Longleftrightarrow \{coherent~crystals~over~X/\Sigma_{e \Inf} \};\\
			\calF & \longmapsto \mu^*\calF; \\
			\mu_*\calG &~ \reflectbox{$\lmt$} ~\calG.
		\end{align*}
	\end{proposition}
	Here we recall from Paragraph \ref{site big small} that the functor $\mu_*$ is the restriction functor from $\Sh(X/\Sigma_{e \Inf})$ to $\Sh(X/\Sigma_{e \inf})$.
	\begin{proof}
		It suffices to show the compositions are equivalences.
		Given a coherent crystal $\calF$ over the small infinitesimal site and a thickening $(U,T)\in X/\Sigma_{e \inf}$, we have
		\begin{align*}
			(\mu_*\mu^*\calF)_T &= (\mu^*\calF)_T\\
			& =\calF_T,
		\end{align*}
		where the second equality follows from the Example \ref{cry, over big site} for the identity map $(U,T)\ra (U,T)$.
		
		Conversely, let $\calG$ be  a coherent crystal over the big infinitesimal site $X/\Sigma_{e \Inf}$.
		For any object $(V,S)\in X/\Sigma_{e \Inf}$, it can always be covered by open affinoid subsets $(V_i,S_i)$ such that each $(V_i,S_i)$ admits a map to a thickening $(U,T)\in X/\Sigma_{e \inf}$.
		\footnote{To see this, 
			we may assume the structure map $V_i\ra V\ra X$ maps into an open affinoid subset $U$ of $X$, 
			where $U$ admits a closed immersion into a smooth rigid space $Y$.
			Then since $V_i \ra S_i$ is a nilpotent thickening, 
			by the smoothness of $Y$, the map $V_i\ra U$ induces a map $S_i$ to some $Y_m$, where $Y_m$ is an infinitesimal neighborhood of $U$ in $Y$.
			Thus the claim follows as $(U,Y_m)$ is in the small site $X/\Sigma_{e \inf}$.}
		We denote $g:S_i \ra T$ to be the associated map of rigid spaces.
		Then by the crystal condition of $\calG$, we have $\calG_{S_i}=g^*(\calG_T)$.
		As an upshot, by the Example \ref{cry, over big site} again we get the following equalities
		\begin{align*}
			(\mu^*\mu_*\calG)_{S_i} & =g^*\left( (\mu_* \calG)_T \right)\\
			& = g^* (\calG_T) \\
			& =\calG_{S_i}.
		\end{align*}
		So we are done.
	\end{proof}

	\textbf{Connection.}
	Recall the definition of general connections for a coherent sheaf.
	\begin{definition}\label{cry conn}
		Let $\calF$ be a coherent sheaf over $X/\Sigma_{e \inf}$.
		A \emph{connection} of $\calF$ is an $\Bdre$-linear morphisms of sheaves
		\[
		\nabla:\calF\rra\calF\otimes_{\calO_{X/\Sigma_{e \inf}}} \Omega_{X/\Sigma_{e \inf}}^1,
		\]
		such that $\nabla$ sends $f\cdot x$ onto $f\nabla(x)+x\otimes df$, for $f$ and $x$ being local sections of $\calO_{X/\Sigma_{e \inf}}$ and $\calF$ respectively.
	\end{definition}
	Here we want to mention that similarly we can define the connection for coherent sheaves over the big infinitesimal site.
	
	Now let $\calF$ be a coherent crystal on $X/\Sigma_{e \inf}$, and let $(U,T)$ be an object in $X/\Sigma_{e \inf}$.
	Then by the definition of crystals, the two projection maps $p_0,p_1:T(1)_1\ra T$ induce an isomorphism of $\calO_{T(1)_1}$-modules:
	\[
	\varepsilon_T:p_0^*\calF_T\simeq \calF_{T(1)_1} \simeq p_1^*\calF_T.
	\]
	This induces a morphism of $\calO_T$-modules given by
	\[
	\xymatrix{\calF_T\ar[r] &\calO_{T(1)_1}\otimes_{\calO_T} \calF_T\ar[r]^{\varepsilon_T}& \calF_T\otimes_{\calO_T} \calO_{T(1)_1},\\
		x\ar@{|->}[r] & 1\otimes x \ar@{|->}[r]& \varepsilon_T(1\otimes x).}
	\]
	Here we identify the sheaf of $\calO_{T(1)_1}$-module $p_1^*\calF_T$ as $\calO_{T(1)_1}\otimes_{\calO_T} \calF_T$ (similarly for $p_0^*\calF_T= \calF_T\otimes_{\calO_T} \calO_{T(1)_1}$). 
	Besides, the pullback of the above sequence along the diagonal map $T\ra T(1)_1$ is the identity, so the image of $\varepsilon_T(1\otimes x)$ under this pullback map is exactly $x$.
	
	The map in fact defines a \emph{canonical connection} structure on the sheaf of the $\calO_T$-module $\calF_T$, by 
	\[
	\nabla_T:\xymatrix{\calF_T\ar[r] & \calF_T\otimes_{\calO_T} \Omega_{T/\Sigma_e}^1.\\
		x\ar@{|->}[r] & \varepsilon_T(1\otimes x)-x\otimes 1.}
	\]
	Here $\calF_T\otimes_{\calO_T} \Omega_{T/\Sigma_e}^1=\calF_T\otimes_{\calO_T} I_T/I_T^2$ can be identified as a subsheaf of $\calF_T\otimes_{\calO_T}\calO_{T(1)_1}$, since $\calO_{T(1)_1}$ decomposes into the direct sum $\calO_T\oplus \Omega_{T/\Sigma_e}^1$ as a left $\calO_T$-module.
	Note that the map satisfies the axiom for the connection, in the sense that for a section $f$ of $ \calO_T$ and $x$ of $ \calF_T$, we have
	\[
	\nabla_T(f\cdot x)=f\nabla_T(x)+x\otimes df,
	\]
	where $df=1\otimes f-f\otimes 1$ is in $\Omega_{T/\Sigma_e}^1$.
	
	We also notice that the above is functorial with respect to $(U,T)\in X/\Sigma_{e \inf}$, in the sense that for a morphism $(i,g):(U_1,T_1)\ra (U_2,T_2)$, we have the following commutative diagram
	\[
	\begin{tikzcd}
		g^*(\calF_{T_2}) \arrow[rr, "g^*\nabla_{T_2}"]\ar[d]&& g^*(\calF_{T_2}\otimes_{\calO_{T_2}} \Omega_{T_2/\Sigma_e}^1) \ar[d]\\
		\calF_{T_1}\arrow[rr, "\nabla_{T_1}"] && \calF_{T_1}\otimes_{\calO_{T_1}} \Omega_{T_1/\Sigma_e}^1.
	\end{tikzcd}
	\]
	In particular, the functoriality leads to the morphism of sheaves over infinitesimal site $X/\Sigma_{e \inf}$
	\[
	\nabla:\calF\rra \calF\otimes_{\calO_{X/\Sigma_{e }}} \Omega_{X/\Sigma_{e \inf}}^1.
	\]
	\begin{definition}
		Let $\calF$ be a coherent crystal over the infinitesimal site $X/\Sigma_{e \inf}$.
		The \emph{canonical connection of $\calF$} is defined as the morphism as above
		\[
		\nabla:\calF\rra \calF\otimes_{\calO_{X/\Sigma_{e }}} \Omega_{X/\Sigma_{e \inf}}^1.
		\]
	\end{definition}
	
	Finally, assume $X\to Y$ is a locally closed immersion of rigid spaces over $\Bdre$ such that $Y$ is smooth, and let $\mathcal{F}$ be a coherent crystal over $X/\Sigma_e$.
	By taking the evaluation of the canonical connection of $\calF$ at the envelope $D=\varinjlim_{n\in \NN} Y_n$, we get a natural $\Bdre$-linear continuous connection over $\mathcal{D}$:
	\[
	\nabla_D: \mathcal{F}_D \longrightarrow \mathcal{F}_D\otimes_{\mathcal{O}_Y} \Omega^1_{Y/\Sigma_e},
	\]
	satisfying the local formula that $\nabla_D(fx)=x\otimes df + f\nabla_D(x)$, where $f$ and $x$ are sections of $\mathcal{O}_D=\varprojlim_{n\in \NN} \mathcal{O}_{Y_m}$ and $\mathcal{F}_D$ respectively.
	More generally, give a coherent sheaf $\mathcal{M}$ over the envelope $D$, we define a \emph{connection} $\nabla_\mathcal{M}$ on $\mathcal{M}$ to be a $\Bdre$-linear continuous map
	\[
	\nabla_\mathcal{M}: \mathcal{M} \longrightarrow \mathcal{M}\otimes_{\mathcal{O}_Y} \Omega^1_{Y/\Sigma_e}
	\]
	that satisfies the same formula above.
	
	\textbf{de Rham complex of a crystal.}\label{cry dR cpx}
	Similar to the flat connection over schemes (\cite{Ber74}, Chapter II, Section 3.2), we can associate a natural de Rham complex to a coherent crystal over the infinitesimal site $X/\Sigma_{e \inf}$ or $X/\Sigma_{e \Inf}$, by the integrability of the canonical connection.

	Let $\calF$ be  a coherent sheaf over $\calO_{X/\Sigma_{e}}$ with a connection $\nabla$.
	For each $k\in \NN_+$ we can associate an $\calO_{X/\Sigma_{e}}$-linear morphism
	\[
	\nabla^k:\calF\otimes_{\calO_{X/\Sigma_{e }}}\Omega_{X/\Sigma_{e}}^k\rra \calF\otimes_{\calO_{X/\Sigma_{e}}}\Omega_{X/\Sigma_{e \inf}}^{k+1},
	\]
	locally given by
	\[
	x\otimes \omega\lmt \nabla(x)\wedge \omega+x\otimes d\omega.
	\]
	This produces a chain of maps
	\[
	(\calF\otimes_{\calO_{X/\Sigma_{e }}}\Omega_{X/\Sigma_{e \inf}}^\bullet,\nabla):=
	\begin{tikzcd}
		0\ar[r] &\calF\arrow[r, "\nabla"] & \calF\otimes_{\calO_{X/\Sigma_e}}\Omega_{X/\Sigma_{e \inf}}^1 \arrow[r,"\nabla^1"] & \cdots.
	\end{tikzcd}
	\]
	
	The connection is called \emph{integrable} if the composition $\nabla^1\circ \nabla$ is zero, under which assumption we call  the above complex the \emph{de Rham complex of $\calF$}.
	Analogously, for a given envelope $D$ of a locally closed immersion $X\to Y$ into a smooth rigid space $Y$, and a coherent module with connection $(\mathcal{M},\nabla_\mathcal{M})$ over $D$, we say $\nabla_\mathcal{M}$ is \emph{integrable} if the composition $\nabla_\mathcal{M}^1\circ \nabla_\mathcal{M}$ is zero.
	
	The following proposition justifies the name of the de Rham complex:
	\begin{proposition}\label{drcpx}
		Let $\calF$ be a coherent crystal over $X/\Sigma_{e \inf}$, and let $\nabla$ be its canonical connection defined in last subsection.
		Then for each $k\in \NN$, we have 
		\[
		\nabla^{k+1}\circ \nabla^k=0.\]
		In particular, the de Rham complex of $\calF$ is in fact a \emph{complex}.
	\end{proposition}
	\begin{proof}
		The proof is identical to that for a crystal over the crystalline site of a scheme, and we refer the reader to\cite[Tag 07J6]{Sta}.
	\end{proof}

	\subsection{Crystal in vector bundles}
	Given a coherent crystal $\calF$ over the infinitesimal site $X/\Sigma_{e \inf}$, we say $\calF$ is a \emph{crystal in vector bundles} if the restriction $\calF_T$ is locally free of finite rank over $\calO_T$, for every object $(U,T)\in X/\Sigma_{e \inf}$.
	In this subsection, we provide a simple criterion when a coherent crystal is a crystal in vector bundles.
	
	\begin{definition}\label{cry flat}
		A coherent crystal $\calF$ is \emph{flat over $\Bdre$} if for any thickening $(U,T)$ in the infinitesimal site with $T$ being flat over $\Bdre$,
		the restriction $\calF_T$ at $T$ is also flat over $\Bdre$.
	\end{definition}
	
	\begin{theorem}\label{cry, vec-bun}
		Let $\calF$ be a coherent crystal over $X/\Sigma_{e \inf}$ that is flat over $\Bdre$ in the sense of \Cref{cry flat}.
		Then $\calF$ is a crystal in vector bundles.
	\end{theorem}
	\begin{proof}
		$~$
		\begin{itemize}
			\item[Step 1]
			We first consider the case when $e=1$.
			Namely let us assume $X$ is defined over $K$ and $\calF$ is a coherent crystal over $X/K_{\inf}$ or $X/K_{\Inf}$, where the flatness of $\calF$ over $\Bdre=K$ is automatic.
			
			For $m\in \NN$, let $T(1)_m$ be the $m$-th infinitesimal neighborhood of $T$ in $T\times_{K} T$.
			The projection map $\mathrm{pr}_0$ to the first factor induces a map from $T(1)_m$ to $T$, via
			\[
			h:\xymatrix{T(1)_m\ar[r] & T\times_{K} T \ar[r]^{~~~~~\mathrm{pr}_0} &T.}
			\]
			Moreover, as $T=T(1)_0$ admits a closed immersion into $T(1)_m$, we can form the following non-commutative diagram of thickenings
			\[
			\xymatrix{&T\ar@{^(->}[rd]&&\\
				T(1)_m\ar[rr]_{\id} \ar^h[ru]&& T(1)_m \ar[r]^h& T. } \tag{$\ast$}
			\]
			Here we notice that the composition $\xymatrix{T\ar@{^(->}[r]&T(1)_m\ar[r]^h&T}$ above is the identity.
			We denote the composition of the map  $h:T(1)_m\ra T$ and the closed immersion $T\ra T(1)_m$ by $g:T(1)_m\ra T(1)_m$.
			Then we get two maps of thickenings of $U$ as follows
			\[
			\xymatrix{ T(1)_m \ar@/^/[r]^{\id} \ar@/_/[r]_g & T(1)_m.}
			\]
			
			Then by the definition of the coherent crystal, pulling back along the above two arrows induces an isomorphism of coherent sheaves over $T(1)_m$
			\[
			g^*\calF_{T(1)_m}\rra \calF_{T(1)_m}.
			\]
			Moreover, by the assumption on $g$, we have
			\[
			g^*\calF_{T(1)_m}= h^*\calF_T. \tag{$**$}
			\]
			
			Now we base change the diagram $(\ast)$ above along the closed immersion $t:\Spa(K)\ra T$ of any $K$-point $t$ of $T$.
			By the construction of the map $h:T(1)_m\ra T$, we get a non-commutative diagram
			\[
			\xymatrix{&t=\Spa(K)\ar@{^(->}[rd]&&\\
				t_m\ar[rr]_{\id} \ar[ur]&&t_m\ar[r]& t,}
			\]
			where $t_m$ is the $m$-th infinitesimal neighborhood of $t$ in $T$, and the base changed map $h:t_m\ra t=\Spa(K)$ is the structure map of $t_m$ over $\Spa(K)$.
			Furthermore, after the base change, 
			the isomorphism $g^*\calF_{T(1)_m}\simeq \calF_{T(1)_m}$ in $(**)$ becomes the following isomorphism of torsion sheaves over $T$ that are supported at $t$
			\[
			\calF_T\otimes t_m\simeq h^*(\calF_T\otimes t).
			\]
			Notice that since the fiber $\calF_T\otimes t$ is flat and finitely generated over $K$, the base change $\calF_T\otimes t$ is in particular a finite free $K$-module.
			As a consequence, the pullback $h^*(\calF_T\otimes t)$, which by the equality above is equal to $\calF_T\otimes_{\mathcal{O}_T} t_m$, is finite free over $t_m$.

			At last, we take the inverse limit of $\calF_T\otimes_{\mathcal{O}_T} t_m$ with respect to $m$.
			By the finite generatedness of $\mathcal{F}_T$ over $\mathcal{O}_T$ and the fact that $T$ is locally noetherian, we know the completion $\varprojlim_{m\in \NN} \calF_T\otimes_{\mathcal{O}_T} t_m$ is equal to the tensor product $\mathcal{F}_{\mathcal{O}_T} \wh\calO_{T,t}$.
			So we see that $\calF_T\otimes_{\calO_T}\wh\calO_{T,t}$ is finite free over the formal completion $\wh\calO_{T,t}$ of the rigid space $T$ at the $K$-point $t$.
			Since $T$ is locally noetherian, the formal completion $\wh\calO_{T,t}$ is isomorphic to the completion $\wh\calO_{T,\ol t}$, where the latter is the formal completion of $T$ along its reduced $K$-valued point $\ol t$.
			In this way, by the faithful flatness of $\wh\calO_{T, t}$ over $\calO_{T, t}$, the stalk of the coherent sheaf $\calF_T$ at $t$ is flat and finitely generated over the local ring, thus projective.
			Hence by the density of $K$-points in $T$, we get the local freeness of $\calF_T$.

			\item[Step 2]
			For general $e\in \NN$, we make the following claim.
			\begin{claim}
				Let $A$ be a flat noetherian algebra over $\Bdre$, and let $M$ be a finite $A$-module that is flat over $\Bdre$.
				Suppose $M/\xi$ is free over $A/\xi$.
				Then $M$ is free over $A$.
			\end{claim}
			\begin{proof}[Proof of Claim]
				We prove the claim by induction on $e$.
				When $e=1$, there is nothing to prove.
				Suppose $e\geq 2$.
				We choose a map of $A$-modules $f:A^{\oplus r} \ra M$ whose reduction mod $\xi$ is an isomorphism.
				Then as $f$ is a map of flat $\Bdre$-modules, the short exact sequence $0\ra \mathrm{B}_{\mathrm{dR},e-1}^+ \ra \Bdre \ra K \ra 0$ induces the following commutative diagram
				\[
				\xymatrix{ 0 \ar[r]& A^{\oplus r} \otimes_{\Bdre} \mathrm{B}_{\mathrm{dR},e-1}^+ \ar[r] \ar[d] & A^{\oplus r}  \ar[r] \ar[d] & A/\xi^{\oplus r} \ar[r] \ar[d] & 0\\
					0 \ar[r] & M\otimes_{\Bdre} \mathrm{B}_{\mathrm{dR},e-1}^+ \ar[r] & M \ar[r] & M/\xi \ar[r] & 0.}
				\]
				Hence the map $f$ is an isomorphism of $A$-modules by induction.
			\end{proof}
			Recall from \Cref{eff-epi} that any infinitesimal thickening $(U,T)\in X/\Sigma_{e \inf}$ locally admits a map to an envelope $D_U(Y)$, where $Y$ is a smooth rigid space over $\Bdre$ and in particular flat over $\Bdre$.
			Notice that by construction, the structure sheaf of the envelope, which is the formal completion of $\mathcal{O}_Y$ along the closed immersion $U\to Y$, is flat over $\Bdre$ as well.
			As a consequence, thanks to the claim above, the sheaf $\mathcal{F}_{D_U(Y)}$ is a vector bundle over $D_U(Y)$, and hence its pullback along $T\to D_U(Y)$ is a vector bundle over $\mathcal{O}_T$.
		\end{itemize}
	\end{proof}
	
	\begin{corollary}\label{cry, vec-bun, K}
		Any coherent crystal over the infinitesimal site $X/K_{\inf}$ or $X/K_{\Inf}$ is a crystal in vector bundles.
	\end{corollary}

	\subsection{Integrable connections over envelope}
	As in the crystalline theory of schemes, there exists an equivalence between the category of coherent crystals over $X/\Sigma_{e \inf}$ and the category of coherent sheaves with integrable connections over the envelope.
	
	Before we state the result, we recall from Remark \ref{env space} that given an envelope $D=\varinjlim_{n\in \NN} Y_n$ of a locally closed immersion $X\ra Y$, we can regard the envelope as a locally ringed space over the adic space $X$.
	The structure sheaf $\calD$ of the envelope is defined as the inverse limit $\varprojlim_{n\in \NN} \calO_{Y_n}$ over $D$.

	\begin{theorem}\label{cry}
		Let $X\ra Y$ be a closed immersion of rigid spaces with $Y$ being smooth over $\Sigma_e$, and let $D=D_X(Y)$ be the envelope of $X$ in $Y$.
		Then we have a natural equivalence of categories:
		\begin{align*}
			\{\text{coherent~crystals~over~}\calO_{X/\Sigma_{e}}\}&\rra \{(M,\nabla)|~M\in\Coh(\calD),~\nabla~\text{integrable~connection}\}\\
			\calF&\lmt (\calF_{D},\nabla_{D}).
		\end{align*}
		Here crystals are over either the big infinitesimal site or small infinitesimal site.
	\end{theorem}
	\begin{corollary}\label{cry, free}
		Let $X\ra Y$ be a closed immersion of rigid spaces with $Y$ being smooth over $\Sigma_e$, and let $D=D_X(Y)$ be the envelope of $X$ in $Y$.
		Then the equivalence above induces the equivalence of the following three categories:
		\begin{itemize}
			\item $\{\text{coherent~crystals~that~are~flat~over~}\Bdre\}$.
			\item $\{\text{crystals~in~vector~bundles}\}$.
			\item $\{(M,\nabla)|~M\in\VVec(\calD),~\nabla\text{~an~integrable~connection}\}$.
		\end{itemize}
	\end{corollary}
	
	Before the proof, we first give a description of the sheaf of differentials over the envelope.
	Below for an affinoid rigid space, we slightly abuse of notations for its sheaf of differentials and its global sections.
	\begin{lemma}\label{diff}
		Let $X=\Spa(A)\ra Y=\Spa(P)$ be a closed immersion of affinoid rigid spaces over $\Bdre$, with $P$ a smooth affinoid algebra over $\Bdre$.
		
		Then we have the following canonical isomorphism
		\[
		\Omega_D^1:=\Omega_{X/\Sigma_{e \inf}}^1(D)\simeq \Omega_{P/\Sigma_e}^1\otimes_P \calD,
		\]
		which is induced from the map $P\ra \calD$.
		
		Moreover, the result is true for $\Omega_{X/\Sigma_{e \Inf}}^1$ over the big infinitesimal site.
	\end{lemma}
	\begin{proof}
		Recall that $\Omega_{D}^1$ is defined as 
		\[
		\Omega_{X/\Sigma_{e \inf}}^1(\varinjlim_{m\in \NN} Y_m),
		\]
		which is equal to the inverse limit of the continuous differentials
		\[
		\Gamma(X,\varprojlim_{m\in \NN} \Omega_{Y_m/\Sigma_e}^1).
		\]
		
		Denote $t_i$ for $i=1,\ldots, r$ to be the \'etale coordinates of $P$.
		This is guaranteed locally by the Jacobian criterion of smoothness, as in \cite[1.6.9]{Hu96}.
		Let $J=(f_1,\ldots, f_s)$ be the kernel of the surjection $P\ra A$, with $f_i$ being its generators.
		Then we have
		\[
		\calO(Y_m)=P/J^{m+1},~~\Omega_{Y_m/\Sigma_e}^1=(\moplus_i \calO(Y_m)dt_i/\sum_{f\in J^{m+1}} \calO(Y_m)df)\tilde.
		\]
		So we get 
		\[
		\calD=\varprojlim_{m\in \NN} P/J^{m+1},~~\Omega_D^1=\varprojlim_{m\in \NN} (\moplus_i \calO(Y_m)dt_i/\sum_{f\in J^{m+1}} \calO(Y_m)df),
		\]
		
		Now consider the natural map $\Omega_{P/\Sigma_e}^1\otimes_P \calD\ra \Omega_D^1$, sending the generator $dt_i$ of $\Omega_{P/\Sigma_e}^1$ onto the $dt_i$ in the limit.
		\begin{itemize}
			\item[Injectivity]
			We first consider the injectivity.
			By writing each $f\in J^{m+1}$ as a finite sum of $af_{j_1}\cdots f_{j_{m+1}}$ for $1\leq j_l\leq r$, each such $df$ is contained in $\sum_j J^m\calO(Y_m)df_j$.
			In particular, the submodule $\sum_{f\in J^{m+1}} \calO(Y_m)df$ of $\moplus \calO(Y_m)dt_i$ is contained the submodule $\sum_j J^m\calO(Y_m)df_j$.
			So it suffices to show the injectivity of
			\[
			\Pi: \moplus_i \calD dt_i\rra \varprojlim_{m\in \NN} (\moplus_i \calO(Y_m)dt_i/\sum_j J^m\calO(Y_m)df_j).
			\]
			
			However, the kernel of each $\moplus_i \calD dt_i\ra \moplus_i \calO(Y_m)dt_i/\sum_j J^m\calO(Y_m)df_j$ equals to $\sum J^{m+1}\calD dt_i+\sum_j J^mdf_j$, which is contained in $\moplus_i J^m  \calD dt_i$.
			In particular, any element $\sum_i g_i dt_i$ in the $\ker(\Pi)$ is contained in the ideal
			\[
			\bigcap_m \moplus_i J^m \calD dt_i.
			\]
			But note that $\calD$ is defined as the $J$-adic completion of $P$, which implies the above ideal is zero.
			So we get the injectivity.
			
			\item[Surjectivity]
			We can write $\Omega_{P/\Sigma_e}^1\otimes_P \calD$ as the limit $\moplus_i \calD dt_i=\varprojlim_{m\in \NN} (\moplus_i \calO(Y_m)dt_i)$.
			Then for each $m$ the map $\moplus_i\calO(Y_m)dt_i\ra \moplus_i \calO(Y_m)dt_i/\sum_{f\in J^{m+1}} \calO(Y_m)df$ is surjective.
			For each $m$, he kernel of the map is $M_m:=\sum_{f\in J^{m+1}} \calO(Y_m)df$, whose image in $M_{m-1}$ is zero.
			Thus we get the surjectivity, by the pro-acyclicity of the kernel.
		\end{itemize}

	\end{proof}
	
	The local freeness of the differential sheaf over the envelope allows us to give a more explicit description of the connection associated with a crystal.
	We assume $X=\Spa(A)\ra \Spa(P)=Y$ be a closed immersion of affinoid rigid spaces over $\Sigma_e$, such that $Y$ is smooth over $\Sigma_e$ with a local coordinates $\{t_1,\ldots,t_r\}$.
	Let $M$ be a coherent sheaf over $\calD$ together with a connection $\nabla$ over $\Bdre$.
	By Lemma \ref{diff} above, the restriction of the infinitesimal differential over $D=D_X(Y)$ is free over $\calD=\calO_{X/\Sigma_{e}}(D)=\varprojlim\calO(Y_m)$ with a basis $dt_i$ for $i=1,\ldots,r$.
	So for any section $x\in M$, we have
	\[
	\nabla(x)=\sum_i\nabla_i(x)\otimes dt_i,
	\]
	where $\nabla_i:M\ra M$ is an $\Bdre$-linear derivation map.
	
	Now we assume $(M,\nabla)$ is integrable.
	We compose $\nabla$ with $\nabla^1$, and get
	\begin{align*}
		\nabla^1(\nabla(x))&=\sum_i \nabla^1(\nabla_i(x)\otimes dt_i)\\
		&=\sum_j\sum_i \nabla_j(\nabla_i(x))\otimes dt_j\wedge dt_i+\sum_i \nabla_i(x)\otimes d(dt_i)\\
		&=\sum_j\sum_i \nabla_j(\nabla_i(x))\otimes dt_j\wedge dt_i.
	\end{align*}
	By the local freeness of $\Omega_D^1$, the element $dt_j\wedge dt_i$ for $j<i$ forms a basis of $\Omega_D^2$.
	So we can rewrite the above as 
	\[
	\nabla^1(\nabla(x))=\sum_{j<i} (\nabla_j(\nabla_i(x))-\nabla_i(\nabla_j(x)))\otimes dt_j\wedge dt_i.
	\]
	By the integrability condition of $\nabla$, the above vanishes for any $x\in \calF_D$.
	So we obtain the following equalities 
	\[
	\nabla_i\circ \nabla_j=\nabla_j\circ \nabla_i.
	\]
	Here we note that the commutativity allows us to write the composition of a finite amount of $\nabla_i$ as
	\[
	\prod_{E=(e_i)} \nabla_i^{e_i},
	\]
	where $E=(e_i)$ is a tuple of non negative integers parametrized by $i$.

	Now we are ready for the proof of Theorem \ref{cry}.
	\begin{proof}[Proof of Theorem \ref{cry}]
		For a crystal $\calF$ over $\calO_{X/\Sigma_{e}}$, we can equip it with its canonical connection $\calF$, which is integrable by Proposition \ref{drcpx}.
		So by taking the associated coherent sheaf of $\calF$ over $D$, we get a coherent sheaf $\calF_D$ together with an integrable connection $\nabla_D$.
		
		Conversely, let $M$ be a coherent sheaf over $\calD$ with an integrable connection $\nabla$.
		By the smoothness of $Y$ over $\Sigma_e$, any object in $X/\Sigma_{e \inf}$ can be covered by an open affinoid covering where each piece admits a map to $(X,Y)$.
		We assume $(U,T)$ is an affinoid thickening fitting in the commutative diagram
		\[
		\xymatrix{ U\ar[r]\ar[d] & T\ar[d]^g\\
			X\ar[r]& Y}
		\]
		Since $T$ is an nilpotent extension of $U$, the map $g:T\ra Y$ factors through the envelope $D=\varinjlim_{m\in \NN} Y_m$ of $X$ in $Y$.
		We denote this map by $f:T\ra D$.
		Then we get a coherent sheaf $f^* M=M\otimes_\calD \calO_T$ over $T$.
		
		Now we make the following claim:
		\begin{claim}\label{cry, formula}
			The pullback $f^*M$ over $T$ is independent of the choice of $f:T\ra D$.
			
			More precisely, let $f_1,f_2,f_3:T\ra D$ be any three maps induced produced as above.
			Then there exists natural isomorphisms of coherent sheaves $h_{ij}:f_i^*M\ra f_j^*M$ over $\calO_T$ such that
			\[
			h_{23}\circ h_{12}=h_{13}.
			\]
		\end{claim}
		We first grant the claim.
		For each thickening $(U,T)$, we pick an arbitrary covering $(U_i,T_i)$ of $(U,T)$, where $(U_i,T_i)$ admits a map to $(X,Y)$.
		Then we get the collection of coherent sheaves $f^*_iM$ over each $T_i$.
		The claim allows us to produce a transition isomorphism for each restriction of $f_i^*M$ on $T_i\cap T_j$, and they satisfy the cocycle condition when restricted at $T_i\cap T_j\cap T_k$.
		Hence by gluing them together, we get a coherent sheaf $\calF_T$ over $(U,T)$.
		This produces a sheaf $\calF$ over the infinitesimal site.
		Moreover, the coherent sheaf $\calF$ is in fact a coherent crystal, namely the pullback $g^*\calF_{T_2}\simeq \calF_{T_1}$ for any map  $(i,g):(U_1,T_1)\ra (U_2,T_2)$ in $X/\Sigma_{e \inf}$.
		This comes from the independence in the claim again, by taking a composition with a map $T_2\ra D$.
		
		Finally, we check that the functors are quasi-inverse to each other.
		We start with a coherent crystal $\mathcal{F}$ over the infinitesimal site, and let $(M=\mathcal{F}_D,\nabla_M)$ be the associated integrable connection over the envelope $D$ defined in the first paragraph of the proof.
		Then the crystal $\mathcal{F}'$ induced by $(M,\nabla_M)$ is locally defined by assigning the module $M\otimes_\mathcal{D} \mathcal{O}_T$ to a thickening $(U,T)$, for any morphism $(U,T)\to (X,D)$ in the infinitesimal site (which locally exists as the envelope covers the final object).
		This coincides with the value of coherent crystal $\mathcal{F}$ at $(U,T)$, thanks to the equality $M=\mathcal{F}_D$.
		Conversely, we let $(M,\nabla_M)$ be an integrable connection over the envelope $D$ and let $\mathcal{F}$ be the associated coherent crystal over the infinitesimal site defined in the second paragraph of the proof.
		The integrable connection over $D$ that is induced by the crystal $\mathcal{F}$ is then defined on the module $M'=\mathcal{F}_D$, which by assumption is also equal to the module $M$.
		Moreover, the induced connection on $M'$ also coincides with the input $\nabla_M$, thanks to the concrete calculation below \Cref{cry conn}.
		So we are done.
		\begin{proof}[Proof of Claim]
			We at last deal with Claim.
			Let $\varphi_j:\calD\ra \calO_T$ be maps of structure sheaves induced from $f_j:T\ra D$.
			We define $h_{jk}$ to be the $\calO_T$-linear map given by
			\[
			x\otimes 1\lmt \sum_{E=(e_i)} (\prod_i \nabla_i^{e_i})(x)\otimes \frac{(\varphi_j(t_i)-\varphi_k(t_i))^{e_i}}{e_i!}.
			\]
			Since $T$ is a nilpotent extension of $U$, for each $t\in \calD$, the difference $\varphi_j(t)-\varphi_k(t)$ is nilpotent in $\calO_T$.
			In particular, the above sum is only finite.
			At last, by the general equality
			\[
			\sum_{n=0}^N \frac{u^n}{n!}\cdot \frac{v^{N-n}}{(N-n)!}=\frac{(u+v)^N}{N!},
			\]
			we have $h_{23}\circ h_{12}=h_{13}$.
		\end{proof}
	\end{proof}

	\section{Cohomology over $\Bdre$}\label{sec coh}
	In this section, we compute the cohomology of crystals over $X/\Sigma_{e \inf}$ using the de Rham complex over the envelope.
	Our strategy is to construct a double complex computing the \v{C}ech-Alexander complex and the de Rham complex in two separate directions, as in \cite{BdJ}.
	
	\begin{remark}
		Before we start, we mention that though our focus is rigid spaces over $\Bdre$, the discussion in this section works alphabetically for cohomology of crystals over $X/K_{0, \inf}$, where $K_0$ is an arbitrary $p$-adic complete non-archimedean field and $X$ is a rigid space over $K_0$.
	\end{remark}

	\subsection{Cohomology of crystals over affinoid spaces}
	We first compute the cohomology of crystals over $X/\Sigma_{e \inf}$, for $X$ being an affinoid rigid space over $\Sigma_e$.
	
	Let $X=\Spa(A)$ be an affinoid rigid space over $\Sigma_e$, together with a closed immersion $X\ra Y=\Spa(P)$ for a smooth affinoid rigid space $Y$ over $\Bdre$.
	Denote by $D$  the envelope of $X$ in $Y$ (Definition \ref{env}), by $\calD$  its structure sheaf $\varprojlim_m \calO_{Y_m}$ (where $Y_m$ is the $m$-th infinitesimal neighborhood of $X$ in $Y$), and by $J$  the kernel ideal for $\calD\ra \calO_X$.
	By construction, the kernel ideal $J$ is equal to the evaluation of the infinitesimal ideal sheaf $\mathcal{J}_{X/\Sigma_e}$ at the envelope $D$.
	We write $\Omega_{D}^i$ as the group of differentials $\Omega_{X/\Sigma_{e\inf}}^i(D)$, which is equal to the inverse limit of continuous differentials
	\[
	\varprojlim_m \Omega_{Y_m/\Sigma_e}^i.
	\]
	By Lemma \ref{diff}, $\Omega_{D}^i$ is isomorphic to the tensor product $\Omega_{Y/\Sigma_e}^i\otimes_{\calO_{Y}} \calD$, and is in particular locally free over $\calD$.

	We then take the section of the infinitesimal de Rham complex $(\calF\otimes_{\calO_{X/\Sigma_{e}}}\Omega_{X/\Sigma_{e \inf}}^\bullet,\nabla)$ at $D$, and get a chain complex of $\Bdre$-modules
	\[
	(M\otimes \Omega_D^\bullet,\nabla_D):= \begin{tikzcd}
		0\ar[r] & M \arrow[r, "\nabla"] & M\otimes_{\calD} \Omega_D^1 \arrow[r, "\nabla^1"] & \cdots,
	\end{tikzcd}
	\]
	where $M$ is the evaluation  $\calF(D)$ of $\calF$ at the envelope $D$.
	The complex is naturally filtered by the infinitesimal filtration, whose $i$-th filtration is the subcomplex
	\[
	0 \rra J^iM \rra J^{i-1}M\otimes_\calD \Omega_D^1 \rra J^{i-2}M\otimes_\calD \Omega_D^2 \rra \cdots.
	\]
	
	We also recall that by taking the powers of the infinitesimal ideal sheaf $\mathcal{J}_{X/\Sigma_e}$, we get a natural filtration on the structure sheaf $\mathcal{O}_{X/\Sigma_e}$ (cf. discussion before \Cref{base K0}) and hence on $\mathcal{F}$, by defining the $i$-th filtration as $\mathcal{J}_{X/\Sigma_e}^i\mathcal{F}$.
	This in particular induces a filtration on the cohomology complex $R\Gamma(X/\Sigma_{e \inf},\calF)$, and we call the latter the \emph{infinitesimal filtration} on the infinitesimal cohomology.
	Our main theorem in this subsection is the following:
	\begin{theorem}\label{aff-coh}
		Let $X$, $Y$, $\calF$ and $M$ be as above.
		Then we have a natural filtered isomorphism in the filtered derived category of abelian groups:
		\[
		R\Gamma(X/\Sigma_{e \inf},\calF)\rra (M\otimes \Omega_D^\bullet,\nabla_D),
		\]
		where the left side is filtered by the infinitesimal filtration.
	\end{theorem}
	\begin{remark}\label{coh, big small, cry}
		Note that by the Corollary \ref{coh, big small}, the above is also isomorphic to the cohomology of the crystal $\calG$ over the big infinitesimal site, when $\calF=\mu_*\calG$ is the restriction of $\calG$ defined over the big site $X/\Sigma_{e \Inf}$.
	\end{remark}
	The rest of this subsection will be devoted to the proof of the theorem.\\
	
	Let us first fix some notations for this section. 
	Denote by $D(n)$  the envelope of $X$ in the $(n+1)$-fold self product of $Y$ over $\Sigma_e$.
	When $n=0$, we write $D(0)$ as $D$.
	The simplicial object $D(\bullet)$ forms a hypercovering of the final object in $\Sh(X/\Sigma_{e \inf})$, as in Proposition \ref{env, coh}.
	
	We fix a coherent crystal $\calF$ on $X/\Sigma_{e \inf}$. 
	Denote $M(n)$ to be the group of sections $\calF(D(n))$ of $\calF$ at $D(n)$, $\calD(n)$ to be $\calO_{X/\Sigma_{e }}(D(n))$,
	$J(n)$ to be the kernel for $\calD(n)\ra \calO_X$, and $\Omega_{D(n)}^i$ to be $\Omega_{X/\Sigma_{e\inf}}^i(D(n))$.
	When $n=0$, we use $M$ and $\Omega_D^i$ to abbreviate $M(0)$ and $\Omega_{D(0)}^i$.
	Here we recall that $\Omega_{D(n)}^i=\Omega_{X/\Sigma_{e \inf}}^i(D(n))$ is isomorphic to the tensor product $\Omega_{Y(n)/\Sigma_e}^i\otimes_{\calO_{Y(n)}} \calD(n)$ (Lemma \ref{diff}), and is in particular locally free over $\calD(n)$. 
	
	\textbf{\v Cech-Alexander complex.}
	First we introduce the \v Cech-Alexander complex of a coherent $\calO_{X/\Sigma_{e}}$ sheaf $\calF$ (not necessarily to be a crystal).
	
	We define $M(\bullet)$ to be the filtered cosimplicial cochain complex
	\[
	M(\bullet):=(\calF(D(0))\rra \calF(D(1))\rra \cdots),
	\]
	where the coboundary map is given by the alternating sum of degeneracy maps, and the filtration is the infinitesimal filtration whose $i$-th filtration at $D(n)$ is $J(n)^i\cdot \calF(D(n))$. 
	It is called the \emph{\v Cech-Alexander complex of $\calF$}.
	
	\begin{proposition}\label{cech}
		Let $\calF$ be a coherent infinitesimal sheaf of  $\calO_{X/\Sigma_e}$-modules as above.
		Then we have a functorial filtered isomorphism
		\[
		R\Gamma(X/\Sigma_{e \inf}, \calF)\simeq M(\bullet),
		\]
		in the filtered derived category of abelian groups.
	\end{proposition}
	\begin{proof}
		We first notice that by the filtered enhancement 
		of Proposition \ref{env, coh}, 
		\footnote{More precisely, the filtered complex $R\Gamma(X/\Sigma_{e \inf}, \calF)$ is defined as the derived limit of the filtered modules $\bigl( (\mathcal{J}_{X/\Sigma_e}^i\mathcal{F})(U,T) \bigr)_i$, ranging over all the infinitesimal thickenings $(U,T)\in X/\Sigma_{e \inf}$.
			In particular, for each $[n]\in \Delta^\op$, by the limit presentation, there is a functorial map of filtered objects from $R\Gamma(X/\Sigma_{e \inf}, \calF)$ to $M(n)$. The functoriality of the construction (with respect to the simplicial diagram $\Delta^\op$) in particular induces a filtered map $R\Gamma(X/\Sigma_{e \inf}, \calF)\to M(\bullet)$. To check it is an isomorphism, it suffices to do so on the $i$-th filtration component for each $i\in \mathbb{Z}$, which is implied by op.\ cit..}
		we get an isomorphism of filtered complexes
		\[
		R\Gamma(X/\Sigma_{e \inf},\mathcal{F})=R\Gamma(D(\bullet),\mathcal{F})= R\lim_{[n]\in \Delta} R\Gamma(D(n), \mathcal{F}).
		\]
		
		Denote by $Y(n)_m$  the $m$-th infinitesimal neighborhood of $X$ in $Y(n)$.
		Since $X$ is the common closed analytic subspace of every $Y(n)_m$, $Y(\bullet)_m$ forms a simplicial object in $X/\Sigma_{e \inf}$ with $D(\bullet)=\varinjlim_{m\in \NN} h_{Y(\bullet)_m}$.
		This leads to the equality
		\[
		R\Gamma(D(\bullet),\mathcal{F})=R\varprojlim_{m\in \NN}R\Gamma(Y(\bullet)_m,\mathcal{F}).
		\]
		Notice that  for each $n$, the rigid space $Y(n)_m$ is affinoid, and the covering of a given infinitesimal thickening $(X,Y(n)_m)$ is defined by analytic covering of the rigid space $Y(n)_m$.
		As a consequence, by the vanishing of the analytic cohomology for coherent sheaves over affinoid rigid spaces in the positive degrees, we know
		\[
		R\Gamma(Y(\bullet)_m,\calF)=\Gamma(Y(\bullet)_m,\calF).
		\]
		Furthermore, by the coherence of $\calF$ and the noetherianity of $\calO(Y(n)_m)$, for each $n\in \NN$ the inverse system $\Gamma(Y(n)_m,\calF)$ satisfies the Mittag--Leffler condition.
		In this way, we get 
		\begin{align*}
			R\varprojlim_{m\in \NN}R\Gamma(Y(\bullet)_m,\calF)&=R\varprojlim_{m\in \NN} \Gamma(Y(\bullet)_m,\calF)\\
			&=\varprojlim_{m\in \NN} \Gamma(Y(\bullet)_m,\calF)\\
			&=\Gamma(\varinjlim_{m\in \NN} Y(\bullet)_m,\calF)\\
			&=M(\bullet).
		\end{align*}
		
	\end{proof}
	
	\textbf{\v Cech-Alexander and the de Rham.}
	We then connect the \v Cech-Alexander complex with the de Rham complex together.
	
	Consider the section of the de Rham complex $(\calF\otimes_{\calO_{X/\Sigma_{e}}}\Omega_{X/\Sigma_{e \inf}}^\bullet,\nabla)$ at the simplicial space $(D(n))_{[n]\in \Delta^\op}$:
	\[
	\Delta^\op \ni [n] \longmapsto (M(n)\otimes_{\calD(n)}\Omega_{D(n)}^\bullet, \nabla).
	\]
	This produces a double complex $M^{n,m}=M(n)\otimes_{\calD(n)}\Omega_{D(n)}^m$ in the first quadrant, with the horizontal coboundary map given by the alternating sum of degeneracy maps for simplicial space $D(\bullet)$, and the vertical coboundary map being the de Rham differential $\nabla^m$.
	Note that the first column $M^{0,\bullet}$ of this double complex is the de Rham complex $M\otimes_{\calD}\Omega_{D}^\bullet$, while the first row $M^{\bullet,0}$ is the \v Cech-Alexander complex $M(\bullet)$.
	So this provides a natural framework for those two types of complexes that we care about.
	
	Moreover, the double complex is naturally filtered via the infinitesimal filtration $\calO_{X/\Sigma_e} \supset \calJ_{X/\Sigma_e} \supset \calJ_{X/\Sigma_e}^2 \cdots$.
	This is a descending filtration on the double complex, compatible with the cosimplicial structure, such that the $i$-th filtration on the $n$-th column is the differential complex
	\[
	J(n)^i \rra J(n)^{i-1}\Omega_{D(n)}^1 \rra \cdots \rra J(n)^0\Omega_{D(n)}^i \rra \Omega_{D(n)}^{i+1} \rra \cdots,
	\]
	as a subcomplex of $\Omega_{D(n)}^\bullet$.
	Here we recall $J(n)$ is the kernel ideal of the surjection $\calD(n)\ra \calO_X$, defined as $\calJ_{X/\Sigma_e}(D(n))$.
	Note that when $X=Y$ are smooth over $\Bdre$, the filtration on $\Omega_{D}^\bullet=\Omega_{X/\Bdre}^\bullet$ is the usual Hodge filtration.
	
	Furthermore, there are two canonical $E_1$ spectral sequences associated with the double complex $M^{n,m}$ (\cite[Tag 0130]{Sta}), with the formations given by
	\begin{align*}
		'E_1^{p,q}&=\cH^q(M(p)\otimes_{\calD(p)}\Omega_{D(p)}^\bullet);\\
		''E_1^{p,q}&=\cH^q(M(\bullet)\otimes_{\calD(\bullet)}\Omega_{D(\bullet)}^p).
	\end{align*}
	Both of those two spectral sequences converge to the hypercohomology of the total complex (\cite[Tag 0132]{Sta}).
	The same applies when we replace the double complex by its $i$-th infinitesimal filtration.
	
	Now we make the following two Lemmas about degeneracy of those two spectral sequences:
	\begin{lemma}\label{lemA}
		For each $p>0$, the filtered cochain complex associated with the cosimplicial complex with its infinitesimal filtration
		\[
		M(\bullet)\otimes_{\calD(\bullet)}\Omega_{D(\bullet)}^p
		\] 
		is filtered acyclic.
	\end{lemma}
	\begin{lemma}\label{lemB}
		Any degeneracy map $D(p)\ra D$ induces an filtered isomorphism of the following two de Rham complexes
		\[
		M\otimes_{\calD}\Omega_D^\bullet\rra M(p)\otimes_{\calD(p)}\Omega_{D(p)}^\bullet,
		\]
		that is functorial with respect to the crystal $\calF$.
		Here $M=\mathcal{F}_D$ is the evaluation of the crystal at the envelope $D$.
	\end{lemma}
	
	We first assume those two lemmas above.
	By Lemma \ref{lemA}, the spectral sequence $''E_1^{p,q}$ is filtered degenerated in its first page and is convergent to the cohomology of the \v Cech-Alexander complex $M(\bullet)$ with its infinitesimal filtration.
	
	On the other hand, Lemma \ref{lemB} implies that the horizontal coboundary map of $'E_1^{p,q}$ is given by
	\[
	\begin{tikzcd}
		\cH^q(M\otimes_{\calD}\Omega_{D}^\bullet) \arrow[r,"0"] & \cH^q(M(1)\otimes_{\calD(1)}\Omega_{D(1)}^\bullet) \arrow[r, "1"] & \cH^q(M(2)\otimes_{\calD(2)}\Omega_{D(2)}^\bullet) \arrow[r,"0"] &\cdots.
	\end{tikzcd}
	\]
	From this, the second page of $'E_1^{p,q}$ vanishes everywhere except for the column $'E_2^{0,\bullet}$, which is exactly the infinitesimal filtered de Rham complex $M\otimes_{\calD}\Omega_{D}^\bullet$.
	
	In this way, since both of those two spectral sequences are convergent to the total complex in the filtered derived category, we get the filtered isomorphism between the de Rham complex $M\otimes_{\calD}\Omega_{D}^\bullet$ and the \v Cech-Alexander complex $M(\bullet)$.
	So by Proposition \ref{cech}, we get Theorem \ref{aff-coh}.
	Here the functoriality follows from that of Lemma \ref{lemB} and Proposition \ref{cech}.
	
	\textbf{Proof of Lemma \ref{lemA}.}
	To complete the proof of Theorem \ref{aff-coh}, we first prove Lemma \ref{lemA} in this paragraph.
	
	We first give a proof for the special case when $\calF$ is the structure sheaf and $p=1$.
	\begin{lemma}\label{cosim}
		The cosimplicial complex 
		\[
		\Omega_{D}^1\rra \Omega_{D(1)}^1\rra \Omega_{D(2)}^1\rra \cdots \tag{$\ast$}
		\]
		is locally (with respect to the topology of $X$; cf. \Cref{env space}) cosimplicial homotopic to zero, as filtered cosimplicial abelian groups.
	\end{lemma}
	Before the proof of this Lemma, we first recall that a \emph{cosimplicial homotopic equivalence} of two maps $f,g:U\ra V$ is defined as a cosimplicial morphism
	\[
	h:U\ra \Hom([1],V),
	\]
	such that
	\[
	h\circ s_0=f,~ h\circ s_1=g,
	\]
	where $s_i:[0]\ra [1]$ are two co-face maps.
	
	A cosimplicial object $U$ is called \emph{cosimplicial homotopic to zero} if its identity map is cosimplicial homotopic to the zero map.
	Here we note that any additive functor $F$ that sends cosimplicial objects to cosimplicial objects will preserve the cosimplicial homotopic equivalence.
	
	We refer the reader to \cite[Tag 019U]{Sta}  for the discussion about cosimplicial homotopic equivalence.
	
	\begin{proof}
		We first recall that since $D(n)$ is the envelope of $X=\Spa(A)$ in the $n+1$-folded self product of $Y$ over $\Sigma_e$, by Lemma \ref{diff} above, we have
		\[
		\Omega_{D(n)}^1=\Omega_{P^{\otimes n+1}/\Sigma_e}^1\otimes_{P^{\otimes n+1}} \calD(n).
		\]
		Besides, any cosimplicial boundaries map $P^{n+1}\ra P^{l+1}$ induces a map $ \Omega_{D(n)}^1\ra \Omega_{D(l)}^1$.
		So the cosimplicial complex $(\ast)$ is the tensor product of the cosimplicial complex $\Omega_{P^{\otimes \bullet+1}/\Sigma_e}^1$ along the cosimplicial ring homomorphism
		\[
		P^{\otimes \bullet+1}\rra \calD(\bullet).
		\]
		Moreover, the $i$-th filtration of the cosimplicial complex $(\ast)$ is 
		\[
		J^{i-1} \Omega_D^1 \rra J(1)^{i-1} \Omega_{D(1)}^1 \rra  J(2)^{i-1}  \Omega_{D(2)}^1 \rra \cdots,
		\]
		which is isomorphic to the fiber of a map between cosimplicial tensor products
		\[
		\left (\Omega_{P^{\otimes \bullet+1}/\Sigma_e}^1\right) \otimes_{P^{\otimes \bullet+1}} \calD(\bullet) \rra \left (\Omega_{P^{\otimes \bullet+1}/\Sigma_e}^1 \right) \otimes_{P^{\otimes \bullet+1}} \left (\calD(\bullet)/J(\bullet)^{i-1}\right).
		\]
		Thus to show the filtered acyclicity, it suffices to show that the cosimplicial module $\Omega_{P^{\otimes \bullet+1}/\Sigma_e}^1$ is homotopic equivalent to zero.
		Here we notice that when $P=\Bdre\langle x_i\rangle$, each $P^{\otimes n+1}$ is a ring of convergent power series over $\Bdre$, and the proof is totally identical to the case for polynomial rings, which is done in \cite{BdJ}, Example 2.16.
		In general, when $P$ is smooth over $\Bdre$, it locally admits an \'etale morphism to an $\Bdre\langle x_i\rangle$.
		So the exactness is true locally, hence globally by a \v Cech-complex arguement associated with a covering.
		
	\end{proof}
	\begin{proof}[End of the proof for Lemma \ref{lemA}]
		Consider the filtered complex $(\ast)$ as below:
		\[
		\Omega_{D}^1\rra \Omega_{D(1)}^1\rra \Omega_{D(2)}^1\rra \cdots \tag{$\ast$}.
		\]
		As the statement is local, by shrieking to an open subsets of $X$ and $Y$ if necessary, 
		we could assume the complex $(*)$ is filtered cosimplicial homotopic to zero as in Lemma \ref{cosim}.
		Then we apply the $p$-th wedge product functor, and the tensor product functor $M(\bullet)\otimes_{\calD(\bullet)}-$ successively to the cosimplicial complex $(\ast)$, then the resulted cosimplicial complex is exactly the one in Lemma \ref{lemA}.
		But note that since any additive cosimplicial functor preserves the cosimplicial homotopic equivalence, the resulted complex is also filtered homotopic to zero.
		So we are done.
	\end{proof}

	\textbf{Proof of Lemma \ref{lemB}.}
	In this paragraph, we prove Lemma \ref{lemB}.

	We first provide  the following simpler description of the envelope $\calD(p)$:
	\begin{lemma}\label{env, D(p)}
		Assume the $\Bdre$-algebra $P$ admits an \'etale map from a ring of convergent power series $\Bdre\langle x_1,\ldots, x_r\rangle$.
		Then the map of global sections of structure sheaves $\calD\ra \calD(p)$ associated with the degeneracy map $D(p)\ra D$ induces an isomorphism
		\[
		\calD(p)\simeq\calD[[\delta_{i,j}, 1\leq i\leq p, 1\leq j\leq r]],
		\]
		where the right side is a ring of formal power series over the topological ring $\calD$.
		
	\end{lemma}
	The notation is explained as follows.
	The projection map $Y(p)\ra Y$ of the $p+1$-th self product onto the first copy induces the zero-th degeneracy map $D(p)\ra D$.
	Then we can rewrite $P^{\otimes p+1}$ as $P\langle \delta_{i,j},~1\leq i\leq p,~1\leq j\leq r\rangle$, where $\delta_{i,j}$ is defined as $x_j\otimes 1\otimes\cdots \otimes 1-1\otimes \cdots \otimes x_j\otimes  \cdots \otimes1$, with $x_j$ being in the $i$-th copy of $P$ in the second term.
	\begin{proof}
		We first consider the case when $P$ is equal to the convergent power series ring.
		
		Denote by $J$  the kernel of the surjection $P\ra A$, and let $I$ be the kernel of the map $P^{\otimes p+1}\ra P$.
		By construction, the ring of sections $\calD(p)=\calO(D(p))$ is equal to the inverse limit
		\[
		\varprojlim_{m\in \NN} P^{\otimes p+1}/(J\otimes1\otimes \cdots \otimes1, I)^m,
		\]
		while $\calD=\calO(D)$ is $\varprojlim_{m\in \NN} P/J^m$.
		So to prove the lemma, it suffices to notice that the above inverse limit is the same as the inverse limit 
		\[
		\calD(p)=\varprojlim_{n\in \NN}(\varprojlim_{m\in \NN} P^{\otimes p+1}/(J\otimes 1\otimes \cdots\otimes 1)^m)/\ol I^n,
		\]
		where $\ol I$ is the image of $I$ along the map $P^{\otimes p+1}\ra \varprojlim_{m\in \NN} P^{\otimes p+1}/(J\otimes \cdots\otimes 1)^m$.
		
		In fact, we have the following more general result:
		\begin{claim}
			Let $R$ be a noetherian ring, and $I, ~J$ be two ideals in $R$.
			Then we have a canonical isomorphism
			\[
			\varprojlim_{m\in \NN}(\varprojlim_{n\in \NN} R/I^n)/\ol J^m\rra 	\varprojlim_{m\in \NN} R/(I,J)^m,
			\]
			where $\ol J$ is the ideal generated by the image of $J$ along the map $R\ra \varprojlim_{m\in \NN} R/I^n$.
		\end{claim}
		\begin{proof}[Proof of Claim]
			First notice that the sequence of ideals $\{(I,J)^m\}$ and $\{(I^m,J^m)\}$ are cofinal to each other, since 
			\[
			(I^{2m},J^{2m})\subset (I,J)^{2m}=(I^iJ^{2m-i},~0\leq i\leq 2m)\subset (I^m,J^m).
			\]
			So the right side $\varprojlim_{m\in \NN} R/(I,J)^m$ can be replaced by $\varprojlim_{m\in \NN} R/(I^m,J^m)$.

			Then we  notice that the $R$-algebra $A:=\varprojlim_{m\in \NN}(\varprojlim_{n\in \NN} R/I^n)/\ol J^m$ is $(I,J)$-adic complete over $R$:
			To show this, by the \cite[Tag 0DYC]{Sta}, it suffices to show that the ring $(\varprojlim_{n\in \NN} R/I^n)/J$ is $I$-adic complete.
			We then note that $(\varprojlim_{n\in \NN} R/I^n)/J=\wh R\otimes_R R/J$, where $\wh R$ is the $I$-adic completion of $R$.
			Since $R/J$ is a finitely generated module over $R$, the tensor product $\wh R\otimes_R  R/J$ is the same as $I$-adic completion of $R/J$.
			Thus the $R$-algebra  $A$ is $(I,J)$-adic complete.
			In particular, we have
			\[
			A=\varprojlim_{l\in \NN} A/(I^l,J^l).
			\]
			
			At last, the quotient ring $A/(I^m,J^m)$ is given as
			\begin{align*}
				A/(I^l,J^l)&=(\varprojlim_{m\in \NN}(\varprojlim_{n\in \NN} R/I^n)/\ol J^m)/(I^l,J^l)\\
				&=(\varprojlim_{n\in \NN} R/I^n)/(\ol I^l,\ol J^l)\\
				&=R/(I^l,J^l).
			\end{align*}
			So we get 
			\begin{align*}
				\varprojlim_{m\in \NN}(\varprojlim_{n\in \NN} R/I^n)/\ol J^m&=:A\\
				&\simeq\varprojlim_{l\in \NN} A/(I^l,J^l)\\
				&=\varprojlim_{l\in \NN} R/(I^l,J^l).
			\end{align*}
		\end{proof}
		
		At last, let us assume $P$ is a smooth affinoid algebra that admits an \'etale map to the ring of convergent power series.
		By the claim above and the noetherianity of the envelope (\cite[Lem.\ 13.4.(ii)]{BMS}), the lemma is reduced to showing that the formal completion $\calD(p)$ for $P^{\otimes p+1} \ra P$ is isomorphic to $P[[\delta_{i,j}]]$, which is proved in \cite[Lem.\ 13.12.(ii)]{BMS}.
	\end{proof}
	Our next observation is about the Euler sequence for the degeneracy map $D(p)\ra D$.
	Denote by  $\Omega_{D(p)/D}^1$  the module of continuous differentials of $\calD(p)$ over $\calD$ under the $(\Delta(p))$-adic topology, where $\Delta(p)$ is the kernel ideal for the diagonal map $\calD(p) \ra \calD$.
	Then we have
	\begin{lemma}\label{eul}
		The Euler sequence for the projection map $Y(p)\ra Y$ over $\Sigma_e$ induces a natural exact sequence of free $\calD(p)$-module:
		\[
		0\rra \Omega_D^1\otimes_\calD \calD(p)\rra \Omega_{D(p)}^1\rra \Omega_{D(p)/D}^1\rra 0,
		\]
		where the map $\Omega_D^1\otimes_\calD\calD(p)\ra \Omega_{D(p)}^1$ sends $dx_i\otimes 1$ to $dx_i$.
		
	\end{lemma}
	\begin{proof}
		We consider the inverse limit of the Euler sequences of differentials for the triple $Y(p)_m\ra Y_m\ra \Sigma_e$, with $m\in \NN$ (\Cref{int-rat} and \Cref{rat cot 3}).
		Then by Lemma \ref{diff}, we see the inverse limit $\varprojlim_{m\in \NN} (\Omega_{Y_m/\Sigma_e}^1\otimes_{\calO(Y_m)} \calO(Y(p)_m))$ is isomorphic to the  $\calD(p)$-module $\Omega_D^1\otimes_\calD \calD(p)$.
		Similarly the inverse limit $\varprojlim_{m\in \NN} \Omega_{Y(p)_m/Y_m}^1$  is isomorphic to $\Omega_{D(p)/D}^1$.
		In particular, we get the following sequence of $\calD(p)$-modules
		\[
		0 \rra \Omega_D^1\otimes_\calD \calD(p)\rra \Omega_{D(p)}^1\rra \Omega_{D(p)/D}^1 \rra 0.
		\]

		To show the sequence is an exact sequence, we may assume $P$ admits an \'etale map from the ring of convergent power series $\Bdre\langle x_1, \ldots, x_r\rangle$.
		We apply Lemma \ref{diff} to the immersion $X\ra Y$ and $X\ra Y(p)=Y\times\cdots \times Y$ separately.
		Then we get an description of differentials as follows
		\[
		\Omega_{D}^1=\moplus_{j=1}^r \calD dx_j,~~~~~\Omega_{D(p)}^1=(\moplus_{j=1}^r \calD(p) dx_j )\oplus( \moplus_{\substack{1\leq i\leq p\\ 1\leq j\leq r}} \calD(p) d\delta_{i,j}),
		\]
		Here the projection map $D(p)\ra D$ induced from $Y(p)\ra Y$ produces the natural monomorphism
		\[
		\Omega_D^1\rra \Omega_{D(p)}^1,
		\]
		sending the generator $dx_j$ onto $dx_j$ in $\Omega_{D(p)}^1$.
		This gives the injectivity from $\Omega_D^1\otimes_\calD\calD(p)$ into $\Omega_{D(p)}^1$.
		
		Moreover, by the explicit formula in Lemma \ref{env, D(p)} for ring of convergent power series,  the $(\delta_{i,j})$-adic continuous differential of $\calD(p)$ over $\calD$ is the free $\calD(p)$-module generated by $d\delta_{i,j}$, for $1\leq i\leq p$ and $1\leq j\leq r$.
		This is exactly the cokernel of the injection above and is the free $\calD(p)$-module generated by $d\delta_{i,j}$.
		Thus we get the short exact sequence as expected.
		
	\end{proof}
	
	We can construct the relative de Rham complex of $D(p)$ over $D$, by taking wedge products of $\Omega_{D(p)/D}^1$ and considering the relative differential operator.
	Then we have the following filtered version of the Poincar\'e Lemma for infinitesimal differentials:
	\begin{lemma}[Poincar\'e Lemma]\label{Poincare}
		There exists  a natural quasi-isomorphism to the relative de Rham complex 
		\[
		\calD \rra \Omega_{D(p)/D}^\bullet.
		\]
		Moreover, for each $m\in \NN$, the natural induced map below is a quasi-isomorphism
		\[
		\calD\ra \Omega_{D(p)/D}^\bullet/\Delta(p)^{m+1-\bullet}.
		\]
	\end{lemma}
	\begin{proof}
		We first assume $Y$ is a unit disc, and by Lemma \ref{env, D(p)} the ring $\calD(p)$ is equal to the ring of formal power series over $\calD$ with coordinates $\delta_{i,j}$.
		For the first argument, it suffices to show that the augmented complex 
		\[
		0\ra \calD\ra \calD(p)\ra \Omega_{D(p)/D}^1\ra \Omega_{D(p)/D}^2\ra \cdots \ra \Omega_{D(p)/D}^N\ra 0 \tag{$\ast$}
		\]
		is homotopic to $0$, where $N=pr$.
		Using the coordinate interpretation, the complex $(\ast)$ is an $N$-th completed tensor product of the complex 
		\[
		0\ra \calD\ra \calD[[x]]\ra \calD[[x]]dx\ra 0
		\]
		over $\calD$, where the map $\calD[[x]]\ra \calD[[x]]dx$ is the $\calD$-linear relative differential.
		But since $\calD$ contains $\Q$, the relative differential is surjective with kernel being $\calD$, which proves the first statement  in this case.
		Moreover, notice that by writing down the differentials $\Omega_{D(p)/D}^i$ in terms of coordinates $\delta_{i,j}$ by Lemma \ref{env, D(p)}, the differential in the complex $(\ast)$ preserves the degree.
		In this way, since the quotient $\Omega_{D(p)/D}^\bullet/\Delta(p)^{m+1-\bullet}$ kills exactly elements of degrees higher than $m$, we get the statement about the quotient complex in this case.
		
		In general, as the statement is \'etale local with respect to the smooth rigid space $Y=\Spa(P)$, we may assume $Y$ admits an \'etale morphism to an unit disc.
		Then the claim follows from a term-wise base change formula in the complex $(\ast)$, thanks to Lemma \ref{diff}.
	\end{proof}

	Here is another observation which we will need in order to compute the cohomology of infinitesimal filtration.
	\begin{lemma}\label{intsec of ideals}
		Let $D(p)\ra D$ be the degeneracy map of envelopes as before, and let $J(p)$, $J$ and $\Delta(p)$ be the kernel ideals for surjections $\calO_{D(p)}\ra \calO_X$, $\calO_D \ra \calO_X$ and $\calO_{D(p)} \ra \calO_D$ respectively.
		Then for $j\leq m $ in $ \NN$, the natural map below is an isomorphism of $\calO_{X}$-modules
		\begin{align*}
			J^{m-j}/J^{m-j+1}\cdot \Delta(p)^{j}/\Delta(p)^{j+1} &\longrightarrow\\
			\left(J^m, J^{m-1}\Delta(p) , \cdots,  J^{m-j}\Delta(p)^{j}, J(p)^{m+1} \right) &/ \left(J^m , \cdots, J^{m-j+1}\Delta(p)^{j-1}, J(p)^{m+1} \right).
		\end{align*}
		
	\end{lemma}
	\begin{proof}
		As the statement is local with respect to $Y$, let us first assume $Y=\Spa(P)$ admits an \'etale map to a ring of convergent power series.
		By Lemma \ref{env, D(p)}, $\calD(p)$ is the formal power series ring $\calD[[\delta_{i,j}]]$, and the ideal $\Delta(p)$ is generated by variables $(\delta_{i,j})$.
		Notice that as the map $\calD[\delta_{i,j}]\ra \calD[[\delta_{i,j}]]$ is flat and the quotient ideals in the statement can be defined over $\calD[\delta_{i,j}]$, it suffices to show the analogous statement for the polynomial ring $\calD[\delta_{i,j}]$.
		
		Then as the ring $\calD[\delta_{i,j}]$ is a free module over $\calD$ with a basis given by monomials of $\delta_{i,j}$, we could express elements $x$ in $\calD[\delta_{i,j}]\cap \left(J^m, J^{m-1}\Delta(p) , \cdots,  J^{m-j}\Delta(p)^{j}, J(p)^{m+1} \right)$ using the coordinates as below
		\begin{align*}
			x&=a_{r_{l_0}}+ \sum_{|r_{l_1}|=1} a_{r_{l_1}}\cdot \delta^{r_{l_1}} + \sum_{|r_{l_2}|=2} a_{r_{l_2}}\cdot \delta^{r_{l_2}} + \cdots, \\
			&a_{r_{l_n}}\in J^{m-n},~for~0\leq n\leq j;\\
			&a_{r_{l_n}}\in J^{m+1-n},~for~j<n\leq m+1;\\
			&a_{r_{l_n}}\in \calD,~for~j>m+1.
		\end{align*}
		Here $\delta^{r_{l_n}}$ are monomials in $\delta_{i,j}$ with multi-indexes.
		Similarly we could do this for elements in $\calD[\delta_{i,j}]\cap \left(J^m , \cdots, J^{m-j+1}\Delta(p)^{j-1}, J(p)^{m+1} \right)$, where in the obtained formula we replace $j$ above by $j-1$.
		Compare those expressions, we see the statement in the lemma holds for $\calD[\delta_{i,j}]$.
		So by extending this along the flat map $\calD[\delta_{i,j}]\ra \calD[[\delta_{i,j}]]$, we get the result for $\calD(p)\simeq \calD[[\delta_{i,j}]]$.

	\end{proof}

	Now we are ready to prove the Lemma \ref{lemB}.
	\begin{proof}[Proof for Lemma \ref{lemB}]
		$~$
		\begin{itemize}
			\item[Step 1]
			We first deal with the underlying complexes and forget the infinitesimal filtration.
			Our goal is to show that the natural map of complexes below is a quasi-isomorphism
			\[
			M\otimes_{\calD}\Omega_D^\bullet\rra M(p)\otimes_{\calD(p)}\Omega_{D(p)}^\bullet.
			\]
			
			The de Rham complex $\Omega_D^\bullet$ is equipped with its Hodge filtration, defined by $F^i\Omega_D^\bullet=\sigma^{\geq i}\Omega_{D}^\bullet$.
			By the  Euler sequence in Lemma \ref{eul}, the Hodge filtration of $\Omega^\bullet_D$ induces a natural descending filtration on the relative de Rham complex $\Omega_{D(p)}^\bullet$, 
			whose graded piece is $gr^i\Omega_{D(p)}^\bullet=\Omega_D^i\otimes_\calD \Omega_{D(p)/D}^\bullet.$
			
			Now consider the de Rham complex $(M\otimes_\calD \Omega_D^\bullet,\nabla_D)$ and $(M(p)\otimes_{\calD(p)} \Omega_{D(p)}^\bullet,\nabla_{D(p)})$ of the crystal $\calF$ at $D$ and $D(p)$.
			The projection $D(p)\ra D$ induces a map of complexes
			\[
			M\otimes\Omega_D^\bullet\rra M(p)\otimes_{\calD(p)}\Omega_{D(p)}^\bullet.
			\]
			By the crystal condition, the base change of $M$ along the map $D(p)\ra D$ is exactly $M(p)$.
			Moreover, by the compatibility of the de Rham complexes, the filtration on $\Omega_{D(p)}^\bullet$ induces a filtration on $M(p)\otimes_{\calD(p)} \Omega_{D(p)}^\bullet=M\otimes_\calD \Omega_{D(p)}^\bullet$, given by 
			\[
			F^i(M(p)\otimes_{\calD(p)} \Omega_{D(p)}^\bullet)=M\otimes_{\calD} F^i\Omega_{D(p)}^\bullet.
			\]
			Each $F^i(M(p)\otimes_{\calD(p)} \Omega_{D(p)}^\bullet)$ is a subcomplex of $M\otimes_{\calD} \Omega_{D(p)}^\bullet$, since $\nabla^i$ sends elements in $M$ into $M\otimes \Omega_D^1\subset M(p)\otimes \Omega_{D(p)}^1$.
			Moreover, the $i$-th graded factor of this filtration is 
			\[
			M\otimes_\calD \Omega_D^i\otimes_\calD \Omega_{D(p)/D}^\bullet,
			\]
			which by Lemma \ref{Poincare} is isomorphic to the $M\otimes_\calD \Omega_D^i$ via the degeneracy map.
			In this way, the projection $D(p)\ra D$ induces a map of filtered complexes
			\[
			M\otimes_\calD \Omega_D^\bullet\rra M(p)\otimes_{\calD(p)}\Omega_{D(p)}^\bullet, \tag{$\ast$}
			\]
			which on each graded factor is an isomorphism.
			Hence the map $(\ast)$ itself is an isomorphism, by the spectral sequence associated with a finite filtration as in \cite[Tag 012K]{Sta}.

			\item[Step 2]
			We then show that the above quasi-isomorphism is filtered under the infinitesimal filtration.
			More precisely, we claim the graded piece of the following map in Step 1 is a filtered quasi-isomorphism under their infinitesimal filtrations:
			\[
			M\otimes_\calD \Omega_D^i \rra M\otimes_\calD \Omega_D^i\otimes_\calD \Omega_{D(p)/D}^\bullet.
			\]
			
			Consider the $(m+i)$-th graded piece for $m\in \NN$.
			On the one hand, the $(m+i)$-th graded piece for the infinitesimal filtration on $M\otimes_\calD \Omega_{D}^\bullet$ induces a subquotient $J^m\cdot M\otimes \Omega_D^i/J^{m+1}$ of the left hand side of the above.
			On the other hand, the $(m+i)$-th graded piece for infinitesimal filtration on $M(p)\otimes_{\calD(p)} \Omega_{D(p)}^\bullet$ induces the following subquotient of the right hand side:
			\[
			J(p)^{m-\bullet}\cdot M\otimes_\calD \Omega_D^i\otimes_\calD \Omega_{D(p)/D}^\bullet/J(p)^{m+1-\bullet}.
			\]
			So we get the map of graded pieces as below
			\[
			J^m\cdot M\otimes \Omega_D^i/J^{m+1} \rra J(p)^{m-\bullet}\cdot M\otimes_\calD \Omega_D^i\otimes_\calD \Omega_{D(p)/D}^\bullet/J(p)^{m+1-\bullet}. \tag{$\ast\ast$}
			\]
			Here we note that as the ideal $J$ maps into $J(p)$, the right hand side is an $\calO_D/J=\calO_X$-linear complex.
			
			To show $(\ast\ast)$ is a quasi-isomorphism, we need to subdivide the right hand side in a finer way.
			We introduce a finite increasing filtration on the right hand side of $(\ast\ast)$, whose $j$-th filtration is 
			\begin{align*}
				& \left(J^{m-\bullet}, J^{m-1-\bullet}\Delta(p) , \cdots J^{m-j}\Delta(p)^{j-\bullet}, J(p)^{m+1 -\bullet} \right)\cdot M\otimes_\calD \Omega_\calD^i \otimes_\calD \Omega_{D(p)/D}^\bullet /J(p)^{m+1-\bullet} \\
				&= complex~ \scalebox{1.7}{(} \left(J^{m}, \cdots, J^{m-j}\Delta(p)^{j}, J(p)^{m+1} \right)\cdot M\otimes \Omega_\calD^i \otimes \calO_{D(p)} /J(p)^{m+1} \\
				& \rra  \left(J^{m-1},  \cdots, J^{m-j}\Delta(p)^{j-1}, J(p)^{m } \right)\cdot M\otimes \Omega_\calD^i \otimes \Omega_{D(p)/D}^1 /J(p)^{m}\\
				& \rra \cdots \\
				& \rra  \left(J^{m-j}, J(p)^{m+1 -j} \right)\cdot M\otimes \Omega_\calD^i \otimes \Omega_{D(p)/D}^j /J(p)^{m+1-j} \scalebox{1.7}{)} .
			\end{align*}
			
			The graded piece of this filtration is the $\calO_X$-linear complex
			\begin{align*}
				\left(J^{m-\bullet}, \cdots, J^{m-j}\Delta(p)^{j-\bullet}, J(p)^{m+1-\bullet} \right)\cdot M\otimes \Omega_\calD^i \otimes \Omega_{D(p)/D}^\bullet /\left(J^{m -\bullet},  \cdots , J^{m-j+1}\Delta(p)^{j-1-\bullet}, J(p)^{m+1-\bullet} \right) 
			\end{align*}
			We apply the Lemma \ref{intsec of ideals} to this complexes, then the graded piece above can be rewritten as 
			\begin{align*}
				&J^{m-j}M\otimes {\Omega_D^i}/J^{m-j+1}\otimes_\calD  \Delta(p)^{j-\bullet}\Omega_{D(p)/D}^{\bullet}/\Delta(p)^{j-\bullet+1} \\
				&= complex~ \scalebox{1.7}{(} J^{m-j}M\otimes {\Omega_D^i}/J^{m-j+1}\otimes_\calD  \Delta(p)^{j}/\Delta(p)^{j+1} \\
				& \rra  J^{m-j}M\otimes {\Omega_D^i}/J^{m-j+1}\otimes_\calD  \Delta(p)^{j-1}\Omega_{D(p)/D}^{1}/\Delta(p)^{j} \\
				&\rra \cdots \\
				& \rra J^{m-j}M\otimes {\Omega_D^i}/J^{m-j+1}\otimes_\calD  \Omega_{D(p)/D}^{j}/\Delta(p) \scalebox{1.7}{)} \\
				& \simeq \left( J^{m-j}M\otimes {\Omega_D^i}/J^{m-j+1}\right) \otimes_\calD  \left(\Delta(p)^{j-\bullet}\Omega_{D(p)/D}^{\bullet}/\Delta(p)^{j+1-\bullet} \right).
			\end{align*}
			
			At last, by the graded version of relative Poincar\'e Lemma in Lemma \ref{Poincare}, we have
			\[
			\Delta(p)^{j-\bullet}\Omega_{D(p)/D}^{\bullet}/\Delta(p)^{j+1-\bullet} \simeq \begin{cases}
				0,~j\geq 1;\\
				\calD,~j=0.
			\end{cases}
			\]
			In this way, the graded pieces of the right hand side of $(\ast\ast)$ are zero, except for its zero-th graded piece which is naturally isomorphic to $J^mM\otimes \Omega_D^i/J^{m+1}$.
			Hence $(\ast\ast)$ is an isomorphism, and we finish the proof.
		\end{itemize}

	\end{proof}
	
	\begin{remark}
		Here we mention that the same study of the infinitesimal filtration works with minor changes for schemes.
		In particular, the schematic analogue of the proof in \cite[Theorem 2.12]{BdJ} can be improved into a filtered version, 
		and we thus obtain the expected filtered isomorphism in the crystalline theory, which is proved in different methods in \cite[Theorem 7.23]{BO78}.
	\end{remark}

	\subsection{Global result}
	Now we generalize the computation of cohomology to the global situation, without assuming $X$ is affinoid.
	
	We recall that the infinitesimal ideal sheaf $\mathcal{J}_{X/\Sigma_e}:=\ker(\mathcal{O}_{X/\Sigma_e}\to \mathcal{O}_X)$ naturally defines a filtration on a coherent crystal $\mathcal{F}$, so that the $j$-th filtration is $\mathcal{J}_{X/\Sigma_e}^j\mathcal{F}$.
	Moreover, one can naturally extend the filtration to each individual $\mathcal{F}\otimes_{\mathcal{O}_{X/\Sigma_e}}\Omega^{i}_{X/\Sigma_e}$ the entire de Rham complex of $\mathcal{F}$, so that 
	\begin{align*}
		\Fil^j \left( \mathcal{F}\otimes_{\mathcal{O}_{X/\Sigma_e}}\Omega^{i}_{X/\Sigma_e} \right) = \begin{cases}
			\mathcal{F}\otimes_{\mathcal{O}_{X/\Sigma_e}}\Omega^{i}_{X/\Sigma_e},~\text{if~}j< i;\\
			\mathcal{J}_{X/\Sigma_e}^{j-i}\mathcal{F}\otimes_{\mathcal{O}_{X/\Sigma_e}}\Omega^{i}_{X/\Sigma_e}~\text{if}~j\geq i,
		\end{cases} 
	\end{align*}
	and the $j$-th filtration of the de Rham complex of $\mathcal{F}$ is
	\[
	\mathcal{J}_{X/\Sigma_e}^j\mathcal{F} \to \mathcal{J}_{X/\Sigma_e}^{j-1}\mathcal{F}\otimes_{\mathcal{O}_{X/\Sigma_e}} \Omega^1_{X/\Sigma_e} \to \cdots \to \mathcal{F}\otimes_{\mathcal{O}_{X/\Sigma_e}}\Omega^j_{X/\Sigma_e} \to \mathcal{F}\otimes_{\mathcal{O}_{X/\Sigma_e}}\Omega^{j+1}_{X/\Sigma_e}\to \cdots.
	\]
	By taking the derived direct image, we in particular get a filtration on  $Ru_{X/\Sigma_e *}(\calF\otimes\Omega_{X/\Sigma_{e \inf}}^\bullet)$.
	
	Our first result in this subsection shows that the above direct image vanishes in higher cohomological degrees.
	\begin{proposition}\label{van}
		Let $X$ be a rigid space over $\Sigma_e$, and let $\calF$ be a coherent crystal over $X/\Sigma_{e \inf}$.
		Then for each $i>0$ and $j\in \NN$, we have
		\[
		Ru_{X/\Sigma_e *}(\calJ_{X/\Sigma_e}^j\calF\otimes_{\calO_{X/\Sigma_e}} \Omega_{X/\Sigma_{e \inf}}^i)=0.
		\]
		In particular, after applying the derived direct image $Ru_{X/\Sigma_e *}$, the truncation map of the de Rham complex induces a filtered quasi-isomorphism:
		\[
		Ru_{X/\Sigma_e *}(\calF\otimes\Omega_{X/\Sigma_{e \inf}}^\bullet)\rra Ru_{X/\Sigma_e *} \calF.
		\]
	\end{proposition}
	\begin{proof}
		Recall from Subsection \ref{sub inf-rig} that $\Gamma(U,u_{X/\Sigma_{e \inf} *}\calG)$ is defined as the $0$-th cohomology $\Gamma(U/\Sigma_{e \inf}, \calG)$,
		and similarly for its filtered analogue.
		So to show the vanishing of $Ru_{X/\Sigma_e *}(\calJ_{X/\Sigma_e}^j\calF\otimes\Omega_{X/\Sigma_{e \inf}}^i)$, it suffices to do this locally and assume $X$ is affinoid, together with a closed immersion into a smooth rigid space $Y$ over $\Sigma_{e \inf}$.
		We then notice that by Proposition \ref{cech}, $Ru_{X/\Sigma_e *}(\calJ_{X/\Sigma_e}^j\calF\otimes\Omega_{X/\Sigma_{e \inf}}^i)$ is computed by the following cosimplicial complex:
		\[
		J^j\calF\otimes\Omega_{X/\Sigma_{e \inf}}^i(D)\rra J(1)^j\calF\otimes\Omega_{X/\Sigma_{e \inf}}^i (D(1))\rra J(2)^j\calF\otimes\Omega_{X/\Sigma_{e \inf}}^i(D(2))\rra \cdots,
		\]
		which by Lemma \ref{lemA} is homotopic to zero when $i>0$.
		So we get the vanishing of $Ru_{X/\Sigma_e *}(\calJ_{X/\Sigma_e}^j\calF\otimes\Omega_{X/\Sigma_{e \inf}}^i)$ for each $i>0$ and $j\in \mathbb{N}$; in other words, the filtered sheaf of complexes $Ru_{X/\Sigma_e *}(\calF\otimes\Omega_{X/\Sigma_{e \inf}}^i)$ vanishes for $i>0$.
		By induction, the latter in particular shows that the truncation map
		\[
		Ru_{X/\Sigma_e *}(\calF\otimes\Omega_{X/\Sigma_{e \inf}}^{\leq i})\longrightarrow Ru_{X/\Sigma_e *} \calF
		\]
		is a filtered quasi-isomorphism for each $i>0$.
		As a consequence, since the filtered complex  $Ru_{X/\Sigma_e *}(\calF\otimes\Omega_{X/\Sigma_{e \inf}}^\bullet)$ is the derived limit of $	Ru_{X/\Sigma_e *}(\calF\otimes\Omega_{X/\Sigma_{e \inf}}^{\leq i})$, we get the isomorphism as claimed in the second half of the statement.
	\end{proof}
	
	Now we can generalize Theorem \ref{aff-coh} to the global case, without assuming the affinoid condition:
	\begin{theorem}\label{glo-coh}
		Let $X\ra Y$ be a closed immersion of $X$ into a smooth rigid space $Y$ over $\Sigma_e$.
		Let $\calF$ be a coherent crystal over $\calO_{X/\Sigma_e}$, and let $\calF_D=\varprojlim_{m\in \NN} \calF_{Y_m}$ be the restriction of $\calF$ at the envelope $D=D_X(Y)=\varinjlim_{m\in \NN} Y_m$,  with its de Rham complex $\calF_D\otimes \Omega_D^\bullet$.
		Then there exists a natural isomorphism in the filtered derived category of sheaves of abelian groups over $X$
		\[
		Ru_{X/\Sigma_e *}\calF \rra \calF_D\otimes \Omega_D^\bullet.
		\]
		
	\end{theorem}
	Before the proof, we want to mention that the strategy is to produce a natural map between those two complexes of sheaves of abelian groups, where the isomorphism will follow from the affinoid computation.
	\begin{proof}

		By Proposition \ref{van}, the truncation map of the de Rham complex $\calF\otimes_{\calO_{X/\Sigma_e}} \Omega_{X/\Sigma_{e \inf}}^\bullet\ra \calF[0]$ produces a canonical filtered isomorphism in the derived category of $\calO_X$-modules
		\[
		Ru_{X/\Sigma_e *}(\calF\otimes \Omega_{X/\Sigma_{e \inf}}^\bullet)\ra Ru_{X/\Sigma_e *}\calF.
		\]
		
		On the other hand, we recall that the envelope $D=D_X(Y)$ is defined as the direct limit $\varinjlim_{m\in \NN} h_{Y_m}$ of representable sheaves, where $Y_m$ is the $m$-th infinitesimal neighborhood of $X$ into $Y$.
		In the infinitesimal topos $\Sh(X/\Sigma_{e \inf})$, the map from the envelope $D$ to the final object $1$ induces a map of derived functors
		\[
		R\Gamma(X/\Sigma_{e \inf},-)\rra R\Gamma(D,-)=R\varprojlim_m R\Gamma(Y_m,-).
		\]
		Similarly for its filtered analogue.
		
		We apply the natural transformation to the filtered de Rham complex $\calF\otimes\Omega_{X/\Sigma_{e \inf}}^\bullet$, and get
		\begin{align*}
			R\Gamma(X/\Sigma_{e \inf}, \calF\otimes\Omega_{X/\Sigma_{e \inf}}^\bullet) & \ra R\varprojlim_m R\Gamma(Y_m,  \calF_{Y_m}\otimes_{\calO_{Y_m}} \Omega_{Y_m/\Sigma_e}^\bullet)\\
			& = R\Gamma(X, R\varprojlim_{m\in \NN}\calF_{Y_m}\otimes_{\calO_{Y_m}} \Omega_{Y_m/\Sigma_e}^\bullet)\\
			& = R\Gamma(X, \calF_D\otimes_{\calO_D} \Omega_D^\bullet),
		\end{align*}
		where the last equality follows from the observation that the inverse system $\{ \calF_{Y_m}\otimes_{\calO_{Y_m}} \Omega_{Y_m/\Sigma_e}^\bullet \}_m$ admits a finite filtration, where each graded piece $\{ \calF_{Y_m}\otimes_{\calO_{Y_m}} \Omega_{Y_m/\Sigma_e}^i \}_m$ is a pro-coherent system satisfying the sheaf vertion Mittag-Leffler condition (\cite[Lemma 7.20]{BO78}).
		Similarly for the subcomplex $\calJ_{X/\Sigma_e}^{m-\bullet}\calF\otimes\Omega_{X/\Sigma_{e \inf}}^{\bullet}$.
		Notice that the map is functorialial with respect to all locally closed immersions $(X,Y)$ into smooth rigid spaces.
		In particular, by varying $X$ among all open subsets $U$ of $X$ and considering the above map for locally closed immersions $(U,Y)$, we could enhance the above into the sheaf version filtered morphism
		\[
		Ru_{X/\Sigma_e *} (\calF\otimes\Omega_{X/\Sigma_{e \inf}}^\bullet) \rra \calF_D\otimes_\calD \Omega_D^\bullet.
		\]
		Thus by composing with (the inverse of) the filtered isomorphism at the beginning, we get a natural map in the filtered derived category of sheaves of abelian groups over $X$:
		\[
		Ru_{X/\Sigma_e *}\calF \rra \calF_D\otimes_\calD \Omega_D^\bullet.
		\]
		
		At last, to show the filtered isomorphism, we note that the evaluation functor defined in the second paragraph of the proof is compatible with restriction onto open affinoid subspaces of $Y$ and $X$.
		In particular, as the map is analytic local with respect to $X$, the compatibility allows us to reduce to the case when both $X$ and $Y$ are affinoid.
		In the latter case, we know by Theorem \ref{aff-coh} that the map is a filtered isomorphism, which finishes the proof.
	\end{proof}
	
	As a consequence, we get a change of bases formula quite easily.
	\begin{proposition}\label{coh, change of bases}
		Let $X$ be a rigid space over $\Sigma_e$, and $e'\geq e$ be an integer.
		Let $\calF'$ be a crystal in vector bundles over $X/\Sigma_{e' \inf}$, and $\calF$ be the pullback of $\calF'$ along the map of infinitesimal topoi $\Sh(X/\Sigma_{e \inf}) \ra \Sh(X/\Sigma_{e' \inf})$.
		Then there exists a natural isomorphism of complexes of sheaves of $\mathrm{B}_{\mathrm{dR},e'}^+$-modules as below
		\[
		(Ru_{X/\Sigma_{e'} *} \calF')\otimes_{\mathrm{B}_{\mathrm{dR},e'}^+}^L \Bdre \rra Ru_{X/\Sigma_e *} \calF.
		\]
		
	\end{proposition}
	\begin{proof}
		We first notice that the natural morphism of infinitesimal sites $X/\Sigma_{e \inf} \ra X/\Sigma_{e' \inf}$ induces a canonical map in the derived category of sheaves over $X$
		\[
		Ru_{X/\Sigma_{e'} *} \calF' \rra Ru_{X/\Sigma_e *} \calF.
		\]
		Moreover, as the target is $\Bdre$-linear, by the adjunction for the forgetful functor (from $\Bdre$-modules to $\mathrm{B}_{\mathrm{dR},e'}^+$-modules) we get a natural map of complexes
		\[
		(Ru_{X/\Sigma_{e'} *} \calF')\otimes_{\mathrm{B}_{\mathrm{dR},e'}^+}^L \Bdre \rra Ru_{X/\Sigma_e *} \calF.
		\]
		So it suffices to show this adjunction map is an isomorphism.
		
		As the statement is analytic local with respect to $X$, by shrinking $X$ if necessarily, we can assume that 
		there exists a  closed immersion $X\ra Y'$ of $X$ into a smooth rigid space over $\Sigma_{e'}$.
		By Theorem \ref{glo-coh}, we have the following natural isomorphisms
		\begin{align*}
			Ru_{X/\Sigma_{e'} *} \calF' & \rra \calF'_{D'} \otimes \Omega_{D'}^\bullet;\\
			Ru_{X/\Sigma_{e} *} \calF & \rra \calF_D \otimes \Omega_D^\bullet,
		\end{align*}
		where $D'$ is the envelope of $X$ in $Y'$, and $D$ is the envelope of $X$ in $Y=Y'\times_{\Sigma_{e'}} \Sigma_e$ where the latter is smooth over $\Sigma_e$.
		Notice that $\calO_{Y'}$ is flat over $\Sigma_{e'}$, and the structure sheaves $\calO_{D'}$ is flat over $\calO_{Y'}$ (for it is defined as the formal completion of $\calO_{Y'}$ along $X\ra Y'$).
		In this way, by the assumption that $\calF'$ is a crystal in vector bundles, the complex $\calF'_{D'} \otimes \Omega_{D'}^\bullet$ is a $\mathrm{B}_{\mathrm{dR},e'}^+$-linear bounded complex of sheaves of flat $\mathrm{B}_{\mathrm{dR},e'}^+$-modules.
		Thus we get the isomorphisms
		\begin{align*}
			(Ru_{X/\Sigma_{e'} *} \calF') \otimes_{\mathrm{B}_{\mathrm{dR},e'}^+}^L \Bdre & \simeq (\calF'_{D'} \otimes \Omega_{D'}^\bullet) \otimes_{\mathrm{B}_{\mathrm{dR},e'}^+}^L \Bdre\\
			& \simeq (\calF'_{D'} \otimes \Omega_{D'}^\bullet)/\xi^{e},
		\end{align*}
		which is then isomorphic to the complex $\calF_D \otimes \Omega_D^\bullet$ as the envelope $D=D_X(Y)$ is equal to the pullback of $D'=D_X(Y')$ along the surjection $\mathrm{B}_{\mathrm{dR},e'}^+ \ra \Bdre=\mathrm{B}_{\mathrm{dR},e'}^+/\xi^e$.
	\end{proof}

	\section{Derived de Rham cohomology over $\Bdre$}\label{sec ddR}
	In this section, we introduce the derived de Rham cohomology of a rigid space over $\Bdre$, and prove the comparison between the derived de Rham cohomology and the infinitesimal cohomology.
	
	Before we start, we want to mention that we will use mildly the language of \ic throughout this section.
	The main reason is to globalize the affinoid constructions and get a good theory of ``sheaf of  derived objects", using the $\infty$-categorical cohomological descent.
	
	\begin{remark}
		The construction of the analytic derived de Rham complex in this section can be applied to more general class of analytic Huber rings, which includes for example rigid spaces over an arbitrary $p$-adic non-archimedean field, and perfectoid spaces.
		
		The results of this section hold true for rigid spaces over a general $p$-adic fields.
		Moreover, in an upcoming work \cite{GL20} by Shizhang Li and the author, we show that the analytic derived de Rham complex of perfectoid rings is isomorphic to the de Rham period sheaves in \cite{Sch13}.
	\end{remark}

	\noindent\textbf{Derived $\infty$-category and filtered $\infty$-category.}\label{par, infty}
	We first setup the convention of derived \ic and its filtered version in this section.
	
	Let $\scrA$ be a Grothendieck abelian category (\cite[Tag 079A]{Sta}).
	We can associate to $\scrA$ a natural \ic  $\scrD(\scrA)$, called the \emph{derived \ic of $\scrA$} (\cite{Lu17}, 1.3.5). 
	This is the $\infty$-categorical enhancement of the classical derived $\infty$-category, and the homotopy category $\rmh\Ch(\scrA)$ of $\scrD(\scrA)$ is the usual derived category $D(\scrA)$.
	Here we want to mention that the derived \ic $\scrD(\scrA)$ is a stable presentable $\infty$-category.
	In the special case when $\scrA$ is the category of modules over an ring $R$, we use $\scrD(R)$ to denote $\scrD(\scrA)$, which is equipped with a symmetric monoidal structure by the derived tensor product of complexes.
	As a convention in this section, we will call $\scrD(R)$ the derived category.
	
	For a presentable \ic $\scrC$, we recall the \emph{filtered \ic in $\scrC$} is defined as the \ic
	\[
	\DF(\scrC):=\Fun(\NN^\op, \scrC).
	\]
	Moreover, $\DF(\scrC)$ admits a full sub-\ic $\wh \DF(\scrC)$, called \emph{filtered complete \ic in $\scrC$}, consisting of objects $C_\bullet$ such that $\lim C_\bullet\simeq 0$.
	The natural inclusion functor $\wh \DF(\scrC) \ra \DF(\scrC)$ admits a left adjoint, called the \emph{filtered completion}.
	When $\scrC=\scrD(R)$ is the derived \ic of $R$-modules, we use $\DF(R)$ and $\wh \DF(R)$ to denote $\DF(\scrC)$ and $\wh \DF(R)$ respectively.
	Here we note that by to their homotopy categories (and induced functors), we recover the ordinary filtered derived category.

	\noindent\textbf{Hypersheaves}\label{par, sheaf}
	We then give a quick review about sheaves of complexes.
	
	Let $X$ be a site, and let $\scrC$ be a presentable $\infty$-category.
	The \ic of presheaves in $\scrC$, denoted as $\mathrm{PSh}(X,\scrC)$, 
	is defined to be the \ic $\Fun(X^\op,\scrC)$ of contravariant functors from $X$ to $\scrC$.
	The \ic $\mathrm{PSh}(X,\scrC)$ admits a full sub-\ic $\Sh(X,\scrC)$ of \emph{(infinity) sheaves in $\scrC$},
	consisting of functors $\calF:X^\op \ra \scrC$ that send coproducts to products and satisfy the descent along \v{C}ech nerves: 
	for any covering $U'\ra U$ in $X$, 
	the natural morphism to the limit below is required to be a weak equivalence
	\[
	\calF(U) \rra \lim_{[n]\in \Delta} \calF(U'_n), \tag{$\ast$}
	\]
	where $U'_\bullet \ra U$ is the \v{C}ech nerve associated with the covering $U'\ra U$.
	Here we note that this is the $\infty$-categorical analogue of the classical sheaf condition in ordinary categories.
	
	There is a stronger descent condition which requires $(\ast)$ above to hold
	with respect to all \emph{hypercovers} $U'_\bullet \ra U$ in the site $X$.
	Sheaves satisfying such stronger condition are called \emph{hypersheaves}.
	For example, given any bounded below complex $C$ of ordinary sheaves on a site $X$,
	the assignment $U \mapsto \mathrm{R}\Gamma(U, C)$
	gives rise to a hypersheaf.
	The collection of hypersheaves in $\scrC$ forms a full 
	sub-\ic $\Sh^{{\mathrm{hyp}}}(X,\scrC)$ inside $\Sh(X,\scrC)$.

	Let $\scrC=\scrD(R)$ be the derived $\infty$-category of $R$-modules.
	Then the $\infty$-category $\Sh^{\mathrm{hyp}}(X,\scrC)$ of hypersheaves over $X$ is in fact equivalent to the derived $\infty$-category $\scrD(X,R)$ of classical sheaves of $R$-modules over $X$, by \cite[Corollary 2.1.2.3]{Lu18}.
	As an upshot, the underlying homotopy category of $\Sh^{\mathrm{hyp}}(X,\scrC)$ is the classical derived category of sheaves of $R$-modules over $X$.
	In particular, given a hypersheaf $\calF$ of $R$-modules over $X$, we can always represent it by an actual complex of sheaves of $R$-modules.

	\subsection{Topological algebras over $\Ainfe$}
	As a preparation, we first setup basics around the topologically finite type algebras over $\Ainfe:=\Ainf/\xi^e$, and the construction of the analytic cotangent complex, generalizing the discussion for $e=1$ in \cite{GR} Section 7.
	
	In this subsection only, we make the convention that $M^\wedge$ is the classical $p$-adic completion of $M$, where $M$ is a $\Z_p$-module.
	
	\begin{definition}
		Let $R$ be an $\Ainfe$-algebra.
		\begin{enumerate}[(i)]
			\item We call $R$ is \emph{topologically finite type over $\Ainfe$} if there exists a surjection of $\Ainfe$-algebras $\Ainfe\langle T_1,\ldots, T_m\rangle\ra R$ for some $m\in \NN$.
			\item We call $R$ is \emph{topologically of finite presentation over $\Ainfe$} if $R$ admits a surjection from $\Ainfe\langle T_1,\ldots, T_m\rangle \ra R$ with kernel being a finitely generated ideal.
		\end{enumerate}
		We denote $\Afe$ the category of $p$-adically complete $p$-torsionfree algebras $R$ over $\Ainfe$ that are of topologically finite presentation, where the morphisms are defined as maps of $\Ainfe$-algebras.
	\end{definition}
	Similarly, we can extend these notions to the relative situation, replacing $\Ainfe$ by any $\Ainfe$-algebra.

	Here we list some basic properties about modules over a given $R\in \Afe$.
	\begin{lemma}[cf. \cite{GR}, 7.1.1]\label{mod over cafe}
		Let $R$ be an algebra in $\Afe$. 
		Then we have
		\begin{enumerate}[(i)]
			\item Every finitely generated $p$-torsionfree $R$-module is finitely presented.
			\item The ring $R$ is coherent.
			\item Let $N$ be a finitely generated $R$-module, $N'\subset N$ a submodule.
			Then there exists an integer $c\geq 0$, such that
			\[
			p^kN\cap N'\subset p^{k-c} N'
			\]
			for every $k\geq c$.
			In particular, the subspace topology on $N'$ induced from the $p$-adic topology on $N$ agrees with the $p$-adic topology of $N'$.
			\item Every finitely generated $R$-module $M$ is $p$-adically complete and separated;
			namely every such $M$ is isomorphic to its $p$-adic completion $M^\wedge$.
			\item Every submodule of a finite type free $R$-module $F$ is closed for the $p$-adic topology of $F$.
		\end{enumerate}		
	\end{lemma}
	\begin{proof}
		\begin{enumerate}[(i)]
			
			\item This is proved in the proof of \cite[13.4.\ (iii.b)]{BMS}; for completeness, we record it here.
			We do this by induction, and note that for $n=1$ the case is given in  \cite{BoLu} 1.2.
			
			Let $\ol M$ be the image of $M$ in $M/\xi[\frac{1}{p}]$, and let $N$ be the kernel of $M\ra \ol M$.
			The image $\ol M$ is finitely generated $p$-torsionfree $R/\xi$-module, which by induction is a finitely presented $R/\xi$-module.
			Note that this also implies the $R/\xi$-module $\ol M$ is a finitely presented $R$-module.
			So by \cite[Tag 0519]{Sta}, $N$ is finitely generated over $R$, and to show the finite presented-ness of $M$ it suffices to show the finite presentedness of $N$.
			But note that for $x\in N$, there exists some $k\in \NN$ such that $p^kx\in \xi M$.
			This implies that $p^k\xi^{n-1}x=0$ in $M$ as the element is contained in $\xi^nM=0$, and by the $p$-torsionfreeness of $M$ we have $\xi^{n-1} x=0$.
			So $N$ is a finitely generated $p$-torsionfree $R/\xi^{n-1}$-module, and by induction we get the result.
			
			\item By definition, a ring $R$ is coherent if every finitely generated ideal of $R$ is finitely presented.
			So by the $p$-torsionfreeness of $R$ and (i), we get the result.
			
			\item Let $M$ be the kernel of the map $N\ra N/N'[\frac{1}{p}]$.
			Namely we have the following short exact sequence
			\[
			0 \rra M \rra N \rra N/N'[\frac{1}{p}].
			\]
			Then since the image of $N$ in $N/N'[\frac{1}{p}]$ is finitely generated and $p$-torsionfree, by (i) we know the image is finitely presented, and thus $M$ is finitely generated (\cite[Tag 0519]{Sta}).
			Note that the quotient $M/N'$ is $p^\infty$-torsion, so by the finitely generatedness there exists some $c\in \NN$ such that $p^cM\subset N'$.
			Besides, for $x\in N$ such that $p^kx\in M$, the image of $x$ in $N/N'[\frac{1}{p}]$ is also zero.
			So the definition of $M$ implies that $x\in M$ and $p^kx\in p^kM$.
			In this way, for $k\geq c$, we have
			\begin{align*}
				p^k N\cap N'&\subset p^k N\cap M\\
				&\subset p^kM\\
				&\subset p^{k-c}N'.
			\end{align*}
			
			\item
			We can fit $M$ into the following short exact sequence of $R$-modules,
			\[
			0\rra N\rra R^{\oplus n}\rra M\rra 0,
			\] 
			We apply the $p$-adic completion to the sequence.
			Then note that since the subspace topology on $N$ is the isomorphic to the $p$-adic topology by (iii), while the quotient topology on $M$ is the same as the $p$-adic topology, by \cite{Mat86} Theorem 8.1, we get an exact sequence of $p$-adically complete $R$-modules with continuous maps
			\[
			0\rra N^\wedge\rra R^{\oplus n}\rra M^\wedge\rra 0.
			\]
			Compare with the above two exact sequences, we see the natural map $N\ra N^\wedge$ is injective while $M\ra M^\wedge$ is surjective.
			
			We then assume the $R$-module $M$ is finitely presented. 
			By the \cite[Tag 0519]{Sta}, we know $N$ is finitely generated.
			In this way, since the surjection of $M\ra M^\wedge$ is true for any finitely generated $R$-module, we see $N\ra N^\wedge$ is an isomorphism.
			In particular, we get $M\simeq M^\wedge$.
			This finishes the (iv) for $M$ being finitely presented over $R$.
			
			In general, let $M$ be any finitely generated module over $R$.
			Take $\ol M$ to be the image of $M$ in $M[\frac{1}{p}]$.
			Then since $\ol M$ is finitely generated and $p$-torsionfree, we know $\ol M$ is finitely presented and hence the kernel $N=\ker(M\ra \ol M)$ is finitely generated by loc. cit.
			Notice that by definition $N$ is $p^\infty$-torsion.
			So there exists some $m\in \NN$ such that $p^mN=0$.
			Now by the $p$-torsionfreeness of $\ol M$, the base change of the exact sequence  $0\ra N\ra M\ra \ol M\ra 0$ along $R\ra R/p^s$ is exact.		
			Moreover, since the inverse system $\{N\otimes_R R/p^s\}_s$ is essentially constant, the inverse limit of the short exact sequence of inverse system is exact, and we get
			\[
			0\rra N^\wedge=N\rra M^\wedge \rra \ol M^\wedge\rra 0,
			\]
			which by the isomorphism $\ol M\simeq \ol M^\wedge$ we get the result
			\[
			M\simeq M^\wedge.
			\]
			So we are done.
			
			\item Let $N$ be a submodule of a finite free $R$-module $F$, and let $M:=F/N$ be the quotient.
			By (iv), since $M$ is finitely generated, we have the canonical isomorphism $M\simeq M^\wedge$.
			As in the proof of (iv), the $p$-adic completion induces the following short exact sequence 
			\[
			0\rra N^\wedge\rra F\rra M^\wedge\simeq M\rra 0.
			\]
			Hence we get the isomorphism $N\simeq N^\wedge$.
			In particular, since $N$ is complete and its $p$-adic topology is isomorphic to its subspace topology, we get the closedness of $N$ in $F$ by standard topological argument.

		\end{enumerate}
	\end{proof}
	\begin{corollary}\label{mod over cafe 2}
		Let $R$ be a topologically finite type algebra over $\Ainfe$.
		\begin{enumerate}[(i)]
			\item The ring $R$ is $p$-adically complete and separated.
			\item The ring $R$ is topologically finitely presented over $\Ainfe$ if it is $p$-torsionfree.
			\item Assume $R$ is in $\Afe$, and $I$ is an ideal of $R$.
			Then $I$ is finitely presented over $R$ if $R/I$ is $p$-torsionfree.
		\end{enumerate}
	\end{corollary}
	\begin{proof}
		\begin{enumerate}[(i)]
			\item Note that $R$ is the quotient of $\Ainfe\langle T_1,\ldots, T_m\rangle$ for some $m$, with the latter being in $\Afe$.
			In particular, $R$ is a finitely generated $\Ainfe\langle T_1,\ldots, T_m\rangle$-module.
			So the result follows from Lemma \ref{mod over cafe}.(iv).
			\item By (i), we know $R$ is $p$-adically complete and $p$-torsionfree.
			So it suffice to check that for a surjection $\Ainfe\langle T_1,\ldots, T_m\rangle \ra R$, the  kernel is finitely generated.
			This then follows from Lemma \ref{mod over cafe}.(i), since $R$ is a finitely generated $\Ainfe\langle T_i\rangle$ module that is $p$-torsionfree.
			\item By Lemma \ref{mod over cafe}.(i) applied at the $p$-torsionfree $R$-module $R/I$, we know $I$ is a finitely generated ideal in $R$.
			So thanks to \ref{mod over cafe}.(ii), we know $I$ is in fact finitely presented.
		\end{enumerate}
	\end{proof}

	\begin{lemma}\label{p-comp}
		Let $R$ be in $\Afe$, and $F$ be a flat $R$-module
		\begin{enumerate}[(i)]
			\item The functor $M\mapsto (M\otimes_R F)^\wedge$ is exact on the category of finitely presented $R$-modules.
			\item Given a  finitely presented $R$-module $M$, the following canonical map is an isomorphism
			\[
			M\otimes_R F^\wedge\rra (M\otimes_R F)^\wedge.
			\]
			\item The $R$-module $F^\wedge$ is flat over $R$, and is $p$-torsionfree.

		\end{enumerate}
	\end{lemma}
	\begin{proof}
		\begin{enumerate}[(i)]
			\item Let $0\ra M'\ra M\ra M''\ra 0$ be a short exact sequence of finitely presented $R$-modules.
			By assumption, the tensor product with $F$ over $R$ is exact, so it suffices to show that the $p$-adic completion is flat on $0\ra M'\otimes F\ra M\otimes F\ra M''\otimes F\ra 0$.
			By Lemma \ref{mod over cafe}.(iii), there exists an integer $c\geq 0$ such that $p^kM\cap M'\subset p^{k-c}M'$. 
			Applying this inclusion with the tensor product functor $-\otimes_R F$, and notice that the flatness of $F$ implies $(p^kM\cap M')\otimes F=(p^kM\otimes F)\cap (M'\otimes F)$, we see the $p$-adic topology on $M'\otimes F$ is isomorphic to the subspace topology induced from $M\otimes F$.
			In particular, by \cite{Mat86} Theorem 8.1, we get the exactness
			\[
			0\rra (M'\otimes F)^\wedge\rra (M\otimes F)^\wedge \rra (M''\otimes F)^\wedge\rra 0.
			\]
			
			\item Assume $M$ has the following presentation 
			\[
			R^{\oplus n}\rra R^{\oplus m}\rra M\rra 0.
			\]
			The tensor product of this with $F$ gives 
			\[
			F^{\oplus n}\rra F^{\oplus m}\rra M\otimes F\rra 0.
			\]
			We then take the $p$-adic completion.
			By (i), we get an exact sequence,
			\[
			(F^\wedge)^{\oplus n}\rra 	(F^\wedge)^{\oplus m}\rra (M\otimes F)^\wedge\rra 0. \tag{1}
			\]
			On the other hand, we replace $F$ by $F^\wedge$ in the second exact sequence above, and get
			\[
			(F^\wedge)^{\oplus n}\rra 	(F^\wedge)^{\oplus m}\rra M\otimes F^\wedge\rra 0. \tag{2}
			\]
			The canonical map from (2) to (1) are identities on $F^{\wedge, \oplus n}$ and $F^{\wedge,\oplus m}$.
			Thus we get the isomorphism.
			
			\item It suffices to show that for any injective map of finitely presented modules $M'\ra M$, the tensor product with $F^\wedge$ is still injective.
			This then follows from (ii) and (i).

		\end{enumerate}
	\end{proof}

	\begin{corollary}\label{rel ft}
		Let $f:A\ra B$ be a map of algebras in $\Afe$.
		Then the kernel of any surjective $A$-homomorphism $\rho:A\langle T_i\rangle \ra  B$  is finitely generated over $A$.
		In particular, $B$ is a topologically finitely presented $A$-algebra.
	\end{corollary}
	\begin{proof}
		By assumption, we can write $A$ as $\Ainfe\langle U_j\rangle /I$ for some finitely presented ideal $I$.
		This allows us to write the surjection $\rho$ as $\Ainfe\langle U_j, T_i\rangle/I \ra B$, and it suffices to show that the ideal $J:=\ker(\Ainfe\langle U_j, T_i\rangle/I \ra B)$ is finitely presented.
		Notice that the quotient ring $B$ by assumption is $p$-torsionfree.
		So by Corollary \ref{mod over cafe 2}.(iii), the finite presentedness of the ideal $J$ follows if we can show that the ring $\Ainfe\langle U_j, T_i\rangle/I$ is in $\Afe$.
		The latter by Corollary \ref{mod over cafe 2}.(ii) is equivalent to showing that the ring $\Ainfe\langle U_j, T_i\rangle/I$ is topologically finite type over $\Ainfe$ and is $p$-torsionfree.
		
		Finally, to check the latter conditions, we notice that that the ring $\Ainfe\langle U_j, T_i\rangle/I$ is by construction topologically finite type over $\Ainfe$.
		On the other hand, since $\Ainfe\langle U_j\rangle /I=A\langle T_i\rangle$ is the $p$-adic completion of the flat $A$-module $A[T_i]$, by Lemma \ref{p-comp}.(iii) we know it is $p$-torsionfree.
	\end{proof}
	
	\subsection{Analytic cotangent complex: affinoid case}
	We then introduce the analytic cotangent complex, for algebras in $\Afe$ and affinoid rigid spaces over $\Bdre$  in this subsection.

	\noindent\textbf{Derived $p$-adic completion}\label{derived com}
	We recall the basics of derived $p$-adic completion.
	
	Let $R$ be a $\Z_p$-algebra.
	For a complex $C=C^\bullet$ of $R$-modules, recall that the \emph{derived $p$-adic completion} of $C$ is defined as 
	\[
	\Rlim_{m\in \NN}( C\otimes^L_R\mathrm{cofib}(\xymatrix@C-0.3cm{R \ar[r]^{p^m}&R})), 
	\]
	as an object in the derived category $\scrD(R)$ of $R$-modules.
	Here the object $\mathrm{cofib}(\xymatrix@C-0.3cm{R\ar[r]^{p^m} &R})$ is the cone of the map $p^m:R\ra R$.
	An object $C\in \scrD(R)$ is called \emph{derived $p$-complete} if $C$ is isomorphic to its derived $p$-adic completion.
	The subcategory $\scrD_{p}(R)$ of derived $p$-complete objects is a full subcategory (\cite[Tag 091U]{Sta}) of $\scrD(R)$, and the derived $p$-adic completion forms a left adjoint functor to the inclusion functor $\scrD_{p}(R)\ra \scrD(R)$ (\cite[Tag 091V]{Sta}).
	
	There exists a natural isomorphism of complexes of $R$-modules
	\[
	R\otimes^L_{\Z_p} \mathrm{cofib}(\xymatrix@C-0.3cm{\Z_p \ar[r]^{p^m}&\Z_p})\simeq \mathrm{cofib}(\xymatrix@C-0.3cm{R \ar[r]^{p^m}&R}).
	\]
	From this, the derived functor $C\mapsto C\otimes_R^L \mathrm{cofib}(\xymatrix@C-0.3cm{R \ar[r]^{p^m}&R})$ in $\scrD(R)$ can be rewritten as 
	\begin{align*}
		C\lmt& C\otimes_R^L R\otimes_{\Z_p}^L \mathrm{cofib}(\xymatrix@C-0.3cm{\Z_p \ar[r]^{p^m}&\Z_p})\\
		\simeq& C\otimes^L_{\Z_p} \mathrm{cofib}(\xymatrix@C-0.3cm{\Z_p \ar[r]^{p^m}&\Z_p}).
	\end{align*}
	Here we note that since $\Z_p$ is $p$-torsionfree, the complex $\mathrm{cofib}(\xymatrix@C-0.3cm{\Z_p \ar[r]^{p^m}&\Z_p})$ is isomorphic to the $\Z_p$-module $\Z_p/p^m[0]$ living at the degree $0$.
	In the case when $C$ is a $p$-torsionfree $R$-module, by the flatness of $C$ over $\Z_p$, its derived $p$-adic completion is exactly its classical $p$-adic completion  $\varprojlim_m C/p^m$. 
	This in fact holds true in full generality for complexes as follows.
	\begin{lemma}\label{derived com lem}
		Let $C$ be a cochain complex of $p$-torsionfree $\ZZ_p$-modules.
		Then the derived $p$-completion of $C$ can be represented by a cochain complex $\wt C \in \Ch(\mathbb{Z}_p)$, which is obtained by the term-wise classical $p$-completion of $C$.
	\end{lemma}
	\begin{proof}
		We first notice that $\wt C$ is derived $p$-complete, as the derived $p$-completeness can be checked by cohomology (\cite[Tag 091N]{Sta}) and each $\rmH^i(\wt C)$ is derived $p$-complete.
		
		When $C$ is bounded to the right, the derived tensor product $C\otimes_{\ZZ_p}^L \Z/p^m$ is represented by the cochain complex $C/p^m  \in \Ch(\mathbb{Z}_p)$, obtained via term-wise quotient by $p^m$. 
		In this case, the claim follows from via \cite[Tag 09AU]{Sta} and can be checked by taking the mod $p^n$ reductions, as $\wt C\otimes_{\ZZ_p}^L \Z/p^m=C/p^m$.
		
		In general, consider the naive truncation 
		\[
		\sigma^{>n} C \rra C \rra \sigma^{\leq n} C.
		\]
		By $\mathbb{Z}_p/p^m$ is quasi-isomorphic to a perfect complex over $\mathbb{Z}_p$, the derived tensor product functor $-\otimes_{\ZZ_p}^L \Z/p^m$ commutes with derived limit functor.
		In particular, we have
		\[
		C\otimes_{\ZZ_p}^L \Z/p^m \simeq R\varprojlim_n \left( (\sigma^{\leq n} C)\otimes_{\ZZ_p}^L \Z/p^m \right).
		\]
		Hence we have
		\begin{align*}
			R\varprojlim_m C\otimes_{\ZZ_p}^L \Z/p^m & \simeq R\varprojlim_m R\varprojlim_n \left( (\sigma^{\leq n} C)\otimes_{\ZZ_p}^L \Z/p^m\right) \\
			& \simeq R\varprojlim_m R\varprojlim_n \left( (\sigma^{\leq n}  C)/p^m \right)\\
			& \simeq R\varprojlim_n R\varprojlim_m \left( (\sigma^{\leq n}  C)/p^m \right)\\
			& \simeq R\varprojlim_n \sigma^{\leq n} \wt C\\
			& = \wt C.
		\end{align*}
	\end{proof}
	
	\noindent\textbf{Analytic cotangent complex for affine formal schemes.}\label{aff cot construction}
	Now we introduce the definition and the basic properties of analytic cotangent complexes, for a map of algebras over $\Ainfe$. 
	The analogous discussion for topologically finite type algebras over $K$ can be found in \cite[Section 7.1]{GR}.
	
	\begin{construction}
		Let $f:A\ra B$ ba a map of $\Ainfe$-algebras in $\Afe$.
		Namely both $A$ and $B$ are $p$-adically complete $p$-torsionfree algebras over $\Ainfe$ that are quotients of $\Ainfe\langle T_1,\ldots,T_m\rangle$ for some $m\in \NN$.
		As an $A$-algebra, the ring $B$ admits a standard simplicial resolution 
		\[
		P_\bullet\ra B,
		\]
		where each $P_i$ is a polynomial over $A$ (\cite[Tag 08PM]{Sta}).
		This allows us to give a simplicial $P_\bullet$-modules $\Omega_{P_\bullet/A}^1$, where each $\Omega_{P_i/A}^1$ is the algebraic  differential of $P_i$ over $A$.
		Recall that the \emph{algebraic cotangent complex $\LL_{B/A}$} is the image of the cochain complex $\Omega_{P_\bullet/A}^1\otimes_{P_\bullet}B$ in the derived category over $A$. 
		The \emph{analytic cotangent complexes $\LL_{B/A}^\an$ for the $\Ainfe$-algebras $B\ra A$} is then defined as the image of the derived $p$-adic completion of the $\Omega_{P_\bullet/A}^1\otimes_{P_\bullet} B$, in the derived category of $A$-modules.
	\end{construction}
	\begin{remark}
		As the polynomial resolution is functorial with respect to the pair $(A,B)$, by Lemma \ref{derived com lem} the analytic cotangent complex $\LL_{B/A}^\an$ can be represented functorially by the cochain complex in  $\Ch(B)$, produced by the term-wise $p$-adic completion of $\Omega_{P_\bullet/A}^1\otimes_{P_\bullet} B$.
	\end{remark}
	There exists a canonical map from the algebraic cotangent complex $\LL_{B/A}$ to the analytic cotangent complex $\LL_{B/A}^\an$.
	This is  given by the counit map of the adjoint pair for the derived $p$-completion and the inclusion functor $\scrD_{p}(A)\ra \scrD(A)$.
	
	Here are some useful results for the analytic cotangent complex of $\Ainfe$-algebras.
	\begin{proposition}\label{cot sm}
		Let $f:A\ra B$ be a map of $\Ainfe$-algebras in $\Afe$.
		Assume $f$ is formally smooth.
		Then we have a canonical isomorphism
		\[
		\LL_{B/A}^\an\rra \Omega_{B/A}^{1,\an}[0],
		\]
		where the right side is the module of ($p$-adic) continuous differential forms.
	\end{proposition}
	\begin{proof}
		Since $f$ is formally smooth, by Elkik's algebraization result of formally smooth adic algebras (for this specific situation, see \cite[Page 11, Footnote 6]{BS}, where the noetherian assumption is not needed), $B$ is isomorphic to the $p$-adic completion of a smooth $A$-algebra.
		In particular, $f$ is flat (Lemma \ref{p-comp}) and the map $f_n:=(A/p^n\ra B/p^n)$ is smooth.
		So by the derived base change formula for the algebraic cotangent complex (\cite[Tag 08QQ]{Sta}), since $B/p^n=B\otimes_A^L A/p^n$, we have
		\[
		\LL_{B/A}\otimes^L_A A/p^n=\LL_{(B/p^n)/(A/p^n)}.
		\]
		Moreover, the smoothness of $f_n$ gives a canonical isomorphism
		\[
		\LL_{(B/p^n)/(A/p^n)}\simeq \Omega_{(B/p^n)/(A/p^n)}^1[0]=\Omega_{B/A}^1\otimes_A A/p^n[0].
		\]
		In this way, by taking the derived $p$-adic completion of $\LL_{B/A}$ and notice the $p$-torsionfreeness of $B$ and $A$, we get
		\begin{align*}
			\LL_{B/A}^\an=&\Rlim_{n\in \NN} \LL_{B/A}\otimes^L_A A/p^n\\
			\simeq& \Rlim_{n\in \NN} \Omega_{B/A}^1/p^n[0]\\
			=&\Omega_{B/A}^{1,\an}[0].
		\end{align*}

	\end{proof}
	In the next result, we show that the analytic cotangent complex for a finite morphism coincides with the associated algebraic cotangent complex.
	Recall that for an object $L^\bullet$ in the derived category of $R$-modules, it is called \emph{pseudo-coherent} if it is isomorphic to a upper-bounded complex of finite free $R$-modules.
	\begin{proposition}\label{cot finite}
		Let $A\ra B$ be a map of two topologically finitely presented $\Ainfe$-algebras in $\Afe$, such that $B$ is a finitely presented $A$-module.
		Then the algebraic cotangent complex $\LL_{B/A}$ is pseudo-coherent. 
		In particular, $\LL_{B/A}$ is derived $p$-complete and we have a canonical isomorphism
		\[
		\LL_{B/A}\rra \LL_{B/A}^\an.
		\]
	\end{proposition}
	\begin{proof}
		We first show that it suffices to assume $A\ra B$ is a surjection.
		Too see this, we first pick a polynomial algebra $A[x_1,\ldots, x_r]$ that maps surjectively onto $B$.
		By the finite presentedness assumption of $B$ over $A$, each $x_i$ satisfies a monic polynomial $f_i(x_i)$ of $x_i$ in $A$, and the induced map $B'=A[x_1,\ldots, x_r]/(f_1, \ldots, f_r)\ra B$ is also surjective.
		Here we note that the ring $B'$, as a finite algebra over $A$ that is $p$-torsionfree, is automatically $p$-complete and is also in $\Afe$.
		Moreover, notice that since the sequence $\{f_1,\ldots, f_r\}$ is a regular sequence in $A[x_i]$, 
		by the distinguished triangle of algebraic cotangent complexes for $A\ra B' \ra B$ we get
		\[
		\LL_{B'/A}\otimes^L_{B'} B \rra \LL_{B/A} \rra \LL_{B/B'}.
		\]
		Here we note that as $B'$ is relatively of local complete intersection over $A$, by  \cite[\href{https://stacks.math.columbia.edu/tag/08SL}{Tag 08SL}]{Sta} we know $\LL_{B'/A}$ is a perfect complex over $B'$.
		Thus to show the pseudo-coherence of $\LL_{B/A}$, it suffices to show this for $\LL_{B/B'}$, where $B'\ra B$ is a surjective map of algebras in $\Afe$.

		Recall that by assumption $B$ is a  finite $A$-module that is $p$-torsionfree.
		So Lemma \ref{mod over cafe} implies that $B$ is a finitely presented $A$-module and there exists an exact sequence of $A$-modules as below
		\[
		\xymatrix{A^{\oplus m}\ar[r]^f &A^{\oplus r} \ar[r]& B\ar[r]&0.}
		\]
		Moreover, as the image of $f$ is a submodule of $A^{\oplus r}$, which by Lemma \ref{mod over cafe}.(i) is finitely presented, we know $\ker(f)$ is also finitely generated (hence finitely presented as it is inside of $A^{\oplus m}$).
		This procedure allows to give a finite free $A$-module resolution of $B$.
		In particular, this shows that $B$ is pseudo-coherent over $A$.

		We then take $P$ to be a simplicial polynomial resolution of $B$ over $A$, and let $J$ be the kernel of the map $P\otimes_A B\ra B$.
		Then by the finite presentedness of $B$ over $A$, the simplicial $A$-algebra $P$ is also pseudo-coherent over $A$.
		So by taking a base change along $A\ra B$, we see $P\otimes_A B$ is a simplicial $B$-algebra that is pseudo-coherent over $B$.
		Moreover, since the map $P\otimes_A B\ra B$ has a natural section, the kernel $J$ is also pseudo-coherent (\cite[Tag 064X]{Sta}).
		Notice that the cotangent complex $\LL_{B/A}$ fits into the distinguished triangle (cf. \cite[Chap.\ III.\ 3.3.2]{Ill71}
		\[
		J\rra \LL_{B/A}\rra J^2[1]. 
		\]

		To show the pseudo-coherence of $\LL_{B/A}$, by \cite[Tag 064U]{Sta} it suffices to show by decreasing induction that $\LL_{B/A}$ is $n$-pseudo-coherent for each $n\leq 1$.
		When $n=1$, the result is clear as $\LL_{B/A}$ has cohomological degree $\leq 0$.
		Suppose the result is true for $n\leq 0$.
		Since $A\ra B$ is a surjection, the induced surjective map $P\otimes_A B\ra B$ is isomorphic on $\pi_0$ with kernel living in cohomological degree $\leq -1$.
		Thus by \cite{Ill71} Chap III, 3.3, we have
		\[
		\cH^m(J^i)=0,~for~i>-m.
		\]
		This implies that when $i> -n$, $J^i$ is $n$-pseudo-coherent.
		On the other hand, by \cite{Ill71} Chap III, 3.3.2, there exists an isomorphism.
		\[
		J^j/J^{j+1}\rra L\Sym^j_B(\LL_{B/A}),~j\geq 0.
		\]
		The derived symmetric product preserves the $n$-pseudo-coherence (\cite{GR} 7.1.18), so $J^j/J^{j+1}$ is $n$-pseudo-coherent for any $j\geq 0$.
		Thus the fiber sequence for the quotient $J^i\ra J^i/J^{i+1}$ allows us to deduce the $n$-pseudo-coherent of every $J^i$.
		In particular, when $i=2$, by taking the cohomological twist we know $J^2[1]$ is $(n-1)$-pseudo-coherent.
		So combining with the quasi-coherence of $J$, we get the $(n-1)$-quasi-coherence of $I/I^2=\LL_{B/A}$ by \cite[Tag 064V]{Sta}.
		
	\end{proof}
	\begin{corollary}\label{comp int}
		Let $A$ be a $p$-torsionfree topologically finitely presented $\Ainfe$-algebra, and $I$ be a finitely generated regular ideal in $A$ such that $B:=A/I$ is $p$-torsionfree.
		Then we have a canonical isomorphism
		\[
		\LL_{B/A}^\an\rra  I/I^2[1].
		\]
	\end{corollary}
	\begin{proof}
		Since $B$ is $p$-torsionfree, by the Corollary \ref{mod over cafe 2} we know $B$ is in $\Afe$.
		So the result follows from Proposition \ref{cot finite} and the case for algebraic cotangent complex.
	\end{proof}
	
	Here is another useful result about the distinguished triangles for triples:
	\begin{proposition}\label{tri}
		Let $A\ra B\ra C$ be maps in $\Afe$.
		Then we have
		\begin{enumerate}[(i)]
			\item The analytic cotangent complex $\LL_{B/A}^\an$ is a pseudo-coherent object in the derived category of $B$-modules.
			\item there exists a natural distinguished triangle of pseudo-coherent objects in the derived category of $C$-modules
			\[
			\LL_{B/A}^\an\otimes_B^L C\rra \LL_{C/A}^\an\rra \LL_{C/B}^\an. 
			\]
			
		\end{enumerate}
	\end{proposition}
	Before the proof, we make the following claim:
	\begin{lemma}
		Let $A\ra B$ be a map of algebras in $\Afe$.
		Let $K$ be a bounded above complex of $A$-modules, and $K'$ be its derived $p$-completion. 
		Then the derived $p$-completion of $K\otimes_A^L B$ is isomorphic to the derived $p$-completion of $K'\otimes_A^L B$.
	\end{lemma}
	\begin{proof}[Proof of Lemma]
		It suffices to check by applying the derived tensor product functor $-\otimes_{\Z_p}^L \Z/p^m$, which is then clear as $K\otimes_{\Z_p}^L \Z/p\simeq K'\otimes_{\Z_p}^L \Z/p^m$.
	\end{proof}
	\begin{proof}[Proof of \Cref{tri}]
		\begin{enumerate}[(i)]
			\item By the Corollary \ref{rel ft}, we may write $B$ as the quotient $P/I$, where $P=A\langle T_i\rangle$ is a convergent power series ring over $A$, and $I$ is a finitely generated ideal by the Corollary \ref{mod over cafe 2}.(iii).
			We take the distinguished triangle of algebraic cotangent complexes for $A\ra P\ra B$, and get
			\[
			\LL_{P/A}\otimes_P^L B\rra \LL_{B/A}\rra \LL_{B/P}.
			\]
			Note that $\LL_{P/A}^\an$ is isomorphic to the finite free $P$-module $\Omega_{P/A}^{1,\an}$ (Proposition \ref{cot sm}), so after applying the derived $p$-completion and using the lemma above, we get a distinguished triangle
			\[
			\Omega_{P/A}^{1,\an}\otimes_P B[0]\rra \LL_{B/A}^\an\rra \LL_{B/P}^\an.
			\]
			Here the $\LL_{B/P}^\an$ is pseudo-coherent by Proposition \ref{cot finite}.
			Thus we are done.
			
			\item For (ii), take the distinguished triangle for algebraic cotangent complexes, we get
			\[
			\LL_{B/A}\otimes_B^L C\rra \LL_{B/A}\rra \LL_{C/B}. 
			\]
			So the result follows from the lemma above and the pseudo-coherence of each analytic cotangent complex.
			
		\end{enumerate}
	\end{proof}
	
	\noindent\textbf{Analytic cotangent complex for affinoid rigid spaces.}
	We then introduce the basics of the analytic cotangent complex for a map of affinoid algebras over $\Bdre$, using the integral construction given in the last paragraph.
	The analogous discussion for topologically finite type algebras over $K$ can be found in \cite[Section 7.2]{GR}.
	
	\begin{construction}\label{ana cot}

		Let $f:A\ra B$ be a map of topologically finite type affinoid algebras over $\Bdre$.
		Namely both $A$ and $B$ are quotients of $\Bdre\langle T_1,\ldots,T_m\rangle$ for some $m\in \NN$.
		Denote by $\calC_{B/A}$  the category of pairs of rings $(B_0,A_0)$, where $A_0$ and $B_0$ are rings of definition of $A$ and $B$ respectively, such that both of them are in $\Afe$, and $f(A_0)\subset B_0$.
		The morphism among pairs is defined by inclusion maps on each entry respectively.
		
		Assume $(B_0,A_0)$ is an object of $\calC_{B/A}$.
		By the construction of the last paragraph, we can construct the analytic cotangent complex $\LL_{B_0/A_0}^\an$ for $B_0/A_0$, as the derived $p$-completion of the algebraic cotangent complex $\LL_{B_0/A_0}$.
		\begin{definition}
			The \emph{analytic cotangent complexes for affinoid algebras $B/A$} is defined as the colimit
			\[
			\LL_{B/A}^\an:=\underset{(B_0,A_0)\in\calC_{B/A}}{\colim} (\LL_{B_0/A_0}^\an[\frac{1}{p}]),
			\]
			as an object in the derived category of $B$-modules.
			
		\end{definition}
	\end{construction}
	\begin{remark}\label{rat cot actual complex}
		As there exists a canonical actual complex representing $\LL_{B_0/A_0}^\an$ (by the term-wise $p$-adic completion of $\Omega_{P/A_0}^1\otimes_P B$, where $P$ is the standard polynomial resolution of $B_0$ over $A_0$), the analytic cotangent complexes can also be represented by a canonical actual complex, defined by taking the colimit of the actual term-wise complete complexes and then invert by $p$.
	\end{remark}
	Here we note that as there exists a canonical map from the algebraic cotangent complex $\LL_{B_0/A_0}$ to $\LL_{B_0/A_0}^\an$, induced by the adjoint pair for the derived completion.
	By inverting $p$ we also have a canonical map from the algebraic cotangent complex $\LL_{B/A}$ to the analytic cotangent complex $\LL_{B/A}^\an$, for a map of algebras $A\ra B$ over $\Bdre$.

	We then give a simple description of the analytic cotangent complex for a smooth morphism.
	\begin{proposition}\label{int-rat}
		Let $A_0\ra B_0$ be a map of algebras in $\Afe$, and let $A=A_0[\frac{1}{p}]$ and $B=B_0[\frac{1}{p}]$ be their generic fibers respectively, with the induced map $f:A\ra B$. 
		Assume the corresponding map of affinoid rigid spaces $\Spa(B)\ra \Spa(A)$ is smooth.
		Then we have a natural isomorphism
		\[
		\LL_{B_0/A_0}^\an[\frac{1}{p}]\rra  (\Omega_{B_0/A_0}^{1,\an}[\frac{1}{p}]).
		\]
	\end{proposition}
	\begin{proof}
		By the Corollary \ref{rel ft}, $B_0$ is a topologically finitely presented $A_0$-algebra.
		So we can write $B_0$ as the quotient ring of the relative convergent power series ring $P_0=A_0\langle T_1,\ldots, T_m\rangle$, by some finitely generated ideal $I_0\subset P_0$.
		Denote by $P$ and $I$  the ring $P_0[\frac{1}{p}]$ and the ideal $I_0[\frac{1}{p}]$ respectively.
		Then the surjection $P\ra B$ induces a closed immersion of $\Spa(B)$ into the $m$-dimensional unit disc $\Spa(P)$ over $\Spa(A)$.
		Since both $\Spa(B)$ and $\Spa(P)$ are smooth over $\Spa(A)$, by the Jacobian criterion for the smoothness of adic spaces (\cite{Hu96}, 1.6.9), for each maximal ideal $\frakP$ of $P$ that contains $I$, we can always find generators $s_1,\ldots,s_l$ of $I_\frakP$ such that their derivatives $ds_1,\ldots ds_l$ can be extended to a basis of the continuous differential $\Omega_{P/A,\frakP}^1$ at $\frakP$.
		We denote by $\frakp$  the intersection $\frakP\cap A$.
		Then the above implies that the image of $s_i$ in $P_\frakP\otimes_A (A_\frakp/\frakp A_\frakp)$ forms a regular sequence.
		So by the flatness of $P_\frakP$ over $A_\frakp$ and \cite[Chap.\ 0, Proposition 15.1.16]{EGA IV}, $(s_i)$ forms a regular sequence in $P_\frakP$.
		Since this is true for every maximal ideal $\frakP$ of $P$ containing $I$, we see $B$ is a local complete intersection of $P$.
		
		Now thanks to the surjectivity of $P_0\ra B_0$, Proposition \ref{cot finite} implies that
		\[
		\LL_{B_0/P_0}^\an\simeq \LL_{B_0/P_0}.
		\]
		Moreover, by the flat base change formula for the algebraic cotangent complex (\cite[\href{https://stacks.math.columbia.edu/tag/08QQ}{Tag 08QQ}]{Sta}), there exits a canonical isomorphism of algebraic cotangent complexes
		\[
		\LL_{B_0/P_0}[\frac{1}{p}]\simeq \LL_{B/P},
		\]
		which by the local complete intersection of $P\ra B$, is isomorphic to $I/I^2[1]$.
		On the other hand, by Proposition \ref{tri} and Proposition \ref{cot sm}, we have a natural distinguished triangle
		\[
		(\Omega_{P_0/A_0}^{1,\an}\otimes_{P_0} B_0[\frac{1}{p}]) \rra \LL_{B_0/A_0}^\an[\frac{1}{p}]\rra \LL_{B_0/P_0}^\an[\frac{1}{p}].
		\]
		Replacing the right side by the ideal $I/I^2[1]$ in degree $1$, we get an isomorphism
		\[
		\LL_{B_0/A_0}^\an[\frac{1}{p}]\simeq \mathrm{cofib}(I/I^2\rra \Omega_{P_0/A_0}^{1,\an}\otimes_{P_0} B_0[\frac{1}{p}]).
		\]
		Notice that by \cite[Prop.\ 1.6.9.(ii)]{Hu96} and the smoothness of $\Spa(B)\ra \Spa(A)$, the above map in the right hand side is injective.
		Hence the cofiber complex above lives in cohomological degree zero and is by definition equal to
		\[
		(\Omega_{B_0/A_0}^{1,\an}[\frac{1}{p}])[0].
		\]
	\end{proof}
	As a quite useful upshot, to compute the analytic cotangent complex for affinoid rings, it suffices to use one single pair of rings of definition.
	\begin{proposition}\label{rat cot 1}
		Let $A_0\ra B_0$ be a map of algebras in $\Afe$, and let $A=A_0[\frac{1}{p}]$ and $B=B_0[\frac{1}{p}]$ be their generic fibers respectively, with the induced map $f:A\ra B$. 
		Then the map below is an isomorphism
		\[
		\LL_{B_0/A_0}^\an[\frac{1}{p}] \rra \LL_{B/A}^\an .
		\]
	\end{proposition}
	\begin{proof}
		It suffices to show that for any commutative diagram of topologically finitely presented rings of definition
		\[
		\xymatrix{
			A_0'\ar[r] &B_0'\\
			A_0\ar[u]\ar[r]& B_0\ar[u]},
		\]
		the induced morphism $\LL_{B_0/A_0}^\an[\frac{1}{p}]\ra \LL_{B_0'/A_0'}^\an[\frac{1}{p}]$ is an isomorphism.
		Moreover, using the distinguished triangles for $A_0\ra A_0'\ra B_0'$ and $A_0\ra B_0\ra B_0'$ respectively, we can reduce to show that 
		\[
		\LL_{A_0'/A_0}^\an[\frac{1}{p}]=0.
		\]
		
		Then we notice that since $A_0\ra A_0'$ is an isomorphism after inverting by $p$, this satisfies the assumption of Proposition \ref{int-rat}.
		So it suffices to prove that $\Omega_{A_0'/A_0}^{1,\an}[\frac{1}{p}]=0$.
		By the Corollary \ref{rel ft}, $A_0'$ is a topologically finitely presented algebra over $A_0$.
		We pick a set of generators $x_i$ of $A_0'$ over $A_0$.
		Then by $A_0'[\frac{1}{p}]=A_0[\frac{1}{p}]$, there exists a positive integer $N$ such that $p^Nx_i\in A_0$.
		Note that the continuous differential forms $\Omega_{A_0'/A_0}^{1,\an}$ is generated by the $dx_i$.
		So we get
		\[
		p^Ndx_i=d(p^Nx_i)=0.
		\]
		This implies that $\Omega_{A_0'/A_0}^{1,\an}$ is $p^\infty$-torsion.
		In particular, we have
		\[
		\Omega_{A_0'/A_0}^{1,\an}[\frac{1}{p}]=0.
		\]
		So we are done.
		
	\end{proof}
	Here are some applications of the above result.
	\begin{corollary}\label{rat cot 2}
		Let $f:A\ra B$ be a map of topologically finite type algebras over $\Bdre$, such that $\Spa(B)\ra \Spa(A)$ is smooth.
		Then the projection onto the zero-th homotopy group induces natural isomorphism
		\[
		\LL_{B/A}^\an\rra \Omega_{B/A}^1[0],
		\]
		where the right side is the modules of continuous differential forms.
	\end{corollary}
	
	\begin{proof}
		This follows from Proposition \ref{rat cot 1} and Proposition \ref{int-rat}.
	\end{proof}
	
	Similar to the integral case, when $B$ is a quotient ring of $A$, the analytic cotangent complex coincides with the algebraic cotangent complex.
	\begin{corollary}\label{cot finite 2}
		Let $f:A\ra B$ ba a map of topologically finite type algebras over $\Bdre$ such that $B$ is a finite $A$-module.
		Then the natural map from the algebraic cotangent complex to the analytic cotangent complex below is an isomorphism
		\[
		\LL_{B/A} \rra \LL_{B/A}^\an.
		\]
	\end{corollary}
	\begin{proof}
		Pick a ring of definition $A_0$ of $A$ that is topologically finitely presented over $\Ainfe$.
		We first notice that under the assumption of $A\ra B$, we can find a ring of definition $B_0$ of $B$ that contains $f(A_0)$ and is finite over $A_0$.
		To find such $B_0$, we pick s set of $A$-module generators $x_i$ of $B$ over $A$.
		Then each $x_i$ satisfies a monic polynomial $f_i(X)=\sum_{j=0}^{r_i}a_{ij}X^j$ with coefficients in $A$.
		Since $A=A_0[\frac{1}{p}]$, we can pick a common integer $N\in \NN$, such that coefficients $p^Na_{ij}$ are inside of $A_0$ for each $i$ and $j$.
		From this, we see the element $p^Nx_i$ satisfies a monic polynomial with coefficient in $A_0$.
		In other words, the subring $B_0=f(A_0)[p^Nx_i]$ of $B$ is finite over $A_0$.
		
		Now the corollary follows easily from \Cref{rat cot 1} and Proposition \ref{cot finite}, since $\LL_{B/A}^\an$ is isomorphic to $\LL_{B_0/A_0}^\an[\frac{1}{p}]$, while the latter is computed by inverting $p$ at the algebraic cotangent complex $\LL_{B_0/A_0}$.
		Notice that thanks to the flat base change of the algebraic cotangent complex, $\LL_{B_0/A_0}[\frac{1}{p}]$ is exactly the algebraic cotangent complex of $B$ over $A$.
		So we get the result.
		
	\end{proof}
	
	As expected, we have the following simple description of the analytic cotangent complex for regular immersion.
	\begin{corollary}\label{comp int 2}
		Let $f:A\ra B$ ba a surjective map of topologically finite type algebras over $\Bdre$, such that the kernel $I$ is a regular ideal in $A$.
		Then we have a natural isomorphism
		\[
		\LL_{B/A}^\an \rra I/I^2[1].
		\]
	\end{corollary}
	\begin{proof}
		This follows from \Cref{cot finite 2} and \cite[\href{https://stacks.math.columbia.edu/tag/08SJ}{Tag 08SJ}]{Sta}.
	\end{proof}
	
	Another quick upshot is the pseudo-coherence of the analytic cotangent complex.
	\begin{corollary}
		Let $A\ra B$ be a map of topologically finite type algebras over $\Bdre$.
		Then the analytic cotangent complex $\LL_{B/A}^\an$ is a pseudo-coherent complex of $B$-modules.
	\end{corollary}
	\begin{proof}
		This follows from Proposition \ref{rat cot 1} and the integral version of the result Proposition \ref{tri}.
	\end{proof}
	
	We also obtain the distinguished triangle for triples as follow.
	\begin{corollary}\label{rat cot 3}
		Let $A\ra B\ra C$ be maps of topologically finite type algebras over $\Bdre$.
		Then there exists a distinguished triangle of analytic cotangent complexes of affinoid algebras
		\[
		\LL_{B/A}^\an\otimes_B^L C\rra \LL_{C/A}^\an\rra \LL_{C/B}^\an.
		\]
	\end{corollary}
	\begin{proof}
		Let $A_0\ra B_0\ra C_0$ be an arbitrary choice of rings of definition of $A\ra B\ra C$ that are topologically finitely presented over $\Ainfe$.
		By Proposition \ref{tri}, we have a distinguished triangle 
		\[
		\LL_{B_0/A_0}^\an\otimes_{B_0}^L C_0\rra \LL_{C_0/A_0}^\an\rra \LL_{C_0/B_0}^\an.
		\]
		Note that for the first term above we have the isomorphism
		\begin{align*}
			(\LL_{B_0/A_0}^\an\otimes_{B_0} C_0)\otimes_{\Z_p}^L \Q_p&\simeq (\LL_{B_0/A_0}^\an\otimes_{\Z_p}^L \Q_p)\otimes_{B_0\otimes^L_{\Z_p}\Q_p}^L (C_0\otimes^L_{\Z_p}\Q_p)\\
			&\simeq \LL_{B/A}^\an \otimes_B^L C.
		\end{align*}
		So the corollary follows from Proposition \ref{rat cot 1} by the base change along $\Z_p\ra \Q_p$.
	\end{proof}
	A quick consequence is the following change of base equality.
	\begin{corollary}\label{rat cot 4, et bases}
		Let $A\ra A'\ra B$ be maps of topologically finite algebras over $\Bdre$, with $A\ra A'$ being \'etale.
		Then the natural map below is an isomorphism
		\[
		\LL^\an_{B/A}\rra \LL^\an_{B/A'}.
		\]
	\end{corollary}
	\begin{proof}
		By the distinguished triangle of the triple in the Corollary \ref{rat cot 3}, it suffices to show that $\LL_{A'/A}^\an$ vanishes.
		So this follows from the assumption and the Corollary \ref{rat cot 2}.
	\end{proof}
	As last, we have the \'etale  base change formula as below.
	\begin{corollary}\label{rat cot 5, et base change}
		Let $A\ra B\ra B'$ be maps of topologically finite algebras over $\Bdre$, with $B'/B$ being \'etale.
		Then we have the following natural isomorphism
		\[
		\LL_{B/A}^\an\otimes_B B'\rra \LL_{B'/A}^\an.
		\]
		
		In particular, when there exists an \'etale morphism $A\ra A'$ such that $B'=A'\otimes_A B$, we get the base change formula
		\[
		\LL_{B/A}^\an\otimes_A A'\simeq \LL_{B'/A'}^\an. 
		\]
	\end{corollary}
	\begin{proof}
		The first isomorphism follows from the distinguished triangle (Corollary \ref{rat cot 3}) for $A\ra B\ra B'$, and  the \'etaleness of $B'/B$ (Corollary \ref{rat cot 2}).
		The second isomorphism follows from the Corollary \ref{rat cot 4, et bases} and the isomorphism 
		\[
		\LL_{B/A}^\an\otimes_B B'=\LL_{B/A}^\an\otimes_B B\otimes_A A'.
		\]
	\end{proof}

	\subsection{Derived de Rham complex: affinoid case}\label{subsec aff ddR}
	In this subsection, we construct the derived de Rham cohomology for topologically finite type algebras over $\Bdre$.
	
	As a preparatory step, we first construct the rational analytic derived de Rham complex for a map $A_0\ra B_0$ of $\Ainf_e$-algebras, which is a complete filtered complex over $A_0[\frac{1}{p}]$, namely an object in $\wh\DF(A_0[\frac{1}{p}])$.
	We also refer the reader to \cite[Const.\ 4.1]{Bha12} for basics on the algebraic derived de Rham complex of algebras.
	\begin{construction}\label{const ddr}
		Let $A_0$ be an $\Ainfe$-algebra in $\Afe$.
		We want to build a functor $F:\Afe_{/A_0} \ra \wh\DF(\Bdre),$ sending $A_0\ra B_0$ to a filtered complete $\mathbb{E}_\infty$-algebra over $A_0[\frac{1}{p}]$.
		\begin{itemize}
			\item[Step 1]
			Let $P$ be the standard polynomial resolution of $B_0$ over $A_0$.
			The de Rham complex $\Omega_{P/A_0}^\bullet$ of $P$ over $A_0$ is then a simplicial complex of $P$-modules.
			Moreover, the (direct sum) totalization $\Tot(\Omega_{P/A_0}^\bullet)$ is a cochain complex of $A_0$-modules that comes with a canonical decreasing filtration, defined by $\Fil^i=\Tot(\Omega_{P/A_0}^{\geq i})$.
			
			\item[Step 2]
			Now we take the derived $p$-adic completion of the filtered cochain complex $(\Tot(\Omega_{P/A_0}^\bullet), \Fil^i)$, to get an object $(E,\Fil^i E)$ in the filtered derived category, such that $E$ and $\Fil^i E$ are all derived $p$-adic complete.
			Then we invert $p$ in $(E, \Fil^i E)$, to get an object $(E[\frac{1}{p}], \Fil^i E[\frac{1}{p}])$ in the filtered derived category of $A_0[\frac{1}{p}]$-modules.
			
			\item[Step 3]
			At last, we denote $F(B_0/A_0)$ to be the filtered completion of $(E[\frac{1}{p}], \Fil^i E[\frac{1}{p}])$.
			\footnote{Roughly speaking, the filtered completion here means we are taking the derived inverse limit $R\varprojlim_j E[\frac{1}{p}]/\Fil^j E[\frac{1}{p}]$, together with the induced filtration with the $i$-th filtration being $R\varprojlim_{j\geq i} \Fil^iE[\frac{1}{p}]/\Fil^j E[\frac{1}{p}]$. This can be understood more canonically as the image of $(E[\frac{1}{p}], \Fil^i E[\frac{1}{p}])$ along the symmetric monoidal \emph{filtered completion functor} $\DF(A_0[\frac{1}{p}]) \to \wh\DF(A_0[\frac{1}{p}])$.}
			Thus we get a functor from maps in $\Afe$ to the filtered complete derived category of $\Bdre$-modules (even $A_0[\frac{1}{p}]$-modules), sending $A_0\ra B_0$ to $F(B_0/A_0)$.
			
		\end{itemize}
	\end{construction}
	
	\begin{remark}
		From the construction above, it is clear that the $i$-th graded piece of $F(B_0/A_0)$ is isomorphic to 
		\[
		\bigl( L\wedge^i \LL_{B_0/A_0}^\an[\frac{1}{p}] \bigr)[-i].
		\]
	\end{remark}
	\begin{remark}
		Recall given a complex of $\Z_p$-modules $C$, it admits a natural map onto its derived $p$-adic completion $\wt C$.
		Applying this to the construction above (Step 2), we see there exists a natural filtered map from the algebraic derived de Rham complex $\wh{\mathrm{dR}}_{B_0[\frac{1}{p}]/A_0[\frac{1}{p}]}$ to $F(B_0,A_0)$.
	\end{remark}
	\begin{remark}\label{ddr, to diff}
		From the construction above, the natural map from $P$ to $B_0$ induces a filtered map from $F(B_0,A_0)$ to the continuous de Rham complex $\Omega_{B_0/A_0}^{\bullet,\an}[\frac{1}{p}]$, which is compatible with the differentials.
		Here the filtration on the latter is the usual Hodge filtration.
	\end{remark}
	\begin{remark}
		As the de Rham complex is equipped with a structure of commutative differential graded algebra, and the above constructions are all lax-symmetric monoidal, the filtered complex $F(B_0/A_0)$ is also naturally a filtered $\mathbb{E}_\infty$-algebra in $\Bdre$-modules.
	\end{remark}

	Now we consider the constructions for affinoid rigid spaces.
	Let $f:A\ra B$ be a map of topologically finite type algebras over $\Bdre$.
	Recall that the category $\calC_{B/A}$ is defined as pairs of rings $(B_0,A_0)$, where $A_0$ and $B_0$ are rings of definition of $A$ and $B$ separately, such that both of them are topologically finitely presented over $\Ainfe$, and $f(A_0)\subset B_0$.
	The morphism among pairs is defined by the natural inclusion map of pairs. 
	Here we note that by the Corollary \ref{rel ft}, $B_0$ is a topologically finitely presented algebra over $A_0$ automatically.
	
	\begin{definition}\label{ddr def}
		Let $f:A\ra B$ be a map of topologically finite type algebras over $\Bdre$, and let $\calC_{B/A}$ be the category of pairs of their rings of definitions as above.
		The \emph{analytic derived de Rham complex of $B$ over $A$}, denoted by $\dR_{B/A}$, is an object in $\wh\DF(\Bdre)$, defined as 
		\[
		\dR_{B/A}:=filtered~completion~of~\underset{(B_0,A_0)\in\calC_{B/A}}{\colim} F(B_0/A_0).
		\]
	\end{definition}
	The filtration of $\dR_{B/A}$ is called the \emph{algebraic Hodge filtration}.
	If we forget the filtered structure, we get the \emph{underlying complex of $\dR_{B/A}$}.
	It is denoted as $\ul {\wh{\mathrm{dR}}}^\an_{B/A}$, and is defined as the 0-th filtration of $\dR_{B/A}$,
	which is the image under the natural projection functor
	\begin{align*}
		\wh\DF(\Bdre)\subset \Fun(\NN^\op, \scrD(\Bdre)) &\rra \scrD(\Bdre);\\
		C_\bullet &\lmt C_0.
	\end{align*}

	\begin{corollary}
		Let $A\ra B$ be a map of topologically finite type algebras over $\Bdre$.
		Then the $i$-th graded piece of $\dR_{B/A}$ is naturally isomorphic to $L\wedge^i \LL_{B/A}^\an[-i]$.
		In particular, for any pair of rings of definition $(B_0,A_0)\in \calC_{B/A}$, the natural map below is a filtered isomorphism
		\[
		F(B_0,A_0)\rra \dR_{B/A}.
		\]
	\end{corollary}
	\begin{proof}
		By the construction above, the algebraic Hodge filtration $\Fil^i \dR_{B/A}$ has the graded factor 
		\[
		\gr^i \dR_{B/A}=\underset{(B_0,A_0)\in\calC_{B/A}}{\colim}L\wedge^i\LL_{B_0/A_0}^\an[\frac{1}{p}][-i].
		\]
		By \Cref{rat cot 1} and the assumption on $(B_0,A_0)$, each $L\wedge^i\LL_{B_0/A_0}^\an[\frac{1}{p}]$ is isomorphic to the $i$-th derived wedge product of the analytic cotangent complex for the affinoid algebras $B/A$.
		In particular, the transition maps in the colimit above are all isomorphisms, and we can replace them by one single term.
		Thus we get
		\[
		\gr^i \dR_{B/A}\simeq L\wedge^i \LL_{B/A}^\an[-i].
		\]
		As a consequence, thanks to the filtered completeness, the filtered isomorphism can be checked on the graded pieces, and hence the natural map $F(B_0,A_0) \rra \dR_{B/A}$ is a filtered isomorphism.
		
	\end{proof}
	
	\begin{remark}
		By taking the colimit for the natural filtered map
		\[ 
		\wh{\mathrm{dR}}_{B_0[\frac{1}{p}]/A_0[\frac{1}{p}]} \rra F(B_0,A_0),
		\] 
		we get a canonical filtered map from the algebraic  derived de Rham complex $\wh{\mathrm{dR}}_{B/A}$   to the analytic derived de Rham complex $\dR_{B/A}$.
		This is compatible with the canonical map of each graded factor
		\[
		L\wedge^i\LL_{B/A} \rra L\wedge^i\LL_{B/A}^\an.
		\]
	\end{remark}
	\begin{remark}
		As the colimit is a lax-symmetric monoidal functor, the analytic derived de Rham complex $\dR_{B/A}$ is naturally a filtered $\mathbb{E}_\infty$-algebra in $\scrD(\Bdre)$.
	\end{remark}

	Here we provide a simple description of the analytic derived de Rham complex for two special cases: the smooth case and the complete intersections.
	\begin{proposition}\label{ddr sm}
		Let $A\ra B$ be a smooth map of topologically finite type algebras over $\Bdre$.
		Then 
		the natural morphism below from the analytic derived de Rham complex to the continuous de Rham complex is a filtered isomorphism:
		\[
		\dR_{B/A}\rra \Omega_{B/A}^\bullet.
		\]
		
	\end{proposition}
	\begin{proof}
		By the Remark \ref{ddr, to diff}, there exists a natural filtered map from $\dR_{B/A}$ to the continuous de Rham complex $\Omega_{B/A}^\bullet$, which is compatible with the differential maps.
		By the assumption and Corollary \ref{rat cot 2}, the analytic cotangent complex $\LL_{B/A}^\an$ is isomorphic to the module of continuous differential forms $\Omega_{B/A}^1[0]$, which is a free $B$-module whose rank is equal to the relative dimension $\dim_A(B)$.
		On the other hand, the de Rham complex of affinoid algebras $B$ over $A$ is bounded above by the relative dimension and is thus complete under the Hodge filtration.
		The derived wedge product $L\wedge^i\LL_{B/A}^\an$ is isomorphic to $\wedge^i\Omega_{B/A}^1[0]=\Omega_{B/A}^i[0]$, which vanishes when $i>\dim_A(B)$.
		So by the Construction \ref{const ddr} above, the natural map from the analytic derived de Rham complex to the de Rham complex of $B/A$ induces an isomorphism from the $i$-th graded factor $\gr^i\dR_{B/A}=L\wedge^i\LL_{B/A}^\an[-i]$ to the $i$-th continuous differential $\Omega_{B/A}^i[-i]$.
		In this way, we get a filtered isomorphism from $\dR_{B/A}$ to the de Rham complex $\Omega_{B/A}^\bullet$.
	\end{proof}

	\begin{proposition}\label{ddr finite}
		Let $A\ra B$ be a surjective map of topologically finite type algebras over $\Bdre$. 
		Then the canonical map below is a filtered isomorphism 
		\[
		\wh{\mathrm{dR}}_{B/A} \rra \dR_{B/A} .
		\] 	
		As an upshot, the underlying  complex $\ul {\wh{\mathrm{dR}}}^\an_{B/A}$ is isomorphic to the formal completion $\wh A$ for the surjection $A\ra B$.
	\end{proposition}
	\begin{proof}
		As both $\dR_{B/A}$ and $\wh{\mathrm{dR}}_{B/A}$ are filtered complete, it suffices to show that the induced map on each graded factor is an isomorphism.
		For each $i\in \NN$, the induced map $\gr^i\wh{\mathrm{dR}}_{B/A} \ra \gr^i\dR_{B/A}$ is exactly the natural map induced from the derived $p$-completion integrally (Construction \ref{ana cot}).
		So the first claim follows from the assumption and the Corollary \ref{cot finite 2}.
		
		For the second claim, it follows from the isomorphism between the underlying complex $\ul{\wh{\mathrm{dR}}}_{B/A}$ of the algebraic derived de Rham complex and the formal completion $\wh A$, which is the main result in \cite{Bha12} (see \cite[4.14, 4.16]{Bha12}).
	\end{proof}
	
	\begin{corollary}\label{ddr lci}
		Let $A\ra B$ be a surjective map of topologically finite type algebras over $\Bdre$, such that the kernel ideal $I$ is regular in $A$.
		Then for each $i\in \NN$, we have a natural isomorphism
		\[
		A/I^i[0] \rra \dR_{B/A}/\Fil^i.
		\]
		In particular, by taking the derived limit, we get a filtered isomorphism of algebras
		\[
		\wh A \simeq \dR_{B/A},
		\]
		where the left side is the (classical) $I$-adic completion of $A$.
	\end{corollary}
	\begin{proof}
		This follows from Proposition \ref{ddr finite}, and the case for algebraic cotangent complex explained in the Example 4.5 in \cite{Bha12}, originally proved in \cite[Theorem 2.2.6]{Ill72}.
	\end{proof}

	\subsection{Global constructions}\label{sub glo ddr}
	In this subsection, we construct the global analytic cotangent complexes and the global analytic derived de Rham complexes.
	Our strategy is to show that the affinoid constructions  satisfy the hyperdescent for the rigid topology, thus can be extended to a complex of sheaves over the given rigid space.

	\noindent\textbf{Unfolding.}\label{par unfold}
	We first recall the unfolding of a hypersheaf in $\infty$-category.
	We refer the reader to \cite[\S6.5.3]{Lu09} for the more detailed discussion on hypercoverings, hypersheaves, and the hypercompletion.
	
	Let $X$ be a site that admits fiber products, and let $\mathcal{B}$ be a basis of $X$,
	namely $\mathcal{B}$ is a subcategory of $X$ such that for each object $U$ in $X$, 
	there exists an object $U'$ in $\mathcal{B}$ covering $U$.
	So any hypercovering of an object in $X$ can be refined to a hypercovering with each term in $\mathcal{B}$.
	
	Let $\scrC$ be a presentable $\infty$-category.
	Consider a hypersheaf $\calF\in \Sh^{\mathrm{hyp}}(\mathcal{B},\scrC)$   over $\mathcal{B}$.
	We can then \emph{unfold} the sheaf $\calF$ to a hypersheaf $\calF'$ on $X$, 
	such that its evaluation at any $V\in X$ is given by
	\[
	\calF'(V)=\underset{U'_\bullet\ra V}{\colim}\varprojlim_{[n]\in \Delta} \calF(U'_n),
	\]
	where the colimit is indexed over all hypercoverings $U'_\bullet\ra V$ with $U'_n\in \mathcal{B}$ for all $n$.
	It follows from \cite[Thm.\ 6.5.3.12]{Lu09} (see also the discussion at the beginning of \cite[\S 6.5.3]{Lu09}) that one hypercovering suffices to compute the value of $\calF'(V)$ in the above formula; in other words for a hypercovering $U'_\bullet\ra V$ with each $U'_n$ in the basis $\mathcal{B}$,
	we have a natural weak-equivalence
	\[
	R\lim_{[n]\in \Delta} \calF(U'_n) \rra \calF'(V).
	\]
	In particular for any $U \in \mathcal{B}$, the natural map  $\calF(U) \rra \calF'(U)$ is a weak-equivalence. 
	
	The above construction is functorial with respect to $\calF\in \Sh^{\mathrm{hyp}}(\mathcal{B}, \scrC)$, 
	and we get a natural unfolding functor
	\[
	\Sh^{\mathrm{hyp}}(\mathcal{B},\scrC) \rra \Sh^{\mathrm{hyp}}(X, \scrC),
	\]
	which is in fact an equivalence, with the inverse given by the restriction functor 
	$\Sh^{\mathrm{hyp}}(X,\scrC)\ra \Sh^{\mathrm{hyp}}(\mathcal{B},\scrC)$.
	Here we mention that the equivalence follows quickly from the limit formula in the last paragraph, which says that the restriction of $\mathcal{F'}$ onto the basis $\mathcal{B}$ naturally coincide with the input sheaf $\mathcal{F}\in \Sh^{\mathrm{hyp}}(\mathcal{B}, \scrC)$.
	
	Recall in the special case when $\scrC=\scrD(R)$ is the derived $\infty$-category of $R$-modules, we have a natural equivalence 
	\begin{align*}
		\scrD(X,R) & \rra \Sh^{\mathrm{hyp}}(X, \scrD(R));\\
		C & \lmt (V \mapsto R\Gamma(V,C)).
	\end{align*}
	As an upshot, to define a complex of sheaves of $R$-modules over $X$, 
	it suffices to specify a contravariant functor from the basis $\mathcal{B}$ to $\scrD(R)$, 
	such that it satisfies the hyperdescent condition within $\mathcal{B}$.

	\noindent\textbf{Hyperdescent of $\LL_{B/A}^\an$ and $\dR_{B/A}$.}
	
	We first consider the analytic cotangent complex.
	\begin{proposition}\label{des, cot}
		Let $A\ra B$ be a map of topologically finite type algebras over $\Bdre$, and let $B \ra B_\bullet$ be a map from $B$ to a cosimplicial algebras over $\Bdre$, such that the associated map of rigid spaces $\Spa(B_\bullet)\ra \Spa(B)$ is a rigid open hypercovering.
		Then the induced map below is an isomorphism
		\[
		\LL_{B/A}^\an \rra R\lim_{[n]\in \Delta} \LL_{B_n/A}^\an.
		\]
	\end{proposition}
	\begin{proof}
		We first notice that by the \'etaleness of the map $B\ra B_n$ and the Corollary \ref{rat cot 5, et base change}, $\LL_{B_n/A}^\an$ is naturally isomorphic to the base change $\LL_{B/A}\otimes_B^L B_n$.
		So it suffices to show the map below  is an isomorphism
		\[
		\LL_{B/A}^\an \rra R\lim_{[n]\in \Delta} \LL_{B/A}^\an \otimes_B^L B_n.
		\]
		Note that since each $\Spa(B_n)\ra \Spa(B)$ is an open covering of the rigid space $\Spa(B)$, the induced map $B\ra B_n$ is flat on structure sheaves (\cite[Proposition 1.7.6]{Hu96}).
		In this way, by the surjectivity of $\Spa(B_n)\ra \Spa(B)$, 
		we see the map of affine schemes $\Spec(B_\bullet) \ra  \Spec(B)$ is a faithfully flat hypercover.
		
		Finally, we recall from the faithfully flat descent of quasi-coherent sheaves over the affine scheme $\Spec(B)$ that for a given $B$-module $M$, the canonical map $M \to  R\lim_{[n]\in \Delta} M \otimes_B^L B_n$ is an isomorphism.
		This by induction implies that for any bounded complex of $B$-modules $C$, we have $C\simeq  R\lim_{[n]\in \Delta} C \otimes_B^L B_n$.
		As a consequence, since $\mathbb{L}_{B/A}^\an$  is bounded above, by taking cohomological truncations of $\mathbb{L}_{B/A}^\an$, we get
		\begin{align*}
			\mathbb{L}_{B/A}^\an & = R\varprojlim_{i\leq 0} \tau^{\geq i} 	\mathbb{L}_{B/A}^\an \\
			& \simeq R\varprojlim_{i\leq 0} R\lim_{[n]\in \Delta}  \bigl( \tau^{\geq i} \LL_{B/A}^\an \otimes_B^L B_n \bigr)\\
			& \simeq  R\lim_{[n]\in \Delta}   R\varprojlim_{i\leq 0} \bigl( \tau^{\geq i} \LL_{B/A}^\an \otimes_B^L B_n \bigr) \\
			& \simeq R\lim_{[n]\in \Delta} R\varprojlim_{i\leq 0} \tau^{\geq i} \bigl(\LL_{B/A}^\an \otimes_B^L B_n \bigr) \\
			& = R\lim_{[n]\in \Delta}  \LL_{B/A}^\an \otimes_B^L B_n.
		\end{align*}
		Here the second isomorphism uses the fact that limit functors commute to each other, and the third isomorphism uses the flatness of $B_n$ over $B$.
	\end{proof}
	Using the unfolding technique in Paragraph \ref{par unfold}, we can extend the affinoid construction of the analytic cotangent complex to the global case.
	\begin{corollary}
		Let $X\ra Y=\Spa(A)$ be a map of rigid spaces over $\Bdre$.
		Then there exists a complex of sheaves of $A$-modules $\LL_{X/Y}^\an$ over $X$, such that for any affinoid open subset $U=\Spa(B)$ of $X$,
		we have a natural isomorphism
		\[
		R\Gamma(U,\LL_{X/Y}^\an) = \LL_{B/A}^\an.
		\]
		The complex $\LL_{X/Y}^\an$ is called the \emph{analytic cotangent complex of $X$ over $Y$}.
	\end{corollary}
	
	Similarly, we could unfold the construction of the analytic derived de Rham complex to an arbitrary rigid space.
	\begin{proposition}
		Let $A\ra B$ be a map of topologically finite type algebras over $\Bdre$, and let $B \ra B_\bullet$ be a map from $B$ to a cosimplicial algebras over $\Bdre$, such that the associated map of rigid spaces $\Spa(B_\bullet)\ra \Spa(B)$ is a rigid open hypercovering.
		Then the induced filtered map below is an isomorphism
		\[
		\dR_{B/A} \rra R\lim_{[n]\in \Delta} \dR_{B_n/A}.
		\]
	\end{proposition}
	\begin{proof}
		As a limit functor preserves the filtered completeness, $R\lim_{[n]\in \Delta} \dR_{B_n/A}$ is an object in $\wh\DF(A)$, and checking the isomorphism above can be reduced to their graded pieces.
		Moreover, notice that the graded piece functor commutes with small limits and colimits (cf. \cite[Lemma 5.2]{BMS2}).
		Thus we get
		\begin{align*}
			\gr^i \dR_{B/A}= L\wedge^i\LL_{B/A}^\an [-i] & \rra \gr^i (R\lim_{[n]\in \Delta} \dR_{B_n/A}) \\
			& \simeq R\lim_{[n]\in \Delta} \gr^i \dR_{B_n/A}\\
			& = R\lim_{[n]\in \Delta} L\wedge^i \LL_{B_n/A}^\an [-i].
		\end{align*}
		Notice that the wedge product functor commutes with the tensor product, and for each $n\in \NN$ we have
		\begin{align*}
			L\wedge^i \LL_{B_n/A}^\an &\simeq L\wedge^i (\LL_{B/A}^\an\otimes_B^L B_n) \\
			& \simeq (L\wedge^i\LL_{B/A}^\an) \otimes_B^L B_n,
		\end{align*}
		In this way, the natural map of graded pieces above is an isomorphism, by the similar fpqc hyperdescent for $B\ra B_\bullet$ as in the proof of Proposition \ref{des, cot}.
		So we are done.
		
	\end{proof}
	\begin{corollary}\label{ddr, global}
		Let $X\ra Y=\Spa(A)$ be a map of rigid spaces over $\Bdre$.
		Then there exists a filtered complex of sheaves of $A$-modules $\dR_{X/Y}$ over $X$, such that for any affinoid open subset $U=\Spa(B)$ of $X$,
		we have a natural isomorphism
		\[
		R\Gamma(U,\dR_{X/Y}) =\dR_{B/A}.
		\]
		The complex $\dR_{X/Y}$ is called the \emph{analytic derived de Rham complex of $X$ over $Y$}.
	\end{corollary}

	\subsection{Infinitesimal cohomology and derived de Rham complex}\label{sub inf-ddr}
	In this subsection, we give a comparison theorem between the infinitesimal cohomology and (the underlying complex of) the analytic derived de Rham complex of a rigid space $X$ over $\Bdre$.

	\noindent\textbf{Affinoid comparison.}
	We first consider the affinoid case.
	Our tool is the \v{C}ech--Alexander complex for the infinitesimal cohomology, and the structure of the analytic derived de Rham complex for closed immersions.
	
	\begin{theorem}\label{inf-ddr aff}
		Let $X=\Spa(A)$ be an affinoid rigid space over $\Bdre$.
		Then there exists a natural isomorphism as below 
		\[
		\ul{\wh{\mathrm{dR}}}_{A/\Bdre}^\an \rra R\Gamma(X/\Sigma_{e \inf},\calO_{X/\Sigma_{e}}) .
		\]
		Here $\ul{\wh{\mathrm{dR}}}_{A/\Bdre}^\an$ is the underlying complex of the analytic derived de Rham complex.
	\end{theorem}
	Before the proof, we want to mention that in the proof below, we will see the isomorphism in the statement is induced from a chosen closed immersion $X \ra Y$, where $Y$ is a smooth rigid space.
	Later on, we will use this observation to globalize a general comparison.
	
	\begin{proof}
		Let $P=\Bdre\langle T_1,\ldots, T_m\rangle\ra A$ be a surjection of topologically finite type algebras over $\Bdre$.
		By Proposition \ref{cech} for the crystal $\calF=\calO_{X/\Sigma_e}$, the infinitesimal cohomology of $X/\Sigma_{e \inf}$ can be computed by the cosimplicial cochain complex
		\[
		R\Gamma(X/\Sigma_{e \inf},\calO_{X/\Sigma_e})\simeq (\calO_{X/\Sigma_e}(D(0))\rra \calO_{X/\Sigma_e}(D(1))\rra \cdots),
		\]
		where $D(\bullet)$ is the cosimplicial object of sheaves over the infinitesimal site, produced by the envelope of $A$ in $P^{\wh\otimes_{\Bdre}\bullet+1}$ (see the discussion before Theorem \ref{aff-coh}).
		Here we recall that by the definition of envelope (cf. \Cref{env}), the sheaf $D(m)$ is the direct limit of all infinitesimal neighborhoods of $\Spa(A)$ in $\Spa(P^{\wh\otimes  m+1})$.
		In particular, we have the following isomorphism
		\[
		\calO_{X/\Sigma_e}(D(m))=\varprojlim_i P^{\wh\otimes m+1}/I(m)^i,
		\]
		where $I(m)$ is the kernel of the surjection $P^{\wh\otimes_{\Bdre} m+1}\ra A$.
		
		Now by Proposition \ref{ddr finite}, there exists a natural filtered morphism inducing an isomorphism of their underlying complexes
		\[
		\ul{\wh{\mathrm{dR}}}^\an_{A/P^{\wh\otimes m+1}} \rra \varprojlim_i P^{\wh\otimes  m+1}/I(m)^i .
		\]
		By taking the cosimplicial version of the above isomorphism, we get isomorphisms of cosimplicial complexes 
		\[
		R\Gamma(X/\Sigma_{e \inf},\calO_{X/\Sigma_e})\rra \calO_{X/\Sigma_e}(D(\bullet)) \longleftarrow \ul{\wh{\mathrm{dR}}}^\an_{A/P^{\wh\otimes \bullet+1}}.
		\]
		So in order to prove the theorem, it suffices to show that the natural map $\Bdre\ra P^{\wh\otimes_{\Bdre}\bullet+1}\ra A$ induces an isomorphism on analytic derived de Rham complexes.
		Moreover, by the filtered completeness, it reduces to show the isomorphism 
		\[
		\LL_{A/\Bdre}^\an\rra \LL_{A/P^{\wh\otimes \bullet+1}}^\an.
		\]
		Note that as both $A$ and $P^{\wh\otimes \bullet+1}$ are term-wise topologically finite type over $\Bdre$, by the distinguished triangle of cotangent complexes for triples (Proposition \ref{rat cot 3}), it suffices to show the vanishing of the following
		\[
		\LL_{P^{\wh\otimes \bullet+1}/\Bdre}^\an.
		\]
		
		At last, we notice that by Proposition \ref{rat cot 1}, the complex $\LL_{P^{\wh\otimes \bullet+1}/\Bdre}^\an$ can be computed by inverting $p$ at the term-wise derived $p$-completion of the algebraic cotangent complex of the \v{C}ech nerve $\Cech(\Ainfe\ra \Ainfe[T_1,\ldots, T_r])$.
		So the vanishing we want follows from the vanishing of the algebraic cotangent complex $\LL_{\Cech(\Ainfe\ra \Ainfe[T_i])/(\Ainfe)}$, which is proved in the first part of the Corollary 2.7 in \cite{Bha12}.
	\end{proof}
	In the special case when $A$ is a complete intersection, the above can be improved into a filtered isomorphism.
	Here we recall that the filtration structure on the infinitesimal cohomology, which is called \emph{infinitesimal filtration}, is defined via $R\Gamma(X/\Sigma_{e \inf}, \calJ_{X/\Sigma_e}^\bullet)$, where $\calJ_{X/\Sigma_e}$ is the kernel of the surjection $\calO_{X/\Sigma_e} \ra \calO_X$ over the infinitesimal site (cf. the discussion above \Cref{base K0}).
	\begin{corollary}\label{inf-ddr reg}
		Let $X=\Spa(A)$ be an affinoid rigid space that admits a regular closed immersion into a smooth affinoid rigid space over $\Bdre$.
		Then there exists a natural filtered isomorphism as below 
		\[
		\dR_{A/\Bdre} \longrightarrow R\Gamma(X/\Sigma_{e \inf},\calO_{X/\Sigma_{e}}).
		\]
	\end{corollary}
	\begin{proof}
		The proof is identical to the proof of Theorem \ref{inf-ddr aff}, with the use of the Corollary \ref{ddr lci}.
	\end{proof}

	\noindent\textbf{Comparison in general.}
	We are now ready to prove the comparison between the infinitesimal cohomology and (the underlying complex of) the analytic derived de Rham complex, for a general rigid space over $\Bdre$.
	
	We first introduce the category of smooth immersions.
	\begin{definition}\label{sm emb, def}
		Let $X$ be a rigid space over $\Bdre$.
		We define the \emph{site $\SE_X$ of smooth immersions of $X$} where:
		\begin{itemize}
			\item objects of $\SE_X$ consist of tuples  $(U,Z,i:U\ra Z)$ with $U$ being an affinoid open subset of $X$, $Z$ a smooth affinoid rigid space over $\Sigma_e$, and $i:U\ra Z$ is a closed immersion;
			\item morphisms from $(U_1,Z_1,i_1:U_1\ra Z_1)$ and $(U_2,Z_2,i_2:U_2\ra Z_2)$ consist of commutative diagrams as below
			\[
			\xymatrix{ U_1 \ar[r]^{i_1} \ar[d]& Z_1\ar[d]\\
				U_2 \ar[r]_{i_2} & Z_2,}
			\]
			where $U_1\ra U_2$ is an open immersion over $X$.
			\item A collection of maps $\{(U_j, Z_j, i_j) \ra (U,Z,i)\}$ is a covering if $\{U_j \ra U\}$ and $\{Z_j \ra Z\}$ are coverings of rigid spaces respectively.
		\end{itemize}
	\end{definition}
	
	There exists a natural projection functor from $\SE_X$ to the category of affinoid open subsets $X_\aff$ of $X$, by sending an object $(U,Z,i:U\ra Z)$ to the open subset $U$ in $X$.
	This functor is continuous under their topology.
	Here the associated \emph{push-forward  functor} $\pi_*$ is the constant functor; namely for an ordinary sheaf $\calF$ in the topos $\Sh(X_\aff)$, the push-forward $\pi_*\calF$ satisfies
	\[
	(\pi_*\calF)(U,Z,i)=\calF(U).
	\]
	The \emph{pullback functor} $\pi^{-1}$ sends a sheaf $\calG$ in $\Sh(\SE_X)$ to the sheaf associated with the presheaf
	\[
	(\pi^{-1}\calG)(U)=\underset{(U,Z,i)}{\colim}\calG(U,Z,i),~U\in X_\rig.
	\]
	The colimit above is a filtered colimit, as given any two closed immersions $i_1:U\ra Z_1$ and $i_2:U\ra Z_2$, we can find a common refinement of them by $i:U\ra Z_1\times_{\Sigma_e} Z_2$, with natural projection maps $Z_1\times_{\Sigma_e} Z_2\ra Z_j$ for $j=1,2$.
	In particular, the colimit (thus the inverse image functor $\pi^{-1}$) is exact.
	So, by translating this into the language of sites (\cite[Tag 00X1]{Sta}), we get a natural morphism of  sites
	\[
	\pi:X_\aff\rra \SE_X.
	\]

	Before we prove the main theorem, we first notice that to check two objects in $\scrD(X_\aff)=\scrD(X_\aff, \ZZ)$ are isomorphic, it suffices to do so by pulling back to the category of smooth immersions. 
	Precisely, we have the following general lemma.
	\begin{lemma}\label{sm emb, test}
		Let $\calF$ and $\calG$ be two objects in the derived category $\scrD(X_\aff)$ of sheaves of abelian groups over the site $X_\aff$.
		Assume $f:R\pi_*\calF\ra R\pi_*\calG$ is an isomorphism in $\scrD(\SE_X)$.
		Then the following natural arrows are all isomorphisms
		\[
		\xymatrix{\pi^{-1} R\pi_*\calF\ar[r]^{~~~~~~\wt f} \ar[d]& \calG\\
			\calF&,}
		\]
		where both arrows are counit maps associated with $f:\calF\ra \calG$ and the identity $\calF\ra \calF$ respectively.
		In particular, there exists a natural isomorphism $\calF\ra \calG$ in the derived category $\scrD(X_\aff)$.
	\end{lemma}
	\begin{proof}
		For each $U\in X_\aff$, we have
		\begin{align*}
			R\Gamma(U, \pi^{-1}R\pi_*\calF)&=\underset{(U,Z,i)}{\colim} ~R\Gamma((U,Z,i),R\pi_*\calF)\\
			&=\underset{(U,Z,i)}{\colim} ~R\Gamma(U,\calF)\\
			&\simeq R\Gamma(U,\calF),
		\end{align*}
		where the last  map is an isomorphism as the colimit above is filtered (thus the geometric realization of the index set is contractible).
		In particular, this implies that the counit maps $\pi^{-1}R\pi_*\calF\ra \calF$ and $\pi^{-1}R\pi_*\calG\to \calG$ are isomorphisms.
		Thus the claim follows from the following diagram of natural isomorphisms
		\[
		\xymatrix{ \pi^{-1}R\pi_*\calF \ar[r]^{\pi^{-1}f} \ar[d] & \pi^{-1}R\pi_*\calG \ar[d]\\
			\calF \ar@{.>}[r]_{\wt f} & \calG.}
		\]
	\end{proof}
	
	Now we are able to prove the comparison theorem.
	\begin{theorem}\label{inf-ddR}
		Let $X$ be a rigid space over $\Bdre$.
		Then there exists a natural filtered morphism from the analytic derived de Rham complex to the Hodge-filtered infinitesimal cohomology sheaf as below
		\[
		\dR_{X/\Sigma_e}\rra Ru_{X/\Sigma_e *} \calJ^\ast_{X/\Sigma_e}.
		\]
		Moroever, the induced map between their underlying complexes is an isomorphism
		\[
		\ul{\wh{\mathrm{dR}}}_{X/\Sigma_e}^\an \rra Ru_{X/\Sigma_e*} \calO_{X/\Sigma_e}.
		\]
	\end{theorem}
	\begin{proof}
		We first notice that the isomorphism can be constructed using smooth embeddings.
		To see this, we recall the equivalence of $\infty$-categories $\scrD(X,R)\simeq \Sh^{\mathrm{hyp}}(X,R)$ (cf. Paragraph \ref{par, sheaf}), where the right side is the full sub-$\infty$-category of contravariant functors from $U\in X_\rig$ to the derived category $\scrD(R)$.
		As $X_\aff$ is a basis of the rigid site $X_\rig$, we may use the equivalence $\Sh^{\mathrm{hyp}}(X,R) = \Sh^{\mathrm{hyp}}(X_\aff,R)$ and regard objects in $\scrD(X,R)$ as contravariant functors on affinoid open subsets of $X$.
		So a map or an isomorphism of sheaves of complexes can be constructed via their evaluations at $X_\aff$.
		Moreover, by Lemma \ref{sm emb, test}, by applying the constant functor $R\pi_*$ to the given two objects, it suffices to show that there is an isomorphism between their images over the site of smooth immersions $\SE_X$.
		
		Now for each smooth immersion $i:U=\Spa(B)\ra Z=\Spa(P)$ for affinoid open subset $U\subset X$, as in the proof of \Cref{inf-ddr aff}, we have the following maps of cosimplicial objects in $\wh\DF(\Bdre)$
		\[
		\xymatrix{\dR_{B/\Bdre} \ar[r] &\dR_{B/P^{\wh\otimes \bullet+1}}  \ar[r] &  \ar[r] \wh{\mathrm{dR}}_{B/P^{\wh\otimes \bullet+1}} \ar[r]& \calO_{D(\bullet)} & R\Gamma(U/\Sigma_{e \inf}, \calO_{X/\Sigma_e}) \ar[l]},
		\]
		where $D(n)$ is the envelope of the surjection 
		\[
		P^{\wh\otimes (n+1)=}P^{\wh\otimes_{\Bdre} (n+1)} \ra B.
		\]
		Here we recall that the first map above is induced from the map of rings $\Bdre\to P^{\wh\otimes \bullet+1}$, the second map is the inverse of the algebraicity isomorphism in Theorem \ref{inf-ddr aff}, the third map is the isomorphism in \cite[Cor.\ 4.14]{Bha12}, and the last map is induced from the limit presentation of the cohomology complex.
		In particular, the above arrows are all functorial with respect to the smooth embedding $i:U\to P$, and the induced maps of their underlying complexes above are all isomorphisms, by the proof of the affinoid comparison in Theorem \ref{inf-ddr aff}.
		
		At last, note that the above maps are functorial with respect to  smooth immersions $i:U\ra Z$, so we can improve the above map into the level of  sheaves over $\SE_X$
		\[
		\xymatrix{R\pi_* \dR_{X/\Sigma_e} \ar[r] &\wh{\mathrm{dR}}_{\SE_X}^\an \ar[r] & \calO_{\SE_X} & R\pi_*  Ru_{X/\Sigma_e *}, \calO_{X/\Sigma_e} \ar[l].}
		\]
		Here 
		where $\wh{\mathrm{dR}}_{\SE_X}^\an$ is the (cosimplicial) sheaf sending $i:U\ra Z$ to the filtered complex $\dR_{B/P^{\wh\otimes \bullet+1}}$, and $\calO_{\SE_X}$ is the (cosimplicial) sheaf sending $i:U\ra Z$ to the structure sheaves $\calO_{D(\bullet)}$ of envelopes $D(\bullet)$.
		In this way, the isomorphism of the underlying complexes over the site $\SE_X$ of smooth immersions
		\[
		R\pi_* \ul{\wh{\mathrm{dR}}}_{X/\Bdre}^\an\rra 	R\pi_* Ru_{X/\Sigma_e *}, \calO_{X/\Sigma_e}
		\]
		follows from the above computation at the sections $i:U\ra Z$, and thus we get the result by Lemma \ref{sm emb, test}.
		So we are done.
		
	\end{proof}
	\begin{remark}
		The above comparison map, though functorial with respect to the rigid space $X$, is constructed in an indirect way.
		It is natural to ask if we can directly produce a natural morphism from the analytic derived de Rham complex to the infinitesimal cohomology sheaf.
		Here we want to mention that for algebraic schemes over $\CC$, this could be achieved by the universal property of the derived de Rham complex in the \ic of filtered $\mathbb{E}_\infty$-algebras.
		
	\end{remark}

	\section{\'Eh descent}\label{sec eh}
	In this section, we prove the \'eh hyperdescent of the infinitesimal cohomology of crystals in vector bundles over $X/K$ and $X/\Bdre_{\inf}$ for a rigid space $X$.
	Our goal is to show the comparison between the \'eh de Rham cohomology and the infinitesimal cohomology of a crystal.

	In order to extend a crystal over $X$ to any rigid space that admits a map to $X$ (which in particular is not necessarily an open immersion), we will work with coherent crystals over the big site.
	Here we note that this is only for the technical convenience, as crystals and their cohomology over $X/\Sigma_{e \inf}$ or $X/\Sigma_{e \Inf}$ are equivalent via pullback and restrictions (Proposition \ref{cry, big and small}, Proposition \ref{coh, big small, cry}).

	\subsection{Descent for blowup coverings}
	We first deal with the descent for the blowup covering over the base $\Sigma_1=\Spa(K)$, for an arbitrary complete non-archimedean $p$-adic field $K$, not necessarily algebraically closed.
	The essential idea follows from Hartshorne's proof for algebraic de Rham cohomology \cite[Chap II, Section 4]{Har75}, where he provides a long exact sequence of the algebraic de Rham cohomology for a blowup square.

	We first give a Mayer-Vietories sequence for the infinitesimal cohomology:
	\begin{proposition}[Mayer-Vietories sequence]\label{des MV}
		Let $X$ be a union of two closed analytic subspaces $X_1$ and $X_2$ over $K$, and let $\calF$ be a coherent crystal over $X/K_{\Inf}$.
		Then the map of rigid spaces $X_1\cap X_2\ra X_1\cup X_2\ra X$ induces a natural distinguished triangle as below 
		\[
		Ru_{X/K *}\calF\rra Ru_{X_1/K *}\calF \moplus Ru_{X_2/K *}\calF \rra Ru_{X_1\cap X_2/K *}\calF.
		\]
	\end{proposition}
	\begin{proof}
		As the functor $u_{X/K*}$ is the sheaf-version of the global section functor $\Gamma(X/K_{\Inf},-)$ (see Subsection \ref{sub inf-rig}), it suffices to show that the maps in the statement above produce a natural distinguished triangle after applying $R\Gamma(U,-)$, for every $U\subset X$ open affinoid.
		So we may assume there exists a smooth affinoid rigid space $Z=\Spa(P)$ over $\Sigma_e$, together with a closed immersion of $X=\Spa(P/I)$ into $Z$, where $I$ is the defining ideal.
		Let $X_i$ be the closed analytic subspace defined by the ideal $I_i$ in $P$.
		Then by assumption, $X$ is defined by the ideal $I_1\cap I_2$, and the intersection $X_3:=X_1\cap X_2$ is defined by $I_3:=I_1+I_2$.
		We denote by $D$ and $D_i$  the envelope of $X$ and $X_i$ in $Z$ respectively.
		Here we regard $D$ and $D_i$ to be the ringed spaces, where the underlying topological spaces are $X$ and $X_i$, and their (global sections of) structure sheaves are $\calO_D=\varprojlim_n P/I^n$ and $\calO_{D_i}=\varprojlim P/I_i^n$ respectively (cf. Remark \ref{env space}).
		
		Now by Theorem \ref{glo-coh}, the infinitesimal cohomology of the coherent crystal can be functorially identified as the derived global sections of the following sequence
		\[
		\calF_D\otimes\Omega_D^\bullet\rra \calF_{D_1}\otimes\Omega_{D_1}^\bullet\moplus \calF_{D_2}\otimes\Omega_{D_2}^\bullet\rra \calF_{D_3}\otimes\Omega_{D_3}^\bullet,
		\]
		where $\calF_D\otimes\Omega_D^\bullet$ (resp. $\calF_{D_i}\otimes\Omega_{D_i}^\bullet$) is the restriction of the de Rham complex of the coherent crystal $\calF$ on the envelope $D_X(Z)$ (resp. $D_{X_i}(Z)$).
		Moreover,  by the crystal condition, the finite projective $\calO_{D_i}$-module $\calF_{D_i}$ is equal to the tensor product $\calF_D\otimes_{\calO_D}\calO_{D_i}$.
		Combining with \Cref{diff}, we know that each term $\calF_{D_i}\otimes\Omega_{D_i}^n$ of the de Rham complex, regarded as a module over $\mathcal{O}_{D_i}$, is equal to the pullback of $\mathcal{F}_D\otimes_{\mathcal{O}_Z} \Omega^1_{Z/K}$ along $D_i\ra D$.
		As a consequence, by taking the naive truncations termwisely, to show the sequence above is distinguished, it suffices to show that the following natural sequence of rings is a short exact sequence
		\[
		0\rra \calO_D\rra \calO_{D_1}\moplus \calO_{D_2}\rra \calO_{D_3}\rra 0.
		\]
		And since each structure sheaf of envelopes are given by the formal completions of the ring $P$, we reduce the question to show that the following sequence of inverse systems is exact
		\[
		0\rra \{P/(I_1\cap I_2)^n\}_n\rra \{P/I_1^n\}_n\moplus \{P/I_2^n\}_n\rra \{P/(I_1+ I_2)^n\}_n\rra 0.
		\]
		Notice that for fixed $n$, we always have the following short exact sequence
		\[
		0\rra P/(I_1^n\cap I_2^n)\rra P/I_1^n\moplus P/I_2^n\rra P/(I_1^n+I_2^n)\rra 0.
		\]
		The proof thus follows since the ring $P$ is noetherian, and the inverse systems below are canonically isomorphic
		\[
		\{P/(I_1^n+ I_2^n)\}_n\rra \{P/(I_1+I_2)^n\}_n,~\{P/(I_1\cap I_2)^n\}_n\rra \{P/(I_1^n\cap I_2^n)\}_n.
		\]
		
	\end{proof}

	Here is our main theorem in this subsection.
	\begin{theorem}\label{des blowup}
		Let $X$ be a rigid space over $K$, and let $Y$ be a smooth analytic closed subset of $X$ over $K$, with the blowup map $f:X':=\Bl_X(Y)\ra X$ and the preimage $Y':=Y\times_X X'$ in $X'$.
		Then for any  coherent crystal $\calF$ over $X/K_{\Inf}$, the blowup square for $X'\ra X$ induces the following distinguished triangle
		\[
		Ru_{X/K *}\calF\rra Rf_*Ru_{X'/K *}\calF  \moplus Ru_{Y/K *}\calF \rra Rf_*Ru_{Y'/K *}\calF .\tag{$\ast$}
		\]
		In particular, by taking the derived global sections at $X$, we get a distinguished triangle of the infinitesimal cohomology
		\[
		R\Gamma(X/K_{\Inf},\calF)\rra R\Gamma(X'/K_{\Inf},\calF)\moplus R\Gamma(Y/K_{\Inf},\calF)\rra R\Gamma(Y'/K_{\Inf},\calF).
		\]
	\end{theorem}
	
	Before we prove the result, first we recall the formal function theorem for a proper map of rigid spaces.
	\begin{theorem}\cite[Thm.\ 3.7]{Kie67}\label{formal function}
		Let $f:X'\ra X$ be a proper map of rigid space over $K$, and let $\calI$ be a sheaf of ideal over $X$, with $Y$ the analytic closed subset of $X$ defined by $\calI$.
		Assume $\calG$ is a coherent sheaf over $X'$.
		Then the following natural map is an isomorphism:
		\[
		(R^jf_*\calG)^\wedge\rra R^jf_*(\wh\calG).
		\]
		Here $(-)^\wedge$ is the classical completion of a sheaf of $\calO_X$-modules (or $\calO_{X'}$-modules) with respect to the ideal $\calI$ (resp. $f^{-1}\calI\cdot \calO_{X'}$).
	\end{theorem}
	\begin{remark}
		Here we note that we may get a more derived version of the above theorem using the derived completion, as in \cite[Tag 0A0H]{Sta}.
		For our uses, we do not jump into this generality.
	\end{remark}

	The rest of this subsection is devoted to prove Theorem \ref{des blowup}.

	\noindent\textbf{Special case: $X$ is smooth.}
	First we deal with the special case, assuming $X$ itself is smooth over $K$.
	
	When $X$ is smooth, as the blowup center $Y$ is assumed to be smooth, the blowup $X'$ itself is also smooth over $K$.
	\footnote{To see this, one can use the algebraization argument in the proof of \cite[Prop.\ 5.1.1]{Guo19} to reduce to a blowup of a smooth closed subscheme in a smooth scheme, where the claim is clear.}
	By Theorem \ref{glo-coh}, the derived direct image of the coherent crystal over $X/K_{\inf}$ and $X'/K_{\inf}$ can be computed by their de Rham complexes $\calF_X\otimes \Omega_{X/K}^\bullet$ and $\calF_{X'}\otimes\Omega_{X'/K}^\bullet=f^*\calF_X\otimes\Omega_{X'/K}^\bullet$ respectively.
	On the other hand, the derived direct image $Ru_{Y/K *}\calF$ and $Ru_{Y'/K *}\calF$ are naturally isomorphic to the de Rham complex over the envelopes $D_Y(X)$ and $D_{Y'}(X')$; namely the complexes
	\[
	\calF_D\otimes\Omega_{D_Y(X)}^\bullet, ~~~\calF_{D'}\otimes\Omega_{D_{Y'}(X')}^\bullet,
	\]
	which are compatible with the Hodge filtrations of $\Omega_{X/K}^\bullet$ and $\Omega_{X'/K}^\bullet$.
	So the sequence $(\ast)$ in Theorem \ref{des blowup} is naturally isomorphic to the following sequence of de Rham complexes
	\[
	\calF_X\otimes\Omega_{X/K}^\bullet\rra Rf_*(f^*\calF_X\otimes\Omega_{X'/K}^\bullet)\moplus\calF_D\otimes\Omega_{D_Y(X)}^\bullet\rra Rf_*(\calF_{D'}\otimes\Omega_{D_{Y'}(X')}^\bullet).
	\]
	In fact, we want to show the following more general statement:
	\begin{proposition}\label{des blowup special}
		Let $f:X'\ra X$ be a proper morphism of smooth connected rigid spaces over $K$, and let $Y$ be a closed analytic subset of $X$, with $Y'=f^{-1}(Y)$.
		Let $\calG'$ and $\calG$ be coherent sheaves over $X'$ and $X$ separately, such that $f:X'\ra X$ induces an injective map of $\calO_X$-modules $\calG\ra f_*\calG'$.
		Assume $f$ induces an isomorphism between open subsets $X'\backslash Y'$ and $X\backslash Y$, and also induces an isomorphism of the restriction of $\calG\ra f_*\calG'$ on $X\backslash Y$.
		Then the following natural sequence is a distinguished triangle
		\[
		\calG\otimes\Omega_{X/K}^i\rra Rf_*(\calG'\otimes\Omega_{X'/K}^i)\moplus\calG\otimes\Omega_{D_Y(X)/K}^i\rra Rf_*(\calG'\otimes \Omega_{D_{Y'}(X')/K}^i),~i\in \NN.
		\]
		
	\end{proposition}
	\begin{proof}
		Let $\calM$ be the coherent sheaf $\calG\otimes \Omega_{X/K}^i$ over $X$, and let $\calM'$ be the coherent sheaf $\calG'\otimes \Omega_{X'/K}^i$ over $X'$.
		By the assumption of smoothness, both $\Omega_{X/K}^i$ and $\Omega_{X'/K}^i$ are locally free, and the natural map $\calM\ra f_*\calM'$ is injective.
		Furthermore, the restriction of the map $\calM\ra f_*\calM'$ on the open subset $X\backslash Y$ is an isomorphism.
		
		Recall the $i$-th differential sheaf $\Omega_{D_Y(X)/K}^i$ of $D_Y(X)$, as a sheaf over $X$, is defined as the inverse limit $\varprojlim_m \Omega_{X_m/K}^i$, where $X_m$ is the $m$-th infinitesimal neighborhood of $Y$ in $X$.
		As is shown in Lemma \ref{diff}, the sheaf $\Omega_{D_Y(X)/K}^i$ is naturally isomorphic to the formal completion of the coherent sheaf  $\Omega_{X/K}^i$ along $Y\ra X$, which is also equal to the tensor product $\Omega_{X/K}^i\otimes\calO_{D_Y(X)}$.
		Moreover, since $\calG$ is coherent over $X$, the tensor product $\calG\otimes\Omega_{D_Y(X)}^i$ is isomorphic to the formal completion $\wh\calM$ of $\calM=\calG\otimes \Omega_{X/K}^i$ along $Y\ra X$.
		The same also holds for $X',Y'$ and $\calM'$.

		We denote by $C_1$ and $C_2$  cones of the map $\calM\ra Rf_*\calM'$ and $\wh \calM\ra Rf_*\wh\calM'$ respectively.
		Consider the following commutative diagram
		\[
		\xymatrix{\calM\ar[r]\ar[d] &Rf_*\calM' \ar[r]\ar[d] & C_1\ar[d]\\
			\wh\calM\ar[r] & Rf_*\wh\calM' \ar[r] & C_2.}
		\]
		Here both rows are distinguished.
		
		Now we make the following claim.
		\begin{claim}\label{des blowup claim}
			The natural map $C_1\ra C_2$ of cones above is an isomorphism.
		\end{claim}	
		\begin{proof}[Proof of Claim]
			First we notice that since the map $f:X'\ra X$ is an isomorphism on the open subsets $X'\backslash Y'\ra X\backslash Y$ and both $X'$ and $X$ are smooth, the sheaves of differentials $\Omega_{X/K}^i$ and $\Omega_{X'/K}^i$ are vector bundles, and the induced map $\Omega_{X/K}^i\ra f_*\Omega_{X'/K}^i$ is injective.
			\footnote{For an open immersion $i:U\to X$, thanks to the smoothness, the restriction map $\Omega^i_{X/K}(X) \to \Omega^i_{X/K}(U)$ is injective. When $U$ is away from $Y$, the open immersion $U\to X$ factors through $X'$, and we get the injectivity of the composition $\Omega^i_{X/K}(X)\to \Omega^i_{X'/K}(X') \to \Omega^i_{X/K}(U)$ as well, hence $\Omega^i_{X/K}(X)\to \Omega^i_{X'/K}(X')$ is injective. The same applies when $X$ and $X'$ are replaced by open subspaces, which justifies the injectivity in the level of sheaves.}
			On the other hand, the map $\calG\ra f_*\calG'$ is assumed to be an injective  map of coherent sheaves in the Proposition.
			Combining the above two conditions, we see the map $\calM\ra f_*\calM'$ is injective, and the cone lives in cohomological degree $[0,+\infty)$.
			
			By the Formal Function Theorem \ref{formal function}, the cohomology sheaf $R^jf_*\wh\calM'$ is naturally isomorphic to the formal completion $(R^jf_*\calM')^\wedge$ of $R^jf_*\calM'$ along $Y\ra X$.
			Moreover, by the exactness of the formal completion on coherent sheaves, the natural map $\wh \calM\ra (f_*\calM')^\wedge$ is injective, and we have a short exact sequence
			\[
			0\rra \wh \calM\rra (f_*\calM')^\wedge\rra \calH^0(C_2)\rra 0.
			\]
			This implies that $C_2$ lives also in cohomological degree no smaller than zero.
			Furthermore, by the exactness of the above sequence, the cohomology sheaf $\calH^0(C_2)$ is isomorphic to the formal completion of $\calH^0(C_1)$ at $Y$.
			But since $\calH^0(C_1)$ is coherent and is already supported at $Y$, we have 
			\[
			\calH^0(C_1)=\calH^0(C_1)^\wedge\simeq \calH^0(C_2).
			\]
			This finishes the degree zero part.
			
			For the higher cohomology, we consider the following diagram of cohomologies
			\[
			\xymatrix{
				R^jf_*\calM'\ar[r] \ar[d] & \calH^j(C_1) \ar[d]\\
				(R^jf_*\calM')^\wedge\ar[r] & \calH^j(C_2).}
			\]
			As $\calM$ and $\wh\calM$ are living in cohomological degree zero, the horizontal maps above are isomorphisms, and it suffices to show for each $i>0$, the left vertical map above is an isomorphism.
			But notice that since $f$ induces an isomorphism between $\calM$ and $\calM'$ over $X\backslash Y$, the higher cohomology sheaf $R^jf_*\calM'$ is coherent and is supported over $Y$.
			In particular, the formal completion of $R^jf_*\calM'$ along $Y\ra X$ is equal to itself; namely the natural map below is an isomorphism
			\[
			R^jf_*\calM'\rra (R^jf_*\calM')^\wedge.
			\]
			This leads to the isomorphism
			\[
			\calH^j(C_1)\simeq\calH^j(C_2),~\forall j\geq 1,
			\]
			and we finish the isomorphism between $C_1$ and $C_2$.
			
		\end{proof}
		We change the notation back to Proposition \ref{des blowup special}.
		Then we get two rows of distinguished triangles
		\[
		\xymatrix{\calG\otimes\Omega^i_{X/K}\ar[r]\ar[d] &Rf_*(\calG'\otimes\Omega^i_{X'/K}) \ar[r]\ar[d] & C_1\ar[d]\\
			\calG\otimes\Omega^i_{D_Y(X)/K}\ar[r] & Rf_*(\calG'\otimes \Omega^i_{D_{Y'}(X')/K}) \ar[r] & C_2,}
		\]
		the third vertical map is an isomorphism.

		Finally, we consider the following two maps, 
		\[
		\begin{tikzcd}
			\phi:& \calG\otimes \Omega_{X/K}^i \arrow[r, "{(+,+)}"] & Rf_*(\calG'\otimes\Omega^i_{X'/K})\moplus (\calG\otimes\Omega_{D_Y(X)/K}^i);\\
			\psi:& Rf_*(\calG'\otimes\Omega^i_{X'/K})\moplus (\calG\otimes\Omega_{D_Y(X)/K}^i) \arrow[r, "{(-, +)}"]& Rf_*(\calG'\otimes\Omega_{D_{Y'}(X')/K}^i),
		\end{tikzcd}
		\]
		where $+$ and $-$ are indicating the signs of the map.
		As the composition of the above two maps is equal to zero, the map $\psi$ factors though a morphism $\Cone(\phi)\ra Rf_*(\calG'\otimes\Omega_{D_{Y'}(X')/K}^i)$ (\cite[Tag 08RI]{Sta}).
		In this way, by  chasing diagrams and the Claim above, the map $\Cone(\phi)\ra Rf_*(\calG'\otimes\Omega_{D_{Y'}(X')/K}^i)$ is an isomorphism, and we get the distinguished triangle we want.
	\end{proof}

	\noindent\textbf{General case.}
	We then deal with the general case of Theorem \ref{des blowup}.
	
	\begin{proof}[Proof of  Theorem \ref{des blowup}.]
		As the theorem is a local statement, by passing to an open covering if necessary it suffices to assume $X$ is affinoid and admits a closed immersion into a smooth affinoid rigid space $Z$.
		Moreover, as any coherent crystal over $X/K_{\Inf}$ is a crystal in vector bundles (Corollary \ref{cry, vec-bun, K}), by further taking  open subsets we may assume $\calF_X\simeq \calO_X^{\oplus m}$ is trivial over $X$.

		As $Y$ is smooth over $K$, the blowup $Z'=\Bl_Z(Y)$ of the smooth rigid space $Z$ at the center $Y$ is also smooth.
		Moreover, as $X\ra Z$ is a closed immersion while $X'$ is the blowup of $X$ at $Y$, the natural map $X'\ra Z'$ is also a closed immersion, which is equal to the preimage of $X$ along the blowup map $f:Z'\ra Z$.
		So we get the following commutative diagrams of rigid spaces over $K$ with both of the square being cartesian
		\[
		\xymatrix{ Y'\ar[r] \ar[d] &X'\ar[r] \ar[d] &Z'\ar[d]\\
			Y\ar[r] &X\ar[r] &Z.}
		\]

		Since the restriction $\calF_{D_X(Z)}$ is a vector bundle over $\calO_{D_X(Z)}$ whose pullback along $X \ra D_X(Z)$ is trivial, 
		by Nakayama's lemma we know $\calF_{D_X(Z)}$ is already a trivial bundle, and we may let $\calG$ be a trivial vector bundle over $Z$ such that the tensor product $\calG\otimes_{\calO_Z} \calO_{D_X(Z)}$ is equal to $\calF_{D_X(Z)}$.
		Let $\calG'$ be the pullback $f^*\calG$, as a trivial bundle over $Z'$.
		Then by the crystal condition of $\calF$ we have
		\[
		\calG'\otimes_{\calO_{Z'}} \calO_{D_{X'}(Z')} =\calF_{D_{X'}(Z')}.
		\]
		Here we note that by our choices and the diagram above, the map $\calG\ra f_*\calG'$ is injective and is an isomorphism when restricted to open subsets $Z\backslash Y$ and $Z\backslash X$.
		Similar to the proof of Proposition \ref{des blowup special}, we let $\calM$ and $\calM'$ be the tensor products $\calG\otimes\Omega_{Z/K}^i$ and $\calG'\otimes\Omega_{Z'/K}^i$ over $Z$ and $Z'$ respectively.

		Now by the proof of Proposition \ref{des blowup special}, we have the following natural commutative diagrams, with each row being distinguished 
		\[
		\xymatrix{
			\calM\ar[r] \ar[d]& Rf_*\calM' \ar[d] \ar[r] &C_1  \ar[d]\\
			\wh{\calM_{/X}} \ar[r] \ar[d]& Rf_* \wh{\calM'_{/X'}} \ar[r] \ar[d]& C_2 \ar[d]\\
			\wh{\calM_{/Y}} \ar[r] & Rf_*\wh{\calM'_{/Y'}} \ar[r]& C_3.}
		\]
		Here the sheaf $\wh{\calM_{/X}}$ and similar for the others is denoted to be the formal completion of $\calM$ along $X\ra Z$.
		Thanks to Claim \ref{des blowup claim}, the map $C_1\ra C_2$ and the $C_1\ra C_3$ are isomorphisms.
		In particular, the map $C_2\ra C_3$ is an isomorphism.
		In this way, as in the last part of the proof for Proposition \ref{des blowup special}, the second and the third rows above produces the following distinguished triangle
		\[
		\begin{tikzcd}
			\wh{\calM_{/X}} \arrow[r,"{(+,+)}"] & Rf_* \wh{\calM'_{/X'}} \bigoplus \wh{\calM_{/Y}} \arrow[r,"{(-,+)}"] & Rf_*\wh{\calM'_{/Y'}}.
		\end{tikzcd}
		\]
		Hence by Lemma \ref{diff}, we may replace those formal completions by their corresponding sheaves over envelopes, and obtain the distinguished triangle below
		\[
		\xymatrix{ \calF_{D_X(Z)}\otimes\Omega_{D_X(Z)}^i \ar[r]& Rf_*(\calF_{D_{X'}(Z')}\otimes\Omega_{D_{X'}(Z')}^i) \moplus \calF_{D_Y(Z)}\otimes\Omega_{D_Y(Z)}^i \ar[r] & Rf_*(\calF_{D_{Y'}(Z')}\otimes\Omega_{D_{Y'}(Z')}^i),}
		\]
		which implies the theorem by taking different $i$ and Theorem \ref{glo-coh}.

	\end{proof}

	\begin{remark}
		With the help of the blowup triangle in Theorem \ref{des blowup}, we could show the following:
		for a universal homeomorphism of rigid spaces $f:X'\ra X$  over $K$ and a coherent crystal $\calF$ over $X/K_{\Inf}$,
		there exists a natural isomorphism of cohomology sheaves as below
		\[
		Ru_{X/K *}\calF\rra Rf_*Ru_{X'/K *}\calF.
		\]
	\end{remark}

	\subsection{\'Eh-hyperdescent in general}
	Now we are ready to prove the \'eh descent for a crystal over the infinitesimal site.
	We first deal with the case for the infinitesimal cohomology over an arbitrary $p$-adic field $K$, where the strategy is to use the blowup square for the \'eh-topology in \cite{Guo19} and the descent results in the first subsection.
	After that, we generalize to the case over $\Bdre$.
	
	Recall that the \'eh site $X_\eh$ is defined over the category $\Rig_K|_X$ of all $K$-rigid spaces over $X$ and is equipped with the \'eh-topology (cf. \cite[Section 2]{Guo19}).
	For an object $X'\ra X$ in $\Rig_K|_X$, we denote by $\pi_{X'}:X_\eh\ra X'_\rig$  the map from the \'eh site of $X$ to the small rigid site of $X'$, and denote by $\pi:X_\et \to \Rig_K|_{X}$ the map from the \'eh site of $X$ to the big rigid site over $X$.

	We first associate an infinitesimal crystal together with its de Rham complex an analogous construction over the \'eh topology.
	\begin{construction}[\'eh de Rham complex]\label{eh-de Rham complex}
		Let $\calF$ be an infinitesimal sheaf of $\calO_{X/K}$-modules over the big infinitesimal site $X/K_{\Inf}$.
		We then associate a sheaf $\calF_\rig:=i_{X/K}^{-1}\calF$ of $\calO_X=i_{X/K}^{-1}\calO_{X/K}$-modules over the big rigid site $\Rig_K|_X$, where $i_{X/K}:\Sh(\Rig_K|_X) \ra \Sh(X/K_{\Inf})$ is the morphism of topoi as in Subsection \ref{sub inf-rig}.
		Here the section of $\calF_\rig$ at an object $f:X'\ra X$ in $\Rig_K|_X$ is the $\calO_{X'}(X')$-module
		\[
		\calF(X',X'),
		\]
		where $(X',X')\in X/\Sigma_{e \Inf}$ is the trivial thickening of $X'$.
		We could then sheafify it with the \'eh-topology, and thus get an \'eh-sheaf $\calF_\eh$ over the \'eh site $X_\eh$.

		Now we specify $\calF$ to be a coherent crystal over big the infinitesimal site.
		As in the discussion of Paragraph \ref{cry dR cpx}, we could associate $\calF$ its de Rham complex over the big infinitesimal site
		\[
		\calF\otimes_{\calO_{X/K}} \Omega_{X/K_{\Inf}}^\bullet.
		\]
		This allows us to get a complex of sheaves over $\Rig_K|_X$ and $X_\eh$ respectively
		\begin{align*}
			\calF_\rig \rra \calF_\rig \otimes_{\calO_\rig} \Omega_\rig^1 \rra \calF_\rig \otimes_{\calO_\rig} \Omega_\rig^2 \rra \cdots;\\
			\calF_\eh \rra \calF_\eh \otimes_{\calO_\eh} \Omega_\eh^1 \rra \calF_\eh \otimes_{\calO_\eh} \Omega_\eh^2 \rra \cdots.
		\end{align*}
		Here the sheaf $\Omega_\rig^i$ is the usual continuous K\"ahler differential sheaf over the rigid site, and $\Omega_\eh^i$ is the \'eh continuous K\"ahler differential, which is equal to the \'eh sheafification of the usual continuous differential.
		The complex $\calF_\eh\otimes\Omega_\eh^\bullet$ is called \emph{\'eh de Rham complex associated with the crystal $\calF$}.

	\end{construction}

	The Construction \ref{eh-de Rham complex} produces two maps of objects in the derived category of the big rigid topos:
	\[
	\xymatrix{ Ru_{X/K*}(\calF\otimes_{\calO_{X/K}} \Omega_{X/K_{\Inf}}^\bullet) \ar[r] & \calF_\rig \otimes_{\calO_\rig} \Omega_\rig^\bullet \ar[r] & R\pi_*(\calF_\eh\otimes_{\calO_\eh} \Omega_\eh^\bullet).}
	\]
	Formally the first map is given by the sheafified version of the natural transformation 
	\[
	\Rig_K|_X \ni X' \lmt \left( R\Gamma(X'/\Sigma_{e \Inf}, -) \ra R\Gamma((X',X'), -) \right)
	\]
	at the complex of infinitesimal sheaves $\calF\otimes_{\calO_{X/K}}\Omega_{X/K_{\Inf}}^\bullet$.
	The second map above comes from the counit morphism for the adjoint pair $(\pi^{-1},\pi_*)$, where $\pi^{-1}$ is the \'eh-sheafification functor.
	Moreover, the above map can be improved into the filtered derived category, where the left side is equipped with the infinitesimal filtration, and the rest two complexes are equipped with their Hodge filtrations.
	
	We also note that the natural map from the de Rham complex $\calF\otimes_{\calO_{X/K}}\Omega_{X/K_{\Inf}}^\bullet$ to the crystal $\calF$ itself induces a natural isomorphism as below
	\[
	\begin{tikzcd}
		Ru_{X/K*}(\calF\otimes_{\calO_{X/K}}\Omega_{X/K_{\Inf}}^\bullet) \arrow[r, "\sim"] & Ru_{X/K*}\calF,
	\end{tikzcd}
	\]
	which is proved in Proposition \ref{van}.
	
	Now we can state the descent result.
	\begin{theorem}\label{inf-eh}
		Let $X$ be a rigid space over $K$, and let $\calF$ be a coherent crystal over the big infinitesimal site $X/K_{\Inf}$.
		Then the natural map of $K$-linear complexes below is an isomorphism.
		\[
		Ru_{X/K*}(\calF\otimes_{\calO_{X/K}} \Omega_{X/K_{\Inf}}^\bullet) \rra R\pi_{X *}(\calF_\eh\otimes_{\calO_\eh} \Omega_\eh^\bullet).
		\]
		In particular, the infinitesimal cohomology of the coherent crystal $\calF$ satisfies the \'eh-hyperdescent.
	\end{theorem}
	\begin{proof}
		$~$
		\begin{itemize}
			\item[Step 1]
			In the Step 1, we show that by restricting to a smooth rigid space $X'$ over $K$ that admits a map to $X$, the morphism in the statement is an isomorphism.
			Namely the natural morphism below is an isomorphism
			\[
			Ru_{X'/K*}(\calF\otimes_{\calO_{X/K}} \Omega_{X/K_{\Inf}}^\bullet) \rra R\pi_{X' *}(\calF_\eh\otimes_{\calO_\eh} \Omega_\eh^\bullet).
			\]
			
			We first apply $R\Gamma(X',-)$ for the smooth rigid space $X'$.
			On the one hand, we apply Theorem \ref{glo-coh} to the trivial closed immersion $X'\ra X'$ to get
			\begin{align*}
				R\Gamma(X'/K_\Inf,\calF) & \simeq R\Gamma(X'/K_{\Inf},\calF\otimes_{\calO_{X/K}} \Omega_{X/K_{\Inf}}^\bullet) \\
				&\simeq R\Gamma(X',\calF_D\otimes \Omega_D^\bullet)\\
				&=R\Gamma(X'_\rig,\calF_{X'}\otimes \Omega_{X'/K}^\bullet),
			\end{align*}
			where the envelope $D$ for the trivial closed immersion $X'\ra X'$ is just $X'$ itself.
			
			On the other hand, the \'eh-cohomology $R\Gamma(X'_\eh, \calF_\eh\otimes\Omega_\eh^\bullet)$ is in fact isomorphic to the cohomology of the de Rham complex of $\calF_{X'}$ given by the restriction of $\calF_\rig\otimes \Omega_\rig^\bullet$ at $X'$.
			To see this, we notice that as the natural map of complexes $\calF_\rig\otimes \Omega_\rig^\bullet \ra R\pi_{X' *} (\calF_\eh \otimes \Omega_\eh^\bullet)$ is  filtered with respect to the Hodge filtrations, it suffices to show for each $i\in \NN$ the following map is an isomorphism
			\[
			R\Gamma(X'_\rig,\calF_{X'}\otimes \Omega_{X'/K}^i) \rra  R\Gamma(X'_\eh,\calF_\eh\otimes\Omega_\eh^i).
			\]
			Here the restriction $\calF_\eh|_{X'}$ at the \'eh site of $X'$ can be given by the \'eh-sheafification of the rigid sheaf $\calF_\rig|_{X'}$ over the big site $\Rig_K|_{X'}$.
			Moreover, as $\calF$ is a crystal in vector bundles (Corollary \ref{cry, vec-bun, K}) and the statement is local on $X'$ (namely both of the above two complexes satisfies the hyperdescent for rigid topology), by passing to an open rigid subspace of $X'$ if necessary we may assume the restriction $\calF_\rig|_{X'}$ of $\calF$ at $X'$ is isomorphic to the direct sum $\calO_{X'}^m$ of structure sheaves.
			Thus we reduce to show that the natural map of cohomology of differentials below for a smooth $K$-rigid space $X'$ is an isomorphism
			\[
			R\Gamma(X'_\rig, \Omega_{X'/K}^i) \rra R\Gamma(X'_\eh,\Omega_\eh^i),
			\]
			which is proved in \cite[Theorem 4.0.2]{Guo19}.
			
			In this way, as the map in the statement is given by the composition
			\[
			R\Gamma(X'/K_{\Inf},\calF\otimes_{\calO_{X/K}} \Omega_{X/K_{\Inf}}^\bullet) \rra R\Gamma(X'_\rig, \calF_{X'}\otimes \Omega_{X'/K}^\bullet) \rra R\Gamma(X'_\eh, \calF_\eh\otimes_{\calO_\eh} \Omega_\eh^\bullet),
			\]
			we see both maps above are isomorphisms when $X'$ is smooth over $K$.
			
			At last, notice that the cofiber $C$ of the map $Ru_{X'/K*}(\calF\otimes_{\calO_{X/K}} \Omega_{X/K_{\Inf}}^\bullet) \rra R\pi_{X' *}(\calF_\eh\otimes_{\calO_\eh} \Omega_\eh^\bullet)$ is a bounded below complex of sheaves over the small rigid site $X'_\rig$.
			If $C$ is not acyclic, then there would exist an open subspace $U$ of $X'$ such that the cohomology $R\Gamma(U_\rig, C)$ does not vanishes, which contradicts to the computation above.
			So we get the isomorphism for smooth $K$-rigid space $X'$ that admits a map to $X$:
			\[
			Ru_{X'/K*}(\calF\otimes_{\calO_{X/K}} \Omega_{X/K_{\Inf}}^\bullet) \rra R\pi_{X' *}(\calF_\eh\otimes_{\calO_\eh} \Omega_\eh^\bullet).
			\]

			\item[Step 2]
			Now we prove the isomorphism of the derived push-forwards as in the statement.
			
			Let $i:X_\red \ra X$ be the closed immersion by the reduced sub rigid space of $X$.
			We first notice that the natural map below is an isomorphism
			\[
			Ru_{X/K *}\calF\rra i_*Ru_{X_\red/K *}(\calF),
			\]
			This is because locally both of them are computed using the de Rham complex $\calF_D\otimes \Omega_D^\bullet$, where $D$ is the envelope of $X$ in a smooth rigid spaces (Theorem \ref{glo-coh}).
			The same isomorphism holds for the derived direct image of the infinitesimal de Rham complex $\calF\otimes\Omega_{X_\red/K_{\Inf}}^\bullet$ by Proposition \ref{van}.
			On the other hand, We notice that as the closed immersion $i:X_\red\ra X$ is an \'eh-covering (\cite[Section 2.4]{Guo19}), which forms a constant \'eh-hypercovering as the product $X_\red\times_X X_\red$ is equal to $X_\red$ itself, we get
			\[
			R\pi_{X*}(\calF_\eh\otimes \Omega_\eh^\bullet)\simeq i_*R\pi_{X_\red*}(\calF_\eh\otimes\Omega_\eh^\bullet).
			\]
			Thus the above two isomorphisms allow us to assume $X$ is reduced, and by passing to an open subset if necessary we may assume $X$ is quasi-compact.
			
			Now we can do the induction on the dimension of $X$.
			When $\dim(X)$ is of dimension zero, as $X$ is quasi-compact and reduced, it is then equal to a disjoint union of finite points and in particular  is smooth over $K$, where the statement follows from the Step 1.
			
			In general, by Temkin's resolution of singularities for rigid spaces (\cite[1.2.1, 5.2.2]{Tem12}), we can find a finite compositions $X_n\ra \cdots X_1\ra X_0=X$, where each $X_{i}\ra X_{i-1}$ is a blowup at a smooth nowhere dense closed subspace $Y_i \subset X_{i-1}$, such that in the last step $X_n$ is smooth over $K$.
			We denote by $Y_i'$  the preimage $Y_i\times_{X_{i-1}} X_i$ in $X_i$, which is of dimension strictly smaller than $\dim(X_i)=\dim(X)$, and we let $f_i$ be the blowup map $X_i\ra X_{i-1}$.
			Then for each $1\leq i\leq n$, we get a natural distinguished triangle by Theorem \ref{des blowup}
			\[
			Ru_{X_{i-1}/K *} (\calF\otimes\Omega_{\Inf}^\bullet) \rra Rf_{i *}Ru_{X_i/K *}(\calF\otimes\Omega_{\Inf}^\bullet) \moplus Ru_{Y_i/K *}(\calF\otimes\Omega_{\Inf}^\bullet) \rra Rf_{i *}Ru_{Y'_i/K *} (\calF\otimes\Omega_{\Inf}^\bullet). \tag{$\ast$}
			\]
			Again here we use Proposition \ref{van} to replace the $\calF$ by its de Rham complex.
			On the other hand, by the sheafified version of the blowup square in the \'eh-topology (\cite[Proposition 5.1.4]{Guo19}), we have a natural distinguished triangle
			\[
			R\pi_{X_{i-1} *} (\calF_\eh\otimes \Omega_\eh^\bullet) \rra R\pi_{X_i *} (\calF_\eh\otimes \Omega_\eh^\bullet) \moplus R\pi_{Y_i *}(\calF_\eh\otimes \Omega_\eh^\bullet)  \rra R\pi_{Y_i' *} (\calF_\eh\otimes \Omega_\eh^\bullet). \tag{$\ast\ast$}
			\]
			The functoriality of the map in the statement allows us to produce a map from the triangle $(\ast)$ to the triangle $(\ast\ast)$.
			Moreover, by the dimension assumption and the induction assumption, we know the statement in the Theorem is true for $X_n$ and all of $Y_i$ and $Y_i'$.
			In this way, since $X_0=X$, by a finite step of descending inductions via comparing the above two triangles $(\ast)$ and $(\ast\ast)$, the natural map below is then an isomorphism
			\[
			Ru_{X/K *} (\calF\otimes \Omega_{X/K_\Inf}^\bullet) \rra R\pi_{X *} (\calF_\eh\otimes \Omega_\eh^\bullet).
			\]
			So we are done.

		\end{itemize}
		
	\end{proof}
	
	\begin{remark}
		As in \Cref{rmk direct summand}, by the construction of the map in Theorem \ref{inf-eh}, we see that (the underlying complex of) the infinitesimal cohomology of a coherent crystal $\calF$ over $X/K_{\Inf}$ is a direct summand of the cohomology $R\Gamma(X_\rig, \calF_X\otimes\Omega_{X/K}^\bullet)$ of the usual de Rham complex over $X_\rig$.
	\end{remark}
	\begin{remark}
		The isomorphism in the Theorem above cannot always be improved into a filtered isomorphism.
		The discrepancy already appears in the schematic theory (see \cite[Example 5.6]{Bha12}).
	\end{remark}
	
	Now we are able to generalize the \'eh-hyperdescent to coherent crystals over $X/\Sigma_{e \Inf}$ for general $e$, not just $K$-linear crystals.
	We assume $K$ is complete and algebraically closed in the next Theorem, so $\Bdre$ is well-defined for $K$.
	\begin{theorem}\label{inf-eh, Bdre}
		Let $X$ be a rigid space over $X$, and let $\calF$ be a crystal in vector bundles over the big infinitesimal site $X/\Sigma_{e \Inf}$.
		Then the infinitesimal cohomology of $\calF$ over $X/\Sigma_{e \Inf}$ satisfies the \'eh-hyperdescent.
		Namely for an \'eh-hypercovering $X_\bullet'\ra X'$ of $K$-rigid spaces over $X$, the following natural map is an isomorphism
		\[
		R\Gamma(X'/\Sigma_{e \Inf},\calF) \rra R\lim_{\Delta} \left( R\Gamma(X'_\bullet/\Sigma_{e \Inf},\calF)\right).
		\] 
	\end{theorem}
	\begin{proof}
		We prove the result by induction on $e$.
		For $e=1$, it is Theorem \ref{inf-eh}.
		In general, we take the derived tensor product of short exact sequence $0\ra K \ra \Bdre \ra \mathrm{B}_{\mathrm{dR},e-1}^+\ra 0$ with the complex of sheaves $Ru_{X/\Sigma_e *} \calF$.
		By the big site version of the base change formula in Proposition \ref{coh, change of bases} (cf. Corollary \ref{coh, big small}), we get a natural distinguished triangle
		\[
		Ru_{X/K*} \calF_1 \rra Ru_{X/\Sigma_e *} \calF \rra Ru_{X/\Sigma_{e-1} *} \calF_{e-1},
		\]
		where $\calF_1$ and $\calF_{e-1}$ are pullbacks of $\calF$ along maps of sites $X/K_{\Inf} \ra X/\Sigma_{e \Inf}$ and $X/\Sigma_{e-1, \Inf} \ra X/\Sigma_{e \Inf}$ respectively.
		In this way, applying the natural transformation $R\Gamma(X',-) \ra  R\lim_{\Delta}R\Gamma(X'_\bullet,-)$ to the above triangle, we get the result by induction.
	\end{proof}

	\subsection{Finiteness}
	With the use of the \'eh-hyperdescent, we show in this subsection the finiteness and the cohomological boundedness of the infinitesimal cohomology, assuming the properness of the rigid space.

	\begin{theorem}\label{fin Bdre}
		Let $X$ be a proper rigid space over $K$, and let $\calF$ be a crystal in vector bundles over the big infinitesimal site $X/\Sigma_{e \Inf}$.
		Then the infinitesimal cohomology $R\Gamma(X/\Sigma_{e \Inf}, \calF)$ is a bounded complex supported in the cohomological degrees $[0,2n]$, where each cohomology is a finite $\Bdre$-module.
	\end{theorem}
	\begin{proof}
		$~$
		\begin{itemize}
			\item[Smooth case] 
			
			We first notice that when $X$ is smooth, the infinitesimal cohomology $R\Gamma(X/K_{\Inf},\calF)$ is computed by the cohomology of the de Rham complex $\calF\otimes_{\calO_X}\Omega_{X/K}^\bullet$ via Theorem \ref{glo-coh}; namely we have
			\[
			R\Gamma(X/K_{\Inf}, \calF) \simeq R\Gamma(X, \calF\otimes \Omega_{X/K}^\bullet).
			\]
			Moreover, each term $\calF\otimes \Omega_{X/K}^i$ of the de Rham complex is a coherent sheaf over $X$.
			Notice that the cohomology of a coherent sheaf over a quasi-compact rigid space vanishes when the degree is above the dimension (\cite[Proposition 2.5.8]{dJvdP}).
			Thus by the Hodge--de Rham spectral sequence for $\calF\otimes\Omega_{X/K}^\bullet$, we get the result for $R\Gamma(X/K_{\Inf},\calF)$ with smooth proper $X$.
			
			In general, we use the  base change formula in Proposition \ref{coh, change of bases}.
			By taking the derived tensor product of $Ru_{X/\Sigma_e *}\calF$ with the short exact sequence $0\ra K \ra \Bdre \ra \mathrm{B}_{\mathrm{dR},e-1} \ra 0$, we get a distinguished triangle
			\[
			R\Gamma(X/K_{\Inf}, \calF)   \rra R\Gamma(X/\Sigma_{e \Inf}, \calF) \rra R\Gamma(X/\Sigma_{e-1 \Inf}, \calF).
			\]
			In this way, the claim for smooth proper $X$ follows from the induction on $e$.
			
			\item[General case]
			In general, we prove by induction on the dimension of $X$.
			When $X$ is of dimension zero, it is equal to a nilpotent extension of several points $\Spa(K)$.
			So the result follows from the \'eh-hyperdescent along closed immersions by reduced subspaces in Theorem \ref{inf-eh, Bdre} (in other words, we apply to the \v{C}ech nerve at closed immersion of the reduced subspace).
			
			Now assume $X$ is reduced of dimension $n$, and the claim is true for any rigid space of smaller dimension.
			By the resolution of singularities of rigid space in \cite{Tem18}, there exists a finite sequence of maps $X_m\ra \cdots X_1\ra X_0=X$, where $X_m$ is smooth, and each map $X_i\ra X_{i-1}$ is a blowup at a closed analytic subspace $Y_{i-1}$ of $X_{i-1}$, such that each $Y_{i-1}$ is nowhere dense in $X_{i-1}$.
			We denote $E_i$ to be the exceptional locus $Y_{i-1}\times_{X_{i-1}} X_i$ of the $i$-th blowup.
			We could then apply the \'eh-hyperdescent in Theorem \ref{inf-eh, Bdre} to the \v{C}ech nerve associated with the blowup covering $X_i\coprod Y_{i-1} \ra X_{i-1}$.
			The limit $R\lim_{\Delta}$ of the infinitesimal cohomology for the hypercovering is isomorphic to the fiber of the blowup square 
			\[
			R\Gamma(X_{i}/K_{\Inf},\calF) \moplus R\Gamma(Y_{i-1}/K_{\Inf},\calF)\rra R\Gamma(Y_i/K_{\Inf},\calF),
			\]
			and thus we get a long exact sequence
			\[
			\cdots \rra \rmH^j(X_{i-1}/K_{\Inf},\calF)\rra \rmH^j(X_{i}/K_{\Inf},\calF) \moplus \rmH^j(Y_{i-1}/K_{\Inf},\calF)\rra \rmH^j(Y_i/K_{\Inf},\calF) \rra \cdots.
			\]
			In this way, with the help of the induction assumption for all $Y_i$, a further descending induction from $X_m$ to $X_0=X$ finishes the proof.
			
		\end{itemize}
	\end{proof}

	\subsection{Algebraic and analytic infinitesimal cohomology}\label{subsec ana-alg}
	At the end of the section, we prove the comparison between the algebraic infinitesimal cohomology and the analytic infinitesimal cohomology, for a proper algebraic variety.
	
	Recall that for an algebraic variety\footnote{For our purpose,  a \emph{variety} is defined to be a locally 	of finite type scheme over a field in the article.} $\scrX$ over a $p$-adic field $K$, we can define its \emph{(algebraic) infinitesimal site $\scrX/K_{\inf}$}, whose objects are schematic infinitesimal thickenings $(\scrU, \scrT)$, where $\scrU$ is an Zariski open subset of $\scrX$.
	The infinitesimal site $\scrX/K_{\inf}$ is equipped with a structure sheaf $\calO_{\scrX/K}$, and its cohomology is called the \emph{algebraic infinitesimal cohomology}.
	Moreover, the infinitesimal structure sheaf admits a surjection $\calO_{\scrX/K}\ra \calO_{\scrX}$ to the Zariski structure sheaf, whose kernel $\calJ_{\scrX/K}$ defines a natural filtration on $R\Gamma(\scrX/K_{\inf},\calO_{\scrX/K})$.
	Similar to the analytic theory, we call this filtration the \emph{(algebraic) infinitesimal filtration}.

	Let $X=\scrX^\an$ be the rigid space over $K$ defined as the analytification of a variety $\scrX$.
	As the analytification functor $\Sch_K \ra \Rig_K$ preserves open and closed immersions, it induces a natural map of ringed sites 
	\[
	(X/K_{\inf}, \calO_{X/K}) \rra  (\scrX/K_{\inf}, \calO_{\scrX/K}).
	\]
	Moreover, as the surjection $\calO_{\scrX/K} \ra \calO_{\scrX}$ is compatible with $\calO_{X/K} \ra \calO_X$, 
	the natural map of infinitesimal structure sheaves above is then a filtered map.
	As a consequence, by passing to their cohomology, we get a natural filtered morphism in derived category
	\[
	R\Gamma(\scrX/K_{\inf}, \calO_{\scrX/K}) \rra R\Gamma(X/K_{\inf}, \calO_{X/K}).
	\]
	
	Our main result in this subsection is the following.
	\begin{theorem}\label{ana-alg}
		Let $\scrX$ be a proper algebraic variety over $K$, and let $X$ be its analytifiation.
		Then the analytification functor induces a filtered isomorphism of the infinitesimal cohomology
		\[
		R\Gamma(\scrX/K_{\inf}, \calO_{\scrX/K}) \rra R\Gamma(X/K_{\inf}, \calO_{X/K}).
		\]
	\end{theorem}
	
	Before the proof, we first recall that there is a natural map of ringed sites $(X_\rig, \calO_X) \ra (\scrX_{\Zar}, \calO_{\scrX})$.
	Here the rigid structure sheaf is flat over the Zariski structure sheaf, 
	and the pullback along the map induces a fully faithful functor from coherent $\calO_\scrX$-modules to coherent $\calO_X$-modules.
	
	Moreover, the above map of sites is compatible with the infinitesimal topos.
	Recall that there exists a natural map of topoi
	\begin{align*}
		u_{\scrX/K}: \Sh(\scrX/K_{\inf}) & \rra \Sh(\scrX_{\Zar});\\
		\calF & \lmt (\scrU\mapsto \Gamma(\scrU/K_{\inf}, \calF|_\scrU)).
	\end{align*}
	By construction, this functor is compatible with its rigid version $u_{X/K}:\Sh(X/K_{\inf}) \ra \Sh(X_\rig)$ (cf. Subsection \ref{sub inf-rig}).
	Namely, the following diagram is commutative 
	\[
	\xymatrix{ \Sh(X/K_{\inf}) \ar[d]_{u_{X/K}} \ar[r] & \Sh(\scrX/K_{\inf}) \ar[d]^{u_{\scrX/K}} \\
		\Sh(X_\rig) \ar[r] & \Sh(\scrX_\Zar).}
	\]

	We then claim the following result.
	\begin{proposition}\label{ana, fil}
		Let $\scrX$ be an algebraic variety over $K$, and let $X$ be its analytifcation.
		Then the complex of coherent $\calO_X$-modules $Ru_{X/K *} (\calJ_{X/K}^n/\calJ_{X/K}^{n+1})$ is naturally isomorphic to the analytification of the complex of coherent $\calO_{\scrX}$-modules $Ru_{\scrX/K *} (\calJ_{\scrX/K}^n/\calJ_{\scrX/K}^{n+1})$.
	\end{proposition}
	\begin{proof}
		We denote the complex of coherent $\calO_{\scrX}$-modules $Ru_{\scrX/K *} (\calJ_{\scrX/K}^n/\calJ_{\scrX/K}^{n+1})$ by $C$, and $Ru_{X/K *} (\calJ_{X/K}^n/\calJ_{X/K}^{n+1})$ by $C'$.
		Then it suffices to show that the natural map below induced from the pullback from $\scrX_\Zar$ to $X_\rig$ is an isomorphism of complexes of $\calO_X$-modules
		\[
		C\otimes_{\calO_\scrX} \calO_X \rra C'.
		\]
		As the result is a local statement for $\scrX$, let us assume $\scrX=\Spec(A)$ is a finite type affine scheme over $K$ and $\scrX \ra \scrY=\Spec(A')$ be a closed immersion into an affine space over $K$.
		Moreover, notice that the isomorphism could be checked locally on $X$, so we may take an open affinoid disc of certain radius $\Spa(B')$ in $\scrY^\an$, with the open subset $X\cap \Spa(B')=\Spa(B)$ in $X$.
		From our choices, we get a cartesian diagram as below, where horizontal maps are surjective and vertical maps are flat
		\[
		\xymatrix{B' \ar@{>>}[r] & B \\
			A'\ar[u] \ar@{>>}[r] & A \ar[u].}
		\]
		So it suffices to show that 
		\[
		R\Gamma(\Spec(A)/K_{\inf},\calJ_{\scrX/K}^n/\calJ_{\scrX/K}^{n+1})\otimes_A B \simeq R\Gamma(\Spa(B)/K_{\inf}, \calJ_{X/K}^n/\calJ_{X/K}^{n+1}).
		\]
		
		We then recall from \cite[Section 2]{BdJ} that the algebraic infinitesimal cohomology can be computed by the \v{C}ech-Alexander complex as below
		\[
		D \rra D(1) \rra D(2) \rra \cdots,
		\]
		where $D(m)$ is the formal completion of $A'(m):={A'}^{\underset{K}{\otimes} m+1}$ along the surjection $A'(m)\ra A$.
		We take the $n$-th graded piece for the algebraic infinitesimal filtration, 
		then the cohomology group $R\Gamma(\Spec(A)/K_{\inf},\calJ_{\scrX/K}^n/\calJ_{\scrX/K}^{n+1})$ is isomorphic to the following map of $A$-linear cosimplicial complexes
		\[
		J_D^n/J_D^{n+1} \rra J_{D(1)}^n/J_{D(n)}^{n+1} \rra J_{D(2)}^n/J_{D(2)}^{n+1} \rra \cdots,
		\]
		where $J_{D(m)}$ is the kernel of the surjection $A'(m) \ra A$.
		On the other hand, by the \v{C}ech-Alexander complex for rigid spaces in Proposition \ref{cech}, we have
		\[
		R\Gamma(\Spa(B)/K_{\inf}, \calJ_{X/K}^n/\calJ_{X/K}^{n+1}) \simeq \left( J_\calD^n/J_\calD^{n+1} \rra J_{\calD(1)}^n/J_{\calD(1)}^{n+1} \rra J_{\calD(2)}^n/J_{\calD(2)}^{n+1} \rra \cdots \right),
		\]
		where $\calD(m)$ is the formal completion for the surjection $B'(m):={B'}^{\underset{K}{\wh\otimes} m+1} \ra B$, and $J_{\calD(m)}$ is the kernel of the map $B'(m)\ra B$.
		Thus we are left to show the quasi-isomorphism for the canonical map of $B$-linear cosimplicial complexes below
		\begin{align*}
			\bigl(J_D^n/J_D^{n+1} \to J_{D(1)}^n/J_{D(n)}^{n+1} \to J_{D(2)}^n/J_{D(2)}^{n+1} &\to \cdots \bigr) \otimes_A B \rra \\
			\bigl(J_\calD^n/J_\calD^{n+1} &\to J_{\calD(1)}^n/J_{\calD(1)}^{n+1} \to J_{\calD(2)}^n/J_{\calD(2)}^{n+1} \to \cdots \bigr).
		\end{align*}

		At last notice that by our choices, the rigid space $\Spa(B')$ is an open disc of some radius in the affine space $\Spec(A')^\an$.
		In particular, the following map of rings is a cartesian diagram such that vertical maps are flat
		\[
		\xymatrix{ B'(m) \ar@{>>}[r] & B' \\
			A'(m) \ar@{>>}[r] \ar[u] & A' \ar[u].}
		\]
		In this way, combining this with the cartesian diagram in the first paragraph, we see the kernel $J_{\calD(m)}$ of the surjection $B'(m)\ra B$ is equal to the base change of $J_{D(m)}$ along the flat map $A'(m) \ra B'(m)$.
		Hence we get the natural equalities
		\begin{align*}
			& J_{D(m)}B'(m)=J_{D(m)}\otimes_{A'(m)} B'(m)=J_{\calD(m)};\\
			& (J_{D(m)}^n/J_{D(m)}^{n+1})\otimes_A B = (J_{D(m)}^n/J_{D(m)}^{n+1}) \otimes_{A'(m)} B'(m) = J_{\calD(m)}^n/J_{\calD(m)}^{n+1}.
		\end{align*}
		So we are done.

	\end{proof}

	At last, we finish the proof of Theorem \ref{ana-alg}.
	\begin{proof}[Proof of Theorem \ref{ana-alg}]
		To show the natural map in the statement is a filtered isomorphism, it suffices to show the isomorphisms for their underlying complexes and each graded piece separately, as both of them are filtered complete.
		\footnote{For each affinoid infinitesimal thickening $(U,T)$, the kernel ideal $\mathcal{J}_T=\ker(\calO_T\to \mathcal{O}_U)$ is nilpotent and hence $\mathcal{O}_T$ is complete under $\mathcal{J}_T$-adic topology.
			This in particular implies the inverse limit formula $\mathcal{O}_{X/K}=\varprojlim_i \mathcal{O}_{X/K}/\mathcal{J}_{X/K}^i$, and similarly for the algebraic version.}
		
		For the underlying complexes, this follows from the \'eh descent.
		To see this, we first notice that when $\scrX$ is smooth and proper over $K$, then the algebraic and the analytic infinitesimal cohomology are isomorphic to the algebraic and the analytic de Rham cohomology respectively (\cite{Gr68}, Theorem \ref{glo-coh}), 
		which are isomorphic to each other by applying the GAGA theorem to their Hodge-filtrations (cf. \cite[Appendix A.1]{Con06}).
		In general, we may assume $\scrX_\bullet\ra \scrX$ is a simplicial smooth varieties by resolving singularities.
		Then its analytification $X_\bullet \ra X$ is an \'eh-hypercovering by smooth rigid spaces, and we get the isomorphism
		\begin{align*}
			R\Gamma(\scrX/K_{\inf}, \calO_{\scrX/K}) &\simeq R\lim_{[n]\in \Delta} 	R\Gamma(\scrX_n/K_{\inf}, \calO_{\scrX_n/K}) \\
			&\simeq R\lim_{[n]\in \Delta} 	R\Gamma(X_n/K_{\inf}, \calO_{X_n/K})\\
			&\simeq 	R\Gamma(X/K_{\inf}, \calO_{X/K}),
		\end{align*}
		where the first equality is the h-hyperdescent of algebraic de Rham cohomology for blowups in \cite{Har75}, and the last is the \'eh-hyperdescent for the analytic infinitesimal cohomology in \ref{inf-eh}.
		
		For the graded pieces, by Proposition \ref{ana, fil} we have
		\[
		R\Gamma(X/K_{\inf}, \calJ_{X/K}^n/\calJ_{X/K}^{n+1}) = R\Gamma(X_\rig, Ru_{X/K *} (\calJ_{X/K}^n/\calJ_{X/K}^{n+1})) \simeq R\Gamma(X_\rig, \left(Ru_{\scrX/K *} (\calJ_{\scrX/K}^n/\calJ_{\scrX/K}^{n+1})\right)^\an).
		\]
		We denote $C$ to be the bounded below complex of coherent $\calO_X$-modules $Ru_{\scrX/K *} (\calJ_{\scrX/K}^n/\calJ_{\scrX/K}^{n+1})$.
		As $R\Gamma(X_\Zar, \tau^{> n} C)$ lives in cohomogical degree larger than $n$, we have the natural equalities
		\[
		R\Gamma(X_\Zar, C)= \colim_n R\Gamma(X_\Zar, \tau^{\leq n} C).
		\]
		Similarly we have
		\[
		R\Gamma(X_\rig, C^\an) =\colim_n  R\Gamma(X_\rig, \tau^{\leq n} (C^\an)).
		\]
		On the other hand, as the rigid structure sheaf $\calO_X$ is flat over $\calO_\scrX$, the analytification functor $(-)^\an=-\otimes_{\calO_\scrX} \calO_X$ on coherent complexes is an exact functor.
		So for each $n\in \NN$, there exists a natural equality
		\[
		\tau^{\leq n} C^\an = (\tau^{\leq n} C)^\an.
		\]
		Notice that for each bounded complex $\tau^{\leq n} C$ of coherent sheaves, by rigid GAGA theorem (\cite[Appendix A.1]{Con06}) we have
		\[
		R\Gamma(X_\Zar, \tau^{\leq n} C)\simeq R\Gamma(X_\rig, (\tau^{\leq n} C)^\an).
		\]
		In this way, combining all of the isomorphisms above, we get
		\begin{align*}
			R\Gamma(X_\Zar, C) &\simeq \colim_n R\Gamma(X_\Zar, \tau^{\leq n} C)\\
			&\simeq \colim_n R\Gamma(X_\rig, (\tau^{\leq n} C)^\an) \\
			&=\colim_n R\Gamma(X_\rig, \tau^{\leq n} (C^\an)) \\
			&\simeq R\Gamma(X_\rig, C^\an).
		\end{align*}
		At last, substituting back the definition of $C$ and Proposition \ref{ana, fil}, we then obtain the formula for graded piece of infinitesimal filtrations:
		\[
		R\Gamma(\scrX/K_{\inf}, \calJ_{\scrX/K}^n/\calJ_{\scrX/K}^{n+1}) \simeq R\Gamma(X/K_{\inf}, \calJ_{X/K}^n/\calJ_{X/K}^{n+1}).
		\]
		
	\end{proof}
	
	As an application, we get the comparison with singular cohomology when $K$ is abstractly isomorphic to the field of complex numbers, proving Theorem \ref{main1}.(v).
	\begin{corollary}\label{ana-sing}
		Assume there exists an abstract isomorphism of fields $K\ra \CC$.
		Then for any proper algebraic variety $\scrX/K$ with its analytification $X$, there exists a filtered isomorphism of cohomology
		\[
		\rmH^i(X/K_{\inf} ,\calO_{X/K}) \simeq \rmH^i_\Sing(\scrX(\CC),\CC),
		\]
		where singular cohomology of $\scrX(\CC)$ is filtered by the algebraic infinitesimal filtration.
	\end{corollary}
	\begin{proof}
		This follows from Theorem \ref{ana-alg} and the classical result of Hartshorne in \cite{Har75}.
	\end{proof}
	
	Using the same idea of the proof for Proposition \ref{ana, fil} and Theorem \ref{ana-alg}, we can prove the base extension formula for the infinitesimal cohomology.
	\begin{corollary}\label{base ext}
		Let $K_0$ be a complete $p$-adic extension of $\QQ_p$, and let $K$ be a complete extension of $K_0$.
		Assume $X$ is a proper rigid space over $K_0$, and let $\calF$ be a coherent crystal over $X/K_{0,\inf}$.
		Then the following natural map of filtered complexes is an isomorphism
		\[
		R\Gamma(X/K_{0, \inf}, \calF)\otimes_{K_0} K \rra R\Gamma(X_K/K_{\inf}, \calF_K).
		\]
	\end{corollary}

	\section{Cohomology over $\Bdr$}\label{sec Bdr}
	In this section, we extend previous results to the infinitesimal site over $\Bdr$, for a rigid space $X$ over $\Sigma_r$ for some fixed $r\in \NN$.
	Our goal is to show  \Cref{main2} from the introduction.
	
	\subsection{Infinitesimal sites and topoi over $\Bdr$}
	We fix a complete algebraic closed $p$-adic field $K$.
	Let $X$ be a rigid space over $\Bdr/\xi^r$ for some fixed $r\in \NN$.
	To build an infinitesimal cohomology theory with the coefficient being $\Bdr=\varprojlim_{e\in \NN} \Bdre$, we construct an infinitesimal site $X/\Sigma_{\inf}$ as a union of all $X/\Sigma_{e \inf}$ for $e\in \NN_{\geq r}$, and consider its relation to each infinitesimal site $X/\Sigma_{e \inf}$.
	
	\noindent\textbf{The site $X/\Sigma_{\inf}$.}

	We first give the definition of the infinitesimal site over $\Sigma=\varinjlim_{e\in \NN} \Sigma_e$, where the latter is regarded as the ringed space whose underlying topological space is $\Spa(K)$ with the structure sheaf given by $\Bdr$.
	\begin{definition}\label{inf site}
		Let $X$ be a rigid space over $\Sigma_r=\Spa(\Bdr/\xi^r)$, for some fixed $r\in \NN$.
		The \emph{infinitesimal site $X/\Sigma_{\inf}$  over $\Bdr$} is defined as follows: 
		\begin{itemize}
			\item The underlying category of $X/\Sigma_{\inf}$ is the category of pairs $(U,T)$, for $(U,T)$ being an thickening in $X/\Sigma_{e \inf}$ for some $e\geq r$.
			
			A morphism between $(U_1,T_1)$ and $(U_2,T_2)$ is a morphism of objects in $X/\Sigma_{e \inf}$, for $e$ large enough such that both pairs are objects in $X/\Sigma_{e \inf}$.
			\item A collection of morphism $(U_i,T_i)\ra (U,T)$ in $X/\Sigma_{\inf}$ is a covering if $\{T_i\ra T\}$ is an open covering for the rigid space $T$.
		\end{itemize} 
	\end{definition}

	As a category, $X/\Sigma_{\inf}$ is the union of $X/\Sigma_{e \inf}$ for all $e\geq r$.
	It is clear that the topology is locally rigid over each object in $X/\Sigma_{\inf}$.
	Thus the description of a sheaf over $X/\Sigma_{\inf}$ is similar to that of a sheaf over $X/\Sigma_{e \inf}$ as in Section \ref{sec inf}.
	
	\begin{remark}\label{inf-lev, big site}
		Similarly to the Discussion in Section \ref{sec inf}, we could define the big version infinitesimal site $X/\Sigma_{\Inf}$, where the objects are infinitesimal thickenings $(U,T)$ for $U$ being a rigid space over $X$ and $U\ra T$ a nil-extension over $\Bdr$.
		The relation between the big infinitesimal sites $X/\Sigma_{\Inf}$ and the small one $X/\Sigma_{\inf}$, including the constructions in the rest of the subsection, are exactly identical to the case over $\Bdre$ in Paragraph \ref{site big small}, and we will not duplicate again here.
		
	\end{remark}
	
	\noindent\textbf{Functoriality of $\Sh(X/\Sigma_{\inf})$.}
	
	The infinitesimal topos $\Sh(X/\Sigma_{\inf})$ is functorial with respect to the rigid space $X$.
	Namely, for a map of $\Bdr$-rigid spaces $f:X\ra Y$ where $\xi$ is nilpotent in both $\mathcal{O}_X$ and $\mathcal{O}_Y$, we have a natural map of topoi
	\[
	f_{\inf}:\Sh(X/\Sigma_{\inf}) \rra \Sh(Y/\Sigma_{\inf}).
	\]
	The corresponding adjoint pair of functors are given by the following:
	\begin{itemize}
		\item For a sheaf $\calG\in \Sh(Y/\Sigma_{\inf})$, the inverse image $f_{\inf}^{-1}\calG$ is given by the restriction of $\mu_Y^{-1}\calG$ to the category $X/\Sigma_{\inf}$ along the map $f$, and is equal to the sheaf associated with the presheaf
		\[
		X/\Sigma_{\inf}\ni (U,T) \lmt \varinjlim_{\substack{(U,T)\ra (V,S)\\ (V,S)\in Y/\Sigma_{\inf},\\
				U\ra V~compatible~with~f}} \calG(V,S).
		\]
		\item The direct image functor $f_{\inf *}$ sends a sheaf $\calF\in \Sh(X/\Sigma_{\inf})$ to the sheaf 
		\[
		f_{\inf *}\calF(V,S)=\varprojlim_{\substack{(U,T)\ra (V,S)\\ (U,T)\in X/\Sigma_{\Inf}\\ U\ra V~compatible~with~f}} \calF(U,T).
		\]
	\end{itemize}
	We want to remind the reader that the construction of those two functors are identical with the construction of the functoriality morphism $\Sh(X/\Sigma_{e \inf})\ra \Sh(Y/\Sigma_{e' \inf})$ for the map of rigid spaces
	\[
	\xymatrix{X \ar[r] \ar[d]& Y \ar[d]\\
		\Sigma_e \ar[r]& \Sigma_{e'},}
	\]
	as in Subsection \ref{subsec functoriality}.

	\noindent\textbf{Relation with $X/\Sigma_{e \inf}$.}
	Topologically, the infinitesimal site $X/\Sigma_{\inf}$ is the limit of $X/\Sigma_{e \inf}$ for $e\geq r$.
	To make this precise, we consider the following morphism of sites:
	\[
	u_e:X/\Sigma_{\inf}\rra X/\Sigma_{e \inf},
	\]
	whose corresponding functor is the canonical inclusion functor that sends $(U,T)\in X/\Sigma_{e \inf}$ to the object $(U,T)\in X/\Sigma_{\inf}$.
	Note that by construction, this cocontinuous functor is a fully faithful embedding.

	This morphism induces an adjoint pair of functors $(u_e^{-1},u_{e *})$ given as follows:
	\begin{itemize}
		\item The functor $u_{e *}$ is the restriction functor, in a way that for a sheaf $\calF\in \Sh(X/\Sigma_{\inf})$ we have
		\[
		(u_{e *}\calF)_T=\calF_T.
		\]
		\item For a sheaf $\calG\in \Sh(X/\Sigma_{e \inf})$, the sheaf $u_e^{-1}\calG$ is the sheaf associated with the presheaf
		\begin{align*}
			(V,S)&\mapsto\varinjlim_{\substack{(V,S)\ra (U,T)\\ (U,T)\in X/\Sigma_{e \inf}}} \calG(U,T)\\
			&=\begin{cases}
				\varnothing,~S\notin \Rig_{\Sigma_e};\\
				\calG(V,S),~S\in \Rig_{\Sigma_e}.
			\end{cases}
		\end{align*}
	\end{itemize}
	So by the definition of the site $X/\Sigma_{\inf}$, the restriction of $u_e^{-1}\calG$ at $(V,S)$ is
	\[
	(u_e^{-1}\calG)_S=\begin{cases}
		\varnothing,~S\notin \Rig_{\Sigma_e};\\
		\calG_S,~S\in \Rig_{\Sigma_e}.
	\end{cases}
	\]
	Here we notice that when $\calG=h_{(U,T)}$ is the representable sheaf for some object $(U,T)\in X/\Sigma_{e \inf}$, the inverse image $u_e^{-1}h_{(U,T)}$ is nothing but the representable sheaf $h_{(U,T)}$ in $\Sh(X/\Sigma_{\inf})$.
	
	The morphism of site $u_e:X/\Sigma_{\inf}\ra X/\Sigma_{e \inf}$ induces a map of topoi
	\[
	u_e:\Sh(X/\Sigma_{\inf}) \rra \Sh(X/\Sigma_{e \inf}).
	\]
	It admits a section $i_e:\Sh(X/\Sigma_{e \inf}) \rra \Sh(X/\Sigma_{\inf})$, where the corresponding adjoint pair of functors is given as follows:
	\begin{itemize}
		\item For a sheaf $\calG\in \Sh(X/\Sigma_{inf})$, the inverse image $i_e^{-1}\calG$ is the sheaf associated with the presheaf
		\[
		X/\Sigma_{e \inf}\ni (U,T) \lmt \varinjlim_{\substack{(U,T)\ra (U,S)\\ (U,S)\in X/\Sigma_{\inf}}} \calG(U,S)=\calG(U,T).
		\]
		Namely, $i_e^{-1}=u_{e *}$ is the restriction functor.
		\item The direct image functor $i_{e *}$ sends a sheaf $\calF\in \Sh(X/\Sigma_{e \inf})$ to the sheaf 
		\[
		i_{e *}\calF(V,S)=\varprojlim_{\substack{(V,T)\ra (V,S)\\ (V,T)\in X/\Sigma_{e \inf}}} \calF(V,T)=\calF(V,S\times_{\Sigma}\Sigma_e).
		\]
	\end{itemize}
	It is  clear that the composition $u_e\circ i_e$ is equal to the identity.
	We also note that the above functors are functorial with respect to $e$.
	
	\begin{remark}
		Here we notice that the map $i_e$ is in fact induced from a natural map of sites
		\begin{align*}
			i_e: X/\Sigma_{e \inf} &\rra X/\Sigma_{\inf};\\
			(U,T\times_{\Sigma} \Sigma_e) & \reflectbox{$\longmapsto$} (U,T).
		\end{align*}
		This is analogous to the nilpotent bases situation, as in the Remark \ref{change of bases}
	\end{remark}

	\begin{remark}
		We also want to reminder the reader that the construction of map $i_e$ could be regarded as the functoriality morphism of infinitesimal topoi associated with the following diagram
		\[
		\xymatrix{X \ar[r]^\id \ar[d]& X \ar[d]\\
			\Sigma_{e}\ar[r] & \Sigma.}
		\]
	\end{remark}
	
	\begin{remark}
		The construction of $u_e$ and $i_e$ is compatible with the functoriality morphism of infinitesimal topoi $f_{\inf}:\Sh(X/\Sigma_{e \inf})\ra \Sh(Y/\Sigma_{e' \inf})$ for a map of rigid spaces $f:X/\Sigma_e\ra Y/\Sigma_{e'}$.
		Namely we have the following commutative diagrams among infinitesimal topoi
		\[
		\xymatrix{ \Sh(X/\Sigma_{\inf}) \ar[r]^{u_e} \ar[d]_{f_{\inf}} & \Sh(X/\Sigma_{e \inf}) \ar[d]^{f_{\inf}} \\
			\Sh(Y/\Sigma_{\inf}) \ar[r]_{u_{e'}} & \Sh(Y/\Sigma_{e' \inf});}~~~~~~~~~~
		\xymatrix{ \Sh(X/\Sigma_{e \inf}) \ar[r]^{i_e} \ar[d]_{f_{\inf}} & \Sh(X/\Sigma_{\inf}) \ar[d]^{f_{\inf}} \\
			\Sh(Y/\Sigma_{e' \inf}) \ar[r]_{i_{e'}} & \Sh(Y/\Sigma_{\inf}).}
		\]
	\end{remark}

	\noindent\textbf{Relation to the rigid topos $\Sh(X_\rig)$.}
	Analogous to Subsection \ref{sub inf-rig}, there exists a natural map of topoi to the rigid site $X_\rig$ as below
	\[
	u_{X/\Sigma}:\Sh(X/\Sigma_{\inf}) \rra \Sh(X_\rig).
	\]
	The corresponding preimage and the direct image functors are given as below
	\begin{itemize}
		\item	 \begin{align*}
			u_{X/\Sigma *}:\Sh(X/\Sigma_{\inf})&\rra \Sh(X_\rig);\\
			\calF&\lmt (U\mapsto \Gamma(U/\Sigma_{\inf},\calF)).\end{align*}
		
		\item \begin{align*}
			u_{X/\Sigma}^{-1}:\Sh(X_\rig)&\rra \Sh(X/\Sigma_{\inf});\\
			\calE&\lmt ((U,T)\mapsto \calE(U)).
		\end{align*}
	\end{itemize}
	Namely the push-forward functor $u_{X/\Sigma*}$ is the sheafified version of the infinitesimal global section functor.
	
	\begin{remark}
		The functor $u_{X/\Sigma}$ is functorial with respect to the rigid space $X$.
		Precisely, given a map of rigid spaces $f:X\ra Y$ over $\Sigma$ where $\xi$ is nilpotent in both $\mathcal{O}_X$ and $\mathcal{O}_Y$, we have the following commutative diagram
		\[
		\xymatrix{ \Sh(X/\Sigma_{\inf}) \ar[r]^{u_{X/\Sigma}} \ar[d]_{f_{\inf}} & \Sh(X_\rig) \ar[d]^f\\
			\Sh(Y/\Sigma_{\inf}) \ar[r]_{u_{Y/\Sigma}} & \Sh(Y_\rig).}
		\]
	\end{remark}
	\begin{remark}
		The functor $u_{X/\Sigma}$ is also compatible with $u_e:\Sh(X/\Sigma_{\inf}) \ra \Sh(X/\Sigma_{e \inf})$ and $i_e:\Sh(X/\Sigma_{e \inf}) \ra \Sh(X/\Sigma_{\inf})$.
		Namely, the following diagrams commute:
		\[
		\xymatrix{
			\Sh(X/\Sigma_{\inf}) \ar[rd]_{u_e} \ar[rr]^{u_{X/\Sigma}} && \Sh(X_\rig) \\
			& \Sh(X/\Sigma_{e \inf}) \ar[ru]_{u_{X/\Sigma_{e}}} &;}~~~~~~~~~~~~
		\xymatrix{
			\Sh(X/\Sigma_{\inf})  \ar[rr]^{u_{X/\Sigma}} && \Sh(X_\rig) \\
			& \Sh(X/\Sigma_{e \inf}) \ar[lu]^{i_e} \ar[ru]_{u_{X/\Sigma_{e}}} &.}
		\]
		Here $u_{X/\Sigma_e}:\Sh(X/\Sigma_{e \inf}) \ra \Sh(X_\rig)$ is the analogous functor of $u_{X/\Sigma}$ onto the rigid site defined in Subsection \ref{sub inf-rig}.
	\end{remark}

	\subsection{Cohomology of crystals over $X/\Sigma_{\inf}$}\label{subsec Bdr coh}
	In this section, we consider the cohomology of a crystal $\calF$ over the infinitesimal site $X/\Sigma_{\inf}$.
	Our strategy is to interpret the cohomology of $\calF$ as the derived inverse limit of the cohomology of the pullback $i_e^*\calF$, where $i_e^*\calF$ is a crystal over the site $X/\Sigma_{e \inf}$.
	
	To start with, we first describe a crystal over the infinitesimal site $X/\Sigma_{\inf}$.
	\begin{definition}
		Let $X$ be a rigid space over $\Sigma_{\inf}$ where $\xi$ is nilpotent.
		\begin{enumerate}[label=\upshape{(\roman*)}]
			\item The \emph{infinitesimal structure sheaf} over $X/\Sigma_{\inf}$, denoted as $\calO_{X/\Sigma}$, is a sheaf over $X/\Sigma_{\inf}$ sending a thickening $(U,T)\in X/\Sigma_{\inf}$ onto the global section of $\calO_T$ at $T$ as below
			\[
			\calO_{X/\Sigma}: (U,T) \lmt \calO_T(T).
			\]
			\item The \emph{infinitesimal ideal sheaf} over $X/\Sigma_{\inf}$, denoted as $\calJ_{X/\Sigma}$, is a sheaf over $X/\Sigma_{\inf}$ sending a thickening $(U,T)\in X/\Sigma_{\inf}$ onto the global section of $\ker(\calO_T\to \mathcal{O}_U)$ at $T$ as below
			\[
			\calO_{X/\Sigma}: (U,T) \lmt \ker(\calO_T(T)\to \mathcal{O}_U(U)).
			\]
			\item A \emph{coherent crystal} over $X/\Sigma_{\inf}$ is a $\calO_{X/\Sigma}$-coherent sheaf $\calF$ over $X/\Sigma_{\inf}$ satisfies the crystal condition as in Definition \ref{cry def}.
			It is called \emph{a crystal in vector bundle} if the restriction $\calF_T$ at each infinitesimal thickening $(U,T)\in X/\Sigma_{\inf}$ is a vector bundle over $\calO_T$.
		\end{enumerate}
	\end{definition}
	Here we mention that similarly to Proposition \ref{cry, big and small}, it can be shown the categories crystals over big and small sites are equivalent.
	
	We notice that the morphism of sites $i_e:X/\Sigma_{e \inf}\ra X/\Sigma_{\inf}$ in the last subsection is naturally a morphism of ringed sites for their structure sheaves.
	Moreover, since the preimage functor $i_e^{-1}$ is equal to the restriction functor onto the subcategory $X/\Sigma_{e \inf}$, we get
	\[
	i_e^{-1}\calO_{X/\Sigma}=\calO_{X/\Sigma_e},
	\]
	and similarly for the infinitesimal ideal sheaves.
	So we can define the \emph{pullback functor} $i_e^*\calF:=i_e^{-1}\calF\otimes_{i_e^{-1}\calO_{X/\Sigma}} \calO_{X/\Sigma_e}$, which is the same as the restriction functor $i_e^{-1}\calF$ itself; namely for an infinitesimal thickening $(U,T)\in X/\Sigma_{e \inf}$, we have
	\[
	(i_e^*\calF)_T=(i_e^{-1}\calF)_T=\calF_T.
	\]
	Here we want to remark that the pullback functor $i_e^*=i_e^{-1}$ is compatible with the pullback functor $f_{\inf}^*$ of the morphism $f_{\inf}:\Sh(X/\Sigma_{e \inf}) \ra \Sh(Y/\Sigma_{e' \inf})$ for a map of rigid spaces $f:X/\Sigma_e \ra Y/\Sigma_{e'}$.
	
	The main tool of the subsection is the following lemma, relating a coherent crystal over $X/\Sigma_{\inf}$ with those over $X/\Sigma_{e \inf}$ of $\xi$-nilpotent coefficients.
	\begin{lemma}\label{inf-lev, tool}
		Let $\calF$ be a coherent crystal over the infinitesimal site $X/\Sigma_{\inf}$ (resp. $X/\Sigma_{\Inf}$), and let $X$ be defined over $\Sigma_r$ for some $r\in \NN$.
		Then we have the following.
		\begin{enumerate}[label=\upshape{(\roman*)}]
			\item The pullback $i_e^*\calF$ for each $e\in \NN_{\geq r}$ is a crystal over $X/\Sigma_{e \inf}$.
			When $\calF$ is a crystal in vector bundles, so is $\calF$ over $X/\Sigma_{e \inf}$.
			\item The counit map for the adjoint pairs $(i_e^*,i_{e *})$ induces the following isomorphism
			\[
			\calF/\xi^e \rra Ri_{e *}i_e^*\calF.
			\]
			In particular, we have the natural equivalences as below
			\[
			\calF \rra R\varprojlim_{e\geq r} \calF/\xi^e \rra R\varprojlim_{e\geq r} Ri_{e *}i_e^*\calF.
			\]
			Here the transition maps in the last limit are given by the map of infinitesimal sites $X/\Sigma_{e \inf} \ra X/\Sigma_{e+1 \inf}$ (resp. $X/\Sigma_{e \Inf} \ra X/\Sigma_{e+1 \Inf}$) for the closed immersions of bases. 
		\end{enumerate}
	\end{lemma}
	\begin{proof}
		$~$
		\begin{enumerate}[(i)]
			\item The proof of the (i) follows from the definition of the crystal condition.
			\item We recall from the last subsection that the push-forward functor $i_{e *}\calG$ is given by
			\[
			(i_{e *}\calG)(U,T)=\calG(U, T\times_{\Sigma} \Sigma_e),
			\]
			for a sheaf  $\calG\in \Sh(X/\Sigma_{e \inf})$.
			We denote the fiber product $ T\times_\Sigma \Sigma_e$ by $T_e$, which is an infinitesimal thickening of $U$ that is defined over $\Sigma_e$.
			Apply the above to the pullback $\calG=i_e^*\calF$ of the crystal $\calF$, and notice that $i_e^*$ is the restriction functor, we get
			\begin{align*}
				(R i_{e *}i_e^*\calF)(U,T) & = R\Gamma((U, T_e), \calF) \\
				& =R\Gamma( T_e, \calF_{T_e}) \\
				& = R\Gamma( T_e, \calF_T/\xi^e) \\
				& = R\Gamma( T, \calF_T/\xi^e),
			\end{align*}
			where the last equality follows from the observation that $T_e\ra T$ has the same underlying topological spaces.
			Hence the cone of $\calF/\xi^e \ra Ri_{e *}i_e^*\calF$, which is bounded below and has no cohomology, vanishes in the derived category.
			
			At last, notice that for a coherent sheaf $\calF$ of $\calO_{X/\Sigma}$-modules over $X/\Sigma_{\inf}$, we always have
			\[
			\calF \simeq R\varprojlim_{e\geq r} \calF/\xi^e \simeq \varprojlim_{e\geq r} \calF/\xi^e.
			\]
			So the last claim in $(ii)$ follows.
			
		\end{enumerate}
	\end{proof}

	Now we are able to give the main result about the cohomology of crystals over the infinitesimal site $X/\Sigma_{\inf}$.
	Analogous to the case over $\Sigma_e$, for a coherent crystal  $\calF$ over $X/\Sigma_{\inf}$, we define a canonical filtration on it by $\Fil^i \mathcal{F}:= \mathcal{J}_{X/\Sigma}^i\mathcal{F}$, for $i\in \mathbb{N}$.
	This then naturally induces a filtration on its derived direct image along the functor $u_{X/\Sigma}$ and on its derived global section respectively.
	\begin{theorem}\label{inf-lev coh}
		Let $X$ be a rigid space over some $\Sigma_r$, and let $\calF$ be a coherent crystal over $X/\Sigma_{\inf}$.
		\begin{enumerate}[label=\upshape{(\roman*)}]
			\item There exists a natural filtered isomorphism of complexes of sheaves of $\Bdr$-modules as below:
			\[
			Ru_{X/\Sigma *} \calF \rra R\varprojlim_{e\geq r} Ru_{X/\Sigma_{e} *} (i_e^*\calF).
			\]
			In particular, by applying the derived global section functor, we get a filtered isomorphism
			\[
			R\Gamma(X/\Sigma_{\inf},\calF) \simeq R\varprojlim_{e\geq r} R\Gamma(X/\Sigma_{e \inf}, i_e^*\calF).
			\]
			\item Let $\{Y_e\}_{e\geq r}$ be a direct system of rigid spaces over $\Sigma_e$, such that each $Y_e$ is smooth over $\Sigma_e$ with $Y_{e+1}\times_{\Sigma_{e+1}} \Sigma_e\simeq Y_e$.
			Assume $X$ admits a closed immersions into $Y_r$.
			Then we have natural filtered isomorphisms of complexes of sheaves of $\Bdr$-modules as below
			\[
			Ru_{X/\Sigma *} \calF \rra \calF_D\otimes \Omega_D^\bullet \simeq R\varprojlim_{e\geq r} (\calF_{D_X(Y_e)} \otimes \Omega_{D_X(Y_e)}^\bullet),
			\]
			where $D=\varinjlim_{e\geq r} D_X(Y_e)$ is the colimit		\footnote{As in \Cref{env}, we again regard $D=\varinjlim_{e\geq r} D_X(Y_e)$ as an ind-representable sheaf in the infinitesimal topos, where the colimit always exits. }
			of envelopes, and $\calF_D\otimes \Omega_D^\bullet$ is the filtered de Rham complex of $\calF$ over $D$ 
			\footnote{Here the filtration on the de Rham complex is defined analogously to the discussion above \Cref{lemA}.}.
			\item Suppose $\calF$ is a crystal in vector bundles over $X/\Sigma_{\inf}$.
			We equip the rings $\Bdr$ and $\Bdre$ with their $\xi$-adic filtrations.
			Then for each $e\geq r$, the natural maps below are filtered isomorphisms
			\begin{align*}
				(Ru_{X/\Sigma *} \calF)\otimes_{\Bdr}^L \Bdre & \rra Ru_{X/\Sigma_e *} (i_e^*\calF);\\
				Ru_{X/\Sigma *} \calF & \rra R\varprojlim_e \left( (Ru_{X/\Sigma *} \calF)\otimes_{\Bdr}^L \Bdre \right).
			\end{align*}
			In particular, when $X$ is quasi-compact quasi-separated, by applying the derived global section functor we obtain the following canonical filtered equivalences
			\begin{align*}
				R\Gamma(X/\Sigma_{\inf}, \calF)\otimes_{\Bdr}^L \Bdre & \simeq R\Gamma(X/\Sigma_{e \inf}, i_e^*\calF); \\
				R\Gamma(X/\Sigma_{\inf}, \calF) &\simeq R\varprojlim_e \left( R\Gamma(X/\Sigma_{\inf}, \calF)\otimes_{\Bdr} \Bdre \right).
			\end{align*}
			
		\end{enumerate}
	\end{theorem}
	Before we prove, we want to remark that the result for crystals over the big site $X/\Sigma_{\Inf}$ are true and the proof is identical to the small site case.
	\begin{proof}
		$~$
		\begin{enumerate}[(i)]
			\item This follows from applying $Ru_{X/\Sigma *}$ to the equivalences $\calF \rra R\varprojlim_{e\geq r} Ri_{e*}i_e^*\calF$ in Lemma \ref{inf-lev, tool}.
			Here we use the identity of maps of topoi in the last subsection 
			\[
			u_{X/\Sigma } \circ i_{e } = u_{X/\Sigma_e}.
			\]
			
			\item For each $e\geq r$, by Theorem \ref{glo-coh} there exists a natural filtered isomorphism of complexes of sheaves of $\Bdre$-modules 
			\[
			Ru_{X/\Sigma_e *} i_e^*\calF \rra \calF_{D_X(Y_e)}\otimes \Omega_{D_X(Y_e)}^\bullet.
			\]
			So the map of ringed sites $X/\Sigma_{e \inf} \ra X/\Sigma_{e+1 \inf}$ induced from the closed immersion of the bases $\Sigma_e\ra\Sigma_{e+1}$ together with (i) produces the inverse limits
			\begin{align*}
				Ru_{X/\Sigma *} \calF & \simeq R\varprojlim_{e\geq r} (\calF_{D_X(Y_e)}\otimes \Omega_{D_X(Y_e)}^\bullet) \\
				& \simeq \calF_D \otimes \Omega_D^\bullet,
			\end{align*}
			where we use the compatibility of the de Rham complexes $\calF_{D_X(Y_e)}\otimes \Omega_{D_X(Y_e)}^\bullet$ for different $e$, by our choices of the direct system of smooth rigid spaces $\{Y_e\}_e$.
			
			\item We first notice that the second half of the statement follows from its sheaf version, by the following isomorphism
			\[
			R\Gamma(U, (Ru_{X/\Sigma *} \calF)\otimes_{\Bdr}^L \Bdre) \simeq R\Gamma(U/\Sigma_{\inf}, \calF)\otimes_{\Bdr}\Bdre.
			\]
			Here the isomorphism follows by applying $R\Gamma(U, -)$ at the distinguished triangle resolving $\Bdre$ over $\Bdr$ as below
			\[
			\xymatrix{Ru_{X/\Sigma *} \calF \ar[r]^{\cdot \xi^{e}} & Ru_{X/\Sigma *} \calF \ar[r] & (Ru_{X/\Sigma *} \calF)\otimes_{\Bdr}^L \Bdre.}
			\]
			On the other hand, to check the sheaf level isomorphism, as the statement is rigid analytic local with respect to $X$, so it suffices to assume that $X$ admits a closed immersion into a direct system of smooth rigid spaces $\{Y_e\}_e$ over $\Sigma_e$, where the results follow from the explicit calculation of the completed de Rham complexes as in part (ii) and Theorem \ref{glo-coh}.
			So we are done.
			
		\end{enumerate}
	\end{proof}
	
	\begin{remark}\label{inf-lev, BMS}
		Recall that for a smooth affinoid rigid space $X=\Spa(R)$ over $K$, the crystalline cohomology of $X$ over $\Bdr$, introduced in \cite[Section 13]{BMS}, is defined as the inverse limit
		\[
		\varprojlim_{e\in \NN} \Omega_{D_X(Y_e)}^\bullet,
		\]
		where $X\ra Y_e=\Spa(\Bdre \langle T_i^{\pm1}\rangle)$ is a closed immersion.
		So Theorem \ref{inf-lev coh} implies that the infinitesimal cohomology $R\Gamma(X/\Sigma_{\inf}, \calO_{X/\Sigma})$ coincides with the crystalline cohomology of $X$ over $\Bdr$ in the sense of \cite{BMS}.
		
	\end{remark}
	
	With the help of Theorem \ref{inf-lev coh}, we can compare the infinitesimal cohomology of $X$ over $\Bdr$ with the derived de Rham complex.
	\begin{definition}\label{ddR over Bdr}
		Let $X$ be a rigid space over $\Sigma_r$.
		Then the \emph{analytic derived de Rham complex of $X$ over $\Bdr$}, denoted as  $\wh{\mathrm{dR}}^\an_{X/\Sigma}$, is defined to be the derived inverse limit of the filtered complexes
		\[
		\wh{\mathrm{dR}}^\an_{X/\Sigma}:=R \varprojlim_{e\geq r} \wh{\mathrm{dR}}^\an_{X/\Sigma_e}.
		\]
	\end{definition}
	Apply Theorem \ref{inf-lev coh}.(i) to the infinitesimal structure sheaf $\calO_{X/\Sigma}$ and the comparison in Theorem \ref{inf-ddR}, 
	we get the following.
	\begin{corollary}\label{inf-lev ddR}
		Let $X$ be a rigid space over $\Sigma_r$.
		There exists a natural filtered map between the analytic derived de Rham complex and the infinitesimal cohomology sheaves, inducing an isomorphism on the underlying complexes
		\[
		\wh{\mathrm{dR}}^\an_{X/\Sigma} \rra Ru_{X/\Sigma *} \calO_{X/\Sigma}.
		\]
		In particular, applying the derived global section on the underlying complexes, we get the following comparison of cohomology
		\[
		R\Gamma(X, \wh{\mathrm{dR}}^\an_{X/\Sigma}) \simeq R\Gamma(X/\Sigma_{\inf}, \calO_{X/\Sigma}).
		\]
	\end{corollary}

	The next result concerns the \'eh descent for cohomology of crystals over the big infinitesimal site $X/\Sigma_{\Inf}$, where $X$ is a rigid space over $K$.
	\begin{proposition}\label{inf-lev des}
		Let $X$ be a rigid space over $K$, and let $\calF$ be a crystal in vector bundles over the big infiniteismal site $X/\Sigma_{\Inf}$.
		Then the cohomology sheaf $Ru_{X/\Sigma*} \calF$ (without the filtration) satisfies the \'eh-hyperdescent.
		Namely for an \'eh-hypercovering $X_\bullet'\ra X'$ of $K$-rigid spaces over $X$, the following natural map is an isomorphism
		\[
		R\Gamma(X'/\Sigma_{\Inf},\calF) \rra R\lim_{[n]\in \Delta} \left( R\Gamma(X'_n/\Sigma_{\Inf},\calF)\right).
		\] 
	\end{proposition}
	\begin{proof}
		By Lemma \ref{inf-lev, tool}.(i), the pullback $i_e^*\calF$ over $X/\Sigma_{e \Inf}$ is a crystal in vector bundles.
		Thanks to Theorem \ref{inf-eh, Bdre}, we know the natural map $X_\bullet' \ra X'$ induces a natural isomorphism as below
		\[
		R\Gamma(X'/\Sigma_{\Inf},i_e^*\calF) \rra R\lim_{[n]\in \Delta} \left( R\Gamma(X'_n/\Sigma_{\Inf},i_e^*\calF)\right).
		\] 
		Thus the result we want follows from taking the derived limit over all $e$, by Theorem \ref{inf-lev coh}.(i).
	\end{proof}
	
	We want to mention that thanks to the Corollary \ref{coh, big small}, it is safe to replace the cohomology of $\calF$ over the big infinitesimal site by the cohomology $R\Gamma(X'/\Sigma_{\inf}, \iota^{-1}\calF)$ of the restriction $\iota^{-1}\calF$ over the small infinitesimal site $X/\Sigma_{\inf}$.
	In particular, by applying the above result to the infinitesimal structure sheaf $\calO_{X/\Sigma}$, we see the infinitesimal cohomology over $\Bdr$ satisfies the \'eh-hyperdescent.
	\begin{corollary}\label{inf-lev des 2}
		Let $X$ be a rigid space over $K$.
		Then the infinitesimal cohomology $R\Gamma(X/\Sigma_{\inf}, \calO_{X/\Sigma})$ satisfies the \'eh-hyperdescent.
	\end{corollary}
	
	Another quick upshot is the finiteness of the infinitesimal cohomology for a proper rigid space $X$.
	\begin{proposition}\label{fin Bdr}
		Let $X$ be a proper rigid space of dimension $n$ over $K$,
		and let $\calF$ be a coherent crystal.
		The infinitesimal cohomology $R\Gamma(X/\Sigma_{\Inf},\calF)$ is then a perfect $\Bdr$-complex in cohomological degrees $[0,2n]$.
	\end{proposition}
	\begin{proof}
		Thanks to Theorem \ref{inf-lev coh}.(i), we can write $R\Gamma(X/\Sigma_{\Inf},\calF)$ as the derived limit of $R\Gamma(X/\Sigma_{e \Inf}, i_e^*\calF)$.
		Here each $R\Gamma(X/\Sigma_{e \Inf},i_e^*\calF)$ is a bounded complex in cohomological degree $[0,2n]$ such that each cohomology group is finite over $\Bdre$ (Proposition \ref{fin Bdre}).
		So the result then follows from the short exact sequence
		\[
		0 \rra R^1\varprojlim_e \rmH^{i-1}(X/\Sigma_{e \Inf}, i_e^*\calF) \rra \rmH^{i}(X/\Sigma_{\Inf}, \calF) \rra \varprojlim_e \rmH^{i}(X/\Sigma_{e \Inf}, i_e^*\calF) \rra 0.
		\]
		Here we note that the inverse system $\{\rmH^{2n}(X/\Sigma_{e \Inf}, i_e^*\calF)\}_e$ satisfies the Mittag-Leffler condition,
		by the finiteness of each $\rmH^{2n}(X/\Sigma_{e \Inf}, i_e^*\calF)$ over $\Bdre$.
	\end{proof}

	\subsection{Comparison with pro-\'etale cohomology}
	In this subsection, we compare the infinitesimal cohomology of $X/\Sigma_{\inf}$ with the pro-\'etale cohomology of the de Rham period sheaf $\BBdr$.
	As an application, we show the degeneracy of the Hodge--de Rham spectral sequence, together with a torsionfreeness of the infinitesimal cohomology $\rmH^i(X/\Sigma,\mathcal{O}_{X/\Sigma_{\inf}})$ over $\Bdr$.
	Throughout the section, we will assume the basics of the pro-\'etale topology defined in \cite{Sch13}.
	
	\noindent\textbf{Comparison theorem.}
	Let $X$ be a rigid space over $K$, and let $X_\pe$ be the pro-\'etale site of $X$.
	The pro-\'etale site admits a basis, which consists of affinoid adic spaces $U=\Spa(B,B^+)$ that are pro-\'etale over $X$ and are  \emph{affinoid perfectoid} (namely, the Huber pair $(B,B^+)$ is a perfectoid algebra over $K$). 
	Over the pro-\'etale site, we can associate the complete structure sheaf $\wh\calO_X$, whose section at an affinoid perfectoid space $U=\Spa(B,B^+)$ is the $K$-algebra $B$.
	Denote $\nu:X_\pe \ra X_{\rig}$ to be the canonical morphism from the pro-\'etale site to the rigid site of $X$.
	
	We recall from \cite{Sch13} that the \emph{de Rham period sheaf $\BBdrp$}, defined as a sheaf of $\Bdr$-algebras over $X_\pe$, sending an affinoid perfectoid space $U=\Spa(B,B^+)$ onto the ring
	\[
	\BBdrp(B,B^+):=\varprojlim_m  \left( W(\varprojlim_{x\mapsto x^p} B^+/p)[\frac{1}{p}]/\xi^m \right).
	\]
	The sheaf $\BBdrp$ admits a canonical surjection $\theta:\BBdrp \ra \wh\calO_X$ that is compatible with the surjection map $\theta:\Bdr \ra K$ for the period ring $\Bdr$.
	It can be shown that $\xi$ is a nonzero-divisor in $\BBdrp$ locally, and the ideal $\ker(\theta)\subset \BBdrp$ is generated by $\xi \in \Bdr$.
	So we could invert the element $\xi$ to get a sheaf of $\Bdrr=\Bdr[\frac{1}{\xi}]$-algebras over $X_\pe$, which we denote by $\BBdr$.
	The sheaf of rings $\BBdr$ then admits a natural descending filtration where the $i$-th filtration for $\forall i\in \mathbb{Z}$ is defined by $\Fil^i\BBdr:=\xi^i\BBdrp \subset \BBdr$.
	Each graded piece $\gr^i\BBdr$, which is locally equal to $\wh\calO_X\cdot \xi^i$, is canonically isomorphic to the pro-\'etale structure sheaf up to a twist.
	
	We first recall the comparison between the infinitesimal cohomology and the pro-\'etale cohomology of $\BBdr$ for smooth rigid spaces.
	\begin{theorem}[\cite{BMS}, Theorem 13.1]\label{pe-inf, sm}
		Let $X$ be a smooth rigid space over $K$.
		Then there exists a natural map of complexes of sheaves of $\Bdr$-modules over $X$
		\[
		Ru_{X/\Sigma *} \calO_{X/\Sigma} \rra R\nu_*\BBdrp.
		\]
		It is an isomorphism after inverting $\xi$.
	\end{theorem}
	\begin{proof}
		This is essentially proved in the \cite{BMS}, Theorem 13.1, and we explain here the relation of their result with our statement.
		
		Let $X$ be a smooth rigid space over $K$ of dimension $d$.
		Assume $U=\Spa(R)$ is a \emph{very small} affinoid open subset in $X$; namely it admits an \'etale morphism onto a torus $\TT^d_K$, where the map can be extended to a closed immersion into a larger torus $\TT^n=\Spa(K \langle  T_i^{\pm1}\rangle)$.
		For any such closed immersion, we could associate the torus $\TT^n$ an affinoid perfectoid space $\TT^{n, \infty}=\Spa( K\langle  T_i^{\pm \frac{1}{p^{\infty}}} \rangle)$.
		The canonical map $\TT^{n,\infty} \ra \TT^n$ is pro-\'etale, and its pullback along $U\ra \TT^n$ produces a pro-\'etale morphism from an affinoid perfectoid space $\Spa(R_\infty, R^+_\infty)$ over $U=\Spa(R)$.
		
		We denote by $D$  the envelope of $U$ inside of the direct system $\{\TT^n_{\Bdre}\}_e$ of tori over $\{\Bdre\}_e$. 
		Then for any such choice of morphisms $(U \ra \TT^d \ra \TT^n)$, we could construct two $\Bdr$-linear complexes
		\begin{itemize}
			\item The de Rham complex $\Omega_D^\bullet$ of $U$ in $\{\TT^n_{\Bdre}\}_e$, that computes the infinitesimal cohomology $R\Gamma(U/\Sigma_{\inf}, \calO_{X/\Sigma})$ by Theorem \ref{inf-lev coh}.
			\item The Koszul complex $K_{\BBdrp(R_\infty)}=K_{\BBdrp(R_\infty)}(\gamma_{u_i}-1)$, that computes the pro-\'etale cohomology $R\Gamma(U_\pe, \BBdrp)$.
		\end{itemize}
		As in the proof of the \cite[Theorem 13.1]{BMS}, for any choice of $(U \ra \TT^d \ra \TT^n)$, there exists a natural map of actual complexes
		\[
		\Omega_D^\bullet \rra K_{\BBdrp(R_\infty)},
		\]
		which is functorial with respect to the choices of triples, such that it becomes an isomorphism after inverting $\xi$.
		Notice that the set of triples for a fixed $U$ is filtered, and the transition map of both complexes are isomorphisms.
		In this way, the induced isomorphism 
		\[
		R\Gamma(U/\Sigma_{\inf}, \calO_{X/\Sigma})[\frac{1}{\xi}] \ra R\Gamma(U_\pe, \BBdr)
		\] is independent of the triples $(U\ra \TT^d \ra \TT^n)$.
		Since the collection of very small open subsets of $X$ form a basis in the rigid topology, we could then get a natural isomorphism as below
		\[
		Ru_{X/\Sigma *} \calO_{X/\Sigma}[\frac{1}{\xi}] \rra R\nu_*\BBdr.
		\]
	\end{proof}
	
	Using the \'eh-hyperdescent, we could improve the above result into the non-smooth situations.
	\begin{theorem}\label{inf-proet}
		Let $X$ be a rigid space over $K$.
		Then there exists a natural map of complexes of sheaves of $\Bdr$-modules over $X$ as below
		\[
		Ru_{X/\Sigma *} \calO_{X/\Sigma} \rra R\nu_*\BBdrp.
		\]
		It is an isomorphism after inverting $\xi$.
	\end{theorem}
	\begin{proof}
		In the proof, we use $\nu_{X}:X_\pe \ra X_\rig$ to denote the natural map of sites associated with the rigid space $X$.
		
		We first notice that the pro-\'etale cohomology sheaf $R\nu_{X *}\BBdrp$ and $R\nu_{X *}\BBdr$ satisfy the \'eh-hyperdescent.
		To see this, we recall from \cite[Section 4]{Guo19} that the derived push-forward $R\nu_{X *}\wh\calO_X$ is naturally isomorphic to $R\pi_{X *} C$, where $C=R\alpha_*\wh\calO_v\in D^{\geq 0}(X_\eh)$ is the derived push-forward of the completed $v$-structure sheaf (see \cite[Section 3.2]{Guo19}), and $\pi_X:X_\eh \ra X_\rig$ is the natural map of sites.
		As an upshot, since $C$ is a bounded below complex of \'eh-sheaves, its direct image $R\pi_{X *} C$ in the rigid site naturally satisfies the \'eh-hyperdescent; namely for an \'eh-hypercovering $\rho:X_\bullet\ra X$ over $K$, the induced map below is an isomorphism
		\[
		R\pi_{X *}C \rra R\rho_* R\pi_{X_\bullet *} C.
		\]
		We could then replace the above by the derived push-forward of the pro-\'etale structure sheaf to get a natural isomorphism 
		\[
		R\nu_{X *}\wh\calO_X \rra R\rho_* R\nu_{X_\bullet *} \wh\calO_{X_\bullet}.
		\]
		On the other hand, notice the de Rham period sheaf $\BBdrp$ is completed under the $\xi$-adic topology such that the $i$-th graded piece is equal to the complete structure sheaf $\wh\calO_X\cdot \xi^i$ up to a twist.
		In this way, by the hyperdescent for graded pieces and the induction on $e$, we get
		\begin{align*}
			R\nu_{X *} \BBdrp & = R\varprojlim_{e\in \NN} R\nu_{X *} \BBdrp/\xi^e \\
			& \simeq R\varprojlim_{e\in \NN} R\rho_* R\nu_{X_\bullet *} \BBdrp/\xi^e \\
			& \simeq R\rho_* R\varprojlim_{e\in \NN} R\nu_{X_\bullet *} \BBdrp/\xi^e \\
			& = R\rho_* R\nu_{X_\bullet *} \BBdrp.
		\end{align*}
		Namely the pro-\'etale cohomology of $\BBdrp$ hence $\BBdr=\BBdrp[\frac{1}{\xi}]$ satisfies the \'eh-hyperdescent.
		
		At last, notice that the collection of maps $f:X'\ra X$ for smooth rigid spaces $X'$ form a basis of the \'eh-site $X_\eh$.
		In this way, the natural comparison map $Ru_{X'/\Sigma *} \calO_{X'/\Sigma} \ra R\nu_{X'*}\BBdrp$ for smooth $X'$ extends to a map for $X$ via the \'eh-hyperdescent (for the infinitesimal cohomology sheaf, this is Theorem \ref{inf-lev des}), and by inverting $\xi$ we get the isomorphism 
		\[
		Ru_{X/\Sigma *} \calO_{X/\Sigma}[\frac{1}{\xi}] \rra R\nu_*\BBdr.
		\]

	\end{proof}
	
	\begin{remark}
		The morphism between the infinitesimal cohomology and the pro-\'etale cohomology is constructed in an indirect way.
		In fact, by enlarging the infinitesimal site $X/\Sigma_{\inf}$ to a bigger site that allows all (adic spectra of) complete Huber rings as in \cite[Construction 5.11]{Yao19}, 
		the de Rham period ring $\BBdrp(R_\infty)$ for a perfectoid algebra $R_\infty$ can then be regarded as a pro-thickening in this enlarged category.
		In this way, the arrow from the associated ind-object to the final object in the enlarged infinitesimal topos will induce a map on their cohomology, and it can be checked via computations in smooth case and the \'eh-hyperdescent that this coincides with our morphism.
	\end{remark}
	
	Below we consider a special case when $X$ comes from a small subfield.
	Precisely, let $K_0$ be a discretely valued subfield of $K$ such that the residue field of $K_0$ is perfect.
	Assume $Y$ is a proper rigid space over $K_0$, and $X=Y\times_{K_0} K$ is the base field extension of $X_0$.
	We recall from \cite[Theorem 8.2.2]{Guo19} that there exists a $\Gal(K/K_0)$-equivariant filtered comparison between the pro-\'etale cohomology $R\Gamma(X_\pe, \BBdr)$ and the tensor product
	\[
	R\Gamma(Y_{\eh}, \Omega_{\eh,/K_0}^\bullet)\otimes_{K_0} \mathrm{B}_{\mathrm{dR}}.
	\]
	Here $\Omega_{\eh,/K_0}^i$ is the \'eh-differential for rigid spaces over $K_0$, and the filtration is defined by the product filtration, where the \'eh de Rham cohomology $R\Gamma(Y_{\eh},\Omega_{\eh,/K_0}^\bullet)$ is equipped with a natural descending filtration by $\Fil^i=R\Gamma(Y_{\eh},\Omega_{\eh,/K_0}^{\geq i})$.
	Moreover, by taking the zero-th graded pieces, we get
	\[
	R\Gamma(X_\pe, \wh\calO_X) \simeq \moplus_i R\Gamma(Y_{\eh}, \Omega_{\eh,/K_0}^i) \otimes_{K_0} K(-i).
	\]
	From this, we get the following:
	\begin{corollary}\label{over K0}
		Let $Y$ be a proper rigid space over the discretely valued subfield $K_0$ of $K$ as above, and let $X$ be its base extension to $K$.
		Then we have a canonical isomorphism
		\[
		R\Gamma(X/\Sigma_{\inf},\calO_{X/\Sigma})[\frac{1}{\xi}] \simeq R\Gamma(Y_{\eh}, \Omega_{\eh,/K_0}^\bullet)\otimes_{K_0} \mathrm{B}_{\mathrm{dR}}.
		\]
		In particular, the infinitesimal cohomology of $Y\times_{K_0} K$ over $\mathrm{B}_{\mathrm{dR}}$ admits a $\Gal(K/K_0)$-equivariant filtration such that the zero-th graded factor is equal to
		\[
		\moplus_i R\Gamma(Y_{\eh}, \Omega_{\eh,/K_0}^i) \otimes_{K_0} K(-i).
		\]
	\end{corollary}
	
	\noindent\textbf{Torsion freeness and Hodge--de Rham degeneracy.}
	For the rest of the subsection, we prove the infinitesimal cohomology $\rmH^n(X/\Sigma_{\inf},\calO_{X/\Sigma})$ is torsionfree over $\Bdr$ (namely \Cref{main2}.(vi)), and show the degeneracy for the \'eh Hodge--de Rham  spectral sequence.

	\begin{theorem}\label{torsionfree and HdR}
		Let $X$ be a proper rigid space over $K$.
		Then we have the following
		\begin{enumerate}[(i)]
			\item the infinitesimal cohomology $\rmH^n(X/\Sigma_{\inf},\calO_{X/\Sigma})$ is torsionfree over $\Bdr$, for each $n\in \NN$.
			\item the \'eh Hodge--de Rham spectral sequence over $K$ below degenerates at its first page:
			\[
			E_1^{i,j}=\rmH^j(X_\eh,\Omega_{\eh,X/K}^i) \Longrightarrow \rmH^{i+j}(X_\eh,\Omega_{\eh,X/K}^{\bullet}).
			\]
		\end{enumerate}
	\end{theorem}
	\begin{remark}
		Note that the part (ii) generalizes the degeneracy result in \cite[Proposition 8.0.8]{Guo19}, where the latter needs the assumption for $X$ to be defined over a discretely valued subfield.
	\end{remark}
	\begin{proof}
		Let $n$ be any integer.
		We first note that the pro-\'etale cohomology $\rmH^n(X_\pe, \BBdrp)$ is finite free over $\Bdr$:
		To see this, we recall the Primitive Comparison Theorem over $\Bdr$ as below (\cite[Thm.\ 3.17]{Sch13b})
		\[
		\rmH^n(X_\et, \mathbb{Q}_p)\otimes_{\mathbb{Q}_p} \Bdr \simeq \rmH^n(X_\pe, \BBdrp).
		\]
		As the \'etale cohomology $\rmH^n(X_\eh ,\mathbb{Q}_p)$ is a finite dimensional vector space over $\mathbb{Q}_p$, the above implies the finite freeness of $\rmH^n(X_\pe, \BBdrp)$ over $\Bdr$.
		In particular, by Theorem \ref{inf-proet} and the finiteness in Proposition \ref{fin Bdr}, we get the following relations
		\begin{align*}
			\dim_{\mathbb{Q}_p} \rmH^n(X_\et, \mathbb{Q}_p) &=\rank_\Bdr \rmH^n(X_\pe, \BBdrp) \\
			&= \dim_{\mathrm{B}_\mathrm{dR}} \rmH^n(X_\pe,\BBdr) \\
			&=\dim_{\mathrm{B}_\mathrm{dR}} \rmH^n(X/\Sigma_{\inf}, \calO_{X/\Sigma})[\frac{1}{\xi}] \\
			&=\rank_\Bdr \rmH^n(X/\Sigma_{\inf}, \calO_{X/\Sigma})/torsion\\
			&\leq \dim_K \rmH^n(X/\Sigma_{\inf}, \calO_{X/\Sigma})/\xi,
		\end{align*}
		where the last equality follows from the structure theorem of finite generated modules over the principal ideal domain $\Bdr$.
		
		On the other hand, by the base change formula in Theorem \ref{inf-lev coh}.(iii), for each $n\in \mathbb{Z}$, we get the following short exact sequence of $K$-vector spaces
		\[
		0\rra \rmH^n(X/\Sigma_{\inf}, \calO_{X/\Sigma})/\xi  \rra \rmH^n(X/K_{\inf}, \calO_{X/K}) \rra  \rmH^{n+1}(X/\Sigma_{\inf}, \calO_{X/\Sigma})[\xi] \rra 0,
		\]
		which implies the inequalities
		\[
		\dim_K \rmH^n(X/\Sigma_{\inf}, \calO_{X/\Sigma})/\xi \leq \dim_K \rmH^n(X/K_{\inf}, \calO_{X/K}),~\forall n\in \mathbb{Z}.
		\]
		In addition, notice that the cohomology complex of a sheaf of abelian groups always lives in the non-negative cohomological degrees.
		So both $ \rmH^n(X/\Sigma_{\inf}, \calO_{X/\Sigma})$ and $\rmH^n(X/K_{\inf}, \calO_{X/K})$ vanish for $n\leq -1$, and the short exact sequence above for $n=-1$ implies the vanishing of $\mathrm{H}^0(X/\Sigma_{\inf},\mathcal{O}_{X/\Sigma})[\xi]$.
		In particular, we see the $\Bdr$-module $\mathrm{H}^0(X/\Sigma_{\inf},\mathcal{O}_{X/\Sigma})$ is torsionfree.
		
		Now using the comparison between the infinitesimal cohomology and the \'eh de Rham cohomology for the trivial crystal $\calF=\calO_{X/K}$ in Theorem \ref{inf-eh}, we have
		\[
		\dim_K \rmH^n(X/K_{\inf}, \calO_{X/K}) = \dim_K \rmH^n(X_\eh,\Omega_\eh^\bullet).
		\]
		Note that the natural Hodge filtration on the \'eh de Rham complex $\Omega_\eh^\bullet$ induces the $E_1$ spectral sequence
		\[
		E_1^{i,j}=\rmH^j(X_\eh,\Omega_\eh^i) \Longrightarrow \rmH^{i+j}(X_\eh,\Omega_\eh^{\bullet}).
		\]
		As a consequence, we get
		\[
		\dim_K \rmH^n(X/K_{\inf}, \calO_{X/K}) \leq \sum_{i+j=n} \dim_K \rmH^j(X_\eh,\Omega_\eh^i).
		\]
		However, by Hodge--Tate decomposition in \cite[Theorem 1.1.3]{Guo19}, we have
		\[
		\dim_{\mathbb{Q}_p} \rmH^n(X_\et, \mathbb{Q}_p) = \sum_{i+j=n} \dim_K \rmH^j(X_\eh,\Omega_\eh^i).
		\]
		Hence combining all the relations of dimensions above, we see all the inequalities should be equalities.
		The latter implies that the $E_1$ spectral sequence degenerates at the first page, and for any $n\in \NN$ we have
		\begin{align*}
			&\rmH^{n+1}(X/\Sigma_{\inf}, \calO_{X/\Sigma})[\xi]=0.\\
			&\rmH^n(X/\Sigma_{\inf}, \mathcal{O}_{X/\Sigma})/\xi= \rmH^n(X/\Sigma_{\inf}, \mathcal{O}_{X/K}).
		\end{align*}
		These, together with the torsionfreeness of the $\Bdr$-module $\rmH^0(X/\Sigma_{\inf}, \mathcal{O}_{X/\Sigma})$, conclude the proof.
	\end{proof}

	

\begin{thebibliography}{99}
		
		\bibitem[Ant20]{Ant20}{J.~Ant\'onio. Spreading out the Hodge filtration in the non-archimedean geometry. Preprint, arXiv: 2005.00774.}
		\bibitem[AK11]{AK11}{D.~Arapura;~S-J.~Kang. K\"ahler-de Rham cohomology and Chern classes. Comm. Algebra, 39(4): 1153–1167, 2011.}
		\bibitem[BO78]{BO78}{P.~Berthelot; A.~ Ogus. Notes on crystalline cohomology. Princeton University Press, Princeton, N.J.; University of Tokyo Press, Tokyo, 1978.}
		\bibitem[Ber74]{Ber74}{P.~Berthelot. Cohomologie cristalline des sch\'emas de caract\'eristique $p>0$. (French) Lecture Notes in Mathematics, Vol. 407. Springer-Verlag, Berlin-New York, 1974. 604 pp.}
		
		\bibitem[BdJ]{BdJ}{B.~Bhatt. A.~J.~de Jong. Crystalline cohomology and de Rham cohomology. Preprint, arXiv: 1110.5001.}
		\bibitem[BMS]{BMS}{B.~Bhatt. M.~Morrow.~P.~Scholze. Integral p-adic Hodge theory. Publ. Math. Inst. Hautes \'Etudes Sci. 128 (2018), 219–397.}
		\bibitem[BMS2]{BMS2}{B.~Bhatt. M.~Morrow.~P.~Scholze. Topological Hochschild homology and integral p-adic Hodge theory. Publ. Math. Inst. Hautes \'Etudes Sci. 129 (2019), 199–310.}
		\bibitem[BS]{BS}{B.~Bhatt.~P.~Scholze. Prisms and prismatic cohomology.	Ann. of Math. (2) , Vol. 196 (2022), No. 3 
			p. 1135-1275.}
		\bibitem[Bha12]{Bha12}{B.~Bhatt. Completion and derived de Rham complex. Preprint, arXiv: 1207.6193.}
		\bibitem[BM08]{BM08}{E.~Bierstone; P.~Milman. Functoriality in resolution of singularities. Publ. Res. Inst. Math. Sci. 44 (2008), no. 2, 609–639.}
		\bibitem[BoL\"u]{BoLu}{S.~Bosch. W.~L\"utkebohmert. Formal and rigid geometry. I. Rigid spaces. Math. Ann. 295 (1993), no. 2, 291–317.}
		
		\bibitem[Bri08]{Bri08}{O. Brinon. Repr\'esentations p-adiques cristallines et de de Rham dans le cas relatif. M\'em. Soc. Math. Fr. (N.S.), (112):vi+159, 2008.}
		\bibitem[Con06]{Con06}{B.~Conrad. Relative ampleness in rigid geometry, Ann. Inst. Fourier (Grenoble) 56 (2006), no. 4, 1049–1126 (English, with English and French summaries).}
		\bibitem[dJvdP]{dJvdP}{J.~de Jong;~M.~van der Put. \'Etale cohomology of rigid analytic spaces. Doc. Math. 1 (1996), No. 01, 1–56.}
		\bibitem[DB81]{DB81}{P.~Du Bois. Complexe de de Rham filtr\'e d’une vari\'et\'e singuli\`ere. Bulletin de la S. M. F., tome 109 (1981), p. 41-81.}
		\bibitem[EGA IV]{EGA IV}{J.~Dieudonn\'e.~A.~Grothendieck. El\'ements de G\'eom\'etrie Alg\'ebrique – Chapitre IV,	partie 1. Publ. Math. IHES 20 (1964).}
		\bibitem[Fon94]{Fon94}{J-M.~Fontaine. Le corps des p\'eriodesp-adiques. No. 223, 1994, With an appendix by Pierre Colmez, P\'eriodesp-adiques (Bures-sur-Yvette, 1988), pp. 59–111. MR 1293971.}
		
		\bibitem[GR]{GR}{O.~Gabber. L.~Ramero. Almost ring theory. Lecture Notes in Mathematics, 1800. Springer-Verlag, Berlin, 2003. vi+307 pp}
		\bibitem[Gr66]{Gr66}{A.~Grothendieck. On de Rham cohomology of algebraic varieties. Inst. Hautes \'Etudes Sci. Publ. Math. No. 29 (1966), 95–103.}
		\bibitem[Gr68]{Gr68}{A.~Grothendieck. Crystals and de Rham cohomology of schemes. Notes by I. Coates and O. Jussila. Adv. Stud. Pure Math., 3, Dix expos\'es sur la cohomologie des sch\'emas, 306–358, North-Holland, Amsterdam, 1968.}
		\bibitem[Guo19]{Guo19}{H.~Guo. Hodge--Tate decomposition for non-smooth spaces. J. Eur. Math. Soc. (JEMS) 25 (2023), no. 4, 1553–1625.}
		\bibitem[Guo21]{Guo21}{H.~Guo. Prismatic cohomology of rigid analytic spaces over de Rham period ring. Preprint, arXiv: 2112.14746.}
		\bibitem[GL20]{GL20}{H.~Guo; S.~Li. Period sheaves via derived de Rham cohomology.  Compositio Mathematica, Volume 157, Issue 11, November 2021, pp. 2377 - 2406.}
		
		\bibitem[Har75]{Har75}{R.~Hartshorne. On the De Rham cohomology of algebraic varieties. Inst. Hautes Études Sci. Publ. Math. No. 45 (1975), 5–99.}
		\bibitem[HL71]{HL71}{M.~Herrera; D.~Lieberman. Duality and de Rham cohomology of infinitesimal neighborhoods. Invent. Math. 13 (1971), 97–124.}
		
		\bibitem[HJ14]{HJ14}{A.~Huber. C.~J\"order. Differential forms in the h-topology. Algebr. Geom. 1 (2014), no. 4, 449–478.}
		
		\bibitem[Hu96]{Hu96}{R.~Huber. \'Etale cohomology of rigid analytic varieties and adic spaces. Aspects of Mathematics, E30. Friedr. Vieweg $\&$ Sohn, Braunschweig, 1996. }
		
		\bibitem[Ill71]{Ill71}{L.~Illusie. Complexe cotangent et d\'eformations. I. (French) Lecture Notes in Mathematics, Vol. 239. Springer-Verlag, Berlin-New York, 1971. xv+355 pp.}
		\bibitem[Ill72]{Ill72}{L.~Illusie. Complexe cotangent et d\'eformations. II. (French) Lecture Notes in Mathematics, Vol. 283. Springer-Verlag, Berlin-New York, 1972. vii+304 pp.}
		
		\bibitem[Kie67]{Kie67}{R.~Kiehl. Der Endlichkeitssatz f\"ur eigentliche Abbildungen in der nichtarchimedischen Funktionentheorie. (German) Invent. Math. 2 (1967), 191–214.}
		
		\bibitem[Lee09]{Lee09}{B.~Lee. Local acyclic fibrations and the de Rham complex. Homology Homotopy Appl. 11(1): 115-140 (2009).}
		\bibitem[Lu09]{Lu09}{J.~Lurie. Higher topos theory. Annals of Mathematics Studies, 170. Princeton University Press, Princeton, NJ, 2009.}
		\bibitem[Lu17]{Lu17}{J.~Lurie. Higher algebra. September 2017. \url{https://www.math.ias.edu/~lurie/papers/HA.pdf}.}
		\bibitem[Lu18]{Lu18}{J.~Lurie. Spectral algebraic geometry. 2018. \url{https://www.math.ias.edu/~lurie/papers/SAG-rootfile.pdf}.}
		\bibitem[Mat86]{Mat86}{H.~Matsumura. Commutative ring theory. Cambridge Univ. Press (1986).}
		
		\bibitem[Rak20]{Rak20}{A.~Raksit. Hochschild homology and the derived de Rham complex revisited. Preprint, arXiv: 2007.02576.}
		
		\bibitem[Sch13b]{Sch13b}{P.~Scholze. Perfectoid spaces: a survey. Current developments in mathematics (2012), 193–227, Int. Press, Somerville, MA, 2013.}
		\bibitem[Sch13]{Sch13}{P.~Scholze.  $p$-adic Hodge theory for rigid-analytic varieties. Forum Math. Pi 1 (2013), e1, 77.}
		\bibitem[Simp]{Simp}{C.~Simpson. Nonabelian Hodge theory. Proceedings of the International Congress of Mathematicians, Vol. I, II (Kyoto, 1990), 747–756, Math. Soc. Japan, Tokyo, 1991.}
		\bibitem[Sta]{Sta}{The Stacks Project Authors. Stacks Project, \url{http://stacks.math.columbia.edu.}}
		\bibitem[Tem12]{Tem12}{M.~Temkin. Functorial desingularization of quasi-excellent schemes in characteristic zero: the nonembedded case. Duke Math. J. 161 (2012), no. 11, 2207–2254.}
		\bibitem[Tem18]{Tem18}{M.~Temkin. Functorial desingularization over $\Q$: boundaries and the embedded case. Israel J. Math. 224 (2018), 455–504.}
		\bibitem[Voe96]{Voe96}{V.~Voevodsky. Homology of schemes. Selecta Math. (N.S.) 2 (1996), no. 1, 111–153.}
		\bibitem[Yao19]{Yao19}{Z.~Yao. The crystalline comparison of Ainf-cohomology: the case of good reduction. Preprint, arXiv: 1908.06366.}
		
	\end{thebibliography}
\end{document}